\documentclass[11pt,reqno]{article} 
\usepackage[latin1]{inputenc}   
\usepackage{amsfonts,amsmath,layout}
\usepackage{mathrsfs}  
\usepackage{amssymb,mathabx,amsfonts,amsthm,amscd,stmaryrd,dsfont,esint,upgreek,constants,todonotes}

\usepackage{color} 

\definecolor{Red}{cmyk}{0,1,1,0.2}
\usepackage{hyperref}

\newcommand{\N}{\mathbb N}

\newcommand{\Z}{\mathbb Z}

\newcommand{\R}{\mathbb R}

\newcommand{\T}{\mathbb{T}}

\newcommand{\essup}{\textrm{\rm essup}}

\def\cC{\mathcal C}
\def\R{\mathbb R}
\def\N{\mathbb N}
\def\Z{\mathbb Z}
\def\E{\mathbb E}
\def\P{\mathbb P}

\def\ep{\epsilon}
\def\rg{\rangle} 
\def\lg{\langle} 

\def\dk{{\bf d}_1}

\newcommand{\be}{\begin{equation}}
\newcommand{\ee}{\end{equation}}
\def\1{{\bf 1}}
\def\inte{\int_{\T^d}}
\def\dive{{\rm div}}

\def\Pw{{\mathcal P}(\T^d)}
\def\Pk{{\mathcal P}(\T^d)}
\def\bx{{\boldsymbol x}}

\def\bX{{\boldsymbol X}}
\def\bY{{\boldsymbol Y}}
\def\bZ{{\boldsymbol Z}}

\def\ds{\displaystyle}

\newtheorem{Theorem}{Theorem}[section]
\newtheorem{Definition}[Theorem]{Definition}
\newtheorem{Proposition}[Theorem]{Proposition}
\newtheorem{Lemma}[Theorem]{Lemma}
\newtheorem{Corollary}[Theorem]{Corollary}
\newtheorem{Remark}[Theorem]{Remark}
\newtheorem{Example}[Theorem]{Example}
\begin{document}


\title{The master equation and the convergence problem in mean field games}
\author{P. Cardaliaguet\thanks{Ceremade, Universit\'e Paris-Dauphine, cardaliaguet@ceremade.dauphine.fr} \and F. Delarue\thanks{Laboratoire Jean-Alexandre Dieudonné, Université de Nice Sophia-Antipolis. delarue@unice.fr} 
 \and J.-M. Lasry\thanks{ 56 rue d'Assas 75006} \and P.-L. Lions\footnotemark[1]  \thanks{College de France, 11 Place Marcelin Berthelot,
Paris 75005}}

\maketitle

\begin{abstract} The paper studies the convergence,  as $N$ tends to infinity, of a system of $N$ coupled Hamilton-Jacobi equations, the Nash system. This system arises in differential game theory. We describe the limit problem in terms of the so-called ``master equation", a kind of second order partial differential equation stated on the space of probability measures. Our first main result is the well-posedness of the master equation. To do so, we first show the existence and uniqueness of a solution to the 
``mean field game system with common noise", which consists in a coupled system made of
a backward stochastic Hamilton-Jacobi equation and a forward stochastic Kolmogorov equation
and which plays the role of characteristics for the master equation. Our second main result is the convergence, in average, of the solution of the Nash system and a propagation of chaos property for the associated ``optimal trajectories". 
\end{abstract}

\setcounter{tocdepth}{1}
\tableofcontents

\section{Introduction}

\subsection{Motivation and summary of the results}

{\textit{\textbf{Statement of the problem.}}} The purpose of this paper is to discuss the behavior, as $N$ tends to infinity, of the strongly coupled system of $N$ parabolic equations 
\be\label{NashIntro}
\left\{ \begin{array}{l}
\ds - \partial_t v^{N,i} (t,\bx)-  \sum_{j=1}^N \Delta_{x_j}v^{N,i}(t,\bx) - \beta  \sum_{j,k=1}^N {\rm Tr} D^2_{x_j,x_k} v^{N,i}(t,\bx) + H(x_i,  D_{x_i}v^{N,i}(t,\bx)) \\
\ds \qquad \qquad + \sum_{j\neq i}  D_pH (x_j, D_{x_j}v^{N,j}(t,\bx))\cdot D_{x_j}v^{N,i}(t,\bx)=
F^{N,i}(\bx) \\
\qquad \qquad \qquad \qquad \qquad \qquad \qquad \qquad {\rm in}\; [0,T]\times (\R^{d})^N,\\
\ds v^{N,i}(T,\bx)= G^{N,i}(\bx)
\qquad {\rm in }\;  (\R^{d})^N.
\end{array}\right.
\ee
The above system is stated in $[0,T]\times (\R^d)^N$, where a typical element is denoted by $(t,\bx)$ with $\bx= (x_1,\dots, x_N)\in (\R^d)^N$. 
The unknowns are the $N$ maps $(v^{N,i})_{i\in \{1, \dots, N\}}$. The data are the Hamiltonian $H:\R^d\times \R^d\to \R$, the maps $F^{N,i},G^{N,i}:(\R^d)^N\to \R$, the nonnegative parameter $\beta$ and the horizon $T\geq 0$. 

System \eqref{NashIntro} describes the Nash equilibria of an $N-$player differential game (see Section \ref{subsec.informal} for a short description). In this game, the set of ``optimal trajectories" solves a system of $N$ coupled stochastic differential equations (SDE): 
\be\label{e.optitraj}
dX_{i,t}= -D_pH\bigl(X_{i,t}, Dv^{N,i}(t, \bX_t) \bigr)dt +\sqrt{2} dB^i_t +\sqrt{2\beta} dW_t, \qquad t\in [0,T],\; i\in \{1, \dots, N\}, 
\ee
where $v^{N,i}$ is the solution to \eqref{NashIntro} and the $((B^i_t)_{t\in [0,T]})_{i=1,\dots,N}$ and $(W_t)_{t\in [0,T]}$ are $d-$dimensional independent Brownian motions. The Brownian motions $((B^i_t)_{t \in [0,T]})_{i=1,\dots,N}$ correspond to the {\it individual noises}, while the Brownian motion $(W_t)_{t \in [0,T]}$ is the same for all the equations and, for this reason, is called the {\it common noise}. 
Under such a probabilistic point of view, 
the collection of random process $((X_{i,t})_{t \in [0,T]})_{i=1,\dots,N}$
forms a dynamical system of interacting particles.
Another, but closely related, objective of our paper is to study the mean-field limit of the $((X_{i,t})_{t \in [0,T]})_{i=1,\dots,N}$ as $N$ tends to infinity.
\vskip 4pt

As explained below, the motivation for investigating 
\eqref{NashIntro}
and 
\eqref{e.optitraj}
asymptotically is to justify the passage to the limit in the theory of mean-field games. 
\bigskip

\noindent 
\textit{\textbf{Link with the mean-field theory.}}
Of course, there is no chance to observe a mean-field 
limit for
\eqref{e.optitraj} under a general choice of the coefficients in
\eqref{NashIntro}. Asking for a mean-field limit certainly requires that 
the system has a specific symmetric structure in such a way that 
the players in the differential game are somewhat exchangeable
(when in equilibrium). For this purpose, 
we suppose that, for each $i\in \{1,\dots, N\}$, the maps 
$(\R^d)^N \ni \bx \mapsto F^{N,i}(\bx)$ and $(\R^d)^N \ni \bx \mapsto G^{N,i}(\bx)$ depend only on $x_i$ and on the empirical distribution of the variables $(x_j)_{j\neq i}$: 
\begin{equation}
\label{eq:intro:running:terminal:costs}
F^{N,i}(\bx)= F(x_i, m^{N,i}_{\bx}) \qquad 
{\rm and}
\qquad 
G^{N,i}(\bx)= G(x_i, m^{N,i}_{\bx}),
\end{equation}
where $m^{N,i}_{\bx}=\frac{1}{N-1} \sum_{j\neq i} \delta_{x_j}$ is the empirical distribution of the $(x_j)_{j\neq i}$ and where $F,G:\R^d\times {\mathcal P}(\R^d)\to \R$ are given functions, ${\mathcal P}(\R^d)$ being the set of Borel measures on $\R^d$. Under this assumption, the solution of the Nash system indeed enjoys strong symmetry properties, 
which imply in particular the required exchangeability property. Namely,
$v^{N,i}$ can be written into 
a similar form to 
\eqref{eq:intro:running:terminal:costs}: 
\begin{equation}
\label{eq:intro:form:vNi}
v^{N,i}(t,{\boldsymbol x})=v^{N}(t,x_i, m^{N,i}_{\boldsymbol x}),
\quad t \in [0,T], \quad {\boldsymbol x} \in (\R^d)^N,
\end{equation}
for a function $v^N(t,\cdot,\cdot)$ taking 
as arguments a state in $\R^d$ and an empirical distribution 
of size $N-1$ over $\R^d$.
\color{black}
\vspace{4pt}

Anyhow, 
even under the above symmetry assumptions, 
it is by no means clear whether 
the system \eqref{e.optitraj} can exhibit a mean-field limit. 
The reason is that the dynamics of the particles
$(X_{1,t},\dots,X_{N,t})_{t \in [0,T]}$
are coupled through the
unknown 
solutions $v^{N,1},\dots,v^{N,N}$ to the Nash system
\eqref{NashIntro}, whose symmetry properties
\eqref{eq:intro:form:vNi}
 may not suffice 
to apply standard results from the theory of propagation of chaos. 
Obviously, the 
difficulty
 is that the function $v^{N}$
 in the right-hand side of \eqref{eq:intro:form:vNi}
 precisely depends upon $N$. 
Part of the challenge in the paper is thus to show that 
the interaction terms in
\eqref{e.optitraj} get closer and closer, as $N$ tends to the infinity, 
to some interaction terms
with a
much more tractable
and much more explicit shape.

In order to get a picture of 
the ideal case under which the mean-field limit can be taken, 
one can
choose for a while $\beta = 0$ 
in \eqref{e.optitraj}
and 
then 
assume that the function $v^{N}$
in the right-hand side of 
\eqref{eq:intro:form:vNi}
is independent of $N$.
Equivalently, 
one can replace 
in \eqref{e.optitraj}
the interaction function 
$(\R^d)^N \ni \bx \mapsto D_{p}H(x_{i},v^{N,i}(t,\bx))$
by 
$(\R^d)^N \ni \bx \mapsto b(x_{i},m_{\bx}^{N,i})$, 
for a map $b: \R^d \times {\mathcal P}(\R^d) \mapsto \R^d$. 
In such a case,  
the coupled system of SDEs
\eqref{e.optitraj} turns into:
\be\label{e.SDEsIntro}
dX_{i,t}= b\Bigl(X_{i,t},
\frac{1}{N-1}\sum_{j \not = i}
\delta_{X_{j,t}} \Bigr)dt + \sqrt{2}dB^i_t, \qquad t\in [0,T],\; i\in \{1, \dots, N\},
\ee
the second argument in $b$ being nothing but the empirical measure of the 
particle system at time $t$. Under suitable assumptions on $b$ (e.g., if $b$ is bounded and Lipschitz continuous in both variables, the space of probability measures being equipped 
with the Wasserstein distance) and on the initial distribution of the 
$((X_{i,t})_{i=1,\dots,N})_{t \in [0,T]}$, 
both the marginal law of $(X^1_t)_{t \in [0,T]}$ (or of any other player) 
and the empirical distribution of the whole system 
converge to the solution of the McKean-Vlasov equation 
$$
\partial_t m-\Delta m +\dive\bigl(m \, b(\cdot,m)\bigr)=0.
$$
(see, among many other references, McKean \cite{McK}, Sznitman \cite{Szn}, M\'el\'eard \cite{Me96},...).
The standard strategy 
for establishing the convergence 
consists in a coupling argument.
Precisely, if one introduces the system of $N$ independent equations
$$
dY_{i,t}= b\bigl(Y_{i,t},{\mathcal L}(Y_{i,t})\bigr)\ dt+ \sqrt{2}dB^i_t, \qquad t\in [0,T],\; i\in \{1, \dots, N\}, 
$$
(where ${\mathcal L}(Y_{i,t})$ is the law of $Y_{i,t}$) with the same (chaotic) initial condition 
as that of the processes $((X_{i,t})_{t \in [0,T]})_{i=1,\dots,N}$, then it is known that 
(under appropriate integrability conditions, see 
Fournier and Guillin 
\cite{FG2014})
$$
\sup_{t \in [0,T]} \E\left[ |X_{1,t}-Y_{1,t}|\right]\leq CN^{-\frac1{\max(2,d)}}
\bigl( 
{\mathbf 1}_{\{d\not =2\}}
+
\ln(1+N) 
{\mathbf 1}_{\{d=2\}}
\bigr).
$$
In comparison with \eqref{e.SDEsIntro}, all the equations in \eqref{e.optitraj} are subject to the common noise $(W_t)_{t \in [0,T]}$, at least when $\beta \not =0$.
This makes a first difference between our limit problem and the above McKean-Vlasov example of interacting diffusions, but, for the time being, it is not
clear how deep this may affect the analysis. Indeed,
the presence of a common noise does not constitute a real challenge in the study
of McKean-Vlasov equations, 
the above coupling argument working in that case as well, provided that 
the distribution of $Y$ is replaced by its conditional distribution given 
the realization of the common noise. However, the key point here is precisely that our problem is not formulated as a McKean-Vlasov equation,
since the 
drifts in \eqref{e.optitraj}
are not of the same explicit mean-field structure as they are in 
\eqref{e.SDEsIntro} 
because of the additional dependence upon $N$ 
in the right-hand side of 
\eqref{eq:intro:form:vNi}
--obviously this is the second main difference between 
\eqref{e.optitraj}
and
\eqref{e.SDEsIntro}--. 
This makes rather difficult any attempt to guess the precise impact of the common noise onto the analysis.  
For sure, as we already pointed out, the major issue
for analyzing  
\eqref{e.optitraj}
is caused by the complex nature of the underlying interactions. 
As the equations depend upon one another 
 through the nonlinear system \eqref{NashIntro},
 the evolution with $N$ of
  the coupling between all of them is 
 indeed
 much more intricate than in \eqref{e.SDEsIntro}.
 And once again, 
on the top of that, the common noise adds another layer of difficulty. 
 For these reasons, the convergence of 
 both \eqref{NashIntro} and \eqref{e.optitraj} has been an open question since Lasry and Lions' initial papers on mean field games \cite{LL06cr1, LL06cr2}.
\bigskip

\noindent \textit{\textbf{The mean field game system.}} The analysis of the Nash system \eqref{NashIntro}  as the number of players is large pops up very naturally in game theory. Similar questions for static games 
were studied a long time ago by Aumann, who introduced the concept of nonatomic games in \cite{Au}; moreover, Schmeidler \cite{Sch73} and Mas-Colell \cite{MC84} defined and investigated non-cooperative Nash equilibria for one shot games with infinitely many small players. 

In the case of differential games,
the theory is known under the name of ``mean-field games'', whose 
principle goes as follows. If one tries, at least in the simpler case $\beta =0$, to describe --in a heuristic way-- the structure of a game with infinitely many
indistinguishable players, i.e., a ``nonatomic differential game", one finds a problem in 
which each (infinitesimal) player optimizes his payoff,
depending upon the collective behavior of the others,
and, meanwhile, the resulting optimal state of each of them 
is exactly distributed 
according to the state of the population. This is the ``mean field game system" (MFG system):
\be\label{e.MFGsyst}
\left\{ \begin{array}{l}
\ds - \partial_t u - \Delta u +H(x,Du)=F(x,m(t))\qquad {\rm in }\; [{0},T]\times \R^d,
\\
\ds \partial_t m - \Delta m -{\rm div}( m D_pH(x, D u))=0\qquad {\rm in }\; [{0},T]\times \R^d, \\
\ds u(T,x)=G(x,m(T)), \; m(0, \cdot)=m_{(0)} 
\qquad {\rm in }\;  \R^d,
\end{array}\right.
\ee
where $m_{(0)}$ denotes the initial state of the population. The system consists in a coupling between a (backward) Hamilton-Jacobi equation, 
describing the dynamics of the 
value function of any of the players, and a (forward) Kolmogorov equation, 
describing the dynamics of the 
distribution of the
population. 
In that framework, $H$ reads as an Hamiltonian, $F$ is understood as a running 
cost and $G$ as a terminal cost.
Since its simultaneous introduction by Lasry and Lions  \cite{LL07mf} and by Huang, Caines and Malhamé \cite{HCMieeeAC06}, this system has been thoroughly investigated: existence, under various assumptions, can be found in \cite{BeFr2012, CaDe2, HCMieeeAC07, HCMieeeDC07, HCMjssc07, KLY, LL07mf, LcoursColl}. Concerning uniqueness of the solution, two regimes were identified in \cite{LL07mf}. 
Uniqueness holds under Lipschitz type conditions when the time horizon $T$ is short (or, equivalently, when $H$, $F$ and $G$ are ``small"), 
but, 
as for finite-dimensional two-point boundary value problems, 
it
may fail when the system is set over a time interval 
of arbitrary length. Over long time intervals, 
uniqueness is guaranteed under the quite fascinating condition that {\it $F$ and $G$ are monotonous}, i.e., if, for any measures $m,m'$, the following holds:
\be\label{ineq.monoIntro}
\int_{\R^d} (F(x,m)-F(x,m')d(m-m')(x) \geq 0\; {\rm and }\; \int_{\R^d} (G(x,m)-G(x,m')d(m-m')(x) \geq 0.
\ee
The interpretation of the monotonicity condition is that the players dislike congested areas and favor configurations in which they are more scattered, see Remark 
\ref{ex:monotonicity:F:G} below for an example.
Generally speaking, condition
\eqref{ineq.monoIntro} plays a key role throughout the paper,
as it guarantees not only uniqueness but also stability of the 
solutions to \eqref{e.MFGsyst}. 

As announced,
a solution to the mean field game system \eqref{e.MFGsyst} can be
indeed interpreted as a Nash equilibrium for a differential game with infinitely many players: in that framework, it plays the role of the Schmeidler's non-cooperative equilibrium. A standard strategy 
to make the connection 
between \eqref{e.MFGsyst}
and differential games consists in inserting 
the optimal strategies 
from the 
Hamilton-Jacobi equation in 
\eqref{e.MFGsyst}
into finitely many player games in order to construct approximate Nash equilibria: see \cite{HCMieeeAC07}, as well as \cite{CaDe2, HCMieeeDC07, HCMjssc07, KLY}. 
However, although it establishes the 
interpretation 
of the system
\eqref{e.MFGsyst}
as a differential game
with infinitely many players, this says nothing about the convergence 
of 
\eqref{NashIntro} and \eqref{e.optitraj}. 
\vspace{4pt}

When $\beta$ is positive, the system
describing Nash equilibria within a population of infinitely many players
subject to the same common noise of intensity $\beta$ cannot 
be longer described by a deterministic system
of the same form as
\eqref{e.MFGsyst}. Owing to 
the theory of propagation of chaos for 
systems of
interacting particles, see the short remark above, 
the unknown $m$ in the forward equation is
then expected 
to represent the conditional law of 
the optimal state of any player
given the realization of the common noise.  
In particular, it must be random. 
This turns the forward Kolmogorov equation into 
a forward stochastic Kolmogorov equation. 
As the Hamilton-Jacobi equation depends on 
$m$, it renders 
$u$ random as well. 
Anyhow, a key fact from the theory of stochastic processes 
is that the solution to a stochastic differential equation must be adapted 
to the underlying observation, as its values at some 
time $t$ cannot anticipate the future of the noise after $t$. 
At first sight, it seems to be very demanding 
as $u$ is also required to match, at time $T$, $G(\cdot,m(T))$, which depends on the whole realization of 
the noise up until $T$. 
The right formulation 
to accommodate both constraints is
given by the theory of backward stochastic differential 
equations, which 
suggests to penalize the backward dynamics by a martingale in order to 
guarantee that the solution is indeed adapted. 
We refer the reader 
to the monograph 
\cite{PardouxRascanu} 
for a complete account on the finite dimensional 
theory and to
the paper 
\cite{Pe92}
for an insight into the infinite dimensional case. 
Denoting by $W$ ``the common noise" (here, a $d-$dimensional Brownian motion) and
by $m_{(0)}$ the initial distribution of the players at time $t_0$, 
the MFG system with common noise then takes the form (in which the unknown are now $(u_t,m_t,v_t)$): 
\be
\label{e.MFGstoch}
\left\{ \begin{array}{l}
\ds 
 d_{t} u_{t} = \bigl\{ - (1+\beta) \Delta u_{t} + H(x,Du_{t}) - F(x,m_{t}) -  \sqrt{2\beta} {\rm div}(v_{t}) \bigr\} dt+ v_{t} \cdot \sqrt{2\beta}dW_{t}\\
 \qquad \qquad \qquad \qquad \qquad \qquad \qquad \qquad \qquad \qquad {\rm in }\; [0,T]\times \T^d,\\
\ds d_{t} m_{t} = \bigl[  (1+\beta) \Delta m_{t} + {\rm div} \bigl( m_{t} D_{p} H(m_{t},D u_{t}) 
\bigr) \bigr] dt - {\rm div} ( m_{t} \sqrt{2\beta} dW_{t} \bigr), \\
 \qquad \qquad \qquad \qquad \qquad \qquad \qquad \qquad \qquad \qquad {\rm in }\; [0,T]\times \T^d,
 \\
 \ds u_T(x)=G(x,m_T), \; m_{0}=m_{(0)},\qquad {\rm in }\;  \T^d
\end{array}\right.
\ee
where we used the standard convention 
from the theory of stochastic processes that consists in 
indicating the time parameter as an index in random functions. 
As suggested right above, 
the map $v_t$ is a random vector field that forces the solution $u_t$ of the backward equation to be adapted to the filtration generated by $(W_t)_{t \in [0,T]}$. 
As far as we know, the system 
\eqref{e.MFGstoch} has never been investigated 
and part of the paper will be dedicated to its analysis
(see however 
\cite{CaDe14} for an informal discussion).
Below, we call the system 
\eqref{e.MFGstoch}
the
{\it MFG system with common noise}.
\vspace{4pt}

It is worth mentioning that
the aggregate equations  \eqref{e.MFGsyst} and \eqref{e.MFGstoch} (see also the master equation \eqref{MasterEqIntro} below) are the continuous time analogues of equations that appear in the analysis of dynamic stochastic general equilibria in  heterogeneous agent models, as introduced in economic theory by Aiyagari \cite{Aiy}, Bewley \cite{Bew} and Huggett \cite{Hug}.  
In this setting, the factor $\beta$ describes the intensity of ``aggregate shocks", as discussed by  Krusell and Smith in the seminal paper \cite{KrSm}. In some sense, the limit problem studied in the paper is an attempt to deduce the macroeconomic models, describing the dynamics of a typical (but heterogeneous) agent in an equilibrium configuration, from the microeconomic ones (the Nash equilibria). 
\bigskip

\noindent \textit{\textbf{ The master equation.}} 
Although the mean field game system has been widely studied since its introduction in \cite{LL07mf} and \cite{HCMieeeAC06}, it has become increasingly clear that this system was not sufficient to take into account the entire complexity of dynamic games with infinitely many players. 
The need for reformulating the original system 
\eqref{e.MFGsyst} into the much more complex stochastic version
\eqref{e.MFGstoch} in order to accommodate 
with the common noise 
(i.e., the case $\beta>0$)
sounds as a hint in that direction.
In the same spirit, we may notice that the original MFG system \eqref{e.MFGsyst} does not accommodate with mean field games with a major player and infinitely many small players, 
see \cite{Hu2010}. And, last but not the least, 
the main limitation is that, so far, the formulation based on the system 
\eqref{e.MFGsyst} (or \eqref{e.MFGstoch} when $\beta >0$) has not permitted to establish a clear connection with the Nash system \eqref{NashIntro}.

These issues led Lasry and Lions  \cite{LcoursColl} to introduce an infinite dimensional equation --the so-called ``master equation"-- that directly describes, at least formally, the limit of the Nash system \eqref{NashIntro} and encompasses the above complex situations. Before writing down this equation, let us explain its main features. One of the key observations has to do with the symmetry properties, to which we already alluded, that are satisfied by the solution of the Nash system \eqref{NashIntro}. 
Under the standing symmetry assumptions 
\eqref{eq:intro:running:terminal:costs}
on the $(F^{N,i})_{i=1,\dots,N}$ and $(G^{N,i})_{i=1,\dots,N}$, 
\eqref{eq:intro:form:vNi}
says that
the $(v^{N,i})_{1,\dots,N}$ can be written into 
a similar form to 
\eqref{eq:intro:running:terminal:costs}, namely
 $v^{N,i}(t,{\boldsymbol x})=v^{N}(t,x_i, m^{N,i}_{\boldsymbol x})$
 (where the empirical measures $m^{N,i}_{\boldsymbol x}$ are defined as in 
\eqref{eq:intro:running:terminal:costs}), but with the obvious but major restriction that
the function 
$v^N$ that appears on the right-hand side of the equality now depends upon $N$. 
With such a formulation, the value function to player $i$ reads as a function 
of the private state of player $i$ and of the empirical distribution 
formed by the others. 
Then, one may guess, at least under the additional assumption 
that 
such a structure
 is preserved as $N\to +\infty$, that the unknown in the limit problem 
takes the form $U=U(t,x,m)$, where $x$ is the position of the (typical) small player at time $t$ and $m$ is the distribution of the (infinitely many) other agents. 

The question is then to write down the dynamics of $U$. Plugging $U=U(t,x_{i},m^{N,i}_{\boldsymbol x})$ into the Nash system \eqref{NashIntro}, one obtains---at least formally---an equation stated in the space of measures (see Subsection \ref{subsec.informal} for a heuristic discussion). This is the so-called master equation. It takes the form:
\be\label{MasterEqIntro}
\left\{\begin{array}{l} 
\ds - \partial_t U  - (1+\beta)\Delta_x U +H(x,D_xU) \\
\ds    \qquad  -(1+\beta)\int_{\R^d} \dive_y \left[D_m U\right]\ d m(y)+ \int_{\R^d} D_m U\cdot D_pH(y,D_xU) \ dm(y) \\
\ds   \qquad  -2 \beta  \int_{\R^d}  \dive_x\left[D_mU\right] dm(y) -\beta \int_{\R^{2d}} {\rm Tr} \left[D^2_{mm}U\right]\ dm\otimes dm =F(x,m)\\
 \ds \qquad \qquad  \qquad {\rm in }\; [0,T]\times \R^d\times {\mathcal P}(\R^d)\\
 U(T,x,m)= G(x,m) \qquad  {\rm in }\;  \R^d\times {\mathcal P}(\R^d)\\
\end{array}\right.
\ee
In the above equation, $\partial_tU$, $D_xU$ and $\Delta_xU$ stand for the usual time derivative, space derivatives and Laplacian with respect to the local variables $(t,x)$ of the unknown $U$, while $D_mU$ and $D^2_{mm}U$ are the first and second order derivatives with respect to the measure $m$. The precise definition of these derivatives is postponed to Section \ref{sec:prelim}. For the time being, let us just note that it is related with the derivatives in the space of probability measures described, for instance,  by Ambrosio, Gigli and Savaré in \cite{AGS} and by Lions in \cite{LcoursColl}. 
It is worth mentioning 
that
the master equation \eqref{MasterEqIntro} is not the first example of an equation studied in the space of measures --by far: for instance Otto \cite{Otto} gave an interpretation of the porous medium equation as an evolution equation in the space of measures, and Jordan, Kinderlehrer and Otto \cite{JoKiOt} showed that the heat equation was also a gradient flow in that framework; notice also that the analysis of Hamilton-Jacobi equations in metric spaces is partly motivated by the specific case when the underlying metric space is the space of measures (see in particular \cite{AmFe, FeKa} and the references therein)--. The master equation is however the first one to combine at the same time the issue of being nonlocal, nonlinear and of second order.

Beside the discussion in \cite{LcoursColl}, the importance of the master equation \eqref{MasterEqIntro} has been acknowledged by several contributions: see for instance the monograph \cite{BFY} and 
the companion papers 
\cite{BFYmaster} and
\cite{BFYmaster2}
 in which Bensoussan, Frehse and Yam generalize this equation to mean field type control problems and reformulate it as a PDE set on an $L^2$ space, \cite{CaDe14} where Carmona and Delarue interpret this equation as a decoupling field of forward-backward stochastic differential equation in infinite dimension. 
\vskip 4pt

If the master equation has been discussed and manipulated thoroughly in the above references, it is mostly at a formal level: The well-posedness of the master equation has remained, to a large extend, open until now. Beside, even if the master equation has been introduced to explain the convergence of the Nash system, the rigorous justification of the convergence has not been understood. 

The aim of the paper is to give an answer to both questions. 
\bigskip

\noindent \textit{\textbf{Well-posedness of the master equation.}} The largest part of this paper is devoted to the proof of the existence and uniqueness of a classical solution to the master equation \eqref{MasterEqIntro}, where, by classical, 
we mean that all the derivatives in \eqref{MasterEqIntro} exist and are continuous. In order to avoid issues related to boundary conditions or conditions at infinity, we work for simplicity with periodic data: the maps $H$, $F$ and $G$ are periodic in the space variable. The state space is therefore the $d$-dimensional  torus $\T^d= \R^d/\Z^d$ and $m_{(0)}$ belongs to $\Pk$, the set of Borel probability measures on $\T^d$. We also assume that $F, G:\T^d\times \Pk\to \R$ satisfy the monotonicity conditions \eqref{ineq.monoIntro}, are sufficiently ``differentiable" with respect to both variables and, of course, periodic with respect to the state variable. Although the periodicity condition is rather restrictive, the extension to maps defined on the full space or to Neumann boundary conditions is probably not a major issue. Anyhow, it would certainly require further technicalities, which would have made the paper even longer than it is if we had decided to include them.
\vspace{4pt}

So far, the existence of classical solutions to the master equation 
has been known in more restricted frameworks. Lions discussed in \cite{LcoursColl} a finite dimensional analogue of the master equation and derived conditions for this hyperbolic system to be well-posed. 
These conditions correspond precisely to the monotonicity 
property \eqref{ineq.monoIntro},
which we here assume to be satisfied by the coupling functions $F$ and $G$.
This parallel strongly indicates --but this should not does not come as a surprise--
that the monotonicity of $F$ and $G$
should play a key role in the unique strong solvability of \eqref{MasterEqIntro}. Lions also explained in \cite{LcoursColl} how to get the well-posedness of the master equation without noise (no Laplacian in the equation) by extending the equation to a (fixed) space of random variables
under a convexity assumption in space of the data. In \cite{BuLiPeRa}
Buckdahn, Li, Peng and Rainer studied equation \eqref{MasterEqIntro},
by means of probabilistic arguments, when there is no coupling nor common noise ($F=G=0$, $\beta=0$) and proved the existence of 
a classical solution in this setting; in a somewhat similar spirit, 
Kolokoltsov, Li and Yang
\cite{KLY}
and
Kolokoltsov, Troeva and Yang 
\cite{KolTroYa14}
investigated
the tangent process to a flow of probability measures solving 
a McKean-Vlasov equation.
Gangbo and Swiech \cite{GS14-2} analyzed the first order master equation in short time (no Laplacian in the equation) for a particular class of Hamiltonians and of coupling functions $F$ and $G$ (which are required to derive from a potential in the measure argument). Chassagneux,  Crisan  and Delarue \cite{CgCrDe} obtained, by a probabilistic approach similar to that used in \cite{BuLiPeRa}, the existence and uniqueness of a solution to \eqref{MasterEqIntro} without common noise (when $\beta=0$) under the monotonicity condition 
\eqref{ineq.monoIntro} in either the non degenerate case (as we do here)
or in the degenerate setting provided that $F$, $H$ and $G$ satisfy 
an additional convexity conditions in the variables $(x,p)$. 
The complete novelty of our result, regarding the
specific question of solvability of the master equation,
is the existence and uniqueness of a classical solution to the problem with common noise.

The technique of proof in \cite{BuLiPeRa, CgCrDe, GS14-2} consists in finding a suitable representation of the solution: indeed a key remark in Lions \cite{LcoursColl} is that the master equation is a kind of transport equation in the space of measures and that its characteristics are, when $\beta =0$, the MFG system \eqref{e.MFGsyst}. Using this idea,  the main difficulty is then to prove that the candidate is smooth enough to perform the computation showing that it is a classical solution of \eqref{MasterEqIntro}. In 
\cite{BuLiPeRa,CgCrDe}  this is obtained by linearizing systems of forward-backward stochastic differential equations, while \cite{GS14-2} relies on a careful analysis of the characteristics of the associated first order PDE. 
%
\vspace{4pt}

Our starting point is the same: we use a representation formula for the master equation. When $\beta=0$, the characteristics are just the solution to the MFG system \eqref{e.MFGsyst}. 
When $\beta$ is positive, these characteristics 
become random under the action of the common noise
and are then given by the solution of 
the MFG system with common noise
\eqref{e.MFGstoch}. 

The construction of a solution $U$ 
to the master equation then relies on the method of characteristics. Namely, 
we {\it define} $U$ by letting $U(t_0,x,m_0):= u_{t_0}(x)$ where the pair $(u_t,m_t)_{t\in [t_0,T]}$ is the solution to \eqref{e.MFGstoch}
when the forward equation is initialized at $m_{(0)} \in {\mathcal P}(\T^d)$
at time $t_{0}$, that is 
\be\label{e.MFGstoch2}
\left\{ \begin{array}{l}
\ds 
 d_{t} u_{t} = \bigl\{ - (1+\beta) \Delta u_{t} + H(x,Du_{t}) - F(x,m_{t}) -  \sqrt{2\beta} {\rm div}(v_{t}) \bigr\} dt+ v_{t} \cdot \sqrt{2\beta}dW_{t}\\
 \qquad \qquad \qquad \qquad \qquad \qquad \qquad \qquad \qquad \qquad {\rm in }\; [t_0,T]\times \T^d,
 \\
\ds d_{t} m_{t} = \bigl[  (1+\beta) \Delta m_{t} + {\rm div} \bigl( m_{t} D_{p} H(m_{t},D u_{t}) 
\bigr) \bigr] dt - {\rm div} ( m_{t} \sqrt{2\beta} dW_{t} \bigr), 
\\
 \qquad \qquad \qquad \qquad \qquad \qquad \qquad \qquad \qquad \qquad {\rm in }\; [t_0,T]\times \T^d\\
 \ds u_T(x)=G(x,m_T), \; m_{t_0}=m_{(0)} \qquad {\rm in }\;  \T^d,
\end{array}\right.
\ee
There are two main difficult steps
in the analysis. The first one is to establish the smoothness of $U$ and
the second one is to show that $U$ indeed satisfies the master equation \eqref{MasterEqIntro}.
In order to proceed, the cornerstone is to make a systematic use of the monotonicity properties of the maps $F$ and $G$: Basically, monotonicity prevents the emergence of singularities in finite time. 
Our approach seems to be very powerful, 
although the reader might have a different feeling due to the length of the paper. As a matter of fact, part of the technicalities in the 
proof are caused by the stochastic aspect of the characteristics \eqref{e.MFGstoch2}.
As a result, we spend much effort to handle
the case with a common noise (for which almost nothing has been known so far), but, in the simpler case $\beta = 0$, our strategy to handle the 
first order master equation provides a much shorter proof than in the 
earlier works \cite{BuLiPeRa, CgCrDe, GS14-2}.
For this reason, we decided to display the proof in this simple context separately (Section \ref{sec:MasterWithoutCN}). 
\bigskip

\noindent
\textit{\textbf{The convergence result.}} 
Although most of the paper is devoted to the
 construction of a solution to the master equation, our main (and \textit{primary}) motivation remains 
to justify the mean field limit. Namely, we show that the solution of the Nash system \eqref{NashIntro} converges to the solution of the master equation. The main issue here is the complete lack of estimates on the solutions to this large system of Hamilton-Jacobi equations: This prevents the use of any compactness method to prove the convergence. So far, 
this question has been almost completely open. The convergence has been 
known in very few specific situations. For instance, it was proved for the ergodic mean field games (see Lasry-Lions \cite{LL06cr1}, revisited by Bardi-Feleqi \cite{BaFe}). In this case, the Nash equilibrium system reduces to a coupled system of $N$ equations in $\T^d$ (instead of $N$ equations in $\T^{Nd}$ as \eqref{NashIntro}) and estimates of the solutions are available. Convergence is also known in the ``linear-quadratic" setting, where the Nash system has explicit solutions: see Bardi \cite{Ba}. Let us finally quote the nice results by Fischer \cite{Fi14} and Lacker \cite{Lc14} on the convergence of {\it open loop Nash equilibria} for the $N-$player game and the characterization of the possible limits. Therein, the authors overcome the lack of strong estimates 
on the solutions to the $N-$player game by using 
the notion of
\textit{relaxed controls} 
for which weak compactness criteria are available.
The problem addressed here---concerning {\it closed loop Nash equilibria}---differs in a substantial way from \cite{Fi14, Lc14}: Indeed, we underline the surprising fact that the Nash system \eqref{NashIntro}, which concerns equilibria in which the players observe each other, converges to an equation in which the players only need to observe the evolution of the distribution of the population. 
\vspace{4pt}

Our main contribution is a general convergence result, in large time, for mean field games with common noise, as well as an estimate of the rate of convergence. The convergence holds in the following sense: 
for any $\bx\in (\T^d)^N$, let $m^N_\bx:= \frac{1}{N} \sum_{i=1}^N \delta_{x_i}$. Then 
\begin{equation}
\label{eq:intro:vitesse:cv}
\frac{1}{N}\sum_{i=1}^N \left|v^{N,i}(t_0,\bx)- U(t_0, x_i, m^{N}_\bx)\right|\leq CN^{-1}.
\end{equation}
We also prove a mean field result for the optimal solutions \eqref{e.optitraj}: if the initial conditions of the $((X_{i,\cdot}))_{i=1,\dots,N}$ are i.i.d. and with the same law $m_{(0)}\in \Pk$, then 
$$
\E\Bigl[\sup_{t\in [0,T]} | X_{i,t}-Y_{i,t}|\Bigr] \leq C N^{-\frac{1}{d+8}},
$$
where the $((Y_{i,t})_{i=1,\dots,N})_{t \in [0,T]}$ are the solutions to the McKean-Vlasov SDE
$$
dY_{i,t}= -D_pH\bigl(Y_{i,t}, D_{x}U\bigl(t, Y_{i,t}, {\mathcal L}(Y_{i,t}|W)\bigr)\bigr)dt +\sqrt{2} dB^{i}_t+\sqrt{2\beta} dW_t, \qquad t\in [t_0,T],\\
$$
with the same initial condition as the $((X_{i,t})_{i=1,\dots,N})_{t \in [0,T]}$. Here $U$ is the solution of the master equation and ${\mathcal L}(Y_{i,t}|W)$ is the conditional law of $Y_{i,t}$ given the realization of the whole path $W$. Since the $((Y_{i,t})_{t \in [0,T]})_{i=1,\dots,N}$ are conditionally independent given $W$, the above result shows that (conditional) propagation of chaos holds for the $N-$Nash equilibria. 

 The technique of proof consists in testing 
 the solution $U$ 
 of the master equation \eqref{MasterEqIntro}
 as a nearly solution to the $N-$Nash system 
 \eqref{NashIntro}.
 On the model of 
 \eqref{eq:intro:running:terminal:costs},  
 a natural candidate for being an approximate solution
to the $N-$Nash system is indeed
\begin{equation*}
u^{N,i}(t,{\boldsymbol x}) = U\bigl(t,x_{i},m_{\boldsymbol x}^{N,i}\bigr), 
\quad t \in [0,T], \ {\boldsymbol x} \in (\T^d)^N. 
\end{equation*} 
 Taking benefit from the smoothness of $U$, 
 we then prove that 
 the ``proxies'' $(u^{N,i})_{i=1,\dots,N}$
 almost solve the $N-$Nash system 
\eqref{NashIntro} up to a remainder term that vanishes as $N$ tends 
to $\infty$. As a by-product, 
we deduce that the $(u^{N,i})_{i=1,\dots,N}$ 
 get closer and closer to the
``true solutions" $(v^{N,i})_{i=1,\dots,N}$ when $N$ tends to $\infty$, 
which yields 
\eqref{eq:intro:vitesse:cv}. 
As the reader may notice, the 
convergence property  
\eqref{eq:intro:vitesse:cv}
is stated in a symmetric form, namely the convergence holds in the mean, the average
being taken over 
all the particles. Of course, this is reminiscent of the 
symmetry properties satisfied by the $N-$Nash system, 
which play a crucial role in the proof.

It is worth mentioning that the monotonicity 
properties \eqref{eq:intro:running:terminal:costs}
play no role in our proof of the convergence. Except structural conditions 
concerning the Lipschitz property of the coefficients, the arguments 
work under
the sole assumption that the master equation has a classical solution. 
\bigskip

\noindent
\textit{\textbf{Conclusion and further prospects.}} 
The fact that the existence of a classical solution 
to the master equation suffices to prove the convergence of 
the Nash system demonstrates the deep interest of the master equation, when regarded 
as a mathematical concept in its own right. 
Considering the problem from a more abstract point of view, 
the master equation indeed captures the evolution of the time-dependent 
semi-group generated by the Markov process formed, on the space of 
probability measures, by the forward component of the MFG system 
\eqref{e.MFGstoch2}. 
Such a semi-group is said to be \textit{lifted} as
the corresponding Markov process has ${\mathcal P}(\T^d)$ as state space.
In other words, the master equation is 
a nonlinear PDE driven by a Markov generator
acting on functions defined on ${\mathcal P}(\T^d)$. 
The general contribution of our paper is 
thus to show that
any classical solution to the master equation accommodates with 
a given perturbation of the 
lifted semi-group
and that the information enclosed in such a classical solution suffices to determine the distance between the semi-group and its perturbation. 
Obviously, as a perturbation of a semi-group on the space of 
probability measures, we are here thinking of a system 
of $N$ interacting particles, exactly as that formed by 
the Nash equilibrium of an $N-$player game. 

Identifying the master equation with a nonlinear PDE driven by the Markov generator 
of a lifted semi-group is a key observation.
As already pointed out, the Markov generator is precisely the operator, acting 
on functions from ${\mathcal P}(\T^d)$ to $\R$, generated by the 
forward component of the MFG system \eqref{e.MFGstoch2}. Put it differently,
the law of the forward component of the MFG system \eqref{e.MFGstoch2},
which lives in ${\mathcal P}({\mathcal P}(\T^d))$, satisfies a forward 
Kolmogorov equation, also referred to as a ``master equation" in physics. 
This says that 
``our master equation'' is somehow the \textit{dual} (in the sense that it is driven by the adjoint operator) of the ``master equation" that would describe, 
according to the terminology used in physics, the law of the Nash equilibrium 
for a game with infinitely many players (in which case the Nash equilibrium itself is 
a distribution). We stress that this interpretation is very close to the point of view developed by 
Mischler and Mouhot \cite{MiMo13} in order to investigate Kac's program
(up to the difference that, differently from ours, Mischler and Mouhot's work
investigates uniform propagation of chaos over
an infinite time horizon;  we refer to the companion paper
by Mischler, Mouhot and Wennberg
\cite{MiMoWe15} for the analysis, based on the same technology, of mean-field models in finite time). 
Therein, the authors introduce the evolution equation satisfied by the (\textit{lifted}) 
semi-group, 
acting on functions from ${\mathcal P}(\R^d)$ to $\R$, generated by 
the $d$-dimensional Boltzmann equation. According to our terminology, such an evolution equation is 
a ``master equation" on the space of probability measures, but it is linear
and of the first-order  
while ours is nonlinear and of the second-order (meaning second-order on ${\mathcal P}(\T^d)$). 

In this perspective, we also emphasize that our strategy for proving the convergence 
of the $N-$Nash system
relies on a similar idea to that used in 
\cite{MiMo13} to establish the convergence of Kac's jump process. 
While our approach consists in inserting the solution of the master equation 
into the $N-$Nash system, Mischler and Mouhot's point of view is to compare 
the semi-group generated by the $N-$particle Kac's jump process, 
which operates on symmetric functions from $(\R^d)^N$ to $\R$ (or equivalently 
on empirical distributions of size $N$),
with the \textit{limiting lifted} semi-group, when acting on the same class of symmetric functions from $(\R^d)^N$ to $\R$. 
Clearly, the philosophy is the same, except that, in our paper, the ``limiting master equation'' is nonlinear and of the second-order (which renders the analysis more difficult)
and is set over a finite time horizon only (which does not 
ask for uniform in time estimates).  
It is worth mentioning that similar ideas have been explored 
by Kolokoltsov in the monograph \cite{Kol10}
and developed, in the McKean-Vlasov framework, in the subsequent works \cite{KLY} and \cite{KolTroYa14} in collaboration with his coauthors.

Of course, these parallels raise interesting questions, but we refrain from comparing these different works in a more detailed way: This would require to address more technical questions regarding, for instance, the topology used on the space of probability measures and the regularity of the various objects in hand; clearly, this would distract us from our original objective. We thus feel better to keep 
the discussion at an informal level and to postpone a more careful comparison to future works on the subject.
\vskip 4pt

We complete the introduction by pointing out possible generalizations of our results. For simplicity of notation, we work in the autonomous case, but the results remain unchanged if $H$ or $F$ are time-dependent provided that the coefficients $F$, $G$ and $H$, and their derivatives (whenever they exist), are continuous in time and that the various quantitative assumptions we put on $F$, $G$ and $H$ hold uniformly with respect to the time variable. We can also remove 
the monotonicity condition \eqref{ineq.monoIntro} provided that the time horizon 
$T$ is assumed to be small enough. The reason is that the analysis of the 
smoothness of $U$ relies on the solvability and stability properties of the 
forward-backward system \eqref{e.MFGstoch2} 
and of its linearized version: As for finite-dimensional two-point boundary value 
problems,  
Lipschitz type conditions on the coefficients (and on their derivatives since we are also 
dealing with the linearized version) are sufficient
whenever $T$ is small enough.

As already mentioned, we also chose to work in the periodic framework. We expect
for similar results under other type boundary conditions, like the entire space $\R^d$ or Neumann boundary conditions.   

Notice also that our results can be generalized without much difficulty to the {\it stationary setting}, corresponding to infinite horizon problems. This framework is particularly meaningful for economic applications. In this setting the Nash system takes the form 
$$
\left\{ \begin{array}{l}
\ds r v^{N,i} (\bx)-  \sum_{j=1}^N \Delta_{x_j}v^{N,i}(\bx) - \beta  \sum_{j,k=1}^N {\rm Tr} D^2_{x_j,x_k} v^{N,i}(\bx) + H(x_i,  D_{x_i}v^{N,i}(\bx)) \\
\ds \qquad \qquad + \sum_{j\neq i}  D_pH (x_j, D_{x_j}v^{N,j}(\bx))\cdot D_{x_j}v^{N,i}(\bx)=
F^{N,i}(\bx) 
\qquad {\rm in}\; (\R^{d})^N,
\end{array}\right.
$$
where $r>0$ is interpreted as a discount factor. The  corresponding master equation is
$$
\left\{\begin{array}{l} 
\ds r U  - (1+\beta)\Delta_x U +H(x,D_xU) \\
\ds    \qquad  -(1+\beta)\int_{\R^d} \dive_y \left[D_m U\right]\ d m(y)+ \int_{\R^d} D_m U\cdot D_pH(y,D_xU) \ dm(y) \\
\ds   \qquad  -2 \beta  \int_{\R^d}  \dive_x\left[D_mU\right] dm(y) -\beta \int_{\R^{2d}} {\rm Tr} \left[D^2_{mm}U\right]\ dm\otimes dm =F(x,m)
\\
 \ds \qquad \qquad  
 \qquad {\rm in }\;\R^d\times {\mathcal P}(\R^d),
\end{array}\right.
$$
where the  unknown is the map $U=U(x,m)$. One can solve again this system by using the method of (infinite dimensional) characteristics, paying attention to the fact that 
these characteristics remain time-dependent.  The MFG system with common noise takes the form (in which the unknown are now $(u_t,m_t,v_t)$): 
$$
\left\{ \begin{array}{l}
\ds 
 d_{t} u_{t} = \bigl\{ ru_{t}- (1+\beta) \Delta u_{t} + H(x,Du_{t}) - F(x,m_{t}) -  \sqrt{2\beta} {\rm div}(v_{t}) \bigr\} dt+ v_{t} \cdot \sqrt{2\beta}dW_{t}\\
 \qquad \qquad \qquad \qquad \qquad \qquad \qquad \qquad \qquad \qquad {\rm in }\; [0,+\infty)\times \T^d\\
\ds d_{t} m_{t} = \bigl[  (1+\beta) \Delta m_{t} + {\rm div} \bigl( m_{t} D_{p} H(m_{t},D u_{t}) 
\bigr) \bigr] dt - {\rm div} ( m_{t} \sqrt{2\beta} dW_{t} \bigr), \\
 \qquad \qquad \qquad \qquad \qquad \qquad \qquad \qquad \qquad \qquad {\rm in }\; [0,+\infty)\times \T^d\\
 \ds m_{0}=\bar m_0\qquad {\rm in }\;  \T^d, \; (u_{t})_t \; \mbox{\rm bounded a.s.}
\end{array}\right.
$$
\\
\textbf{\textit{Organization of the paper.}} We present our main results in Section \ref{sec:prelim}, where we also explain the notation, state the assumption and rigorously define the notion of derivative on the space of measures. The well-posedness of the master equation is proved in Section \ref{sec:MasterWithoutCN} when $\beta =0$. Unique solvability of the MFG system with common noise is discussed in Section \ref{se:common:noise}. 
Results obtained in Section \ref{se:common:noise}
are implemented in the next Section \ref{subse:partie:2:first:order:derivative}
to derive the existence of a classical solution to the master equation in the general case.  The last section is devoted to the  convergence of the Nash system. In appendix, we revisit the notion of derivative on the space of probability measures and discuss some useful auxiliary properties. 

\subsection{Informal derivation of the master equation}
\label{subsec.informal}

Before stating our main results, it is worthwhile explaining the meaning of the Nash system, the heuristic derivation of the master equation from the Nash system and its main properties. We hope that this (by no means rigorous) presentation might help the reader to be acquainted with our notation and the main ideas of proof. To emphasize the informal aspect of the discussion, we state all the ideas in $\R^d$, without bothering about the boundary issues
(whereas in the rest of the paper we always work with periodic boundary conditions). 

\subsubsection{The differential game}

The Nash system \eqref{NashIntro} arises in  differential game theory. Differential games are just optimal control problems with many (here $N$) players. 
In this game, Player $i$ (for $i=1,\dots, N$) controls  his state $(X_{i,t})_{t \in [0,T]}$ through his control $(\alpha_{i,t})_{t \in [0,T]}$. The state  $(X_{i,t})_{t \in [0,T]}$ evolves according to the stochastic differential equation (SDE)
\begin{equation}
\label{eq:intro:N:player:game:alpha}
dX_{i,t}= \alpha_{i,t}dt +\sqrt{2} dB^i_t +\sqrt{2\beta} dW_t, \qquad X_{t_0}= x_{i,0}.
\end{equation}
Recall that the $d$-dimensional Brownian motions $((B^i_t)_{t\in [0,T]})_{i=1,\dots,N}$ and $(W_t)_{t\in [0,T]}$ are independent, $(B^i_t)_{t \in [0,T]}$ corresponding to the {\it individual noise} (or \textit{idiosyncratic noise}) to player $i$ and $(W_t)_{t \in [0,T]}$ being the {\it common noise},
which affects all the players. Controls $((\alpha_{i,t})_{t \in [0,T]})_{i=1,\dots,N}$ are required to be progressively-measurable 
with respect to the filtration generated by all the noises.
Given an initial condition
${\boldsymbol x}_{0}=(x_{1,0},\dots,x_{N,0}) \in (\T^d)^N$ 
for the whole system at time $t_{0}$,
each player aims at minimizing the cost functional:
$$
J^N_i\bigl(t_0,\bx_0,(\alpha_{j,\cdot})_{j=1,\dots,N}\bigr)= \E\left[ \int_{t_0}^T \left(L(X_{i,s}, \alpha_{i,s})+F^{N,i}(\bX_s)\right) ds +G^{N,i}(\bX_T)\right],
$$
where $\bX_t=(X_{1,t}, \dots, X_{n,t})$ and where $L:\R^{d}\times \R^d\to \R$,  $F^{N,i}:\R^{Nd}\to \R$ and $G^{N,i}:\R^{Nd}\to \R$ are given Borel maps. 
If we assume that, for each player $i$, the other players are undistinguishable, we can suppose that $F^{N,i}$ and $G^{N,i}$ take the form
$$
F^{N,i}(\bx)= F(x_i, m^{N,i}_{\bx}) \qquad 
{\rm and}
\qquad 
G^{N,i}(\bx)= G(x_i, m^{N,i}_{\bx}).
$$
In the above expressions, $F,G:\R^d\times {\mathcal P}(\R^d)\to \R$, where ${\mathcal P}(\R^d)$ is the set of Borel measures on $\R^d$. The Hamiltonian of the problem is related to $L$ by the formula: 
$$
\forall (x,p)\in \R^d\times \R^d, \qquad H(x,p)= \sup_{\alpha\in \R^d} \left\{-\alpha\cdot p- L(x,\alpha)\right\}.
$$
Let now $(v^{N,i})_{i=1,\dots,N}$ be the solution to \eqref{NashIntro}. By It\^{o}'s formula, it is easy to check that $(v^{N,i})_{i=1,\dots,N}$ corresponds to an optimal solution of the problem in the sense of Nash, i.e., a {\it Nash equilibrium} of the game. Namely, the feedback strategies 
\be\label{e.optiFeed}
\left(\alpha^*_i(t,\bx):=-D_pH(x_i,D_{x_i}v^{N,i}(t,\bx))\right)_{i=1,\dots, N}
\ee
provide a feedback Nash equilibrium for the game:
$$
v^{N,i}\bigl(t_0,\bx_0\bigr)= J_i^N\bigl(t_0,\bx_0, (\alpha^*_{j,\cdot})_{j=1,\dots,N}\bigr)\leq J_i^N(t_0,\bx_0, \alpha_{i,\cdot}, (\hat \alpha^*_{j,\cdot})_{j\neq i})
$$
for any $i\in \{1,\dots, N\}$ and any control $\alpha_{i,\cdot}$, progressively-measurable 
with respect to the filtration 
 generated by $((B^j_{t})_{j=1,\dots, N})_{t \in [0,T]}$ and 
$(W_{t})_{t \in [0,T]}$. 
In the left-hand side, $\alpha_{j,\cdot}^*$ is an abusive 
notation for the process 
$(\alpha^*_{j}(t,X_{j,t}))_{t \in [0,T]}$, where 
$(X_{1,t},\dots,X_{N,t})_{t \in [0,T]}$
solves the system of SDEs \eqref{eq:intro:N:player:game:alpha}
when $\alpha_{j,t}$ is precisely given under the 
implicit form
$\alpha_{j,t}=\alpha^*_{j}(t,X_{j,t})$.
Similarly,
in the right-hand side, $\hat{\alpha}_{j}^*$, for $j \not =i$,
denotes  
$(\alpha^*_{j}(t,X_{j,t}))_{t \in [0,T]}$, where 
$(X_{1,t},\dots,X_{N,t})_{t \in [0,T]}$
now solves the system of SDEs \eqref{eq:intro:N:player:game:alpha}
for the given $\alpha_{i,\cdot}$, 
the other $(\alpha_{j,t})_{j \not i}$'s being given under the 
implicit form
$\alpha_{j,t}=\alpha^*_{j}(t,X_{j,t})$.
In particular, system  \eqref{e.optitraj}, in which all the players play the optimal feedback \eqref{e.optiFeed},  describes the dynamics of the optimal trajectories. 

\subsubsection{Derivatives in the space of measures} 
In order to describe the limit of the maps $(v^{N,i})$, let us introduce---in a completely informal manner---a notion  of derivative in the space of measures ${\mathcal P}(\R^d)$. A rigorous description of the notion of derivative used in this paper is given in section \ref{subsec.Derivatives}. 

In the following discussion, we argue as if all the measures 
had a density. 
Let $U:{\mathcal P}(\R^d)\to \R$. {\it Restricting} the function $U$ to the elements $m$ of ${\mathcal P}(\R^d)$ which have a density in $L^2(\R^d)$ and assuming that $U$ is defined in a neighborhood
${\mathcal O} \subset L^2(\R^d)$ of ${\mathcal P}(\R^d)\cap L^2(\R^d)$, we can use the Hilbert structure on $L^2(\R^d)$. We denote by $\frac{\delta U}{\delta m}$ the gradient of $U$ in $L^2(\R^d)$, 
namely
$$
\frac{\delta U}{\delta m}(p)(q)
 = \lim_{\varepsilon \rightarrow 0} \frac{1}{\varepsilon}
 \Bigl( U(p + \varepsilon q) - U(p) \Bigr), \quad p \in {\mathcal O}, \ q \in L^2(\R^d).
 $$
Of course, way can identify $\frac{\delta U}{\delta m}(p)$ with an element 
of $L^2(\R^d)$. Then, the duality product 
$\frac{\delta U}{\delta m}(p)(q)$
reads as the inner product 
$\langle \frac{\delta U}{\delta m}(p),q \rangle_{L^2(\R^d)}$.
Similarly, we denote by $\frac{\delta^2 U}{\delta m^2}$ the second order derivative of $U$ (which can be identified with a symmetric bilinear form on $L^2(\R^d)$):
$$
\frac{\delta U}{\delta m}(p)(q,q')
 = \lim_{\varepsilon \rightarrow 0} \frac{1}{\varepsilon}
 \Bigl( 
 \frac{\delta U}{\delta m}(p + \varepsilon q)(q')
 -\frac{\delta U}{\delta m}(p)(q')
\Bigr),  \quad p \in {\mathcal O},\ q,q' \in L^2(\R^d).
 $$

We set, when possible, 
 \be\label{relationsDeri}
D_mU(m,y)=D_y \frac{\delta U}{\delta m}(m,y),  \quad D^2_{mm}U(m,\cdot,y,y')=D^2_{y,y'} \frac{\delta U}{\delta m}(m,y,y').
\ee
To explain the meaning of $D_mU$,  let us compute the action of a vector field on a measure $m$ and the image by $U$. For a given vector field $B:\R^d\to \R^d$ and $m\in {\mathcal P}(\R^d)$ absolutely continuous with a smooth density, let $m(t)=m(x,t)$ be the solution to 
$$
\left\{ \begin{array}{l}
\frac{\partial m}{\partial t}+{\rm div}( Bm)=0\\
m_0=m
\end{array}\right.
$$
This expression directly gives
\be\label{hbqsfnjkd}
\frac{d}{dh}U(m(h))_{|_{h=0}} =  \lg  \frac{\delta U}{\delta m}, -{\rm div}(Bm)\rg_{L^2(\R^d)} = \int_{\R^d} D_mU(m,y)\cdot B(y)\ dm(y), 
\ee
where we used an integration by parts in the last equality. 

Another way to understand these derivatives is to project the map $U$ to the finite dimensional space $(\R^d)^N$ via the empirical measure: if $\bx=(x_1,\dots, x_N)\in (\R^d)^N$, let $m^N_\bx:= (1/N)\sum_{i=1}^N \delta_{x_i}$ and set
$u^N(\bx)= U(m^N_{\boldsymbol x})$. Then one can check the following  relationships (see Proposition \ref{uNC2}): for any $j\in \{1,\dots, N\}$, 
\be\label{e.derivUN1Intro}
D_{x_j}  u^{N}({\bx})= \frac{1}{N} D_mU(m^N_{\bx},x_j),
\ee
\be\label{e.derivUN2Intro}
D^2_{x_j,x_j}  u^{N}({\bx})= \frac{1}{N} D_y\left[D_mU\right](m^{N}_{\bx},x_j)+\frac{1}{N^2} D^2_{mm}U(m^{N}_{\bx},x_j,x_j)  
\ee
while, if $j\neq k$, 
\be\label{e.derivUN3Intro}
D^2_{x_j,x_k}  u^{N}({\bx})= \frac{1}{N^2} D^2_{mm}U(m^{N}_{\bx},x_j,x_k).
\ee

\subsubsection{Formal asymptotic of the $(v^{N,i})$} 

Provided that \eqref{NashIntro} has a unique solution, each $v^{N,i}$, 
for $i=1,\dots,N$, is symmetric with respect to permutations on $\{1,\dots,N\}\backslash \{i\}$ and, for $i \not =j$,
the role played by $x^i$ in $v^{N,i}$
is the same as the role played by $x^j$ in $v^{N,j}$
(see Subsection 
\ref{subsubse:partie3:convergence}).
Therefore, it makes sense to expect, as limit as $N\to+\infty$, 
$$
v^{N,i}(t,\bx) \simeq U(t, x_i,  m^{N,i}_\bx)
$$
where $U:[0,T]\times \R^d\times {\mathcal P}(\R^d)\to \R$. Starting from this \textit{ansatz}, our aim is now to provide heuristic arguments explaining why $U$ should satisfy \eqref{MasterEqIntro}.  The sense in which the $(v^{N,i})_{i=1,\dots,N}$ actually converge to $U$ is stated in Theorem \ref{thm:mainCV} and the proof given in Section \ref{sec.convergence}. 

The informal idea is to assume that $v^{N,i}$ is already of the form $U(t, x_i,  m^{N,i}_\bx)$ and to plug this expression into the equation of the Nash equilibrium \eqref{NashIntro}: the time derivative and the derivative with respect to $x_i$ are understood in the usual sense, while the derivatives with respect to the other variables are computed by using the relations in the previous section. 

The terms $\partial_t v^{N,i}$ and $H(x_i,  D_{x_i}v^{N,i})$ easily become $\frac{\partial U}{\partial t}$ and $H(x,D_xU)$. We omit for a while the second order terms and concentrate on the expression 
$$
\sum_{j\neq i} D_pH (x_j, D_{x_j}v^{N,j})\cdot D_{x_j}v^{N,i}\;.
$$ 
Note that $D_{x_j}v^{N,j}$ is just like $D_xU(t,x_j, m^{N,j}_\bx)$. In view of \eqref{e.derivUN1Intro}, 
$$
D_{x_j}v^{N,i}\simeq \frac{1}{N-1}D_mU(t,x_i,m^{N,i}_\bx,x_j),
$$ 
and the sum over $j$ is like an integration with respect to $m^{N,i}_\bx$. So we find, ignoring the difference between  $m^{N,i}_\bx$ and $m^{N,j}_\bx$, 
$$
\sum_{j\neq i} D_pH (x_j, D_{x_j}v^{N,j})\cdot D_{x_j}v^{N,i}
\simeq
\inte D_pH(y, D_xU(t,m^{N,i}_\bx, y))\cdot D_mU(t,x_i,m^{N,i}_\bx,y) dm^{N,i}_\bx(y).
$$
We now study the term $\ds \sum_{j} \Delta_{x_j}v^{N,i}$. As $\Delta_{x_i}v^{N,i} \simeq \Delta_x U$, we  have to analyze the quantity  $\ds \sum_{j\neq i} \Delta_{x_j}v^{N,i}$. In view of \eqref{e.derivUN2Intro},  we expect
$$
\begin{array}{l}
\ds  \sum_{j\neq i} \Delta_{x_j}v^{N,i} \; \simeq \; \ds \frac{1}{N-1} \sum_{j\neq i} \dive_y\left[D_mU\right](t,x_i,m^{N,i}_\bx,x_j)+\frac{1}{(N-1)^2} \sum_{j\neq i} {\rm Tr}\left[D^2_{mm}U\right](t,x_i,m^{N,i}_\bx,x_j,x_j) \\
\qquad \ds \simeq \inte \dive_y\left[D_mU\right](t,x_i,m^{N,i}_\bx,y)dm^{N,i}_\bx(y)+\frac{1}{N-1} \inte {\rm Tr}\left[D^2_{mm}U\right](t,x_i,m^{N,i}_\bx,y,y)dm^{N,i}_\bx(y),
\end{array}
$$
where we can drop the last term since it is of order $1/N$. 

Let us finally discuss the limit of the term $\ds \sum_{k,l}  {\rm Tr}( \frac{\partial^2 v^{N,i}}{\partial x_k\partial  x_l})$ that we rewrite
\be\label{hgvjfg}
\Delta_{x_i} v^{N,i} 
+
2\sum_{k\neq i} {\rm Tr}(  \frac{\partial }{\partial  x_i}  \frac{\partial v^{N,i}}{\partial x_k})
+
\sum_{k,l\neq i } {\rm Tr}(  \frac{\partial^2 v^{N,i}}{\partial x_k\partial  x_l}) 
\ee
The first term gives $\Delta_x U$. Using \eqref{e.derivUN1Intro} the second one becomes 
$$
\begin{array}{rl}
\ds 2\sum_{k\neq i} {\rm Tr}(  \frac{\partial }{\partial  x_k}  \frac{\partial v^{N,i}}{\partial x_i}) \;
\simeq & \ds 
\frac{2}{N-1} \sum_{k\neq i}{\rm Tr}\left[D_xD_mU\right](t,x_i, m^{N,i}_\bx,x_k) \\
\simeq & \ds 2\inte \dive_x \left[D_mU\right](t,x_i, m^{N,i}_\bx,y)dm^{N,i}_\bx(y). 
\end{array}
$$
As for the last term in \eqref{hgvjfg}, we have by \eqref{e.derivUN3Intro}: 
$$
\begin{array}{rl}
\ds \sum_{k,l\neq i } {\rm Tr}(  \frac{\partial^2 v^{N,i}}{\partial x_k\partial  x_l}) \;  \simeq & \ds  
\frac{1}{(N-1)^2} \sum_{k,l\neq i } {\rm Tr} \left[D^2_{mm}U\right](t,x_i,m^{N,i}_{\bx},x_j,x_k) \\
\simeq & \ds \inte\inte  {\rm Tr} \left[D^2_{mm}U\right](t,x_i,m^{N,i}_{\bx},y,y') dm^{N,i}_\bx (y) dm^{N,i}_\bx (y').
\end{array}
$$
Collecting the above relations, we expect that the Nash system 
$$
\left\{ \begin{array}{l}
\ds -\frac{\partial v^{N,i} }{\partial t}-\sum_{j} \Delta_{x_j}v^{N,i} -\beta \sum_{k,l} {\rm Tr}( \frac{\partial^2 v^{N,i}}{\partial x_k\partial  x_l})
+H(x_i, D_{x_i}v^{N,i}) \\
\ds \qquad \qquad + \sum_{j\neq i} D_pH(x_j, D_{x_j}v^{N,j})\cdot D_{x_j}v^{N,i}=F(x_i,m^{N,i}_\bx)\\
\ds v^{N,i}(T,\bx)= G(x_i, m^{N,i}_{\bx})
\end{array}\right.
$$ 
has for limit
$$
\left\{ \begin{array}{l}
\ds - \frac{\partial U }{\partial t}- \Delta_x U- \int_{\R^d} \dive_y\left[ D_mU\right]dm +H(x, m, D_{x}U)\\
\qquad \ds -\beta \left( \Delta_x U+2 \int_{\R^d}  \dive_x \left[ D_mU\right]dm +  \int_{\R^d} \dive_y \left[ D_mU\right]dm+ \int_{\R^{2d}} {\rm Tr}\left[D^2_{mm}U\right]dm\otimes dm\right)  \\
\ds \qquad \qquad +  \int_{\R^d}  D_mU\cdot D_pH(y,D_xU)dm(y) =F(x,m)\\
\ds U(T,x,m)= G(x, m).
\end{array}\right.
$$
This  is the master equation. Note that there are only two genuine approximations in the above computation. One is where we dropped the term of order $1/N$ in the computation of the sum $\sum_{j\neq i} \Delta_{x_j}v^{N,i}$. The other one was at the very beginning, when we replaced $D_xU(t,x_j, m^{N,j}_\bx)$ by $D_xU(t,x_j, m^{N,i}_\bx)$. This is again of order $1/N$. 

\subsubsection{The master equation and the MFG systems}

We complete this informal discussion by explaining the relationship between the master equation and the MFG systems. This relation
 plays a central role in the paper. It is indeed
 the cornerstone for constructing a solution to the master equation \textit{via} a method of (infinite dimensional) characteristics.
 However, for pedagogical reasons, we here go the other way round: 
 While, in the next sections, we 
 start from the unique solvability of the system of characteristics to prove the existence 
 of a classical solution to the master equation, 
 we now assume for a while that the master equation has a classical solution and, 
 from this solution, we construct a solution to the MFG system. 

Let us start with the first order case, i.e., when $\beta =0$, since this is substantially easier. Let $U$ be the solution to the master equation \eqref{MasterEqIntro} and, for a fixed initial position $(t_0,m_{(0)})\in [0,T]\times {\mathcal P}(\R^d)$, $(u,m)$ be a solution of the MFG system \eqref{e.MFGsyst} with initial condition $m(t_0)=m_{(0)}$. We claim that 
\be\label{chareqINTRO}
\begin{array}{l}
\ds \partial_t m -\Delta m -\dive
\Bigl(mD_pH\bigl(x,D_xU(t,x,m(t))\bigr)\Bigr)=0, \\
\ds u(t,x)=U(t,x,m(t)), \qquad t\in [t_0,T]. 
\end{array}
\ee
In other words, to compute $U(t_0,x,m_{(0)})$, we just need to compute the solution $(u,m)$ of the MFG system \eqref{e.MFGsyst} and let $U(t_0,x,m_{(0)}):=u(t_0,x)$. This is exactly the method of proof of Theorem \ref{theo:ex}. 

To check \eqref{chareqINTRO}, we solve the McKean-Vlasov equation 
$$
\partial_t m'-\Delta m'  -\dive\Bigl(m'D_pH\bigl(x,D_xU(t,x,m'(t))\bigr)\Bigr)=0, \qquad m'(t_0,\cdot)=m_{(0)},
$$
and set $u'(t,x)=U(t,x,m'(t))$. Then 
\be\label{eq:intro:1storder:chainrule:measure}
\begin{array}{rl}
\ds \partial_t u'(t,x)\; = & \ds \partial_t U+ \Big\lg \frac{\delta U}{\delta m},  \partial_tm' \Big\rg_{L^2} = \partial_t U+ \Big\lg \frac{\delta U}{\delta m},  \Delta m'  +\dive\bigl(m'D_pH(\cdot,D_xU)\bigr)\Big\rg_{L^2} \vspace{3mm}\\
= & \ds \partial_t U+ \int_{\R^d}\Bigl(\dive_y\left[D_mU\right]  - D_mU\cdot D_pH(y,D_xU)\Bigr) dm'(y) \vspace{3mm}\\
= & \ds -\Delta_xU +H(x,D_xU)-F(x,m)
\end{array}
\ee
where we used the equation satisfied by $U$ in the last equality. 
Therefore the pair $(u',m')$ is a solution to \eqref{e.MFGsyst}, which, 
provided that the MFG system is at most uniquely solvable, shows that $(u',m')=(u,m)$. \bigskip

For the second order master equation ($\beta>0$) the same principle applies except that, now, the MFG system becomes stochastic. Let $(t_0,m_{(0)})\in [0,T]\times {\mathcal P}(\R^d)$ and $(u_t,m_t)$ be a solution of the MFG system with common noise \eqref{e.MFGstoch2}. 
Provided that the master equation has a classical solution, 
we claim that 
\be\label{chareqINTRObis}
\begin{array}{l}
\ds d_t m_t = \Bigl\{(1+\beta)\Delta m_t +\dive\Bigl(m_tD_pH\bigl(x,D_xU(t,x,m_t)\bigr)
\Bigr)
\Bigr\}dt+\sqrt{2\beta}\dive(m_tdW_t), \\
\ds u_t(x)=U(t,x,m_t), \qquad t\in [t_0,T], \; a.s.. 
\end{array}
\ee
Once again, we stress that this formula (whose derivation here is informal) underpins the rigorous construction of the second order master equation performed in Section \ref{subse:partie:2:first:order:derivative}. As a matter of fact, it says that, in order to define $U(t_0,x,m_{(0)})$
(meaning that $U$ is no more \textit{a priori} given as we assumed a few lines above), one ``just needs" to solve the MFG system \eqref{e.MFGstoch2} with $m_{t_0}=m_{(0)}$ and then set $U(t_0,x,m_{(0)})=u_{t_0}(x)$. Here one faces the additional issue that, so far, there has not been any solvability result for \eqref{e.MFGstoch} and that the regularity of the map $U$ that is defined in this way is much more involved to investigate than in the first order case.
 
Returning to the proof of  
\eqref{chareqINTRObis} (and thus assuming again 
that the master equation has a classical solution),
the argument is the same in the case $\beta=0$, but with extra terms coming from the stochastic contributions. First, we (uniquely) solve the stochastic McKean-Vlasov equation 
$$
\ds d_t m_t' = \left\{(1+\beta)\Delta m_t' +\dive\Bigl(m_t'D_pH\bigl(x,D_xU(t,x,m_t')\bigr)\Bigr)\right\}dt+\sqrt{2\beta}\dive(m_t'dW_t), \qquad m'_{t_0}=m_0,
$$
and set $u_t'(x)=U(t,x,m_t')$. Then, by It\^{o}'s formula,  
\be\label{eq:intro:2ndorder:chainrule:measure}
\begin{array}{rl}
\ds d_t u_t'(x)\; = & \ds \Bigl\{ \partial_t U+ \Bigl\lg  \frac{\delta U}{\delta m}, (1+\beta)\Delta m_t' +\dive\bigl(m_t'D_pH(\cdot,D_{x} U)\bigr) \Bigr\rg_{L^2}
+ \beta \Big\lg \frac{\delta^2 U}{\delta^2 m}Dm_t', Dm_t' \Bigr\rg_{L^2}\Bigr\} dt 
\vspace{4pt}
\\
& \ds \qquad + \Bigl\lg \frac{\delta U}{\delta m}, \sqrt{2\beta}\dive(m_t'dW_t)
\Bigr\rg_{L^2}.
\end{array}
\ee
In comparison with the first-order formula 
\eqref{eq:intro:1storder:chainrule:measure}, 
equation 
\eqref{eq:intro:2ndorder:chainrule:measure}
involves two additional terms: The stochastic term on the second line 
derives directly from the Brownian part in the forward part of 
\eqref{e.MFGstoch2} whilst the second order term on the first line
is reminiscent of the second order term that appears 
in the standard It\^o calculus. We provide 
a rigorous proof of 
\eqref{eq:intro:2ndorder:chainrule:measure}
in Section 
\ref{subse:partie:2:first:order:derivative}.

Using \eqref{relationsDeri}, we obtain 
$$
\begin{array}{rl}
\ds d_t u_t'(x)\; = & \ds \Bigl\{ \partial_t U+ \int_{\R^d} \Bigl((1+\beta)\dive_y \left[D_mU\right]- D_mU\cdot D_pH(\cdot,D_{x}U)\Bigr)dm_t 
\vspace{4pt}
\\
&\hspace{150pt} \ds + \beta\int_{\R^d\times \R^d} {\rm Tr}\bigl[D^2_{mm}U\bigr]dm'_t\otimes m'_t \Bigr\} dt 
\vspace{4pt}
\\
& \ds \qquad +  \Bigl(\int_{\R^d}D_mUdm'_t \Bigr)\cdot \sqrt{2\beta}dW_t
\end{array}
$$
Taking into account the equation satisfied by $U$, we get
$$
\begin{array}{rl}
\ds d_t u_t'(x)\; = & \ds \Bigl\{ -(1+\beta)\Delta_xU +H(\cdot,D_{x}U) -2\beta \int_{\R^d}\dive_x\left[D_mU\right]dm'_t-F\Bigr\}dt 
\vspace{4pt}
\\
& \ds \qquad 
+  \Bigl(\int_{\R^d}D_mUdm'_t \Bigr)\cdot \sqrt{2\beta}dW_t 
\vspace{4pt}
\\
= & \ds  \ds \bigl\{ -(1+\beta)\Delta_xU +H(\cdot,D_{x} U) -2\beta \dive(v_t')-F\bigr\}dt +  v_t'\cdot \sqrt{2\beta}dW_t 
\end{array}
$$
for $\ds v_t':= \int_{\R^d}D_mUdm'_t$. 

This proves that $(u_t', m_t',v_t')$ is a solution to the MFG system \eqref{e.MFGstoch2} and, provided that the MFG system is at most uniquely solvable, proves the claim.

\newpage

\section{Main results}\label{sec:prelim}

In this section we collect our main results. We first state the notation used in the paper, 
specify the notion of derivatives in the space of measures, and describe the assumptions on the data. 

\subsection{Notations}\label{sub.notation}

Throughout the paper, $\R^d$ denotes the $d-$dimensional euclidean space, with norm $|\cdot|$, the scalar product between two vector $a,b\in \R^d$ being written $a\cdot b$. We work in the $d-$dimensional torus (i.e., periodic boundary conditions) that we denote $\T^d:=\R^d/\Z^d$. When $N$ is a (large) integer, we use bold symbols for elements of $(\T^d)^N$: for instance, $\bx=(x_1,\dots, x_N)\in (\T^d)^N$.

The set $\Pk$  of Borel probability measures on $\T^d$ is endowed with the  Monge-Kantorovich distance 
$$
\dk(m,m') =\sup_\phi \inte \phi(y)\ d(m-m')(y),
$$
where the supremum is taken over all Lipschitz continuous maps $\phi:\T^d\to\R$ with a Lipschitz constant bounded by $1$. Let us recall that this distance metricizes the weak convergence of measures. If $m$ belongs to $\Pk$ and $\phi:\T^d\to \T^d $ is a Borel map, then $\phi\sharp m$ denotes the push-forward of $m$ by $\phi$, i.e., the Borel probability measure such that $[\phi\sharp m](A)= m(\phi^{-1}(A))$ for any Borel set $A\subset \T^d$. 
When the probability measure $m$ is absolutely continuous with respect to 
the Lebesgue measure, we use the same letter $m$ to denote 
its 
density. Namely, we write $m : \T^d \ni x \mapsto m(x) \in \R_{+}$. 
Besides we often consider flows of time dependent measures
of the form
$(m(t))_{t\in [0,T]}$, with $m(t) \in \Pk$ for any $t \in [0,T]$. 
When, at each time $t \in [0,T]$, 
$m(t)$ is absolutely continuous with respect to the Lebesgue measure on $\T^d$, 
we identify $m(t)$ with its density and we sometimes
denote by $m : [0,T] \times \T^d \ni (t,x)\mapsto m(t,x) \in \R_{+}$
the collection of the densities. In all the examples considered below,
such an $m$ has a time-space continuous version and, implicitly, we identify $m$ with it. 

If $\phi:\T^d\to \R$ is sufficiently smooth and $\ell=(\ell_1,\dots, \ell_d)\in \N^d$, then $D^\ell \phi$ stands for the derivative $\frac{\partial^{\ell_1}}{\partial x_1^{\ell_1}}\dots\frac{\partial^{\ell_d}}{\partial x_d^{\ell_d}}\phi$. The order of derivation $\ell_1+\dots+\ell_d$ is denoted by $|\ell|$.  Given $e\in \R^d$, we also denote by $\partial_{e} \phi$ the directional derivative of $\phi$ in the direction $e$. 
For $n\in \N$ and $\alpha\in (0,1)$, $\cC^{n+\alpha}$ is the set of maps for which $D^\ell \phi$ is defined and $\alpha-$H\"{o}lder continuous for any $\ell\in \N^d$ with $|\ell|\leq n$. We set 
$$
\|\phi\|_{n+\alpha} := \sum_{|\ell|\leq n} \sup_{x\in \T^d} |D^\ell \phi(x)|+ \sum_{|\ell|=n} \sup_{x\neq x'} \frac{|D^\ell \phi(x)-D^\ell \phi(x')|}{|x-x'|^\alpha}.
$$
The dual space of $\cC^{n+\alpha}$ is denoted by $(\cC^{n+\alpha})'$ with norm 
$$
\forall \rho\in  (\cC^{n+\alpha})', \qquad 
\|\rho\|_{-(n+\alpha)} := \sup_{\|\phi\|_{n+\alpha}\leq 1} \lg \rho, \phi\rg_{(\cC^{n+\alpha})',\cC^{n+\alpha}}. 
$$
If a smooth map $\psi$ depends on two space variables, e.g. $\psi=\psi(x,y)$, and $m,n\in \N$ are the order of derivation of $\psi$ with respect to  $x$ and $y$ respectively, we set 
$$
\|\psi\|_{(m,n)} := \sum_{|\ell|\leq m, |\ell'|\leq n} \|D^{(\ell,\ell')}\psi\|_\infty, 
$$
and, if moreover the derivatives are H\"{o}lder continuous, 
$$ 
\|\psi\|_{(m+\alpha,n+\alpha)} := \|\psi\|_{(m,n)} + \sum_{|\ell|= m, |\ell'|= n} \sup_{(x,y)\neq (x',y')} \frac{|D^{(\ell,\ell')} \phi(x,y)-D^{(\ell,\ell')} \phi(x',y')|}{|x-x'|^\alpha+|y-y'|^\alpha}.
$$
The notation is generalized in an obvious way to mappings depending on 3 or more variables.  

If now the (sufficiently smooth) map $\phi$ depends on time and space, i.e., $\phi=\phi(t,x)$, we say that $\phi\in \cC^{l/2,l}$ (where $l=n+\alpha$, $n\in \N$, $\alpha\in (0,1)$) if $D^\ell D_t^j\phi$ exists for any $\ell\in \N^d$ and $j\in \N$ with $|\ell|+2j\leq n$ and is $\alpha-$H\"{o}lder in $x$ and $\alpha/2-$H\"{o}lder in $t$. We set 
$$
\|\phi\|_{n/2+\alpha/2, n+\alpha}:= \sum_{|\ell|+2j\leq n} \|D^\ell D_t^j\phi\|_\infty + \sum_{|\ell|+2j= n} \lg D^\ell D_t^j\phi\rg_{x,\alpha}+
\lg D^\ell D_t^j\phi\rg_{t,\alpha/2} 
$$ 
with 
$$
\lg D^\ell D_t^j\phi\rg_{x,\alpha}:= \sup_{t, x\neq x'}\frac{|\phi(t,x)-\phi(t,x')|}{|x-x'|^\alpha}, \; 
\lg D^\ell D_t^j\phi\rg_{t,\alpha}:= \sup_{t\neq t', x}\frac{|\phi(t,x)-\phi(t',x)|}{|t-t'|^\alpha}. 
$$

If $X,Y$ are a random variables on a probability space $(\Omega, {\mathcal A},\P)$, ${\mathcal L}(X)$  is the law of $X$ and ${\mathcal L}(Y|X)$ is the conditional law of $Y$ given $X$. 
Recall that, whenever $X$ and $Y$ take values in Polish spaces (say ${\mathcal S}_{X}$
and ${\mathcal S}_{Y}$ respectively), 
we can always find a regular version of the conditional law ${\mathcal L}(Y|X)$, 
that is 
a mapping $q : {\mathcal S}_{X} \times {\mathcal B}({\mathcal S}_{Y}) 
\rightarrow [0,1]$ such that:
\begin{itemize}
\item
 for 
each $x \in {\mathcal S}_{X}$, $q(x,\cdot)$ is a probability measure on 
${\mathcal S}_{Y}$ equipped with its Borel $\sigma$-field ${\mathcal B}({\mathcal S}_{Y})$, 
\vspace{-10pt}
\item
for any $A \in {\mathcal B}({\mathcal S}_{Y})$, 
the mapping ${\mathcal S}_{X} \ni x \mapsto q(x,A)$ is Borel measurable, 
\vspace{-10pt}
\item
$q(X,\cdot)$ is a version of the conditional law of $X$ given $Y$, in the sense that
\begin{equation*}
\E \bigl[ f(X,Y) \bigr] = \int_{{\mathcal S}_{X}}
\biggl( \int_{{\mathcal S}_{Y}} 
f(x,y)q(x,dy) \biggr) d \bigl( {\mathcal L}(X) \bigr) (x)
= 
\E \biggl[ \int_{{\mathcal S}_{Y}}
f(X,y) q(X,dy) \biggr], 
\end{equation*}
for any bounded Borel measurable mapping $f : {\mathcal S}_{X} \times {\mathcal S}_{Y}
\rightarrow \R$. 
\end{itemize}

\subsection{Derivatives}
\label{subsec.Derivatives}

One of the striking features of the master equation is that it involves derivatives of the unknown with respect to the measure. In the paper, we use two notions of derivatives.  The first one, denoted by $\frac{\delta U}{\delta m}$ is, roughly speaking, the $L^2$ derivative when one looks at  the restriction of $\Pk$ to densities in $L^2(\T^d)$. It is widely used in linearization procedures. The second one, denoted by $D_mU$,  is more intrinsic and is related with the so-called Wasserstein metric on $\Pk$. It can be introduced as in  Ambrosio, Gigli and Savaré \cite{AGS} by defining a kind of manifold structure on $\Pk$ or,  as in Lions \cite{LcoursColl}, by embedding $\Pk$ into an $L^2(\Omega, \T^d)$ space of random variables. We introduce this notion here in a slightly different way, as the derivative in space of $\frac{\delta U}{\delta m}$. In appendix we briefly compare the different notions. 

\subsubsection{First order derivatives}

\begin{Definition}\label{def:Diff} We say that $U:\Pw\to \R$ is $\cC^{1}$ if there exists a continuous  map $\ds \frac{\delta U}{\delta m}:\Pw\times \T^d\to \R$ such that, for any $m,m'\in \Pk$,  
$$
\lim_{s\to 0^+} \frac{U((1-s)m+sm')-U(m)}{s}  = \inte \frac{\delta U}{\delta m}(m, y) d(m'-m)(y).
$$
\end{Definition}

Note that $\frac{\delta U}{\delta m}$ is defined up to an additive constant. We adopt the normalization convention
\be\label{ConvCondDeriv}
\inte \frac{\delta U}{\delta m}(m,y)dm(y)=0.
\ee
For any $m\in \Pk$ and any signed measure $\mu$ on $\T^d$, we will use indifferently the notations
$\ds \frac{\delta U}{\delta m}(m)(\mu)$ and $\ds \inte \frac{\delta U}{\delta m}(m,y)d\mu(y)$. 

Note also that 
\be\label{e.iubsdiazsd}
\forall m,m'\in \Pk, \qquad U(m')-U(m) = \int_0^1\inte \frac{\delta U}{\delta m}((1-s)m+sm',y)\ d(m'-m)(y)ds.
\ee

Let us explain the relationship between the derivative in the above sense and the Lipschitz continuity of $U$ in $\Pk$. 
If $\ds \frac{\delta U}{\delta m}= \frac{\delta U}{\delta m}(m,y)$ is Lipschitz continuous with respect to the second variable with a Lipschitz constant bounded independently of $m$, then $U$ is Lipschitz continuous: indeed, by \eqref{e.iubsdiazsd}, 
$$
\begin{array}{rl}
\ds \left|U(m')-U(m)\right|\; \leq & \ds   \int_0^1\left\|D_y \frac{\delta U}{\delta m}((1-s)m+sm',\cdot)\right\|_\infty ds\  \dk(m,m')\\
 \leq & \ds \sup_{m''} \left\|D_y \frac{\delta U}{\delta m}(m'',\cdot)\right\|_\infty \  \dk(m,m') .  
\end{array}
$$
This leads us to define  the ``intrinsic derivative" of $U$.

\begin{Definition} 
\label{def:intrinsic:derivative}
 If $\ds \frac{\delta U}{\delta m}$ is of class $\cC^1$ with respect to the second variable, 
the intrinsic derivative  $D_mU:\Pw\times \T^d\to \R^d$ is defined by 
$$
D_mU(m,y):= D_y \frac{\delta U}{\delta m}(m,y)
$$
\end{Definition}

The expression $D_mU$ can be understood as a derivative of $U$ along vector fields: 

\begin{Proposition}\label{prop:DmUderiv} Assume that $U$ is $\cC^1$, with $\frac{\delta U}{\delta m}$ $\cC^1$ with respect to $y$ and $D_mU$ is continuous in both variables. Let $\phi:\T^d\to\R^d$ be a Borel measurable and bounded vector field.  Then 
$$
\lim_{h\to 0} \frac{U((id+h\phi)\sharp m)-U(m)}{h} = \inte  D_mU(m,y)\cdot \phi(y)\ dm(y).
$$
\end{Proposition}

\begin{proof} Let us set $m_{h,s}:= s(id+ h \phi)\sharp m + (1-s) m $. Then 
$$
\begin{array}{rl}
\ds U((id+h\phi)\sharp m)-U(m) \; = & \ds \int_0^1 \inte \frac{\delta U}{\delta m}(m_{h,s},y)d((id+h\phi)\sharp m-m)(y)ds \\
 = & \ds \int_0^1 \inte ( \frac{\delta U}{\delta m}(m_{h,s},y+h\phi(y))- \frac{\delta U}{\delta m}(m_{h,s},y)) dm(y)ds\\
 = & \ds h \int_0^1 \inte \int_0^1 D_m U (m_{h,s},y+t h\phi(y))\cdot \phi(y)\ dtdm(y)ds.
\end{array}
$$
Dividing by $h$ and letting $h\to 0$ gives the result thanks to the continuity of $D_mU$. 
\end{proof}

 Note also that, if $U:\Pk\to \R$ and $\frac{\delta U}{\delta m}$ is $\cC^2$ in $y$, then $D_yD_mU(m,y)$ is a symmetric matrix since
$$
D_yD_mU(m,y)= D_y \left(D_y \frac{\delta U}{\delta m}\right)(m,y) = {\rm Hess_y} \frac{\delta U}{\delta m}(m,y).
$$

\subsubsection{Second order derivatives.} 

If, for a fixed $y\in \T^d$, the map $\ds m\mapsto \frac{\delta U}{\delta m}(m,y)$ is $\cC^1$, then we say that $U$ is $\cC^2$ and denote by $\ds \frac{\delta^2 U}{\delta m^2}$ its derivative. (Pay attention that $y$ is fixed. At this stage, nothing is said about 
the smoothness in the direction $y$.)
By Definition \ref{def:Diff} we have that $\ds \frac{\delta^2 U}{\delta m^2}:\Pw\times \T^d\times \T^d\to \R$ with
$$
\frac{\delta U}{\delta m}(m',y)-\frac{\delta U}{\delta m}(m,y)= \int_0^1\inte \frac{\delta^2 U}{\delta m^2}((1-s)m+sm',y,y')\ d(m'-m)(y').
$$
If $U$ is $\cC^2$ and if $\ds \frac{\delta^2 U}{\delta m^2}= \frac{\delta^2 U}{\delta m^2}(m,y,y')$ is $\cC^2$ in the variables $(y,y')$, then we set
$$
D^2_{mm}U(m,y,y'):= D^2_{y,y'}  \frac{\delta^2 U}{\delta m^2}(m,y,y').
$$ 
We note that $D^2_{mm}U:\Pk\times \T^d\times \T^d \to \R^{d\times d}$. The next statement asserts that $\ds \frac{\delta^2 U}{\delta m^2}$ enjoys the classical symmetries of second order derivatives. 

\begin{Lemma} 
\label{lem:schwarz}
 Assume that $\ds \frac{\delta^2 U}{\delta m^2}$ is jointly continuous in all the variables. Then 
$$
\frac{\delta^2 U}{\delta m^2}(m,y,y')= \frac{\delta^2 U}{\delta m^2}(m, y',y), \quad m \in {\mathcal P}(\T^d), \ y,y' \in \T^d.
$$
In the same way, if 
$\ds \frac{\delta U}{\delta m}$ is $\cC^1$ in the variable $y$
and
$\ds \frac{\delta^2 U}{\delta m^2}$ is also $\cC^1$ in the variable $y$,
$\ds D_{y} \frac{\delta^2 U}{\delta m^2}$
being jointly continuous in all the variables, then, 
for any fixed $y \in \T^d$, 
the map 
$m \mapsto D_{m} U(m,y)$ is ${\mathcal C}^1$ 
and  
$$
D_y\frac{\delta^2 U}{\delta m^2}(m,y,y') = \frac{\delta}{\delta m} \bigl(D_mU(m,y)\bigr)(y'),
\quad m \in {\mathcal P}(\T^d), \ y,y' \in \T^d,
$$
while, if $\ds \frac{\delta^2 U}{\delta m^2}$ is also $\cC^2$ in the variables $(y,y')$, then,
for any fixed $y \in \T^d$, the map 
$\ds \frac{\delta}{\delta m}(D_{m} U(\cdot,y))$ is ${\mathcal C}^1$ in the variable $y'$ and
$$
D_m\bigl(D_mU(\cdot,y)\bigr)(m,y') = D^2_{mm}U(m,y, y').
$$
\end{Lemma}

\begin{proof} 
\textit{First step.}
We start with the proof of the first claim.
By continuity, we just need to show the result when $m$ has a smooth positive density. Let $\mu,\nu\in L^\infty(\T^d)$, such that $\int_{\T^d} \mu=\int_{\T^d} \nu=0$,
with a small enough norm so that 
$m+s \mu + t \nu$ is a probability measure for any $(s,t) \in [0,1]^2$. 

Since $U$ is ${\mathcal C}^2$, the mapping 
${\mathcal U} : [0,1]^2 \in (s,t) \mapsto U(m + s \mu + n \nu)$
is twice differentiable and, by standard Schwarz' Theorem,
$D_{t} D_{s} {\mathcal U}(s,t) = 
D_{s} D_{t} {\mathcal U}(s,t)$, for any $(s,t) \in [0,1]^2$.
Notice that 
\begin{equation*}
\begin{split}
&D_{t}
D_{s} {\mathcal U}(s,t)
= \int_{[\T^d]^2} \frac{\delta^2 U}{\delta m^2}
\bigl( m +  s \mu + t \nu, y,y' \bigr) \mu(y) \nu(y') dy dy' 
\\
&D_{s}
D_{t} {\mathcal U}(s,t)
= \int_{[\T^d]^2} \frac{\delta^2 U}{\delta m^2}
\bigl( m +  s \mu + t \nu, y',y \bigr) \mu(y) \nu(y') dy dy'. 
\end{split}
\end{equation*}
Choosing $s=t=0$,
the first claim easily follows. 

\textit{Second step.}
The proof is the same for the second assertion, except that now we have to consider the 
mapping
$\ds {\mathcal U}' : [0,1] \times \T^d \ni (t,y) \mapsto 
\frac{\delta U}{\delta m}(m + t \mu,y)$, 
for a general probability measure $m \in {\mathcal P}(\T^d)$
and a general finite signed measure $\mu$ on $\T^d$, 
such that 
$\mu(\T^d)=0$ and $m + \mu$ is a probability measure. 
(In particular, $m + t \mu = (1-t) m + t(m + \mu)$ is also a probability 
measure for any $t \in [0,1]$.) 
By assumption, ${\mathcal U}'$ is 
${\mathcal C}^1$ in each variable 
$t$ and $y$ with
\begin{equation*}
D_{t} {\mathcal U}'(t,y) = \int_{\T^d}
\frac{\delta^2 U}{\delta m^2}(m+t \mu,y,y') d\mu(y'),
\quad D_{y} {\mathcal U}'(t,y)
= D_{m} U(m + t\mu,y). 
\end{equation*}
In particular, 
$D_{t} {\mathcal U}'$ 
is ${\mathcal C}^1$ in $y$
and 
\begin{equation*}
D_{y}
D_{t} {\mathcal U}'(t,y) = 
\int_{\T^d}
D_{y} \frac{\delta^2 U}{\delta m^2}(m+t \mu,y,y') \mu(y') dy'.
\end{equation*} 
By assumption, $D_{y} D_{t} {\mathcal U}'$ is jointly continuous 
and, by standard Schwarz' Theorem, 
the mapping $D_{y} {\mathcal U}'$
is differentiable in $t$, with
\begin{equation*}
D_{t} \bigl( D_{y} {\mathcal U}' \bigr)(t,y)
=
D_{t} \bigl( 
D_{m} U(m + t\mu,y) \bigr)
= \int_{\T^d}
D_{y} \frac{\delta^2 U}{\delta m^2}(m+t \mu,y,y') \mu(y') dy'.
\end{equation*}
Integrating in $t$, this shows that 
\begin{equation*}
D_{m} U(m + \mu,y)
- D_{m} U(m,y) 
= \int_{0}^1 
\int_{\T^d}
D_{y} \frac{\delta^2 U}{\delta m^2}(m+t \mu,y,y') 
\mu(y') dy'
dt.
\end{equation*}
Choosing $\mu = m'-m$, for another probability measure 
$m' \in {\mathcal P}(\T^d)$
and noticing that
(see Remark 
\ref{rem:convention:D^2m} below):
 $$\int_{\T^d}
D_{y}\frac{\delta^2 U}{\delta m^2}(m,y,y') 
dm(y')=0,$$ we complete the proof of the 
second claim. 

%
For the last assertion, one just need to take the derivative in $y$ in the second one. 
\end{proof}

\begin{Remark}
\label{rem:convention:D^2m}
Owing to the convention 
\eqref{ConvCondDeriv}, 
we have
\begin{equation*}
\forall y \in \T^d, \quad \int_{\T^d} \frac{\delta^2 U}{\delta m^2}(m,y,y') d 
m(y') = 0,
\end{equation*}
when $U$ is ${\mathcal C}^2$. 
By symmetry, we also have
\begin{equation*}
\forall y' \in \T^d, \quad \int_{\T^d} \frac{\delta^2 U}{\delta m^2}(m,y,y') d 
m(y) = 0.
\end{equation*}
And, of course,
\begin{equation*}
\int_{[\T^d]^2} \frac{\delta^2 U}{\delta m^2}(m,y,y') 
d 
m(y)
d 
m(y') = 0.
\end{equation*}
\end{Remark}

\subsubsection{Comments on the notions of derivatives} 

Since several concepts of derivatives have been used in the mean field game theory, we now discuss the link between these notions. For simplicity, we argue as if our state space was $\R^d$ and not $\T^d$, since most results have been stated in this context.
(We refer to the Appendix for an exposition on $\T^d$.) 

A first idea consists in looking at the restriction of the map $U$ to the subset of measures with a density which is in $L^2(\R^d)$, and take the derivative of $U$ in the $L^2(\R^d)$ sense. This is partially the point of view adopted by Lions in \cite{LcoursColl} and followed by Bensoussan, Frehse and Yam \cite{BFY}. In the context of smooth densities, this is closely related to our first and second derivatives $\ds \frac{\delta U}{\delta m}$ and  $\ds \frac{\delta^2 U}{\delta m^2}$.

Many works on mean field games (as in Buckdahn, Li, Peng and Rainer \cite{BuLiPeRa},  Carmona and Delarue \cite{CaDe14}, 
Chassagneux, Crisan and Delarue \cite{CgCrDe}, Gangbo and Swiech \cite{GS14-2}) make use of an idea introduced by Lions in \cite{LcoursColl}. It consists in working in a sufficiently large probability space $(\Omega, {\mathcal A},\P)$ and in looking at maps $U:{\mathcal P}(\R^d)\to \R$ through their lifting to $L^2(\Omega,{\mathcal A},\P; \R^d)$ defined by 
$$
\widetilde U(X)= U({\mathcal L}(X)) \qquad \forall X\in L^2(\Omega, \R^d),
$$
where ${\mathcal L}(X)$ is the law of $X$. It is clear that the derivative of $\widetilde U$---if it exists---enjoys special properties because $\widetilde U(X)$  depends only on the law of $X$ and not on the full random variable. 
As explained in \cite{LcoursColl}, if $\widetilde U$ is differentiable at some point $X_0\in L^2(\Omega,{\mathcal A},\P; \R^d)$, then its gradient can be written as 
$$
\nabla \widetilde U(X_0) = \partial_{\mu} U({\mathcal L}(X_0))(X_0),
$$
where $\partial_{\mu} U:{\mathcal P}(\R^d)\times \R^d \ni (m,x) \mapsto 
\partial_{\mu} U(m)(x) \in 
\R^d$. We explain in the Appendix that the maps $\partial_{\mu} U$ and $D_mU$ introduced in Definition \ref{def:intrinsic:derivative} coincide, as soon as one of the two derivatives exists. Let us also underline that this concept of derivative is closely related with the notion introduced by Ambrosio, Gigli and Savaré \cite{AGS} in a more general setting.

\subsection{Assumptions}
\label{subsec:hyp}

Throughout the paper, we assume that $H:\T^d\times \R^d\to \R$ is smooth, globally Lipschitz continuous and satisfies the coercivity condition: 
\be
\label{HypD2H}
C^{-1}\frac{ I_d }{1+|p|}\leq D^2_{pp}H(x,p) \leq CI_d \qquad {\rm for }\; (x,p)\in \T^d\times \R^d.
\ee
We also always assume that the maps $F,G:\T^d\times \Pk\to\R$ are globally Lipschitz continuous and monotone: for any $m,m'\in \Pk$, 
\be\label{e.monotoneF}
\inte (F(x,m)-F(x,m'))d(m-m')(x)\geq 0, \; \inte (G(x,m)-G(x,m'))d(m-m')(x)\geq 0.
\ee
Note that assumption \eqref{e.monotoneF} implies that $\frac{\delta F}{\delta m}$ and $\frac{\delta G}{\delta m}$ satisfy the following monotonicity property (explained for $F$): 
$$
\inte\inte \frac{\delta F}{\delta m}(x,m,y)\mu(x)\mu(y)dxdy\geq 0
$$
for any centered measure $\mu$. Throughout the paper the conditions \eqref{HypD2H} and \eqref{e.monotoneF} are in force. \\

Next we describe assumptions that might differ according to the results. Let us fix 
$n\in \N$ and $\alpha\in (0,1)$.  We set (with the notation introduced in subsection \ref{sub.notation})
$$
{\rm Lip}_n(\frac{\delta F}{\delta m}):= \sup_{m_1\neq m_2}\left(\dk(m_1,m_2)\right)^{-1}
\left\| \frac{\delta F}{\delta m}(\cdot,m_1,\cdot)-\frac{\delta F}{\delta m}(\cdot,m_2,\cdot)\right\|_{(n+\alpha,n+\alpha)}
$$
and use the symmetric notation for $G$.
We call {\bf (HF1(${\boldsymbol n}$))} the following regularity conditions on $F$: 
$$
{\rm {\bf (HF1({\boldsymbol n}))}} \qquad \sup_{m\in \Pk} \left(\left\|F(\cdot, m)\right\|_{n+\alpha}+ \left\|\frac{\delta F(\cdot, m,\cdot)}{\delta m}\right\|_{(n+\alpha, n+\alpha)}\right)+ {\rm Lip}_n(\frac{\delta F}{\delta m})\; <\; \infty.
$$
and  {\bf (HG1(${\boldsymbol n}$))} the symmetric condition on $G$: 
$$
{\rm {\bf (HG1({\boldsymbol n}))}} \qquad \sup_{m\in \Pk} \left(\left\|G(\cdot, m)\right\|_{n+\alpha}+ \left\|\frac{\delta G(\cdot, m,\cdot)}{\delta m}\right\|_{(n+\alpha, n+\alpha)}\right)+ {\rm Lip}_n(\frac{\delta G}{\delta m})\; <\; \infty.
$$
We use similar notation when dealing with second order derivatives: 
$$
\ds {\rm Lip}_n(\frac{\delta^2 F}{\delta m^2}):= \sup_{m_1\neq m_2}\left(\dk(m_1,m_2)\right)^{-1} 
\left\| \frac{\delta^2 F}{\delta m^2}(\cdot,m_1,\cdot,\cdot)-\frac{\delta^2 F}{\delta m^2}(\cdot,m_2,\cdot,\cdot)\right\|_{(n+\alpha, n+\alpha, n+\alpha)}
$$
and call {\bf (HF2(${\boldsymbol n}$))} (respectively {\bf (HG2(${\boldsymbol n}$))}) the second order regularity conditions on $F$: 
$$
\begin{array}{rl}
\ds {\rm {\bf (HF2({\boldsymbol n}))}} \; & \ds  \sup_{m\in \Pk} \left(\left\|F(\cdot, m)\right\|_{n+\alpha}+ \left\|\frac{\delta F(\cdot, m,\cdot)}{\delta m}\right\|_{(n+\alpha, n+\alpha)}\right)\\
& \ds \qquad  +\sup_{m\in \Pk}\left\|\frac{\delta^2 F(\cdot, m,\cdot,\cdot)}{\delta m^2}\right\|_{(n+\alpha,n+\alpha,n+\alpha)}
+ {\rm Lip}_n(\frac{\delta^2 F}{\delta m^2})\; <\; \infty.
\end{array}$$
and on $G$: 
$$
\begin{array}{rl}
\ds {\rm {\bf (HG2({\boldsymbol n}))}} \; & \ds  \sup_{m\in \Pk} \left(\left\|G(\cdot, m)\right\|_{n+\alpha}+ \left\|\frac{\delta G(\cdot, m,\cdot)}{\delta m}\right\|_{(n+\alpha,n+\alpha)}\right)\\
& \ds \qquad  +\sup_{m\in \Pk}\left\|\frac{\delta^2 G(\cdot, m,\cdot,\cdot)}{\delta m^2}\right\|_{(n+\alpha,n+\alpha,n+\alpha)}
+ {\rm Lip}_n(\frac{\delta^2 G}{\delta m^2})\; <\; \infty.
\end{array}$$

\begin{Example}
\label{ex:monotonicity:F:G}
{\rm Assume that $F$ is of the form: 
$$
F(x, m)= \int_{\R^d} \Phi(z, (\rho\star m)(z))\rho(x-z)dz, 
$$
where $\star$ denotes the usual convolution product (in $\R^d$) and where $\Phi:\R^2\to \R$ is a smooth map which is nondecreasing with respect to the second variable and $\rho$ is a smooth, even function with compact support. Then $F$ satisfies the monotonicity condition \eqref{e.monotoneF} as well as the regularity conditions {\bf (HF1(}${\boldsymbol n}${\bf))} and {\bf (HF2(}${\boldsymbol n}${\bf))} for any $n\in \N$. 
}\end{Example}

\begin{proof} Let us first note that, for any $m,m'\in \Pk$, 
\begin{multline*}
\inte (F(x,m)-F(x,m'))d(m-m')(x)\\ 
= \inte \left[\Phi(y, \rho\star m(y))-\Phi(y, \rho\star m'(y))\right] \left(\rho\star m(y)-\rho\star m'(y)\right)dy \geq 0,
\end{multline*}
since $\rho$ is even and $\Phi$ is nondecreasing with respect to the  second variable. So $F$ is monotone. Writing $\Phi=\Phi(x,\theta)$, the derivatives of $F$ are given by 
$$
\frac{\delta F}{\delta m}(x,m,y)= \int_{\R^d} \frac{\partial \Phi}{\partial \theta}
\bigl(z,\rho\star m(z)\bigr)\rho(x-z) \rho(z-y) dz 
$$
and
%
$$
\frac{\delta^2 F}{\delta m^2}(x,m,y,y')= \int_{\R^d} \frac{\partial^2 \Phi}{\partial \theta^2}\bigl(z,\rho\star m(z)\bigr)
\rho(z-y)\rho(z-y')
\rho(x-z)dz.
$$
Then {\bf (HF1(}${\boldsymbol n}${\bf))} and {\bf (HF2(}${\boldsymbol n}${\bf))} hold because of the smoothness of $\rho$. 
\end{proof}

\subsection{Statement of the main results}

The paper contains two main results: on the one hand the well-posedness of the master equation, and, on the other hand, the convergence of the Nash system with $N$ players as $N$ tends to infinity. We start by considering the first order master equation ($\beta=0$), because, in this setting, the approach is relatively simple (Theorem \ref{theo:ex}). In order to handle the second order master equation, we build solutions to the mean field game system with common noise, which play the role of ``characteristics" for the master equation (Theorem \ref{thm:partie:2:existence:uniquenessINTRO}). Our first main result is Theorem \ref{theo:2nd-order:master:equation}, which states that the master equation has a unique classical solution under our regularity and monotonicity assumptions on $H$, $F$ and $G$. Once we know that the master equation has a solution, we can use this solution to build approximate solutions for the Nash system with $N-$players. This yields to our main convergence results, either in term of functional terms (Theorem \ref{thm:mainCV}) or in term of optimal trajectories (Theorem \ref{thm:CvMFG}). 

\subsubsection{First order master equation}

We first consider the first order master equation (or master equation  without common noise): 
\be\label{MFGf}
\begin{array}{l}
\left\{\begin{array}{l} 
\ds - \partial_t U(t,x,m)  - \Delta_x U(t,x,m) +H(x,D_xU(t,x,m)) -\inte \dive_y\left[D_m U\right](t,x,m,y) \ d m(y)\\
\ds    \qquad+ \inte  D_m U(t,y,m,y)\cdot D_pH(y,D_xU(t,y,m))\ dm(y) =F(x,m), \\
 \ds  \qquad\qquad\qquad \qquad {\rm in }\; [0,T]\times \T^d\times \Pw,\\
 \;\\
 U(T,x,m)= G(x,m) \qquad  {\rm in }\; \T^d\times \Pw.\\
\end{array}\right.
\end{array}
\ee
We call it the first order master equation since it only contains first order derivatives with respect to the measure variable. 
Let us first explain the notion of solution. 

\begin{Definition} 
\label{def:master:eq:1st:order}
We say that a map $U:[0,T]\times \T^d\times \Pw\to \R$ is a classical solution to the first order master equation if
\begin{itemize}
\item $U$ is continuous in all its arguments (for the $\dk$ distance on $\Pw$), is of class $\cC^2$ in $x$ and $\cC^1$ in time (the derivatives of order one in time and space 
and of order two in space being continuous in all the arguments),  

\item $U$ is of class $\cC^1$ with respect to $m$, 
the first order derivative
\begin{equation*}
\begin{split}
&[0,T] \times \T^d \times \Pk 
\times \T^d \ni
(t,x,m,y) \mapsto \frac{\delta U}{\delta m}(t,x,m,y),
\end{split}
\end{equation*}
being continuous in all the arguments, 
$\delta U/\delta m$ 
being
 twice differentiable in $y$, the derivatives being continuous 
in all the arguments, 

\item $U$ satisfies the master equation
\eqref{MFGf}. 
\end{itemize}
\end{Definition}

\begin{Theorem}\label{theo:ex} Assume that $F$, $G$ and $H$ satisfy  \eqref{HypD2H} and  \eqref{e.monotoneF} in Subsection \ref{subsec:hyp}, 
and that  {\bf (HF1(${\boldsymbol n}$+1))} and {\bf (HG1(${\boldsymbol n}$+2))} hold for some $n\geq 1$
and some $\alpha \in (0,1)$.
Then the first order master equation \eqref{MFGf} has a unique solution. 

Moreover, $U$ is $\cC^1$ (in all variables), $\frac{\delta U}{\delta m}$ is continuous in all variables and $U(t,\cdot,m)$ and $\frac{\delta U}{\delta m}(t,\cdot, m, \cdot)$ are bounded in $\cC^{n+2+\alpha}$ and $\cC^{n+2+\alpha}\times \cC^{n+1+\alpha}$ respectively, independently of $(t,m)$. Finally,  $\frac{\delta U}{\delta m}$ is Lipschitz continuous with respect to the measure variable: 
$$
\sup_{t\in[0,T]}  \sup_{m_1\neq m_2}\left(\dk(m_1,m_2)\right)^{-1}
\left\| \frac{\delta U}{\delta m}(t,\cdot,m_1,\cdot)-\frac{\delta U}{\delta m}(t,\cdot,m_2,\cdot)\right\|_{(n+2+\alpha,n+\alpha)} \; <\; \infty.
$$
\end{Theorem}

Section \ref{sec:MasterWithoutCN} is devoted to the proof of Theorem \ref{theo:ex}. We also discuss in this section the link between the solution $U$ and the derivative of the solution of a Hamilton-Jacobi equation in the space of measure. 

The proof of Theorem \ref{theo:ex} relies on the representation of the solution in terms of the mean field game system: for any $(t_0,m_0)\in [0,T)\times \Pw$, the MFG system is the system of forward-backward equations: 
\be\label{MFGSec2}
\left\{ \begin{array}{l}
\ds - \partial_t u - \Delta u +H(x,Du)=F(x,m(t))\\
\ds \partial_t m - \Delta m -{\rm div}( m D_pH(x, D u))=0 \\
\ds u(T,x)=G(x,m(T)), \; m(t_0, \cdot)=m_0
\end{array}\right.
\ee
As recalled below (Proposition \ref{prop:reguum}), under suitable assumptions on the data, there exists a unique solution $(u,m)$ to the above system. Our aim is to show that the map $U$ defined by 
\be\label{defUSec2}
U(t_0,\cdot, m_0):= u(t_0,\cdot)
\ee
is a solution to \eqref{MFGf}. The starting point is the obvious remark that, for $U$ defined by \eqref{defUSec2} and for any $h\in [0,T-t_0]$,  
$$
u(t_0+h, \cdot)= U(t_0+h,\cdot, m(t_0+h)).
$$
Taking the derivative with respect to $h$ and letting $h=0$ shows that $U$ satisfies \eqref{MFGf}. 

The main issue is to prove that the map $U$ defined by \eqref{defUSec2} is sufficiently smooth to perform the above computation. In order to prove the differentiability of the map $U$, we use a flow method and differentiate the MFG system \eqref{MFGSec2} with respect to the measure argument $m_{0}$. The derivative system then reads as a \textit{linearized system}
initialized with a signed measure. 
Fixing a solution $(u,m)$ to \eqref{MFGSec2} and  
allowing for a more singular initial 
distribution
$\mu_0\in (\cC^{n+1+\alpha}(\T^d))'$ (instead of a signed measure),
the \textit{linearized system}, with $(v,\mu)$ as unknown, takes the form:
$$
\left\{ \begin{array}{l}
\ds - \partial_t v - \Delta v + D_pH(x,Du)\cdot Dv =\frac{\delta F}{\delta m}\bigl(x,m(t)\bigr)(\mu(t))\\
\ds \partial_t \mu - \Delta \mu -{\rm div}\bigl( \mu  D_pH(x, D u)\bigr)-{\rm div}\bigl( m  D^2_{pp}H(x, D u)Dv\bigr)=0 \\
\ds v(T,x)=\frac{\delta G}{\delta m}\bigl(x,m(T)\bigr)(\mu(T)), \; \mu(t_0,\cdot)=\mu_0.
\end{array}\right.
$$
We prove that $v$ can be interpreted as the directional derivative of $U$ in the direction $\mu_0$: 
$$
v(t_0,x)= \inte \frac{\delta U}{\delta m}(t_0,x,m_0,y) \mu_0(y)dy.
$$
Note that this shows at the same time the differentiability of $U$ and the regularity of its derivative. For this reason the introduction of the directional derivative appears extremely useful in this context.

\subsubsection{The mean field game system with common noise}

As explained in the previous subsection, the characteristics of the first order master equation \eqref{MFGf} are the solution to the mean field game system \eqref{MFGSec2}. The analogous construction for the second order master equation (with $\beta>0$) yields to a system of stochastic partial differential equations, the  {\it mean field game system with common noise}. Given an initial distribution ${m}_{0} \in \Pk$ at an initial time $t_0\in [0,T]$, this system 
reads\footnote{
In order to emphasize the random nature of the 
functions $u$ and $m$, the time variable 
is now indicated as an index, as often done in the theory of stochastic processes.
}
\be\label{e.MFGsto}
\left\{
\begin{array}{l}
 d_{t} u_{t} = \bigl\{ -  (1+\beta) \Delta u_{t} + H(x,Du_{t}) - F(x,m_{t}) -  2\beta {\rm div}(v_{t}) \bigr\} dt+ v_{t} \cdot \sqrt{2\beta} dW_{t},\\
 d_{t} m_{t} = \bigl[  (1+\beta) \Delta m_{t} + {\rm div} \bigl( m_{t} D_{p} H(m_{t},D u_{t}) 
\bigr) \bigr] dt - \sqrt{2\beta} {\rm div} ( m_{t} dW_{t} \bigr), \qquad {\rm in}\; [t_0,T]\times \T^d,\\
 m_{t_0}=m_0, \; u_T(x)= G(x, m_T) \qquad {\rm in}\; \T^d.
\end{array}
\right.
\ee
Here $(W_t)_{t\in [0,T]}$ is a given $d-$dimensional Brownian motion, generating a  filtration $({\mathcal F}_t)_{t\in [0,T]}$. The solution is the process $(u_t,m_t,v_t)_{t\in [0,T]}$, adapted to $({\mathcal F}_t)_{t\in [t_0,T]}$, where, for each $t\in [t_0,T]$, $v_t$ is a vector field which ensures the solution $(u_t)$ to the backward equation to be adapted to the filtration $({\mathcal F}_t)_{t\in [t_0,T]}$. 
Up to now, the  well-posedness  of this system has never been investigated, 
but it is reminiscent of the theory of forward-backward stochastic 
differential equations in finite dimension, see for instance 
the monograph \cite{PardouxRascanu}.

To analyze \eqref{e.MFGsto}, we take advantage of the additive structure of the common noise and perform the (formal) change of variable
\begin{equation*}
\tilde{u}_{t}(x) = u_{t}(x+ \sqrt{2\beta} W_{t}), 
\quad \tilde{m}_{t}(x) = m_{t}(x+ \sqrt{2\beta} W_{t}), \quad x \in \T^d, \quad t \in [0,T]. 
\end{equation*}
Setting $\tilde{H}_{t}(x,p) = H(x+ \sqrt{2}W_{t},p)$, $\tilde{F}_{t}(x,m) = F(x+\sqrt{2} W_{t},m)$ and $\tilde{G}_{t}(x,m) = G(x+\sqrt{2} W_{t},m)$
and invoking 
the  It\^o-Wentzell formula (see Section 
\ref{se:common:noise} for a more precise account), 
the pair $(\tilde{u}_t, \tilde{m}_{t})_{t\in [t_0,T]}$ formally satisfies 
the system
\begin{equation}
\label{eq:se:3:tilde:HJB:FPINTRO}
\left\{
\begin{array}{l}
d_{t} \tilde{u}_{t}
=  \bigl\{ -   \Delta \tilde{u}_{t} + \tilde{H}_{t}(\cdot,D\tilde{u}_{t}) - 
\tilde{F}_{t}(\cdot,m_{t}) 
\bigr\} dt
+ d\tilde{M}_{t},
\\
d_{t} \tilde{m}_{t} = \bigl\{ \Delta \tilde{m}_{t}
+{\rm div} \bigl( \tilde{m}_{t} D_{p} \tilde{H}_{t}(\cdot,D \tilde{u}_{t})
\bigr) \bigr\} dt, \\
 \;  \tilde{m}_{t_0} = m_{0},\; \tilde{u}_{T}
= \tilde{G}(\cdot,m_{T}).
\end{array}
\right.\end{equation}
where (still formally) $d\tilde{M}_{t}=v_{t}(x+\sqrt{2} W_{t})dW_{t}$. 

Let us explain how we understand the above system. The solution  
$(\tilde{u}_{t})_{t \in [0,T]}$ is
seen as an $({\mathcal F}_{t})_{t \in [0,T]}$-adapted process
with paths in the space ${\mathcal C}^0([0,T],{\mathcal C}^{n+2}({\mathbb T}^d))$, 
for some fixed $n \geq 0$. 
The process $(\tilde{m}_{t})_{t \in [0,T]}$ reads 
as an $({\mathcal F}_{t})_{t \in [0,T]}$-adapted process
with paths in the space 
$\cC^0([0,T],{\mathcal P}({\mathbb T}^d))$. 
We shall look for solutions satisfying 
\begin{equation}
\label{eq:se:3:bsde:1INTRO}
\sup_{t \in [0,T]} \bigl( \| \tilde{u}_{t} \|_{n+2+\alpha}
\bigr) \in L^\infty(\Omega,{\mathcal A},\P),
\end{equation}
(for some fixed $\alpha \in (0,1)$). 
The process $(\tilde{M}_{t})_{t \in [0,T]}$ is
seen as an $({\mathcal F}_{t})_{t \in [0,T]}$-adapted process
with paths in the space ${\mathcal C}^0([0,T],
{\mathcal C}^{n}({\mathbb T}^d))$, 
such that, for any $x \in {\mathbb T}^d$, 
$(\tilde{M}_{t}(x))_{t \in [0,T]}$
is an $({\mathcal F}_{t})_{t \in [0,T]}$ martingale. 
It is required to satisfy 
\begin{equation}
\label{eq:se:3:bsde:2INTRO}
\sup_{t \in [0,T]} \bigl( \| \tilde{M}_{t} \|_{n+\alpha}
\bigr) \in L^\infty(\Omega,{\mathcal A},\P).
\end{equation}

\begin{Theorem}
\label{thm:partie:2:existence:uniquenessINTRO}
Assume that
$F$, $G$ and $H$ satisfy 
\eqref{HypD2H}
and 
\eqref{e.monotoneF}
and that
{\bf (HF1(${\boldsymbol n}$+1))}
and 
{\bf (HG1(${\boldsymbol n}$+2))}
hold true for some 
$n \geq 0$
and some $\alpha \in (0,1)$. 
Then,
there exists a unique solution $(\tilde{u}_{t},\tilde{m}_{t},\tilde{M}_{t})_{t \in [0,T]}$
to \eqref{eq:se:3:tilde:HJB:FPINTRO},
satisfying 
\eqref{eq:se:3:bsde:1INTRO}
and \eqref{eq:se:3:bsde:2INTRO}. 
\end{Theorem}

We postpone the discussion of the existence of the solution to the true MFG system with common noise \eqref{e.MFGsto} to the next section, where the master equation allows to identify the correction term $(v_t)_{t\in [0,T]}$. 

Theorem \ref{thm:partie:2:existence:uniquenessINTRO} is proved in section \ref{se:common:noise} (see Theorem \ref{thm:partie:2:existence:uniqueness} for more precise estimates). The main difference with the deterministic mean field game system is that 
the solution $(\tilde{u}_{t},\tilde{m}_{t})_{0 \leq t \leq T}$ 
is sought in a much bigger space, namely 
$[
{\mathcal C}^0([0,T],{\mathcal C}^{n}(\T^d))
\times
{\mathcal C}^0([0,T],{\mathcal P}(\T^d))]^{\Omega}$, 
which is not well-suited to the use of 
 compactness 
arguments. Because of that, one can 
can no longer invoke Schauder's Theorem to prove the existence of 
a solution. For this reason, the proof uses instead a continuation method, 
directly inspired
from the literature on 
finite dimensional forward-backward stochastic systems (see \cite{PeWu99}).
Notice also that, due to the presence of the noise $(W_{t})_{t \in [0,T]}$, 
the analysis of the time-regularity of the solution becomes a challenging issue
and that the continuation method permits to bypass this difficulty. 

\subsubsection{Second order master equation}

The second main result of the paper
concerns the analogue of
Theorem \ref{theo:ex}
 when 
the underlying
mean-field game problem incorporates an additive common noise.
Then the master 
equation  
\eqref{MFGf}
then
involves additional terms, including 
second order derivatives in the direction of 
the measure. It has the form (for some fixed level of common noise $\beta>0$):
\begin{equation}
\left\{
\begin{array}{l}
\ds
- \partial_{t} U(t,x,m)
- (1+\beta) \Delta_{x} U(t,x,m) 
+ H\bigl(x,D_{x} U(t,x,m) \bigr)
- F \bigl(x,m \bigr)
\vspace{5pt}
\\
\ds \ -(1+\beta) \int_{\T^d}
\textrm{div}_{y} 
\bigl[
D_{m} U
\bigr]
 \bigl( t,x,m,y
 \bigr) d {m}(y)
 +  \int_{\T^d}
 D_{m} U  \bigl( t,x,m,y \bigr)\cdot  D_{p} H\bigl(y, D_{x} U(t,y,m)  \bigr)  d m(y)
 \vspace{5pt}
 \\
\ds \ - 2\beta 
\int_{\T^d}
\textrm{div}_{x}
\bigl[
D_{m} 
U
\bigr]
\bigl( t, x,
 m,y
\bigr) dm(y)
-\beta 
\int_{\T^d \times \T^d}
{\rm Tr} \Bigl[ D^2_{mm}
U
\bigl( t, x ,
 m,y,y'\bigr) 
\Bigr] d m
(y)
d m
(y')\; = \; 0,
\vspace{5pt}
\\
\hspace{130pt} \textrm{for} \ (t,x,m) \in [0,T] \times \T^d \times \Pk,
\vspace{5pt}
\\
U(T,x,m) = G(x,m), \quad \quad 
\textrm{for} \ (x,m) \in  \T^d \times \Pk.
\end{array}
\right.
\label{eq:master:equation:2nd:order:def}
\end{equation}
Following 
Definition \ref{def:master:eq:1st:order}, 
we let
\begin{Definition} 
\label{def:master:eq:2nd:order}
We say that a map $U:[0,T]\times \T^d\times \Pw\to \R$ is a classical solution to the 
second order master equation 
\eqref{eq:master:equation:2nd:order:def}
if
\begin{itemize}
\item $U$ is continuous in all its arguments (for the $\dk$ distance on $\Pw$), is of class $\cC^2$ in $x$ and $\cC^1$ in time (the derivatives of order one in time and space 
and of order two in space being continuous in all the arguments),  

\item $U$ is of class $\cC^2$ with respect to $m$, 
the first and second order derivatives 
\begin{equation*}
\begin{split}
&[0,T] \times \T^d \times \Pk 
\times \T^d \ni
(t,x,m,y) \mapsto \frac{\delta U}{\delta m}(t,x,m,y),
\\
&[0,T] \times \T^d \times \Pk
\times \T^d \times \T^d 
\ni (t,x,m,y,y') 
\mapsto \frac{\delta^2 U}{\delta m^2}(t,x,m,y),
\end{split}
\end{equation*}
being continuous in all the arguments, the first order derivative
$\delta U/\delta m$ 
being
 twice differentiable in $y$, the derivatives being continuous 
in all the arguments, and the second order derivative 
$\delta^2 U/\delta m^2$ 
being also twice differentiable in the pair $(y,y')$, the derivatives being continuous 
in all the arguments, 
\item the function $D_{y}(\delta U/\delta m)=D_{m} U$ is differentiable in 
$x$, the derivatives being continuous in all the arguments,

\item $U$ satisfies the master equation
\eqref{eq:master:equation:2nd:order:def}. 
\end{itemize}
\end{Definition}

On the model of Theorem 
\ref{theo:ex}, we claim
\begin{Theorem}
\label{theo:2nd-order:master:equation}
Assume that
$F$, $G$ and $H$ satisfy 
\eqref{HypD2H}
and 
\eqref{e.monotoneF}
in Subsection \ref{subsec:hyp}
and 
that
{\bf (HF2(${\boldsymbol n}$+1))}
and 
{\bf (HG2(${\boldsymbol n}$+2))}
hold true for some $n \geq 2$ and for some $\alpha \in (0,1)$.

Then, the second-order master equation 
\eqref{eq:master:equation:2nd:order:def}
has a unique solution $U$.  

The solution $U$ enjoys the following regularity: for any
$\alpha' \in [0,\alpha)$,
 $t \in [0,T]$ 
and $m \in {\mathcal P}(\T^d)$, 
$U(t,\cdot,m)$, 
$[\delta U/\delta m](t,\cdot,m,\cdot)$
and 
$[\delta^2 U/\delta m^2](t,\cdot,m,\cdot,\cdot)$
are
in $\cC^{n+2+\alpha'}$,
$\cC^{n+2+\alpha'} \times \cC^{n+1+\alpha'}$
and $\cC^{n+2+\alpha'}\times \cC^{n+\alpha'} \times \cC^{n+\alpha'}$ respectively, 
independently of $(t,m)$. 
Moreover, 
the mappings 
\begin{equation*}
\begin{split}
&[0,T] \times {\mathcal P}(\T^d) \ni (t,m) \mapsto 
U(t,\cdot,m) \in \cC^{n+2+\alpha'},\\
&[0,T] \times {\mathcal P}(\T^d) \ni (t,m) \mapsto 
[\delta U/\delta m](t,\cdot,m,\cdot) \in 
\cC^{n+2+\alpha'} \times \cC^{n+1+\alpha'},\\
&[0,T] \times {\mathcal P}(\T^d) \ni (t,m) \mapsto 
[\delta^2 U/\delta m^2](t,\cdot,m,\cdot,\cdot) \in 
\cC^{n+2+\alpha'}\times [\cC^{n+\alpha'}]^2
 \end{split}
\end{equation*}
are continuous.
When $\alpha'=0$, these mappings are Lipschitz continuous in 
$m$, uniformly in time. 
\end{Theorem}

Section \ref{subse:partie:2:first:order:derivative} is devoted to the proof of Theorem \ref{theo:2nd-order:master:equation}. As for the first order master equation, the starting point consists in letting, given $(t_0,m_0)\in [0,T]\times \Pk$, 
\begin{equation*}
U(t_{0},x,m_{0}) = \tilde{u}_{t_{0}}(x), \quad x \in \T^d, 
\end{equation*}
where $(\tilde{u}_{t},\tilde{m}_{t},\tilde{M}_{t})_{t \in [0,T]}$ is the solution to the mean field game system with common noise \eqref{eq:se:3:tilde:HJB:FPINTRO}, when 
$(W_{t})_{t \in [0,T]}$
in the definition of the 
coefficients $\tilde{F}$, $\tilde{G}$ and $\tilde{H}$ is replaced
by 
$(W_{t}-W_{t_{0}})_{t \in [t_{0},T]}$. The key remark (see Lemma \ref{lem:partie2:master:equation}), is that, if we let ${m}_{t_{0},t} = [id + 
\sqrt{2}(W_{t}-W_{t_{0}})] \sharp 
\tilde{m}_{t}$, then, for any $h \in [0,T-t_{0}]$,
$\P$ almost surely,
\begin{equation*}
\tilde{u}_{t_{0}+h}(x) =
U\bigl(t_{0}+h,x + \sqrt{2}(W_{t_{0}+h}-W_{t_{0}}),m_{t_{0},t_{0}+h}\bigr), \quad x \in \T^d. 
\end{equation*}
Taking the derivative with respect to $h$ at $h=0$ on both sides of the equality shows that the map $U$ thus defined satisfies the master equation
(up to a tailor-made It\^{o}'s formula, see section \ref{subsub:Tailor-madeItoFormula}). Of course, the main issue is to prove that $U$ is sufficiently smooth to perform the above computation: for this we need to show that $U$ has a first and second order derivative with respect to the measure. As for the deterministic case, this is obtained by linearizing the mean field game system (with common noise). This linearization procedure is complicated by the fact that the triplet $(\tilde{u}_{t},\tilde{m}_{t},\tilde{M}_{t})_{t \in [0,T]}$ solves an equation in which the coefficients have little time regularity. 

As a byproduct of the construction of the master equation, we can come back to the MFG system with common noise. Let $U$ be the solution of the master equation \eqref{eq:master:equation:2nd:order:def}. 
\begin{Corollary}\label{c.sec5.MFGstoch}
Given $t_{0} \in [0,T]$, 
we call a solution to 
\eqref{e.MFGsto} 
a triplet 
$({u}_{t},{m}_{t},v_{t})_{t \in [t_{0},T]}$ 
of $({\mathcal F}_{t})_{t \in [t_{0},T]}$-adapted processes with paths in the space ${\mathcal C}^0([t_{0},T],{\mathcal C}^{2}({\mathbb T}^d)\times {\mathcal P}({\mathbb T}^d)
\times {\mathcal C}^1(\T^d))$
such that 
$\sup_{t \in [t_{0},T]}
( \| u_{t} \|_{2} + \|v_{t}\|_{1})
\in L^{\infty}(\Omega,{\mathcal A},\P)$
and 
\eqref{e.MFGsto} 
holds true with probability $1$.
Under the assumptions of Theorem \ref{theo:2nd-order:master:equation}, for any initial data $(t_0,m_0)\in [0,T]\times \Pk$, the stochastic mean field game system \eqref{e.MFGsto} 
has a unique solution $(u_t,m_t,v_t)_{t\in [0,T]}$, where $({u}_{t},{m}_{t})_{t \in [0,T]}$ is an $({\mathcal F}_{t})_{t \in [0,T]}$-adapted processes with paths in the spaces ${\mathcal C}^0([0,T],{\mathcal C}^{n}({\mathbb T}^d)\times {\mathcal P}({\mathbb T}^d))$ and where the vector field $(v_t)_{t \in [0,T]}$ is given by 
$$
v_t(x) =  \inte D_mU (t,x,m_t,y)dm_t(y).
$$
\end{Corollary}

\subsubsection{The convergence of the Nash system for $N$ players}

We finally study the convergence of Nash equilibria of differential games with $N$ players to the limit system given by the master equation.

We consider the solution $(v^{N,i})$ of the Nash system:  
\be\label{NashvResults}
\left\{ \begin{array}{l}
\ds - \partial_t v^{N,i} -  \sum_{j} \Delta_{x_j}v^{N,i} - \beta  \sum_{j,k} {\rm Tr} D^2_{x_j,x_k} v^{N,i} + H(x_i,  D_{x_i}v^{N,i}) \\
\ds \qquad \qquad + \sum_{j\neq i}  D_pH (x_j, D_{x_j}v^{N,j})\cdot D_{x_j}v^{N,i} =
F(x_i, m^{N,i}_\bx)\qquad {\rm in }\; [0,T]\times \T^{Nd}\\
\ds v^{N,i}(T,\bx)= G(x_i,m^{N,i}_\bx)\qquad {\rm in }\;  \T^{Nd}
\end{array}\right.
\ee
where we have set, for $\ds \bx=(x_1, \dots, x_N)\in (\T^{d})^N$, $\ds m^{N,i}_\bx=\frac{1}{N-1}\sum_{j\neq i} \delta_{x_j}$.

Let us recall that, under 
the 
same
assumptions 
on $H$, $F$ and $G$
as in the statement 
of 
Theorem
\ref{theo:2nd-order:master:equation}, the above system has a unique solution (see for instance \cite{LSU}).

Our main result says that the $v^{N,i}$ ``converges" to the solution of the master equation as $N\to+\infty$. This result, conjectured in Lasry-Lions \cite{LL07mf},  is somewhat subtle because in the Nash system players observe each other (closed loop form) while in the limit system the players just need to observe the theoretical distribution of the population, and not the specific behavior of each player. We first study  the convergence of the functions $v^{N,i}$ and then the convergence of the optimal trajectories. 


\bigskip

We have two different ways to express the convergence of the $v^{N,i}$, described in the following result: 

\begin{Theorem}
\label{thm:mainCV} 
Let the assumption of Theorem 
\ref{theo:2nd-order:master:equation}
be in force for some $n \geq 2$
and let $(v^{N,i})$ be the solution to \eqref{NashvResults} and $U$ be the classical solution to the second order master equation.  Fix $N\geq 1$ and $(t_0,m_0)\in [0,T]\times \Pk$.
\begin{itemize}
\item[(i)] For any $\bx\in (\T^d)^N$, let $m^N_\bx:= \frac{1}{N} \sum_{i=1}^N \delta_{x_i}$. Then 
$$
\frac{1}{N}\sum_{i=1}^N \left|v^{N,i}(t_0,\bx)- U(t_0, x_i, m^{N}_\bx)\right|\leq CN^{-1}.
$$

\item[(ii)]  For any $i\in\{1,\dots, N\}$ and $x\in \T^d$, let us set 
$$
w^{N,i}(t_0,x,m_0) := \inte\dots \inte v^{N,i}(t_0, \bx) \prod_{j\neq i}m_0(dx_j) \qquad {\rm where }\; \bx=(x_1,\dots, x_N).
$$ 
Then 
$$
\left\| w^{N,i}(t_0,\cdot, m_0)-U(t_0,\cdot, m_0)\right\|_{L^1(m_0)} \leq \left\{\begin{array}{ll}
C N^{-1/d} & {\rm if}\; d\geq 3\\
C N^{-1/2}\log(N) & {\rm if}\; d=2\end{array}\right. 
$$
\end{itemize} 
In (i) and (ii), the constant $C$ does not depend on $i$, $t_0$, $m_0$, $i$ nor $N$. 
\end{Theorem}

Theorem \ref{thm:mainCV} says, in two different ways, that ``in average", the $(v^{N,i})$ are close to $U$. The first statement explains that, for a fixed $\bx\in (\T^d)^N$, the quantity  $|v^{N,i}(t_0,\bx)- U(t_0, x_i, m^{N,i}_\bx)|$ is, in average over $i$, of order $N^{-1}$. In  the second statement, 
 one fixes a measure $m_0$ and an index $i$, and one averages in space $v^{N,i}(t_0,\cdot)$ over $m_0$ for all variables but the $i-$th one. The resulting map $w^{N,i}$ is at a distance of order $N^{-1/d}$ of $U(t_0,\cdot, m_0)$.

Because of the lack of estimates for the $v^{N,i}$ uniform with respect to $N$, we do not know if it is possible to avoid the two averaging procedures in the above results.  However, if one knows that the solution of the Nash system has a (locally uniform) limit, then this limit is necessarily $U$: 

\begin{Corollary}\label{cor.NashLim} 
Under the assumption of
Theorem  
\ref{thm:mainCV},
let $(t,x_1,m)\in [0,T]\times \T^d\times \Pk$ be fixed and assume that there exists $v\in \R$ such that 
$$
\limsup_{N\to+\infty, \ x_1'\to x_1, \ m^{N,1}_{\bx'}\to m} \left| v^{N,1}(t,\bx')-v\right| =0.
$$
Then, if $x_1$ belongs to the support of $m$, we have  $v=U(t,x_1,m)$. 
\end{Corollary}
\bigskip

We can also describe the convergence in terms of optimal trajectories. Let $t_0\in [0,T)$, $m_0\in \Pw$ and let $(Z_i)$ be an i.i.d family of $N$ random variables of law $m_0$. We set $\bZ=(Z_1,\dots, Z_N)$.  Let also $((B^{i}_{t})_{t \in [ 0,T]})_{i \in \{1,\dots,N\}}$ be a family of $N$ independent Brownian motions which is also independent of $(Z_i)$ and let $(W_{t})_{t \in [0,T]}$ be a Brownian motion independent of the $(B^i)$ and $(Z_i)$. 
We consider the optimal trajectories $(\bY_t=(Y_{1,t},\dots, Y_{N,t}))_{t\in [t_0,T]}$ for the $N-$player game: 
$$
\left\{\begin{array}{l}
dY_{i,t}= -D_pH(Y_{i,t}, D_{x_i}v^{N,i}(t, \bY_t))dt +\sqrt{2} dB^{i}_t+\sqrt{2\beta} dW_t, \qquad t\in [t_0,T]\\
Y_{i,t_0}= Z_i
\end{array}\right.
$$
and the solution $(\tilde \bX_t=(\tilde X_{1,t}, \dots, \tilde X_{N,t}))_{t\in[t_0,T]}$ of stochastic differential equation of McKean-Vlasov type:
$$
\left\{\begin{array}{l}
d\tilde X_{i,t} =  -D_pH\left(\tilde X_{i,t}, D_xU
\bigl(t,\tilde X_{i,t},{\mathcal L}(\tilde X_{i,t}\vert W)\bigr)\right)dt +\sqrt{2}dB^i_t+\sqrt{2\beta} dW_t, \\
\tilde X_{i,t_0}= Z_i. 
\end{array}\right.
$$
Both system of SDEs
are set on $(\R^d)^N$. Since both are driven by periodic coefficients, solutions generate (canonical) flows of probability measures on 
$(\T^d)^N$:
The flow of probability measures 
generated
in ${\mathcal P}((\T^d)^N)$ 
by each solution 
is independent of the representatives in $\R^d$ of
the $\T^d$-valued random variables $Z_{1},\dots,Z_{N}$.

The next result says that the solutions of the two systems are close: 

\begin{Theorem}\label{thm:CvMFG} 
Let the assumption of Theorem 
\ref{thm:mainCV} be in force. Then,
for any $N\geq 1$ and any $i\in \{1,\dots, N\}$, we have
$$
\E\biggl[ \sup_{t\in [t_0,T]} \left|Y_{i,t}-\tilde X_{i,t}\right| \biggr]\leq C N^{-\frac{1}{d+8}}
$$
for some constant $C>0$ independent of $t_0$, $m_0$ and $N$. 
\end{Theorem}

In particular, since the $(\tilde X_{i,t})$ are independent conditioned on $W$, the above result is a (conditional) propagation of chaos. 

The proofs of Theorem \ref{thm:mainCV} and Theorem \ref{thm:CvMFG} rely on the existence of the solution $U$ of the master equation \eqref{eq:master:equation:2nd:order:def} and 
constitute the aim of Section \ref{sec.convergence}. 
Our starting point is that, for any $N\geq 1$, the ``projection" of $U$ onto the finite dimensional space $[0,T]\times (\T^d)^N$ is  almost a solution to the Nash system \eqref{NashvResults}. Namely, if we set, for any $i\in \{1,\dots, N\}$ and any $\bx=(x_1,\dots, x_N)\in (\T^d)^N$, 
$$
u^{N,i}(t,\bx):= U(t, x_i, m^{N,i}_\bx),
$$
then $(u^{N,i})_{i \in \{1,\dots,N\}}$ satisfies \eqref{NashvResults} up to an error term of size $O(1/N)$ for each equation (Proposition \ref{Prop:equNi}). Note that, as the number of equations in \eqref{NashvResults} is $N$, this could yield to a serious issue because the error terms could add up. The strategy of proof consists in controlling the error terms by exploiting the symmetry of the Nash system along the optimal paths. 

One of the thrust of our approach is that, somehow, the proofs 
work under the sole assumption that the 
master equation \eqref{eq:master:equation:2nd:order:def}
admits a classical solution. Here existence of a
classical solution is guaranteed under the assumption 
of Theorem \ref{theo:2nd-order:master:equation}, 
which includes in particular 
the monotonicity properties of $F$ and $G$,
but the analysis provided in Section 
 \ref{sec.convergence}
shows that monotonicity plays no role 
in the proofs of Theorems \ref{thm:mainCV} and 
\ref{thm:CvMFG}.
Basically, 
only the global Lipschitz properties of $H$ and $D_{p}H$, together with the various bounds obtained for the solution 
of the master equation and its derivatives, matter. 
This is a quite remarkable fact, which 
demonstrates the efficiency of our strategy.

\newpage

\section{A starter: the first order master equation}\label{sec:MasterWithoutCN}

In this section we prove Theorem \ref{theo:ex}, i.e., we establish the well-posedness of the master equation  without common noise: 
\be\label{masterwithoutCN}
\begin{array}{l}
\left\{\begin{array}{l} 
\ds - \partial_t U(t,x,m)  - \Delta_x U(t,x,m) +H\bigl(x,D_xU(t,x,m) \bigr) -\inte \dive_y \left[ D_m U\right](t,x,m,y) \ d m(y)\\
\ds    \qquad+ \inte  D_m U(t,x,m,y)\cdot D_pH\bigl(y,D_xU
(t,y,m)
\bigr)\ dm(y) =F(x,m) \\
 \ds  \qquad\qquad\qquad \qquad {\rm in }\; [0,T]\times \T^d\times \Pw\\
 U(T,x,m)= G(x,m) \qquad  {\rm in }\; \T^d\times \Pw\\
\end{array}\right.
\end{array}
\ee
The idea is to represent $U$ by solutions of the MFG system: let us recall that, for any $(t_0,m_0)\in [0,T)\times \Pw$, the MFG system is the system of forward-backward equations: 
\be\label{MFG}
\left\{ \begin{array}{l}
\ds - \partial_t u - \Delta u +H(x,Du)=F(x,m(t))\\
\ds \partial_t m - \Delta m -{\rm div}( m D_pH(x, D u))=0 \\
\ds u(T,x)=G(x,m(T)), \; m(t_0, \cdot)=m_0
\end{array}\right.
\ee
As recalled below, under suitable assumptions on the data, there exists a unique solution $(u,m)$ to the above system. Our aim is to show that the map $U$ defined by 
\be\label{defU}
U(t_0,\cdot, m_0):= u(t_0,\cdot)
\ee
is a solution to \eqref{masterwithoutCN}. 
 
Throughout this section assumptions \eqref{HypD2H} and \eqref{e.monotoneF} are in force. Let us however underline that the global Lipschitz continuity of $H$ is not absolutely necessary. We just need to know that the solutions of the MFG system are uniformly Lipschitz continuous, independently of the initial conditions: sufficient conditions for this can be found in \cite{LL07mf} for instance.

The proof of Theorem \ref{theo:ex} requires several preliminary steps. We first recall the existence of a solution to the MFG system \eqref{MFG} (Proposition \ref{prop:reguum}) and show that this solution depends in a Lipschitz continuous way of the initial measure $m_0$ (Proposition \ref{prop:Ulip}). Then we show by a linearization procedure that the map $U$ defined in \eqref{defU} is of class $\cC^1$ with respect to the measure (Proposition \ref{prop:diff}, Corollary  \ref{cor:diff0}). The proof relies on the analysis of a linearized system with a specific structure, for which well-posedness and estimates are given in Lemma \ref{lem:BasicEsti0} and Lemma \ref{lem:BasicEsti2}. 
We are then ready to prove Theorem \ref{theo:ex} (subsection \ref{subsec:prooftheomain}). We also show, for later use, that the first order derivative of $U$ is Lipschitz continuous with respect to $m$ (Proposition \ref{prop:DmULip}). We complete the section  by explaining how one obtains the solution $U$ as the derivative with respect to the measure $m$ of the value function of an optimal control problem set over flows of probability measures (Theorem \ref{theo:HJB}). 
\vspace{4pt}

Some of the proofs given in this section consist of a sketch only.
One of the reason is that some of the arguments we use here in order to investigate the 
MFG system \eqref{MFG} have been already developed in the literature. 
Another reason is that this section constitutes a starter only, 
specifically devoted to the simpler case without common noise. 
Arguments will be expanded in detail in the two next sections, 
when handling mean-field games with a common noise, 
for which there are much less available results in the literature.

\subsection{Space regularity of $U$}

In this part we investigate the space regularity of $U$ with respect to $x$. Recall that $U(t_0,\cdot,m_0)$ is defined by 
$$
U(t_0,x,m_0)= u(t_0,x)
$$
where $(u,m)$ is a classical solution to \eqref{MFG} with initial condition $m(t_0)=m_0$. 
By a classical solution to \eqref{MFG} we mean a pair $(u,m)\in \cC^{1,2}\times \cC^0([t_0,T], \Pk)$ such that the equation for $u$ holds in the classical sense while the equation for $m$ holds in the sense of distribution. 

\begin{Proposition}\label{prop:reguum} Assume that {\bf (HF1(${\boldsymbol n}$))} and {\bf (HG1(${\boldsymbol n}$+2))} hold for some $n\geq 0$. Then, for any initial condition $(t_0,m_0)\in [0,T]\times \Pk$, the MFG system \eqref{MFG} has a unique classical solution $(u,m)$ and this solution satisfies 
\be\label{esti:um}
\sup_{t_1\neq t_2} \frac{\dk(m(t_1),m(t_2))}{|t_2-t_1|^{1/2}}+ \sum_{|\ell |\leq n} \|D^\ell u\|_{1+\alpha/2,2+\alpha} \leq C_n,
\ee
where the constant $C_n$ does not depend on $(t_0,m_0)$. 

If moreover $m_0$ is absolutely continuous with a smooth positive density, then $m$ is of class $\cC^{1+\alpha/2,2+\alpha}$ with a smooth, positive density.  
\end{Proposition}

 Note that further regularity of $F$ and $G$ improves the space regularity of $u$ but not its time regularity (as the time regularity of the coefficients depends upon that of 
 $m$, see Proposition 
 \ref{prop:reguum} right above). By \eqref{esti:um}, we have, under assumptions  {\bf (HF1(${\boldsymbol n}$))} and {\bf (HG1(${\boldsymbol n}$+2))}
 $$
\sup_{t\in [0,T]} \sup_{m\in \Pk} \|U(t, \cdot, m)\|_{n+2+\alpha}\leq C_n.
$$

\begin{proof} 
We provide a sketch of proof only.
Existence and uniqueness of classical solutions for \eqref{MFG} under assumptions {\bf (HF1(${\boldsymbol n}$))} and  {\bf (HG1(${\boldsymbol n}$+2))} for $n=0$ are standard: see, e.g., \cite{LL06cr2, LL07}. Note that we use here the Lipschitz continuity assumption on $H$, which guaranties uniform Lipschitz estimates on $u$. 

We obtain further regularity on $u$ by deriving in space $n$ times  the equation for $u$.

When $m_0$ has a smooth density, $m$ satisfies an equation with $\cC^{\alpha/2,\alpha}$ exponents, so that by Schauder theory $m$ is $\cC^{1+\alpha/2,2+\alpha}$. If moreover, $m_0$ is positive, then $m$ remains positive by strong maximum principle. 
\end{proof}

\subsection{Lipschitz continuity of $U$}

\begin{Proposition}\label{prop:Ulip} Assume that {\bf (HF1(${\boldsymbol n+1}$))} and {\bf (HG1(${\boldsymbol n}$+2))} hold for some $n\geq 0$. 
Let $m_0^1, m_0^2\in \Pw$, $t_0\in [0,T]$ and $(u^1,m^1)$, $(u^2,m^2)$  be the solutions of the MFG system \eqref{MFG} with initial condition $(t_0,m^1_0)$ and $(t_0,m^2_0)$ respectively.  Then 
$$
\sup_{t\in [0,T]} \left\{ \dk(m^1(t),m^2(t)) +   \left\|u^1(t,\cdot)-u^2(t,\cdot)\right\|_{n+2+\alpha}\right\} \leq C_n\dk(m^1_0,m^2_0),
$$
for a constant $C_{n}$ independent of 
$t_{0}$, $m^1_{0}$ and $m^2_{0}$.
In particular, 
$$
\left\|U(t_0,\cdot, m^1_0)-U(t_0,\cdot,m^2_0)\right\|_{n+2+\alpha}\leq C_n\dk (m_0^1,m^2_0).
$$
\end{Proposition}

\begin{proof} {\it First step.} To simplify the notation, we show the result for $t_0=0$. We use the well-known Lasry-Lions monotonicity argument (see the proof of Theorem 2.4 and Theorem 2.5 of \cite{LL07mf}): 
\begin{multline*}
\frac{d}{dt}\inte \bigl(u^1(t,y)-u^2(t,y)\bigr)\bigl(
m^{1}(t,y)-m^2(t,y) \bigr)dy \\
\leq -C^{-1} \inte \frac12 |Du^1(t,y)-Du^2(t,y)|^2
\bigl(m^1(t,y)+m^2(t,y)\bigr)dy
\end{multline*}
since $F$ is monotone, $Du^1$ and $Du^2$ are uniformly bounded and $H$ satisfies \eqref{HypD2H}. 
So 
\begin{multline*}
\int_0^T \inte   |Du^1(t,y)-Du^2(t,y)|^2\bigl(m^1(t,y)+m^2(t,y)\bigr) \ dy dt \\ \leq C \left[\inte \bigl(u^1(t,y)-u^2(t,y) \bigr) \bigl(m^1(t,y)-m^2(t,y) \bigr)dy\right]_0^T.
\end{multline*}
At time $T$ we use the monotonicity of $G$ to get 
\begin{multline*}
\inte \bigl(u^1(T,y)-u^2(T,y) \bigr) \bigl(m^1(T,y)-m^2(T,y) \bigr)dy\\ =\inte 
\bigl(G(y,m^1(T))-G(y,m^2(T))\bigr)\bigl(m^1(T,y)-m^2(T,y)\bigr)dy\geq 0.
\end{multline*}
 At time $0$ we have by the definition of $\dk$,
$$
\inte \bigl(u^1(0,y)-u^2(0,y)\bigr)\bigl(m^1_0(y)-m^2_0(y)\bigr)dy \leq  C \|D(u^1-u^2)(0,\cdot)\|_{\infty} \dk(m^1_0,m^2_0).
$$
Hence 
\be\label{mmDuDu}
\int_0^T \inte \bigl(m^1(t,y)+m^2(t,y)\bigr)  |Du^1(t,y)-Du^2(t,y)|^2dydt \leq C \|D(u^1-u^2)(0,\cdot)\|_{\infty} \dk(m^1_0,m^2_0).
\ee

{\it Second step:} Next we estimate $m^1-m^2$: to do so, let $(\Omega,{\mathcal F},\P)$ be a standard probability space, $X^1_0$, $X^2_0$ be random variables on $\Omega$ with law $m^1_0$ and $m^2_0$ respectively and such that 
$\E[|X^1_0-X^2_0|]=\dk (m^1_0,m^2_0)$. Let also $(X^1_t)$, $(X^2_t)$ be the solutions to 
$$
dX^i_t= - D_pH(X^i_t, Du^i(t,X^i_t))dt +\sqrt{2}dB_t\qquad t\in [0,T], \; i=1,2,
$$
where $(B_t)_{t\in [0,T]}$ is a $d-$dimensional Brownian motion. Then the law of $X^i_t$ is $m^i(t)$ for any $t$. We have
$$
\begin{array}{rl}
\ds \E\bigl[ |X^1_t-X^2_t| \bigr] \; \leq & \ds  \E\bigl[|X^1_0-X^2_0|\bigr]+\E\biggl[ \int_0^t \Bigl( \bigl| D_pH \bigl(X^1_s, Du^1(s,X^1_s) \bigr)-D_pH\bigl(
X^2_s, Du^1(s,X^2_s) \bigr) \bigr|
\\
 & \ds \qquad \qquad \qquad + \bigl|D_pH \bigl(X^2_s, Du^1(s,X^2_s) \bigr)-D_pH
 \bigl(X^2_s, Du^2(s,X^2_s) \bigr) \bigr|\ ds\Bigr)\biggr].
\end{array}
$$
As the maps $x\mapsto D_pH(x,Du^1(s,x))$ and $p\mapsto D_pH(x,p)$  are  Lipschitz continuous (see 
\eqref{HypD2H}
and
Proposition 
\ref{prop:Ulip}): 
$$
\begin{array}{l}
\ds \E\bigl[ |X^1_t-X^2_t| \bigr] \\
\;   \leq    \ds  \E\bigl[|X^1_0-X^2_0|\bigr]+ C\int_0^t \E\bigl[|X^1_s-X^2_s|\bigr]ds
 +C\int_0^t\inte  |Du^1(s,x)-Du^2(s,x)| m^2(s,x)dxds \\
 \; \ds \leq  \ds \ds  \dk(m^1_0,m^2_0)+ C\int_0^t \E\bigl[
 |X^1_s-X^2_s|\bigr]ds
 +C\left(\int_0^t\inte  |Du^1(s,x)-Du^2(s,x)|^2 m^2(s,x)dxds\right)^{1/2}.
\end{array}
$$
In view of \eqref{mmDuDu} and Gronwall inequality, we obtain
\be\label{X1t-X2t}
\E\bigl[ |X^1_t-X^2_t|\bigr]  \leq  \ds  C\left[ \dk(m^1_0,m^2_0)+  \|D(u^1-u^2)(0,\cdot)\|_{\infty}^{1/2} \dk(m^1_0,m^2_0)^{1/2}\right].
\ee
As $\ds \dk(m^1(t),m^2(t))\leq \E[ |X^1_t-X^2_t|]$, we get therefore 
\be\label{suptm1m2}
\sup_{t\in [0,T]} \dk(m^1(t),m^2(t))\leq C\left[ \dk(m^1_0,m^2_0)+  \|D(u^1-u^2)(0,\cdot)\|_{\infty}^{1/2} \dk(m^1_0,m^2_0)^{1/2}\right].
\ee

{\it Third step:} We now estimate the difference $w:=u^1-u^2$. We note that $w$ satisfies: 
$$
\left\{\begin{array}{l}
\ds -\partial_t w(t,x) -\Delta w(t,x) +  V(t,x)\cdot D w(t,x) = R_1(t,x)\qquad {\rm in} \; [0,T]\times \T^d \\
\ds w(T,x)= R_T(x)\qquad {\rm in} \;  \T^d
\end{array}\right.
$$
where, for $(t,x)\in[0,T]\times \T^d$, 
$$
V(t,x)= \int_0^1D_pH(x,sDu^1(t,x)+(1-s)D u^2(t,x)) \ ds,
$$
$$
 R_1(t,x)= \int_0^1\inte  \frac{\delta F}{\delta m}(x, s m^1(t)+(1-s)m^2(t),y)(m^1(t,y)-m^2(t,y))\ dyds
$$
and  
$$
 R_T(x) =\int_0^1 \inte  \frac{\delta G}{\delta m}(x, s m^1(T)+(1-s)m^2(T),y)(m^1(T,y)-m^2(T,y))\  dyds. 
 $$
By assumption {\bf (HF1(${\boldsymbol n+1}$))} and inequality \eqref{suptm1m2}, we have, for any $t\in [0,T]$,  
$$
\begin{array}{rl}
\ds  \left\|D^\ell R_1(t,\cdot)\right\|_{n+1+\alpha}\; \leq & \ds \int_0^1 \left\|D_y\frac{\delta F}{\delta m}(\cdot, s m^1(t)+(1-s)m^2(t),\cdot)\right\|_{\cC^{n+1+\alpha}\times L^\infty} ds \ \dk(m^1(t),m^2(t))  \\
\leq & \ds C \left[ \dk(m^1_0,m^2_0)+  \|Dw(0,\cdot)\|_{\infty}^{1/2} \dk(m^1_0,m^2_0)^{1/2}\right]
\end{array}
$$
and, in the same way (using assumption {\bf (HG1(${\boldsymbol n}$+2))}),  
$$
\left\|R_T\right\|_{n+2+\alpha}\leq C \left[ \dk(m^1_0,m^2_0)+  \|Dw(0,\cdot)\|_{\infty}^{1/2} \dk(m^1_0,m^2_0)^{1/2}\right].
$$
On another hand, $V(t,\cdot)$ is bounded in  $\cC^{n+1+\alpha}$ in view of the regularity of $u^1$ and $u^2$ (Proposition \ref{prop:reguum}). 
Then Lemma \ref{l.estilinear} below states that 
$$
\begin{array}{rl}
\ds \sup_{t\in [0,T]} \|w(t,\cdot)\|_{n+2+\alpha} \; \leq & \ds C\Bigl\{ \left\|R_T\right\|_{n+2+\alpha}+
\sup_{t\in [0,T]} \|R_1(t,\cdot)\|_{n+1+\alpha}\Bigr\} \\
\leq & \ds C \left[ \dk(m^1_0,m^2_0)+  \|Dw(0,\cdot)\|_{\infty}^{1/2} \dk(m^1_0,m^2_0)^{1/2}\right].
\end{array}
$$
Rearranging, we find 
$$
\sup_{t\in [0,T]} \|w(t,\cdot)\|_{n+2+\alpha} \; \leq \; C\dk(m^1_0,m^2_0),
$$
and coming back to inequality \eqref{suptm1m2}, we also obtain 
$$
\sup_{t\in [0,T]} \dk(m^1(t),m^2(t))\leq C \dk(m^1_0,m^2_0).
$$
\end{proof}

In the proof we used the following estimate: 
\begin{Lemma}\label{l.estilinear} Let $n\geq 1$, $V\in \cC^0([0,T], \cC^{n-1+\alpha}(\T^d,\R^d))$ and $f\in \cC^0([0,T], \cC^{n-1+\alpha}(\T^d))$. 
Then, for any $z_T\in \cC^{n+\alpha}(\T^d)$, the (backward) equation
$$
\left\{\begin{array}{l}
-\partial_t z-\Delta z +V(t,x)\cdot Dz = f(t,x),\qquad {\rm in }\; [0,T]\times \T^d\\
z(T,x)=z_T(x)
\end{array}\right.
$$
has a unique solution which satisfies 
$$
\sup_{t\in [0,T]} \|z(t,\cdot)\|_{n+\alpha} +\sup_{t\neq t'} \frac{\|z(t',\cdot)-z(t,\cdot)\|_{n+\alpha}}{|t'-t|^\frac12}
\leq C\left\{ \|z_T\|_{n+\alpha}+\sup_{t\in [0,T]} \|f(t,\cdot)\|_{n-1+\alpha}\right\},
$$
where $C$ depends  on $\sup_{t\in [0,T]} \|V(t,\cdot)\|_{n-1+\alpha}$. 
\end{Lemma}

\begin{proof} Beside the time estimate, Lemma \ref{l.estilinear} is a particular case (in the deterministic setting) of Lemma \ref{lem:cas:vartheta=0}. So we postpone this part of the proof to section \ref{se:common:noise}. 

We now prove the time regularity. By Duhamel formula, we have,
$$
z(t+h,\cdot)-z(t,\cdot) = (P_{T-t-h}-P_{T-t})z_T + \int_{t+h}^T P_{s-t-h}\psi(s,\cdot) ds -\int_{t}^T P_{s-t}\psi(s,\cdot) ds ,
$$
where $P_t$ is the heat semi-group and $\psi(s,\cdot):= V(s,\cdot)\cdot Dz(s,\cdot)-f(s,\cdot)$. Hence, 
for $2h \leq T-t$, 
\begin{equation}
\label{aboveformula}
\begin{array}{rl}
\ds \|z(t+h,\cdot)-z(t,\cdot)\|_{n+\alpha} \; \leq &\ds
\hspace{-5pt}
 \|(P_{T-t-h}-P_{T-t})z_T\|_{n+\alpha} +\int_{t}^{t+2h} \|P_{s-t}\psi(s,\cdot)\|_{n+\alpha} ds
\\
 &
\hspace{-55pt}
\ds
 +\int_{t+h}^{t+2h} \|P_{s-t-h}\psi(s,\cdot)\|_{n+\alpha} ds + \int_{t+2h}^T \bigl\|(P_{s-t-h}-P_{s-t})\psi(s,\cdot)\bigr\|_{n+\alpha} ds .
\end{array}
\end{equation}
Recalling the standard estimates $\|(P_{T-t-h}-P_{T-t})z_T\|_{n+\alpha}\leq C h^\frac12\|z_T\|_{n+\alpha}$, $\|P_{s-t}\psi(s,\cdot)\|_{n+\alpha}\leq C(s-t)^{-\frac12}\|\psi(s)\|_{n-1+\alpha}$ and $\|(P_{s-t-h}-P_{s-t})\psi(s,\cdot)\|_{n+\alpha} \leq Ch(s-t-h)^{-\frac32}\|\psi(s,\cdot)\|_{n-1+\alpha}$, we find the result
when $2h \leq T-t$. 

When $2h > T-t$,  there is no need to consider the integral from 
$t+2h$ to $T$ in the above formula \eqref{aboveformula}, and the result follows in the same way.
\end{proof}

\subsection{Estimates on a linear system} 
\label{subsec:LS}

In the sequel we need to estimate several times solutions of a forward-backward system of linear equations. In order to minimize the computation, we collect in this section  two different results on this system. The first one provides existence of a solution and estimates for smooth data.  The second one deals with general data. 

We consider systems of the form 
\be\label{eq:prem}
\left\{ \begin{array}{rl}
 (i) & \ds - \partial_t z - \Delta z + V(t,x)\cdot Dz  = \frac{\delta F}{\delta m}(x, m(t))(\rho(t))+ b(t,x) \qquad {\rm in}\; [t_0,T]\times \T^d\\
 (ii) & \ds \partial_t \rho - \Delta \rho - \dive(\rho V)- \dive (m\Gamma Dz+c)=0 \qquad  {\rm in}\; [t_0,T]\times \T^d\\
 (ii) & \ds z(T,x)=\frac{\delta G}{\delta m}(x, m(T))(\rho(T))+z_T(x), \; \rho(t_0)=\rho_0\qquad  {\rm in}\;  \T^d
\end{array}\right.
\ee
where  $V:[t_0,T]\times \R^d\to \R^d$ is a given vector field, $m\in \cC^0([0,T], \Pk)$, $\Gamma: [0,T]\times \T^d\to \R^{d\times d}$ is a continuous map with values into the family of symmetric matrices and where  the maps $b:[t_0,T]\times \T^d\to \R$, 
$c : [t_{0},T] \times \T^d \to \R^d$ and $z_T:\T^d\to \R$ are given. We always assume that there is a constant $\bar C>0$  such that
\be
\label{condGamma}
\begin{array}{l}
\forall t,t'\in [t_0,T], \qquad \ds \dk(m(t),m(t'))\leq \bar C |t-t'|^{1/2}, 
\\
\ds \forall (t,x)\in [t_0,T]\times \T^d, \qquad \bar C^{-1}I_d\leq \Gamma(t,x)\leq  \bar CI_d.
\end{array}
\ee

Typically, $V(t,x)= D_pH(x,Du(t,x))$, $\Gamma(t,x)= D^2_{pp}H(x,Du(t,x))$ for some solution $(u,m)$ of the MFG system \eqref{MFG} starting from some initial data $m(t_0)=m_0$. Recall that the derivative $Du$ is globally Lipschitz continuous with a constant independent of $(t_0,m_0)$, so that assumption \eqref{HypD2H} gives the existence of a constant $\bar C$ for which \eqref{condGamma} holds.  We note for later use that this constant does not depend on $(t_0,m_0)$. 

To simplify the notation, let us set, for $n\in\N$, $X_n= \cC^{n+\alpha}(\T^d)$ and let $(X_n)'$ be its dual space ($(X_n)'=({\mathcal C}^{n+\alpha}(\T^d))'$).  We first establish the existence of a solution and its smoothness for smooth data: 

\begin{Lemma}\label{lem:BasicEsti0} Assume that $b$, $c$, $z_T$ and $\rho_0$ are smooth, $V$ is of class $\cC^{1+\alpha/2,2+ \alpha}$, $\Gamma$ is of class $\cC^1$ and $(m(t))_{t\in [t_0,T]}$ is a $\cC^1$ family of densities, which are uniformly bounded above and below by positive constants. Suppose furthermore that {\bf (HF1(${\boldsymbol n}$))} and {\bf (HG1(${\boldsymbol n}$+2))} hold for some $n\geq  0$.  Then system \eqref{eq:prem} has a classical solution $(z,\rho)\in \cC^{1+\alpha/2,2+\alpha}\times  \cC^{1+\alpha/2,2+\alpha}$. 

Moreover, the pair $(z,\rho)$ satisfies the following estimates: 
\be\label{esti1theo}
\ds\sup_{t\in [t_0,T]} \|z(t,\cdot)\|_{n+2+\alpha} +\sup_{t\neq t'} \frac{\|z(t',\cdot)-z(t,\cdot)\|_{n+2+\alpha}}{|t'-t|^\frac12}\leq C_nM.
\ee
and
\be\label{esti3theo}
\sup_{t\in [t_0,T]} \|\rho(t)\|_{(X_{n+1})'} + \sup_{t\neq t'} \frac{\|\rho(t')-\rho(t)\|_{(X_{n+1})'}}{|t-t'|^{\frac12}}
 \leq  C_nM,
\ee
where the constant $C_n$ depends on $n$, $T$,  $\sup_{t\in [t_0,T]}\|V(t,\cdot)\|_{X_{n+1}}$, the constant $\bar C$ in \eqref{condGamma}, $F$ and $G$ (but not on the smoothness assumption on  $b$, $c$, $z_T$, $\rho_0$, $V$, $\Gamma$ and $m$) and where $M$ is given by 
\be\label{defMtheo}
M:= \|z_T\|_{X_{n+2}}+  \|\rho_0\|_{({X_{n+1}})'}+ \sup_{t\in[t_0,T]}( \|b(t,\cdot)\|_{X_{n+1}} +\|c(t)\|_{(X_{n})'}) . 
\ee
\end{Lemma}

Remark: if $m_0$ has a smooth density which is bounded above and below by positive constants and if $(u,m)$ is the solution to \eqref{MFG}, then $V(t,x):= D_pH(x,Du(t,x))$ and $\Gamma(t,x):= D^2_{pp}H(x,Du(t,x))$ satisfy the conditions of Lemma \ref{lem:BasicEsti0}.

\begin{proof}
Without loss of generality we assume $t_0=0$.   We prove the existence of a solution to \eqref{eq:prem} by Leray-Schauder argument. The proof requires several steps, the key argument being precisely the estimates \eqref{esti1theo} and \eqref{esti3theo}. \\

\noindent {\it Step 1: Definition of the map ${\bf T}$.} Let $\beta\in (0,1/2)$ and set $X:=\cC^\beta([0,T], (X_{n+1})')$. For $\rho\in X$, we define ${\bf T}(\rho)$ as follows: let $z$ be the solution to 
\be\label{e.uhbzosdhu}
\left\{ \begin{array}{l}
\ds - \partial_t z -  \Delta z + V(t,x)\cdot Dz  = \frac{\delta F}{\delta m}(x, m(t))(\rho(t))+ b \qquad {\rm in}\; [0,T]\times \T^d,\\
\ds z(T)=\frac{\delta G}{\delta m}(x, m(T))(\rho(T))+z_T \qquad {\rm in } \; \T^d
\end{array}\right.
\ee
By our assumptions on the data, $z$ solves a parabolic equation with $\cC^{\beta/2,\beta}$ coefficients, and, by Schauder estimates, is therefore bounded in $\cC^{1+\beta/2, 2+\beta}$ when $\rho$ is bounded in $X$. Next we define $\tilde \rho$ as the solution to 
$$
\left\{ \begin{array}{l}
\ds \partial_t \tilde\rho - \Delta \tilde \rho - \dive(\tilde \rho V) - \dive (m\Gamma Dz+c)=0 \qquad  {\rm in}\; [0,T]\times \T^d\\
\ds \tilde \rho(0)=\rho_0\qquad {\rm in}\; \T^d.
\end{array}\right.
$$
Again by Schauder estimates $\tilde \rho$ is bounded in $\cC^{1+\beta/2, 2+\beta}$ for bounded $\rho$. Setting ${\bf T}(\rho):= \tilde \rho$ defines the continuous and compact map ${\bf T}:X\to X$.  

In the rest of the proof we show that, if  $\rho=\sigma {\bf T}(\rho)$ for some $(\rho,\sigma)\in X\times [0,1]$, then $\rho$ satisfies \eqref{esti3theo}. This estimate proves that the norm in $X$ of $\rho$ is bounded independently of $\sigma$. Then we can conclude by Leray-Schauder Theorem the existence of a fixed point for ${\bf T}$, which, by definition, yields a classical solution to \eqref{eq:prem}. 

From now on  we fix $(\rho,\sigma)\in X\times [0,1]$ such that $\rho=\sigma {\bf T}(\rho)$ and let $z$ be the solution to \eqref{e.uhbzosdhu}. Note that the pair $(z,\rho)$ satisfies 
$$
\left\{ \begin{array}{l}
\ds - \partial_t z -  \Delta z + V(t,x)\cdot Dz  = \sigma\left(\frac{\delta F}{\delta m}(x, m(t))(\rho(t))+ b\right) \qquad {\rm in}\; [0,T]\times \T^d,\\
\ds \partial_t \rho - \Delta  \rho - \dive( \rho V) - \sigma\dive (m\Gamma Dz+c)=0 \qquad  {\rm in}\; [0,T]\times \T^d\\
\ds  \rho(0)=\sigma\rho_0,\qquad z(T)=\sigma\left(\frac{\delta G}{\delta m}(x, m(T))(\rho(T))+z_T\right) \qquad {\rm in } \; \T^d.
\end{array}\right.
$$
Our goal is to show that \eqref{esti1theo} and \eqref{esti3theo} hold for $z$ and $\rho$ respectively. Without loss of generality we can assume  that $\sigma$ is positive, since otherwise $\rho=0$. \\

\noindent {\it Step 2: Use of the monotonicity condition.} We note that 
$$
\begin{array}{l}
\ds \frac{d}{dt} \inte z(t,x)\rho(t,x)dx \\ 
\qquad \qquad =  \ds -\sigma  \inte \Bigl[
\frac{\delta F}{\delta m}(x, m(t))\bigl(\rho(t) \bigr)+b(t) \Bigr] \rho(t,x)dx \\
\qquad \qquad \qquad \ds -\sigma \inte  Dz(t,x)\cdot
\bigl[ \Gamma(t,x) Dz(t,x) \bigr] \ m(t,x)dx - \sigma \inte  Dz(t,x)\cdot c(t,x)dx.
 \end{array}
$$
Using  the monotonicity of $F$ and $G$ and dividing by $\sigma$, we have:
$$
\begin{array}{l}
\ds \int_0^T\inte  \Gamma(t,x) Dz(t,x)\cdot Dz(t,x) m(t,x)dxdt \\
 \qquad \ds \leq  \ds -\inte 
 \bigl[ \frac{\delta G}{\delta m}(x, m(T))(\rho(T))+z_T(x) \bigr] \rho(T,x)dx\\
 \qquad \ds  +\inte z(0,x)\rho_0(x)dx-\int_0^T  \inte \bigl( b(t,x) \rho(t,x) + Dz(t,x) \cdot c(t,x)\bigr) dxdt \\
 \leq   \ds  \sup_{t\in [0,T]} \|\rho(t)\|_{({X_{n+1}})'}
 \bigl(\|z_T\|_{X_{n+1}}+\|b\| \bigr)+
 \sup_{t\in [0,T]} \|z(t,\cdot)\|_{X_{n+1}}
 \bigl(\|\rho_0\|_{({X_{n+1}})'}+\|c\| \bigr)
 \end{array}
$$
where we have set $\ds \|b\|= \sup_{t\in [0,T]} \|b(t,\cdot) \|_{X_{n+1}}$,  $\|c\|:= \sup_{t\in [0,T]} \|c(t)\|_{({X_{n}})'}$. 
Using assumption \eqref{condGamma} on $\Gamma$, we get: 
\be\label{jhblbj}
\begin{array}{l}
\ds \int_0^T\inte |\Gamma(t,x) Dz(t,x)|^2 m(t,x)dx \\
\ds \qquad \leq  \ds  C\biggl( \sup_{t\in [0,T]} \|\rho(t)\|_{({X_{n+1}})'}
\bigl(\|z_T\|_{X_{n+1}}+\|b\| \bigr)+
 \sup_{t\in [0,T]} \|z(t)\|_{X_{n+1}}
 \bigl(\|\rho_0\|_{({X_{n+1}})'}+\|c\| \bigr)\biggr).
 \end{array}
 \ee

{\it Second step: Duality technique.} Next we use a duality technique for checking the regularity of $\rho$. Let $\tau\in(0,T]$, $\xi \in X_{n+1}$ and $w$ be the solution to the backward equation
\be\label{eq:wpreli}
- \partial_t w -  \Delta w + V(t,x)\cdot Dw  =0\; {\rm in}\; [0,\tau]\times \T^d, \qquad w(\tau)= \xi \ {\rm in} \ \T^d.
\ee
Lemma \ref{l.estilinear} states that
\be\label{e.ineqwxi}
\sup_{t\in [0,T]} \|w(t,\cdot)\|_{n+1+\alpha} +\sup_{t\neq t'} \frac{\|w(t',\cdot)-w(t,\cdot)\|_{n+1+\alpha}}{|t'-t|^\frac12}
\leq C\|\xi\|_{n+1+\alpha},
\ee
where $C$ depends  on $\sup_{t\in [0,T]} \|V(t,\cdot)\|_{n+\alpha}$. 
As
\begin{equation}
\label{eq:duality:eqenplus}
\frac{d}{dt} \inte w(t,x)\rho(t,x)dx =  - \sigma\inte  Dw(t,x)\cdot \bigl(m(t,x)\Gamma(t,x) Dz(t,x)+c(t,x)\bigr)dx \;,
\end{equation}
we get (recalling that $\sigma \in (0,1]$)
$$
\begin{array}{l}
\ds \inte \xi(x)\rho(\tau,x)dx \\ 
\qquad =  \ds \sigma \inte w(0,x)\rho_0(x)dx-\sigma \int_0^\tau \inte  Dw(t,x)\cdot \bigl(
m(t,x)\Gamma(t,x) Dz(t,x)+c(t,x)\bigr)dxdt  \\
\qquad  \leq  \ds \|w(0)\|_{X_{n+1}}  \|\rho_0\|_{({X_{n+1}})'}+ \sup_{t\in [0,T]}\|Dw(t,\cdot)\|_{X_{n}}\|c\|\\
\qquad  \ds  \qquad \ds   +  \left(\int_0^T\inte |Dw(t,x)|^2m(t,x)dxdt\right)^{1/2}\left(\int_0^T\inte |\Gamma(t,x) Dz(t,x)|^2 m(t,x)dxdt 
\right)^{1/2} \\
\qquad  \leq  \ds C\left[ \|\xi\|_{X_{n+1}}  \|\rho_0\|_{({X_{n+1}})'}+ \sup_{t\in [0,T]}\|Dw(t,\cdot)\|_{X_{n}}\|c\|
\right.\\
\ds \qquad \qquad \qquad \ds \left.  
+  \|Dw\|_{\infty} \left(\int_0^T\inte |\Gamma(t,x) Dz(t,x)|^2 m(t,x)dxdt \right)^{1/2}  \right].
\end{array}
$$
Using  \eqref{jhblbj} and \eqref{e.ineqwxi} we obtain therefore
$$
\begin{array}{rl}
\ds \inte \xi(x)\rho(\tau,x)dx \; \leq &   \ds C\|\xi\|_{X_{n+1}} \biggl[   \|\rho_0\|_{({X_{n+1}})'}+ \|c\|+  \sup_{t\in [0,T]} \|\rho(t)\|_{({X_{n+1}})'}^{1/2}
\bigl(\|z_T\|_{X_{n+1}}^{1/2}+\|b\|^{1/2} \bigr)
\\
& \qquad \qquad \qquad \qquad \qquad \qquad \qquad \ds +
 \sup_{t\in [0,T]} \|z(t,\cdot)\|_{X_{n+1}}^{1/2}
 \bigl(\|\rho_0\|_{({X_{n+1}})'}^{1/2}+\|c\|^{1/2}\bigr)\biggr].
\end{array}
$$
Taking the supremum over $\xi$ with $\|\xi\|_{X_n+1}\leq 1$ and over $\tau\in [0,T]$ yields to 
$$
\begin{array}{rl}
\ds \sup_{t\in [0,T]} \|\rho(t)\|_{({X_{n+1}})'} \; \leq &   \ds C\biggl[   \|\rho_0\|_{({X_{n+1}})'}+ \|c\|+  \sup_{t\in [0,T]} \|\rho(t)\|_{({X_{n+1}})'}^{1/2}
\bigl(\|z_T\|_{X_{n+1}}^{1/2}+\|b\|^{1/2}\bigr)
\\
& \qquad \qquad \qquad \qquad \qquad \qquad \qquad \ds +
 \sup_{t\in [0,T]} \|z(t,\cdot)\|_{X_{n+1}}^{1/2}
 \bigl(\|\rho_0\|_{({X_{n+1}})'}^{1/2}+\|c\|^{1/2}\bigr)\biggr].
\end{array}
$$
Rearranging and using the definition of $M$ in \eqref{defM}, we obtain 
\be\label{esti:rho}
\sup_{t\in [0,T]} \|\rho(t)\|_{({X_{n+1}})'} \leq  C \biggl[  M+
 \sup_{t\in [0,T]} \|z(t,\cdot)\|_{X_{n+1}}^{1/2}
 \bigl(\|\rho_0\|_{({X_{n+1}})'}^{1/2}+\|c\|^{1/2}\bigr)\biggr].
\ee
We can use the same kind of argument to obtain the regularity of $\rho$ with respect to the time variable: integrating 
\eqref{eq:duality:eqenplus}
 in time and using the H\"{o}lder estimate in \eqref{e.ineqwxi} we have, for any $\tau\in [0,T]$,  
$$
\begin{array}{l}
\ds \inte  \xi(x) \bigl( \rho(\tau,x)-\rho(t,x) \bigr)dx\\
\qquad  =   \ds \inte  \bigl(w(t,x)-w(\tau,x) \bigr)\rho(t,x)dx -\sigma\int_t^\tau \inte  Dw(s,x)\cdot  \bigl[ \Gamma(s,x) Dz(s,x)  \bigr] m(s,x) dxds
\\
\qquad \qquad \qquad \ds -\sigma\int_t^\tau \inte  Dw(s,x)\cdot c(s,x)dxds\\
\qquad \leq \ds C (\tau-t)^{\frac12}\|\xi\|_{X_{n+1}}\sup_{t\in [0,T]} \|\rho(t)\|_{({X_{n+1}})'}\\
\qquad \qquad \qquad \ds +(\tau-t)^{1/2}\|Dw\|_\infty \left(\int_0^T \inte \left|\Gamma(s,x) Dz(s,x)\right|^2m(s,x) dxds\right)^{1/2}
\\
 \qquad \qquad \qquad \ds 
 +(\tau-t)\sup_{t\in [0,T]} \|w(t,\cdot)\|_{X_{n+1}}\|c(t)\| .
\end{array}
$$ 
Plugging 
\eqref{esti:rho}
into
\eqref{jhblbj}, we get that the 
root of the
left-hand side in 
\eqref{jhblbj} 
satisfies the same bound as the left-hand side
in \eqref{esti:rho}.
Therefore,
$$
\begin{array}{l}
\ds \inte  \xi(x) \bigl( \rho(\tau,x)-\rho(t,x) \bigr)dx
\\
\hspace{15pt} \leq    \ds C  (\tau-t)^{\frac12}\|\xi\|_{X_{n+1}}    \biggl[M+\sup_{t\in [0,T]} \|z(t,\cdot)\|_{X_{n+1}}^{1/2}
\bigl(\|\rho_0\|_{({X_{n+1}})'}^{1/2}+\|c\|^{1/2} \bigr)\biggr].
\end{array}
$$ 
Dividing by $(\tau-t)^{1/2}$
and taking the supremum over $\xi$ yields
\be\label{esti:rhobis}
\sup_{t\neq t'} \frac{\|\rho(t')-\rho(t)\|_{(X_{n+1})'}}{|t-t'|^{\frac12}} \leq C  \left[M
+
 \sup_{t\in [0,T]} \|z(t,\cdot)\|_{X_{n+1}}^{1/2}
 \bigl(\|\rho_0\|_{({X_{n+1}})'}^{1/2}+\|c\|^{1/2} \bigr) \right] .
\ee

\noindent {\it Third step: Estimate of $z$.} In view of the equation satisfied by $z$, we have, by Lemma \ref{l.estilinear}, 
$$
\begin{array}{l}
\ds\sup_{t\in [0,T]} \|z(t,\cdot)\|_{n+2+\alpha} +\sup_{t\neq t'} \frac{\|z(t',\cdot)-z(t,\cdot)\|_{n+2+\alpha}}{|t'-t|^\frac12} \\
\qquad \leq  \ds C\sigma \biggl[\sup_{t\in [0,T]} \left\|\frac{\delta F}{\delta m}\bigl(x, m(t)
\bigr)(\rho(t))+ b(t,\cdot)\right\|_{n+1+\alpha}+ \left\|\frac{\delta G}{\delta m}
\bigl(x, m(T)\bigr)(\rho(T))+z_T\right\|_{n+2+\alpha}\biggr],
\end{array}
$$
where $C$ depends on  $\sup_{t\in [0,T]} \|V(t,\cdot)\|_{n+1+\alpha}$.
Assumptions {\bf (HF1(${\boldsymbol n}$+1))} and {\bf (HG1(${\boldsymbol n}$+2))} on $F$ and $G$ and the fact that $\sigma\in [0,1]$ imply that the right-hand side of the previous inequality is bounded above by  
$$
C\left[\sup_{t\in [0,T]} \|\rho(t)\|_{(X_{n+1})'}+ \|b\|+ \|\rho(T)\|_{(X_{n+2})'}+\|z_T\|_{X_{n+2}}\right].
$$
Estimate  \eqref{esti:rho}  on $\rho$ then implies (since $\|\rho(T)\|_{(X_{n+2})'}\leq \|\rho(T)\|_{(X_{n+1})'}$): 
$$
\begin{array}{l}
\ds\sup_{t\in [0,T]} \|z(t,\cdot)\|_{n+2+\alpha} +\sup_{t\neq t'} \frac{\|z(t',\cdot)-z(t,\cdot)\|_{n+2+\alpha}}{|t'-t|^\frac12} \\
 \ds
\qquad \qquad \qquad \leq C\left[M+
 \sup_{t\in [0,T]} \|z(t,\cdot)\|_{X_{n+1}}^{1/2}
 \bigl(\|\rho_0\|_{({X_{n+1}})'}^{1/2}+\|c\|^{1/2}\bigr)\right].
 \end{array}
$$
Rearranging we obtain \eqref{esti1theo}. 
Plugging this estimate into \eqref{esti:rho} and  \eqref{esti:rhobis} then  gives \eqref{esti3theo}. 
%
%
\end{proof}

We now discuss the existence and uniqueness of the solution for general data. 
Given $n\geq 2$, $z_T\in X_{n+2}$, $\rho_0\in X_{n}'$, $b\in L^\infty([0,T], X_{n})$, $c\in L^\infty([0,T], [(X_{n})']^d)$, we define a solution to \eqref{eq:prem} to be a pair $(z,\rho)\in  \cC^0([0,T], X_{n+2}\times (X_{n})')$ that satisfies \eqref{eq:prem} in the sense of distribution. 


Here is our main estimate on system \eqref{eq:prem}.

\begin{Lemma}
\label{lem:BasicEsti2} Let $n\geq 0$. Assume that {\bf (HF1(${\boldsymbol n}$+1))} and {\bf (HG1(${\boldsymbol n}$+2))} hold, that $V\in C^0([0,T], X_{n+1})$ and that 
\be\label{defM}
M:= \|z_T\|_{X_{n+2}}+  \|\rho_0\|_{({X_{n+1}})'}+ \sup_{t\in[t_0,T]}( \|b(t,\cdot)\|_{X_{n+1}} +\|c(t)\|_{(X_{n})'}) \; <\; \infty. 
\ee
Then there exists a unique  solution  $(z,\rho)$ of \eqref{eq:prem} with initial condition $\rho(t_0)=\rho_0$. This solution satisfies 
$$
\ds \sup_{t\in [t_0,T]} \|(z(t,\cdot),\rho(t))\|_{{X_{n+2}}\times ({X_{n+1}})'} \;  \leq \; \ds  CM,
$$
where the constant $C$ depends on $n$, $T$,  $\sup_{t\in [0,T]}\|V(t,\cdot)\|_{X_{n+1}}$, the constant $\bar C$ in \eqref{condGamma}, $F$ and $G$. 

Moreover this solution is stable in the following sense: assume that 
\begin{itemize}
\item the data $V^k$, $m^k$, $\Gamma^k$ and $\rho_0^k$ converge to $V$, $m$, $\Gamma$ and $\rho_0$ respectively in the spaces $\cC^0([0,T]\times \T^d,\R^d)$, $\cC^0([t_0,T],\Pk)$, $\cC^1([0,T]\times \T^d,\R^d)$  and $(X_{n+1})'$, 
\item the perturbations $(b^k)$, $(c^k)$ and $z_T^k$ converge to $b$, $c$ and $z_T$, uniformly in time, in $X_{n+1}$, in $[(X_{n})']^d$ and in $X_{n+2}$ respectively.
\end{itemize}
Suppose also that the $M^k$ (defined by \eqref{defM} for the $(b^k)$, $(c^k)$, $z_T^k$ and $\rho_0^k$) are bounded above by $M+1$ and that the $\sup_{t\in [0,T]}\|V^k(t,\cdot)\|_{X_{n+1}}$ are uniformly bounded. Then the corresponding solutions $(z^k,\rho^k)$ converge to the solution $(z,\rho)$ of \eqref{eq:prem} in $\cC^0([t_0,T], \cC^{n+2+\alpha}(\T^d)\times (\cC^{n+1+\alpha}(\T^d))')$. 
\end{Lemma}

\begin{proof}  By Lemma \ref{lem:BasicEsti0}, existence of a solution holds for smooth data. We now address the case where the data are not smooth
(so that $M$ cannot be zero). We smoothen $m$, $\Gamma$, $b$, $c$, $z_T$ and $\rho_0$ into $m^k$, $\Gamma^k$, $(b^k)$, $(c^k)$, $z_T^k$ and $\rho_0^k$ in such a way that the corresponding $M^k$ is bounded by
$2M$
 and $m^k$ is a smooth density bounded above and below by positive constants. Let $(z^k,\rho^k)$ be a classical solution to \eqref{eq:prem} as given  by Lemma \ref{lem:BasicEsti0}. Using the linearity of the equation, estimates \eqref{esti1theo}, \eqref{esti3theo} imply that $(z^k,\rho^k)$ is a Cauchy sequence in $C^0([0,T], \cC^{n+2+\alpha}(\T^d)\times \cC^{n+\alpha}(\T^d))')$ and therefore converges in that space to some limit $(z,\rho)$.  Moreover $(z,\rho)$ still satisfies the estimates \eqref{esti1theo},  \eqref{esti3theo}. By {\bf (HF1(${\boldsymbol n}$+1))} and {\bf (HG1(${\boldsymbol n}$+2))}, $(\frac{\delta F}{\delta m}(\cdot, m^k)(\rho^k))$ converges uniformly in time to $(\frac{\delta F}{\delta m}(\cdot, m)(\rho))$ while $(\frac{\delta G}{\delta m}(\cdot, m^k)(\rho^k))$ converges to $(\frac{\delta G}{\delta m}(\cdot, m)(\rho))$. Therefore $(z,\rho)$ is a weak solution to \eqref{eq:prem}. 

Note also that any solution of  \eqref{eq:prem} satisfies estimates \eqref{esti1theo}, \eqref{esti3theo}, so that uniqueness holds by linearity of the problem.
\end{proof}

\subsection{Differentiability of $U$ with respect to the measure}
\label{subse:partie:1:differentiability}

In this section we show that the map $U$ has a derivative with respect to $m$. To do so, we linearize the MFG system \eqref{MFG}. Let us fix $(t_0,m_0)\in [0,T]\times \Pk$ and let $(m,u)$ be the solution to the MFG system \eqref{MFG} with initial condition $m(t_0)=m_0$. Recall that, by  definition, $U(t_0,x,m_0)=u(t_0,x)$. 

For any $\mu_0$ in a suitable space, we consider the solution $(v,\mu)$ to the {\it linearized system} 
\be\label{eq:vmubis}
\left\{ \begin{array}{l}
\ds - \partial_t v - \Delta v + D_pH(x,Du)\cdot Dv =\frac{\delta F}{\delta m}(x,m(t))(\mu(t))\\
\ds \partial_t \mu - \Delta \mu -{\rm div}( \mu  D_pH(x, D u))-{\rm div}( m  D^2_{pp}H(x, D u)Dv)=0 \\
\ds v(T,x)=\frac{\delta G}{\delta m}(x,m(T))(\mu(T)), \; \mu(t_0,\cdot)=\mu_0
\end{array}\right.
\ee
Our aim is to prove that  $U$ is of class $\cC^1$ with respect to $m$ with
$$
v(t_0,x)= \inte \frac{\delta U}{\delta m}(t_0,x,m_0,y) \mu_0(y)dy.
$$

Let us start by showing that the linearized system \eqref{eq:vmubis} has a solution and give estimates on this solution. 

\begin{Proposition}\label{prop:estiDmU} Assume that {\bf (HF1(${\boldsymbol n}$+1))} and {\bf (HG1(${\boldsymbol n}$+2))} hold for some $n\geq 0$.
\begin{itemize}
\item[(i)] Let $m_0$ be a smooth density bounded below by
a positive constant and let $\mu_0$ be smooth map on $\T^d$. Then there exists a unique solution $(v,\mu)\in \cC^{1+\alpha/2,2+\alpha}\times  \cC^{1+\alpha/2,2+\alpha}$ to system \eqref{eq:vmubis}.

\item[(ii)] If $m_0\in \Pk$ and $\mu_0\in (\cC^{n+1+\alpha})'$, there is a unique solution  $(v,\mu)$ of \eqref{eq:vmubis} (in the sense given in section \ref{subsec:LS}) and this solution satisfies 
$$
\ds \sup_{t\in [t_0,T]} \left\{ \|v(t,\cdot)\|_{n+2+\alpha}+\|\mu(t)\|_{-(n+1+\alpha)} \right\}\;  \leq \; \ds  C\|\mu_0\|_{-(n+1+\alpha)}, 
$$
where the  constant $C$ depends on $n$, $T$,  $H$, $F$ and $G$ (but not on $(t_0,m_0)$). 

\item[(iii)] The solution is stable in the following sense: assume that the triplet $(t_0^k,m_0^k,\mu_0^k)$ converges to $(t_0,m_0,\mu_0)$ in $[0,T]\times \Pk\times (\cC^{n+1+\alpha})'$. Then  the corresponding solutions $(v^k,\mu^k)$ to \eqref{eq:vmubis} (where  $(u^k,m^k)$ solves  \eqref{MFG} with initial condition $m^k(t^k_0)=m_0^k$) converge to the solution $(v,\mu)$ in $\cC^0([t_0,T], \cC^{n+2+\alpha}(\T^d)\times (\cC^{n+1+\alpha}(\T^d))')$.  
\end{itemize}
\end{Proposition}

\begin{proof} It is a straightforward application of Lemmata \ref{lem:BasicEsti0} and \ref{lem:BasicEsti2} respectively, with $V(t,x)= D_pH(x,Du(t,x))$, $\Gamma(t,x)=D^2_{pp}H(x,Du(t,x))$ and $z_T=b=c=0$.  Note that $V$ satisfies the condition that $D^\ell V$ belongs to $C^0([0,T], \cC^{n+1+\alpha}(\T^d))$ in view of Proposition \ref{prop:reguum}. 
\end{proof}

\begin{Corollary}\label{cor:repv} Under the assumptions of Proposition \ref{prop:estiDmU}, there exists, for any $(t_0,m_0)$, a $\cC^{n+2+\alpha}(\T^d)\times \cC^{n+1+\alpha}(\T^d)$ map $(x,y)\mapsto K(t_0,x,m_0,y)$ such that, for any $\mu_0\in  (\cC^{n+1+\alpha}(\T^d))'$, the $v$ component of the solution of \eqref{eq:vmubis} is given by 
\be\label{repvPart3}
v(t_0,x)= \lg K(t_0,x,m_0,\cdot), \mu_0\rg_{\cC^{n+1+\alpha}(\T^d), (\cC^{n+1+\alpha}(\T^d))'}.
\ee
Moreover 
$$
\|K(t_0,\cdot,m_0,\cdot)\|_{(n+2+\alpha,n+1+\alpha)} \leq C_n
$$
and $K$ and its derivatives in $(x,y)$ are continuous on $[0,T]\times \T^d\times \Pk\times\T^d$.
\end{Corollary}

\begin{proof} For $\ell\in \N^d$ with $|\ell |\leq n+1$ and $y\in \T^d$, let $(v^{(\ell)}(\cdot,\cdot, y), \mu^{(\ell)}(\cdot, \cdot,y))$ be the solution to \eqref{eq:vmubis} with initial condition $\mu_0=D^\ell \delta_y$  (the $\ell-$th derivative of the Dirac mass at $y$). Note that $\mu_0\in (\cC^{n+1+\alpha}(\T^d))'$. We set 
$K(t_0,x,m_0,y):= v^{(0)}(t_0,x,y)$. 

Let us check that $\partial_{y_1} K(t_0,x,m_0,y)= -v^{(e_1)}(t_0,x,y)$ where $e_1=(1,0, \dots, 0)$. Indeed, since $\ep^{-1}(\delta_{y+\ep e_1}-\delta_y)$ converges to $-
D^{e_{1} \delta_{y}}$
 in $(\cC^{n+1+\alpha})'$ while, by linearity, the map $\ep^{-1}( K(\cdot,\cdot,m_0,y+\ep e_1)-K(\cdot,\cdot,m_0,y))$ is the first component of the solution of  \eqref{eq:vmubis} with initial condition $ \ep^{-1}(\delta_{y+\ep e_1)}-\delta_y)$, this map must converge by stability (point (iii) in Proposition \ref{prop:estiDmU}) to the first component of the solution with initial condition $ -D^{e_1}\delta_y$, which is $-v^{(e_1)}(\cdot,\cdot, y)$. This proves our claim. 

One can then check in the same way by induction that, for $|\ell |\leq n+1$, 
$$
D^{\ell}_yK(t_0,x,m_0,y):= (-1)^{|\ell |}v^{(\ell)}(t_0,x,y).
$$
Finally, if $|\ell |\leq n+1$, point (ii) in Proposition \ref{prop:estiDmU} combined with the linearity of system \eqref{eq:vmubis} implies that
$$
\begin{array}{rl}
\ds \left\| K^{(\ell)}(t_0,\cdot,m_0,y)- K^{(\ell)}(t_0,\cdot,m_0,y')\right\|_{n+2+\alpha} \; \leq & \ds C\| D^\ell \delta_y-D^\ell \delta_{y'}\|_{-(n+1+\alpha)} \\
\leq &  C\| \delta_y-\delta_{y'}\|_{-\alpha} \leq 
C|y-y'|^{\alpha}. 
\end{array}
$$
Therefore $K(t_0,\cdot,m_0,\cdot)$ belongs to $\cC^{n+2+\alpha}\times \cC^{n+1+\alpha}$. 
Continuity of $K$ and its derivatives in $(t_{0},m_{0})$
follows from 
point (iii) in Proposition \ref{prop:estiDmU}.
\end{proof}

We now show that $K$ is indeed the derivative  of $U$ with respect to $m$. 

\begin{Proposition}\label{prop:diff} Assume that {\bf (HF1(${\boldsymbol n}$+1))} and {\bf (HG1(${\boldsymbol n}$+2))} hold for some $n\geq  0$. Fix $t_0\in [0,T]$, $m_0,\hat m_0\in \Pk$.  Let $(u,m)$ and $(\hat u,\hat m)$ be the solution of the MFG system \eqref{MFG} starting from $(t_0,m_0)$ and $(t_0,\hat m_0)$ respectively and let $(v,\mu)$ be the solution to \eqref{eq:vmubis} with initial condition $(t_0,\hat m_0-m_0)$. Then
$$
\ds \sup_{t\in [t_0,T]}\left\{\|\hat u(t,\cdot)-u(t,\cdot)-v(t,\cdot)\|_{n+2+\alpha}+\|\hat m(t,\cdot)-m(t,\cdot)-\mu(t,\cdot)\|_{-(n+1+\alpha)} \right\}
\; \leq C\dk^2(m_0,\hat m_0).
$$
\end{Proposition}

As a straightforward consequence, we obtain the differentiability of $U$ with respect to the measure: 

\begin{Corollary}
\label{cor:diff0} Under the assumption of Proposition \ref{prop:diff}, the map $U$ is of class $\cC^1$ (in the sense of Definition \ref{def:Diff}) with
$$
\frac{\delta U}{\delta m}(t_0,x,m_0,y)= K(t_0,x, m_0,y),
$$
whose regularity is given by Corollary \ref{cor:repv}.
Moreover, 
$$
\left\|U(t_0,\cdot,\hat m_0)-U(t_0,\cdot,m_0)- \inte \frac{\delta U}{\delta m}(t_0,\cdot,m_0,y)d(\hat m_0-m_0)(y)\right\|_{n+2+\alpha}\leq C\dk^2(m_0,\hat m_0).
$$
\end{Corollary}

\begin{Remark}\label{rem:normalization}{\rm Let us recall that the derivative $\delta U/\delta m$ is defined up to an additive constant and that our normalization condition is
$$
\inte \frac{\delta U}{\delta m}(t_0,x,m_0,y)dm_0(y)=0.
$$
Let us check that we have indeed 
\be\label{e.condmeanK}
\inte K(t,x,m_0,y)dm_0(y)=0. 
\ee
For this let us chose $\mu_0=m_0$ in \eqref{eq:vmubis}. Since, by normalization condition, $\frac{\delta F}{\delta m}(t,m(t))(m(t))=0$, for any $t\in [0,T]$, and $\frac{\delta G}{\delta m}(t,m(T))(m(T))=0$,  it is clear that the solution to \eqref{eq:vmubis} is just $(v,\mu)=(0,m)$. So, by \eqref{repvPart3}, \eqref{e.condmeanK} holds. 
}\end{Remark}

\begin{proof}[Proof of Proposition \ref{prop:diff}.] Let us set $z:=\hat u-u-v$ and $\rho:= \hat m-m-\mu$. The proof consists in estimating the pair $(z,\rho)$, which satisfies: 
$$
\left\{ \begin{array}{l}
\ds - \partial_t z - \Delta z + D_pH(x,Du)\cdot Dz =\frac{\delta F}{\delta m}(x,m(t))(\rho(t))+b\\
\ds \partial_t \rho - \Delta \rho -{\rm div}( \rho  D_pH(x, D u))-{\rm div}( m  D^2_{pp}H(x, D u)Dz)-\dive(c)=0 \\
\ds z(T,x)=\frac{\delta G}{\delta m}(x,m(T))(\rho(T))+z_T(x), \; \rho(t_0,\cdot)=0,
\end{array}\right.
$$
where 
$$
b(t,x)= A(t,x)+B(t,x)
$$
with 
$$
A(t,x)= -\int_0^1\left( D_pH(x,sD\hat u+(1-s)Du)-D_pH(x,Du)\right) \cdot D(\hat u-u)\ ds
$$
and 
$$
B(t,x)= \int_0^1\inte \left(\frac{\delta F}{\delta m}(x,s\hat m(t)+(1-s)m(t),y)-\frac{\delta F}{\delta m}(x,m(t),y)\right)d(\hat m(t)-m(t))(y)ds,
$$
$$
\begin{array}{rl}
\ds c(t)\; = & \ds  (\hat m-m)(t)D^2_{pp}H\bigl(\cdot,D u(t,\cdot)\bigr)(D\hat u-Du)(t,\cdot)
\\
& \ds  + \hat m \int_0^1\Bigl(D^2_{pp}H\bigl(\cdot, sD\hat u(t,\cdot)+(1-s)Du(t,\cdot) \bigr)-D^2_{pp}H\bigl(\cdot,Du(t,\cdot))\bigr) \Bigr)(D\hat u-Du)(t,\cdot) ds
\end{array}
$$
(note that $c(t)$ is a signed measure) and 
$$
z_T(x)=  \int_0^1\inte \left(\frac{\delta G}{\delta m}(x,s\hat m(T)+(1-s)m(T),y)-\frac{\delta G}{\delta m}(x,m(T),y)\right)d(\hat m(T)-m(T))(y)ds.
$$
We apply Lemma \ref{lem:BasicEsti2}  to get (recalling the notation $X_n= \cC^{n+\alpha}(\T^d)$):
$$
\ds \sup_{t\in [t_0,T]} \|(z(t),\rho(t))\|_{X_{n+2}\times (X_{n+1})'} \;  \leq \; \ds  C\left[\|z_T\|_{X_{n+2}}+  \|\rho_0\|_{(X_{n+1})'}+ \sup_{t\in [t_0,T]}  (\|b(t)\|_{X_{n+1}} +\|c(t)\|_{(X_{n})'})  \right].
$$
It remains to estimate the various quantities in the right-hand side. We have 
$$
 \sup_{t\in [0,T]}  \|b(t)\|_{X_{n+1}} \leq   \sup_{t\in [0,T]}  \|A\|_{X_{n+1}} +\sup_{t\in [0,T]}  \|B\|_{X_{n+1}}, 
 $$
 where,  
 $$
 \sup_{t\in [0,T]}  \|A\|_{X_{n+1}}\leq C\sup_{t\in [0,T]} \|\hat u-u \|_{X_{n+2}}^2\leq C \dk^2(m_0,\hat m_0)
 $$
according to Proposition \ref{prop:Ulip}. To estimate $B$ and $\|z_T\|_{X_{n+2}}$, we argue  in the same way:
$$
\|z_T\|_{X_{n+2}}+ \sup_{t\in [0,T]}  \|B\|_{X_{n+1}}\leq C\dk^2(m_0,\hat m_0),
$$
where we have used as above Proposition \ref{prop:Ulip} now combined with assumptions {\bf (HF1(${\boldsymbol n}$+1))} and {\bf (HG1(${\boldsymbol n}$+2))} which imply (e.g., for $F$) that, for any $m_1,m_2\in \Pk$, 
$$
\begin{array}{l}
\ds \left\|\inte \left(\frac{\delta F}{\delta m}(\cdot,m_1,y)-\frac{\delta F}{\delta m}(\cdot, m_2,y)\right)d(m_1- m_2)(y)\right\|_{X_{n+1}}\\ 
\qquad \qquad \ds \leq 
\dk(m_1,m_2)\left\| D_y\frac{\delta F}{\delta m}(\cdot,m_1,\cdot)-D_y\frac{\delta F}{\delta m}(\cdot,m_2,\cdot)\right\|_{X_{n+1}\times L^\infty} 
\; \leq \; C\dk^2(m_1,m_2).
\end{array}
$$
Finally, 
$$
\sup_{t\in [0,T]} \|c(t)\|_{(X_{n})'} \leq \sup_{t\in [0,T]} \sup_{\|\xi\|_{X_{n}}\leq 1} \lg \xi, c(t) \rg_{X_{n}, (X_{n})'},
$$
where, for  $\|\xi\|_{X_{n}}\leq 1$, 
\begin{equation*}
\begin{split}
&\lg \xi, c(t) \rg_{X_{n}, (X_{n})'} 
\\
&=    \inte \bigg\langle \xi, (\hat m-m)D^2_{pp}H
\bigl(\cdot,D u(t,\cdot)\bigr)(D\hat u-Du)(t,\cdot)
\\
&\hspace{5pt} + \hat m(t) \int_0^1\left(D^2_{pp}
H\bigl(\cdot, [sD\hat u +(1-s)Du](t,\cdot) \bigr)-D^2_{pp}H\bigl(\cdot,Du(t,\cdot)
\bigr)\right)(D\hat u-Du)(t,\cdot) ds \biggr\rangle_{X_{n}, (X_{n})'} 
\\ 
&\leq  C
\Bigl( \|\xi\|_{1}\|u-\hat u\|_{2}\dk(m_0, \hat m_0)+ 
\|\xi\|_{\cC^0}\|u-\hat u\|_{1}^2 \Bigr).
\end{split}
\end{equation*}
So again by Proposition \ref{prop:Ulip} we get 
$$
\sup_{t\in [0,T]} \|c(t)\|_{(X_{n})'} \leq C\dk^2(m_0,\hat m_0).
$$
This completes the proof.
\end{proof}

%
%

\subsection{Proof of Theorem \ref{theo:ex}}\label{subsec:prooftheomain}

\begin{proof}[Proof of Theorem \ref{theo:ex} (existence).] We check in a first step that the map $U$ defined by \eqref{defU} is a solution of the first order master equation. Let us first assume that $m_0\in \cC^1(\T^d)$, with $m_0>0$. Let $t_0>0$, $(u,m)$ be the solution of the MFG system \eqref{MFG} starting from $m_0$ at time $t_0$. Then  
\begin{equation*}
\begin{split}
\frac{U(t_0+h,x, m_0)-U(t_0,x,m_0)}{h} &= 
\frac{U(t_0+h,x, m_0)-U(t_0+h, x, m(t_0+h))}{h}
\\
& \ds \qquad + \frac{U(t_0+h, x, m(t_0+h))-U(t_0,x,m_0)}{h}.
\end{split}
\end{equation*}
Let us set $m_s= (1-s)m(t_0)+sm(t_0+h)$. Note that, by the equation satisfied by $m$ and the regularity of $U$ given by Corollary \ref{cor:diff0},  
\begin{equation*}
\begin{split}
&U\bigl(t_0+h, x, m(t_0+h)\bigr)-U\bigl(t_0+h,x,m(t_0)\bigr) 
\\ 
&\qquad =  
\int_0^1\inte \frac{\delta U}{\delta m}(t_0+h, x, m_s, y)(m(t_0+h,y)-m(t_0,y))\ dyds
\\
&\qquad =  \int_0^1\inte\int_{t_0}^{t_0+h}  \frac{\delta U}{\delta m}(t_0+h, x, m_s, y)
\Bigl(\Delta m(t,y)+\dive\bigl(
m(t,y)D_pH(y,D u(t,y)) \bigr) \Bigr)
\ dtdyds 
\\
&\qquad =  \ds  \int_0^1\inte\int_{t_0}^{t_0+h}  \Delta_y \frac{\delta U}{\delta m}(t_0+h, x, m_s, y) m(t,y)\ dtdyds 
\\
&\qquad \qquad \ds - \int_0^1\inte\int_{t_0}^{t_0+h}   D_y \frac{\delta U}{\delta m}(t_0+h, x, m_s, y)\cdot D_pH\bigl(y,D u(t,y)\bigr)\ m(t,y)\ dtdyds.
\end{split}
\end{equation*}
We can then divide by $h$ to obtain, using the continuity of $D_mU$ and its smoothness with respect to the space variables: 
$$
\begin{array}{l}
\ds \lim_{h\to0} \frac{U(t_0+h, x, m(t_0+h))- U(t_0+h,x, m_0)}{h}\\
\qquad \qquad \qquad \ds = \inte  \Bigl( \dive_y\left[D_mU\right](t_0, x, m_0, y) -
  D_mU(t_0, x, m_0, y)\cdot D_pH\bigl(y,D u(t_0,y)\bigr) \Bigr)\  m_0(y)\ dy.
\end{array}
$$
On the other hand, for $h>0$, 
$$
U(t_0+h,x, m(t_0+h))-U(t_0,x,m_0)= u(t_0+h,x)-u(t_0,x)= h\partial_t u (t_0,x)+ o(h),
$$
since $u$ is smooth, so that 
$$
\lim_{h\to0^+}  \frac{U(t_0+h, x, m(t_0+h))-U(t_0,x,m_0)}{h}= \partial_t u (t_0,x).
$$
Therefore $\partial_t U (t_0,x,m_0)$ exists and, using the equation satisfied by $u$,  is equal to
\begin{equation}
\label{jhbsldlbh}
\begin{split}
\partial_t U (t_0,x,m_0)
 &=  - \inte \dive_y \left[D_m U\right](t_0, x, m_0,y) m_0(y)dy
 \\
&\hspace{15pt} 
+ \inte  D_m U(t_0, x, m_0,y)\cdot D_pH\bigl(x,D_xU(t_0,y,m_0) \bigr) m_0(y)dy
\\
&\hspace{15pt} - \Delta_{x} U(t_0,x,m_0)+H\bigl(x,D_xU(t_0,x,m_0)\bigr)-F(x,m_0).
\end{split}
\end{equation}
This means that $U$ has a continuous time derivative at any point $(t_0,x,m_0)$ where $m_0\in \cC^1(\T^d)$ with $m_0>0$ and satisfies  
\eqref{MFGf} at such a point. By continuity of the right-hand side of \eqref{jhbsldlbh}, $U$ has a time derivative everywhere and \eqref{MFGf} holds at any point. 
\end{proof}


Next we turn to the uniqueness part of the Theorem: 

\begin{proof}[Proof of Theorem \ref{theo:ex} (uniqueness).] In order to prove the uniqueness of the solution for the master equation, we explicitly show that the solutions of the MFG system \eqref{MFG} coincide with the characteristics of the master equation.  Let $V$ be another solution to the master equation. The main point is that,  by the definition of a solution, $D^2_{x,y}\frac{\delta V}{\delta m}$ is bounded, and therefore $D_xV$ is Lipschitz continuous with respect to the measure variable. 

Let us fix $(t_0,m_0)$. In view of the Lipschitz continuity of $D_xV$, one can easily uniquely solve in $\cC^0([0,T], \Pk)$ the 
Fokker-Planck equation: 
$$
\left\{\begin{array}{l}
\partial_t \tilde m -\Delta \tilde  m -\dive \bigl(
\tilde mD_pH(x,D_xV(t,x,\tilde m)) \bigr)=0\qquad {\rm in }\; [t_0,T]\times \T^d\\
\tilde m(t_0)=m_0\qquad {\rm in }\;  \T^d.
\end{array}\right.
$$
Then let us set $\ds \tilde  u(t,x)= V(t,x,\tilde m(t))$. By the regularity properties of $V$, $\tilde u$ is at least of class $\cC^{1,2}$ with 
\begin{equation*}
\begin{split}
\partial_t \tilde u (t,x) &=
\partial_t V(t,x,\tilde m(t))+ \big\langle \frac{\delta V}{\delta m}(t,x,\tilde m(t),\cdot), \partial_t \tilde m(t)\big\rangle_{\cC^2,(\cC^2)'} 
\\
&= \partial_t V(t,x,\tilde m(t))+
\big\langle
 \frac{\delta V}{\delta m}(t,x,\tilde m(t),\cdot), \Delta \tilde m +\dive (\tilde mD_pH\bigl(x,D_xV(t,x,\tilde m) \bigr)
 \big\rangle_{\cC^2,(\cC^2)'}  
\\
&= \partial_t V(t,x,\tilde m(t))+\inte \dive_y \left[ D_mV\right](t,x,\tilde m(t),y)\ d\tilde m(t)(y)\\
& \qquad \ds - \inte  D_mV(t,x,\tilde m(t),y)\cdot D_pH(y,D_xV(t,y,\tilde m))\ d\tilde m(t)(y).
\end{split}
\end{equation*}
Recalling that $V$ satisfies the master equation, we obtain: 
\begin{equation*}
\begin{split}
\ds \partial_t \tilde u(t,x)
 &=  - \Delta_x V(t,x,\tilde m(t)) +H\bigl(x,D_xV(t,x,\tilde m(t))\bigr) -F(x,\tilde m(t)) 
 \\
&=  - \Delta \tilde u(t,x) +H(x,D\tilde u(t,x)) -F(x,\tilde m(t))
\end{split}
\end{equation*}
with terminal condition $\tilde u(T,x)=V(T,x,\tilde m(T))=G(x,\tilde m(T))$. Therefore the pair $(\tilde u,\tilde m)$ is a solution of the MFG system \eqref{MFG}. As the solution of this system is unique, we get that $V(t_0,x,m_0)= U(t_0,x,m_0)$. 
\end{proof}

\subsection{Lipschitz continuity of $\frac{\delta U}{\delta m}$ with respect to $m$}

We later need the Lipschitz continuity of the derivative of $U$ with respect to the measure. 

\begin{Proposition}\label{prop:DmULip}  Let us assume that {\bf (HF1(${\boldsymbol n}$+1))} and {\bf (HG1(${\boldsymbol n}$+2))} hold for some $n\geq 2$. Then
$$
\sup_{t\in[0,T]}  \sup_{m_1\neq m_2}\left(\dk(m_1,m_2)\right)^{-1}
\left\| \frac{\delta U}{\delta m}(t,\cdot,m_1,\cdot)-\frac{\delta U}{\delta m}(t,\cdot,m_2,\cdot)\right\|_{(n+2+\alpha,n+\alpha)} \; <\; \infty,
$$
where $C$ depends on $n$, $F$, $G$, $H$ and $T$. 
\end{Proposition}

\begin{proof} By continuity of $\frac{\delta U}{\delta m}$ 
in the measure argument (see Corollaries
\ref{cor:diff0} and \ref{cor:repv}),
we can assume without loss of generality that $m^1_0$ and $m^2_0$ are two smooth, positive  densities. 
Let $\mu_0\in \cC^\infty(\T^d)$.  We consider  $(u^1, m^1)$ and $(u^2,m^2)$ the classical solutions to the MFG system \eqref{MFG} associated with the initial condition $(t_0, m^1_0)$ and $(t_0,m^2_0)$ respectively and $(v^1,\mu^1)$ and $(v^2, \mu^2)$ the associated classical solutions to \eqref{eq:vmubis}
with $\mu^1(t_{0},\cdot) = \mu^2(t_{0},\cdot) = \mu_{0}$. 

Let us set $(z, \rho):=(v^1-v^2, \mu^1-\mu^2)$. We first write an equation for $(z, \rho)$. 
To avoid too heavy notation, we set $H_1'(t,x)=D_pH(x,Du^1(t,x))$, $H_1''(t,x)= D^2_{pp}H(x,Du^1(t,x))$,  $F'_1(x,\mu)= \int_{\T^d}
\frac{\delta F}{\delta m}(x,m^1,y) \mu(y) dy$, etc... 
 Then $(z, \rho)$ satisfies
$$
\left\{\begin{array}{l}
\ds  -\partial_t z -\Delta z +H_1'Dz=F_1'(\cdot,\rho)+b 
\\
\ds \partial_t \rho -\Delta \rho -\dive( \rho H_1')  -\dive(m^1H_1''Dz)-\dive(c)=0\\
z(T)= G_1'(\rho(T))+z_T, \; \tilde m(t_0)=0 
\end{array}\right.
$$
where 
\begin{equation*}
\begin{split}
&b(t,x) :=F_1'\bigl(x,\mu^2(t)\bigr)-F_2'\bigl(x,\mu^2(t)\bigr)  - 
\bigl[ (H_1'-H_2')Dv^2 \bigr](t,x), 
\\
&c(t,x) := \mu^2(t,x)(H_1'-H_2')(t,x)
+ \bigl[ 
(m^1H_1''-m^2H_2'')Dv^2 \bigr](t,x),
\\
&z_T(x) := G_1'(\mu^2(T))-G_2'(\mu^2(T)).
\end{split}
\end{equation*}
We apply Lemma \ref{lem:BasicEsti2} with  $V=H_1'$. Recalling the notation $X_n=\cC^{n+\alpha}(\T^d)$,  it says that, under assumptions {\bf (HF1(${\boldsymbol n}$+1))} and {\bf (HG1(${\boldsymbol n}$+2))},
$$
\begin{array}{rl}
\ds \sup_{t\in [t_0,T]}\|z(t,\cdot)\|_{X_{n+2}}\; \leq & \ds C\Bigl[\|z_T\|_{X_{n+2}}+ \sup_{t\in [0,T]}
\bigl( \|b(t,\cdot)\|_{X_{n+1}}+ \|c(t,\cdot)\|_{(X_{n})'}
\bigr)\Bigr].
\end{array}
$$
Let us estimate the various terms in the right-hand side: 
\begin{equation*}
\begin{split}
\|z_T\|_{X_{n+2}} &\leq 
\biggr\| \inte \Bigl(\frac{\delta G}{\delta m}(0,\cdot, m^1(T),y)-\frac{\delta G}{\delta m}(0,\cdot, m^2(T),y)\Bigr)
\mu^2(T,y)dy\biggr\|_{n+2+\alpha} 
\\
&\leq  \left\| \frac{\delta G}{\delta m}(0,\cdot, m^1(T),\cdot)-\frac{\delta G}{\delta m}(0,\cdot, m^2(T), \cdot)\right\|_{(n+2+\alpha,n+1+\alpha)}\|\mu^2(T)\|_{-(n+1+\alpha)}
\\
& \leq C\dk(m^1_0,m^2_0)\ \|\mu_0\|_{-(n+1+\alpha)} 
\end{split}
\end{equation*}
where we have used Proposition \ref{prop:estiDmU}-(ii) in the last inequality. Moreover, we have 
$$
\|b(t,\cdot)\|_{X_{n+1}} \leq \bigl\|F_1'\bigl(\cdot,\mu^2(t) \bigr)-F_2'
\bigl(\cdot,\mu^2(t) \bigr) \bigr\|_{X_{n+1}}
+
\bigl\| \bigl(H_1'-H_2' \bigr)(t,\cdot)Dv^2(t,\cdot) \bigr\|_{X_{n+1}},
$$
where the first term can be estimated as $z_T$: 
$$
\bigl\|F_1'\bigl(\cdot,\mu^2(t)\bigr)-F_2'
\bigl(\cdot,\mu^2(t) \bigr) \bigr\|_{X_{n+1}}\leq C \dk(m^1_0,m^2_0)  \|\mu_0\|_{-(n+1+\alpha)} .
$$
 The second one is bounded by 
\begin{equation*}
\begin{split}
\bigl\|\bigl(H_1'-H_2'\bigr)(t,\cdot)Dv^2(t,\cdot)\bigr\|_{X_{n+1}}
&=
\bigl\| \bigl(D_pH\bigl(\cdot, Du^1(t,\cdot) \bigr)-
D_pH(\cdot, Du^2(t,\cdot)) \bigr)Dv^2(t,\cdot)
\bigr\|_{n+1+\alpha} 
\\
&\leq \|(u^1-u^2)(t,\cdot)\|_{n+2+\alpha}\|v^2(t,\cdot)\|_{n+2+\alpha} 
\\ 
&\leq C \dk(m^1_0,m^2_0) \ \|\mu_0\|_{-(n+1+\alpha)},
\end{split}
\end{equation*}
where the last inequality comes from Proposition \ref{prop:Ulip} and Proposition \ref{prop:estiDmU} thanks to assumptions {\bf (HF1(${\boldsymbol n}$+1))} and {\bf (HG1(${\boldsymbol n}$+2))}. 
Finally, by a similar argument,
\begin{equation*}
\begin{split}
\|c(t)\|_{-(n+\alpha)}  
&=  \ds \sup_{\|\phi\|_{n+\alpha}\leq 1} \inte \phi(x) \Bigl[ \mu^2\bigl(H_1'-H_2'\bigr) 
+\bigl((m^1-m^2)H_1''+m^2(H''_1-H_2'')Dv^2\bigr)\Bigr](t,x) dx  \\
&\leq \sup_{\|\phi\|_{n+\alpha}\leq 1} 
\bigl\| \phi  (H_1'-H_2')(t,\cdot) 
\bigr\|_{n+\alpha}
\|\mu^2(t,\cdot)\|_{-(n+\alpha)}
\\
&\hspace{15pt} + \dk
\bigl(m^1(t),m^2(t)\bigr)
\sup_{\|\phi\|_1\leq 1} 
\bigl\|\phi  
\bigl(H_1''Dv^2\bigr)(t,\cdot)
\bigr\|_{1}+ \sup_{\|\phi\|_{0}\leq 1} 
 \bigl\|\phi(H_1''-H_2'')(t,\cdot)Dv^2(t,\cdot)
 \bigr\|_{0} 
 \\
&\leq  C\bigl\|
(u^1-u^2)(t,\cdot)
\bigr\|_{n+\alpha}\|\mu_0\|_{-(n+\alpha)}
\\
&\hspace{15pt}
+ C \dk\bigl(m^1(t),m^2(t)\bigr)
\|v^2(t,\cdot)\|_{2}
+ C \bigl\|
\bigl(
u^1-u^2
\bigr)(t,\cdot)
\bigr\|_{1}\|v^2(t,\cdot)\|_{1}
\\
&\leq  C\dk(m^1_0,m^2_0)\|\mu_0\|_{-(n+\alpha)}.
\end{split}
\end{equation*}
This shows that 
$$
\begin{array}{rl}
\ds \sup_{t\in [t_0,T]}\|z(t,\cdot)\|_{n+2+\alpha}\; \leq & \ds  C\dk(m^1_0,m^2_0)\|\mu_0\|_{-(n+\alpha)}.
\end{array}
$$
As 
$$
z(t_0,x)= \inte \left(\frac{\delta U}{\delta m}(t_0, x, m^1_0, y)-\frac{\delta U}{\delta m}(t_0, x, m^2_0,y) \right) \mu_0(y)dy, 
$$
we have proved 
$$
 \sup_{m_1\neq m_2}\left(\dk(m_1,m_2)\right)^{-1}
\left\| \frac{\delta U}{\delta m}(t_0,\cdot,m_1,\cdot)-\frac{\delta U}{\delta m}(t_0,\cdot,m_2,\cdot)\right\|_{(n+2+\alpha,n+\alpha)} \; \leq\; C,
$$
where the supremum is taken over smooth densities. The map $\frac{\delta U}{\delta m}$ being continuous, we can remove the restriction of the measures $m_1$ and $m_2$ by approximation to get the full result.
\end{proof}

\subsection{Link with the optimal control of Fokker-Planck equation}

We now explain that, when $F$ and $G$ derive from potentials functions ${\mathcal F}$ and ${\mathcal G}$, the space derivative $D_xU$ is nothing but the derivative with respect to the measure of the solution ${\mathcal U}$ of a Hamilton-Jacobi equation stated in the space of measures. The fact that the mean field game system can be viewed as a necessary condition for an optimal transport of the Kolmogorov equation goes back to  Lasry and Lions \cite{LL07mf}. As explained by Lions \cite{LcoursColl}, one can also write the value function of this optimal control problem, which turns out to be a Hamilton-Jacobi equation in the space of measure. The (directional) derivative with respect to the measure of the value function is then (at least formally) the solution of the master equation. This is  rigorously derived, for short horizon and first order (in space and measure) master equation by Gangbo and Swiech \cite{GS14-2}. We show here that this holds true for the master equation without common noise.

Let us assume that $F$ and $G$ derive from $\cC^1$  potential maps ${\mathcal F}:\Pk\to\R$ and ${\mathcal G}:\Pk\to \R$: 
\be\label{e.mathcalFG}
F(x,m)= \frac{\delta {\mathcal F}}{\delta m}(x,m), \qquad G(x,m)= \frac{\delta {\mathcal G}}{\delta m}(x,m).
\ee
Note for later use that the monotonicity of $F$ and $G$ implies the convexity of ${\mathcal F}$ and ${\mathcal G}$. 

\begin{Theorem}\label{theo:HJB} Under the assumptions of Theorem \ref{theo:ex}, let $U$ be  the solution to the master equation \eqref{masterwithoutCN} and suppose that \eqref{e.mathcalFG} holds.  Then the Hamilton-Jacobi-Bellmann equation
\be\label{e.HJB}
\left\{\begin{array}{l}
\ds -\partial_t {\mathcal U}(t,m)+ \inte H\left(y, D_m{\mathcal U}(t,m,y)\right)dm(y) -\inte \dive\left[D_m{\mathcal U}\right](t,m,y)dm(y)= {\mathcal F}(m)\\
\ds \qquad \qquad \qquad \qquad \qquad {\rm in} \; [0,T]\times \Pk,
\\
\ds {\mathcal U}(T,m)= {\mathcal G}(m)  \qquad {\rm in} \;  \Pk,
\end{array}\right.
\ee
has a unique classical solution ${\mathcal U}$ and 
\be\label{e.MdmathcalU}
D_m{\mathcal U}(t,x,m)= D_xU(t,x,m)\qquad \forall (t,x,m)\in [0,T]\times \T^d\times \Pk. 
\ee
\end{Theorem}

We represent the solution ${\mathcal U}$ of \eqref{e.HJB} as the value function of an optimal control problem: for an initial condition $(t_0,m_0)\in [0,T]\times \Pk$, let 
\be\label{def:mathcalU}
{\mathcal U}(t_0,m_0) := \inf_{(m,\alpha)} \int_{t_0}^T \biggl[ \inte H^*\left(x,\alpha(t,x)\right)m(t,dx) \biggr] dt + \int_{t_0}^T {\mathcal F}\bigl(m(t)
\bigr)dt + {\mathcal G}\bigl(m(T)\bigr)
\ee
(where $H^*$ is the convex conjugate of $H$ with respect to the second variable) under the constraint that 
$m \in 
\cC^0([0,T], \Pk)$,
$\alpha$ is a bounded and Borel measurable 
function from $[0,T] \times \T^d$ into $\R^d$
and the pair $(m,\alpha)$ satisfies in the sense of distribution: 
\begin{equation}
\label{eq:mkv:control}
\partial_t m -\Delta m - \dive \bigl(
 \alpha 
m \bigr)
 = 0\; {\rm in}\; [0,T]\times \T^d, \qquad m(t_0)=m_0\; {\rm in}\; \T^d.
\end{equation}
Of course, \eqref{eq:mkv:control} is understood
as the Fokker-Planck equation describing the flow of measures
generated on the torus by the SDE
\begin{equation*}
dZ_{t} = - \alpha(t,Z_{t}) dt + dB_{t}, \quad t \in [0,T],
\end{equation*} 
which is is known to be uniquely solvable in the weak sense. 
Notice that, throughout the subsection, 
we shall use, as in 
\eqref{def:mathcalU}, 
the notation $m(t,dx)$ to denote
the integral on the torus with respect to the (time-dependent)
measure $m(t)$.
\vspace{4pt}

The following characterization of the optimal path for ${\mathcal U}$ is due to Lasry and Lions \cite{LL07mf}: 
\begin{Proposition}\label{p.minim} For an initial position $(t_0,m_0)\in [0,T]\times \Pk$, let $(u,m)$ be the solution of the MFG system \eqref{MFG}. Then 
$(m,\alpha) = (m,D_pH(\cdot,Du(\cdot,\cdot)))$ is a minimizer for ${\mathcal U}(t_0,m_0)$.  
\end{Proposition}

\begin{proof} For 
a function 
$\hat{m} \in \cC^0([0,T],\Pk)$
and 
a bounded and measurable 
function 
$\hat{\alpha}$
from $[0,T] \times \T^d$ into $\R^d$, we let
$$
J(\hat m, \hat\alpha):= \int_{t_0}^T \inte H^*\bigl(x, \hat{\alpha}(t,x)\bigr)d\hat m(t)+ \int_{t_0}^T {\mathcal F}\bigl(\hat m(t)\bigr)dt + {\mathcal G}\bigl(\hat m(T)\bigr)
$$
where $\hat m$ solves 
$$
\partial_t \hat m -\Delta \hat m - \dive 
\bigl(\hat{\alpha}  \hat m 
\bigr) = 0\; {\rm in}\; [0,T]\times \T^d, \qquad \hat m(t_0)=m_0\; {\rm in}\; \T^d.
$$
As, for any 
$m' \in {\mathcal P}(\T^d), \alpha' \in \R^d$, 
$$
H^*(x,\alpha') =\sup_{p\in \R^d} \left( \alpha' \cdot p -
H(x,p)\right), 
$$
we have, by convexity of ${\mathcal F}$ and ${\mathcal G}$, 
\begin{equation*}
\begin{split}
&J(\hat m, \hat \alpha) 
\\
&\geq  \int_{t_0}^T
\biggl[ \inte
\Bigl[
 \hat \alpha(t,x) \cdot Du(t,x) -H \bigl(x, Du(t,x)\bigr)
 \Bigr] \hat{m}(t,dx) 
 \biggr] dt 
 \\
&\hspace{5pt} + \int_{t_0}^T 
\Bigl[{\mathcal F}\bigl(m(t)\bigr)+ F\bigl(\cdot,m(t) \bigr) 
\bigl(\hat m(t)-m(t)
\bigr)
\Bigr] dt + {\mathcal G}\bigl(m(T) \bigr)+ G\bigl(\cdot,m(T))
\bigl(\hat m(T)-m(T)\bigr) 
\\
&= J(m, \alpha) 
\\
&\hspace{5pt}+\int_{t_0}^T \biggl[\inte 
 \Bigl[
 Du(t,x) \cdot 
\bigl( 
 \hat \alpha(t,x)
 \hat{m}(t,dx)
 -
\alpha(t,x)
{m}(t,dx)
\bigr)  
-H\bigl(x, Du(t,x)\bigr)
\bigl( \hat{m} - \hat{m}\bigr)(t,dx)
\Bigr]
\biggr] dt
\\
&\hspace{5pt}
 + \int_{t_0}^T F\bigl(\cdot,m(t)\bigr)(\hat m-m)(t)dt
 +  G\bigl(\cdot,m(T)\bigr)\bigl(\hat m(T)-m(T)\bigr). 
\end{split}
\end{equation*}
because
$$
\alpha(t,x) \cdot Du(t,x) - H\bigl(x, Du(t,x)\bigr) = H^*
\bigl(x,\alpha(t,x)\bigr).
$$
Using the equation satisfied by $(m,w)$ and $(\hat m, \hat w)$ we have
\begin{equation*}
\begin{split}
&\int_{t_0}^T \biggl[ \inte 
Du(t,x) 
\cdot \bigl( \hat{\alpha}(t,x) 
\hat{m}(t,dx) 
- \alpha(t,x) m(t,dx)
\bigr)
\biggr] dt
\\
&\hspace{5pt} =  -\left[ \inte u(t,x) (\hat m-m)(t,dx) \right]_0^T +\int_0^T
\biggl[ \inte (\partial_t u+\Delta u)(t,x) \bigl(\hat{m}-m\bigr)
(t,dx) \biggr] dt
\\
&\hspace{5pt} =   - G\bigl(\cdot,m(T)\bigr)\bigl(\hat m(T)-m(T)
\bigr)+ \int_0^T
\biggl[ \inte 
\Bigl(H\bigl(x,Du(t,x)\bigr)-F\bigl(x,m(t)\bigr)\Bigr)
\bigl(\hat m-m)(t,dx)\biggr] dt .
 \end{split}
\end{equation*}
This proves that $J(\hat m, \hat \alpha)\geq J( m, \alpha)$ and shows the optimality of $(m, \alpha)$. 
\end{proof}

\begin{proof}[Proof of Theorem \ref{theo:HJB}] 
\textit{First step.}
Let us first check that ${\mathcal U}$, defined by \eqref{def:mathcalU}, is $\cC^1$ with respect to $m$ and satisfies
\be
\label{mathcalU=U}
\frac{\delta {\mathcal U}}{\delta m}(t,x,m)= U(t,x,m)-\inte U(t,y,m)dm(y)\qquad \forall (t,x,m)\in [0,T]\times \T^d\times \Pk. 
\ee
Assume for a while that 
\eqref{mathcalU=U} holds true. 
Then, taking the derivative with respect to $x$ on both sides shows \eqref{e.MdmathcalU}.

We now prove \eqref{mathcalU=U}.
 Let $m_0, \hat m_0$ be two initial measures, $(u,m)$ and $(\hat u, \hat m)$ be the solutions of the MFG system \eqref{MFG} with initial conditions $(t_0,m_0)$ and $(t_0, \hat m_0)$ respectively. 
Let also $(v,\mu)$ be the solution of the linearized system \eqref{eq:vmubis} with initial condition $(t_0, \hat m_0-m_0)$. Let us recall that, according to Proposition \ref{prop:diff}, we have 
\be\label{e.estiqjhfbdjhsl}
\ds \sup_{t\in [t_0,T]} \left\{ \|\hat u-u-v\|_{n+2+\alpha}+ \|\hat m-m-\mu \|_{-(n+1+\alpha)} \right\} \; \leq C\dk^2(m_0,\hat m_0)
\ee
while Proposition \ref{prop:Ulip} and Proposition \ref{prop:estiDmU} imply that 
$$
\sup_{t\in [0,T]} \left\{ \|\hat u-u\|_{n+2+\alpha}+ \|\mu \|_{-(n+1+\alpha)} \right\}\leq C\dk(m_0,\hat m_0).
$$
Our aim is to show that 
\begin{equation}
\label{eq:proof:mkv}
{\mathcal U}(t_0,\hat m_0)-{\mathcal U}(t_0, m_0) -\inte U(t_0,x,m_0)d(\hat m_0-m_0)(x) = O\bigl(\dk^2(m_0,\hat m_0)
\bigr).
\end{equation}
Indeed, if \eqref{eq:proof:mkv}
 holds true, then
$U$ is a derivative of ${\mathcal U}$ and, by convention \eqref{ConvCondDeriv}, proves \eqref{mathcalU=U}. 
\vspace{4pt}

\textit{Second step.}
We now turn to the proof of 
\eqref{eq:proof:mkv}.
Since $(u,m)$ and $(\hat u, \hat m)$ are optimal in ${\mathcal U}(t_0, m_0)$ and ${\mathcal U}(t_0,\hat m_0)$ respectively, we have 
\begin{equation*}
\begin{split}
&{\mathcal U}(t_0,\hat m_0)-{\mathcal U}(t_0, m_0)
\\
&=  \int_{t_0}^T \left( \inte H^*
\bigl(x, D_pH(x,D\hat u(t,x))\bigr)
\hat m(t,dx)-
\inte H^*\bigl(x, D_pH(x,D u(t,x))\bigr)
m(t,dx)\right)dt
\\
&\hspace{15pt}+ \int_{t_0}^T \Bigl(
{\mathcal F}\bigl(\hat m(t)\bigr)-{\mathcal F}\bigl(m(t)\bigr)\Bigr)dt + {\mathcal G}\bigl(\hat m(T)\bigr)- 
{\mathcal G}\bigl(m(T)\bigr).
\end{split}
\end{equation*}
Note that, by \eqref{e.estiqjhfbdjhsl},  
\begin{equation*}
\begin{split}
&\int_{t_0}^T \left( 
\inte H^*\Bigl(x, D_pH\bigl(x,D\hat u(t,x)\bigr)\Bigr) \hat m(t,dx)
-
\inte H^*\Bigl(x, D_pH\bigl(x,D u(t,x)\bigr)\Bigr) m(t,dx)
\right)dt
\\
&= \int_{t_0}^T\biggl( \inte H^*\Bigl(x, D_pH\bigl(x,D u(t,x)\bigr)
\Bigr) \mu(t,dx) 
\\
&\hspace{5pt} + \inte   D_q
H^*\Bigl(x, D_pH\bigl(x,Du(t,x)\bigr)
\Bigr)
\cdot  \bigl[ D^2_{pp}H\bigl(x,Du(t,x) \bigr)Dv(t,x)
\bigr]
m(t,dx) \biggr)
dt
+O\bigl(\dk^2(m_0,\hat m_0)\bigr)
\\
&= \int_{t_0}^T \biggl(\inte  \Bigl(Du(t,x)
\cdot  D_pH\bigl(x,D u(t,x)\bigr)-H\bigl(x,Du(t,x)\bigr)
\Bigr)\mu(t,dx) 
\\
&\hspace{5pt}
+\inte Du(t,x)\cdot 
\bigl[ D^2_{pp}H\bigl(x,Du(t,x)\bigr) Dv(t,x)
\bigr] m(t,dx)
\biggr)dt
+O\bigl(\dk^2(m_0,\hat m_0)\bigr),
\end{split}
\end{equation*}
where we have used the properties of the Fenchel conjugate in the last equality, while
$$
\begin{array}{l}
\ds \int_{t_0}^T \Bigl[{\mathcal F}\bigl(\hat m(t)\bigr)-{\mathcal F}\bigl(m(t)\bigr)\Bigr]dt + {\mathcal G}\bigl(\hat m(T)\bigr)- {\mathcal G}\bigl(m(T)\bigr)
\\
\ds \qquad = \int_{t_0}^T \biggl( 
\inte F\bigl(x,m(t)\bigr)\mu(t,dx) \bigr) dt + \inte 
G\bigl(x, m(T)\bigr) \mu(T,dx) +O\bigl(\dk^2(m_0,\hat m_0)\bigr).
\end{array}
$$
Recalling the equation satisfied by $u$ and $\mu$, we have 
\begin{equation*}
\begin{split}
&\frac{d}{dt} \inte u(t,x)\mu(t,dx) 
\\
&= 
\inte \Bigl[H\bigl(x,Du(t,x)
\bigr)-F\bigl(x,m(t)\bigr)\Bigr] \mu(t,dx) -
\inte  
Du(t,x) \cdot  D_pH\bigl(x,Du(t,x) \bigr) 
\mu(t,dx)
\\
&\hspace{15pt} - \inte  Du(t,x) \cdot
\Bigl[ D^2_{pp}H
\bigl(x,Du(t,x)\bigr)
Dv(t,x)
\Bigr]
 m(t,dx).
\end{split}
\end{equation*}
Putting the last three identities
together,
 we obtain 
\begin{equation*}
\begin{split}
&{\mathcal U}(t_0,\hat m_0)-{\mathcal U}(t_0, m_0)
\\
&= -\int_{t_0}^T 
\biggl( \frac{d}{dt} \inte u(t,x)\mu(t,dx)
\biggr) dt
+\inte G\bigl(x, m(T)\bigr)\mu(T,dx) +O
\bigl(\dk^2(m_0,\hat m_0)\bigr) 
\\
&= \inte u(t_0,x)\mu(t_0,dx)+O\bigl(\dk^2(m_0,\hat m_0)\bigr) 
 = \inte U(t_0,x,m_0)d(\hat m_0-m_0)(x)+O
 \bigl(\dk^2(m_0,\hat m_0)\bigr).
\end{split}
\end{equation*}
This completes the proof of \eqref{mathcalU=U}.
\vspace{4pt}

\textit{Third step.}
Next we show that ${\mathcal U}$ is a classical solution to the Hamilton-Jacobi equation \eqref{e.HJB}. 
Let us fix $(t_0,m_0)\in [0,T)\times \Pk$, where $m_0$ has a smooth, positive density. Let also $(u,m)$ be the solution of the MFG system \eqref{MFG} with initial condition $(t_0,m_0)$. Proposition \ref{p.minim} states that $(m, D_pH(\cdot,Du(\cdot,\cdot)))$ is a minimizer for ${\mathcal U}(t_0,m_0)$. By standard dynamic programming principle, we have therefore, for any $h\in (0,T-t_0)$, 
\begin{equation}
\label{e.PPD}
\begin{split}
{\mathcal U}(t_0,m_0)&= \int_{t_0}^{t_0+h} \inte H^*\Bigl(x, D_pH\bigl(x,Du(t,x)\bigr)\Bigr)m(t,x)dx dt
\\
&\hspace{20pt} + \int_{t_0}^{t_0+h} {\mathcal F}\bigl(m(t)\bigr)dt + {\mathcal U}(t_0+h,m(t_0+h)). 
\end{split}
\end{equation}
Now we note that
\begin{multline}\label{e.mathcalUt0+h}
\frac{{\mathcal U}(t_0+h,m_0)-{\mathcal U}(t_0,m_0)}{h} \\
= 
\frac{{\mathcal U}(t_0+h,m_0)-{\mathcal U}(t_0+h,m(t_0+h))}{h}
+
\frac{{\mathcal U}(t_0+h,m(t_0+h))-{\mathcal U}(t_0,m_0)}{h}.
\end{multline}
We can handle the  first term in the right-hand side of \eqref{e.mathcalUt0+h}  by using the fact that ${\mathcal U}$ is $\cC^1$ with respect to $m$. Letting $m_{s,h}:= (1-s) m_0+sm(t_0+h)$), we have:
\begin{equation*}
\begin{split}
&{\mathcal U}\bigl(t_0+h,m(t_0+h)\bigr)-{\mathcal U}\bigl(t_0+h,m_0\bigr) 
\\
&\hspace{15pt}= 
\int_0^1 \inte \frac{\delta {\mathcal U}}{\delta m}\bigl(t_0+h, m_{s,h},y\bigr)d\bigl(m(t_0+h)-m_0\bigr)(y)ds 
\\
&\hspace{15pt}= 
- \int_0^1 \inte \int_{t_0}^{t_0+h} D_m {\mathcal U}\bigl(t_0+h, m_{s,h},y\bigr)\cdot \Bigl(D m(t,y)+ D_pH\bigl(y,Du(t,y) \bigr)m(t,y)\Bigr)\ dt dy ds.
\end{split}
\end{equation*}
Dividing by $h$, letting $h\to 0^+$ and rearranging gives
$$
\begin{array}{l}
\ds \lim_{h\to0^+} \frac{{\mathcal U}(t_0+h,m(t_0+h))-{\mathcal U}(t_0+h,m_0)}{h} 
\\
\ds \qquad = 
\inte \dive\left[D_m {\mathcal U}\right] (t_0, m_0,y) dm_0(y)-\inte  D_m {\mathcal U}(t_0, m_0,y) \cdot D_pH\bigl(y,Du(t_0,y)\bigr)\ dm_0(y) .
\end{array}
$$
To handle the second term in the right-hand side of \eqref{e.mathcalUt0+h}, we use \eqref{e.PPD} and get 
$$
\lim_{h\to0^+}
 \frac{{\mathcal U}(t_0+h,m(t_0+h))-{\mathcal U}(t_0,m_0)}{h}
=
-\inte H^*\bigl(x, D_pH(x,Du(t_0,x))\bigr)dm_0(x)- {\mathcal F}(m_0).
$$
As $Du(t_0,x)= D_xU(t_0, x,m_0)= D_m{\mathcal U}(t_0,m_0,x)$, we have 
\begin{equation*}
\begin{split}
&-H^*\bigl(x, D_pH\bigl(x,Du(t,x)\bigr)\bigr)- D_m {\mathcal U}(t_0, m_0,x) \cdot D_pH\bigl(y,Du(t_0,y)
\bigr) 
\\
&\hspace{15pt} =
-H^*\bigl(x, D_pH\bigl(x,D_m{\mathcal U}(t_0,m_0,x)\bigr)\bigr)+ D_m {\mathcal U}(t_0, m_0,x) \cdot D_pH
\bigl(x,D_m{\mathcal U}(t_0,m_0,x)\bigr) 
\\
&\hspace{15pt}= H\bigl(x,D_m{\mathcal U}(t_0,m_0,x)\bigr).
\end{split}
\end{equation*}
Collecting the above equalities, we obtain therefore 
\begin{multline*}
\lim_{h\to0^+}\frac{{\mathcal U}(t_0+h,m_0)-{\mathcal U}(t_0,m_0)}{h} \\
= 
-\inte \dive\left[D_m {\mathcal U}\right] (t_0, m_0,y) dm_0(y) + \inte H\bigl(x,D_m{\mathcal U}(t_0,m_0,x)\bigr)dm_0(x) - {\mathcal F}(m_0). 
\end{multline*}
As the right-hand side of the above equality is continuous in all variables, this shows that ${\mathcal U}$ is continuously derivable with respect to $t$ and satisfies \eqref{e.HJB}. 
\vspace{4pt}

\textit{Last step.}
We finally check that ${\mathcal U}$ is the unique classical solution to \eqref{e.HJB}. For this we use the standard comparison argument. Let ${\mathcal V}$ be another classical solution and assume that ${\mathcal V}\neq {\mathcal U}$. To fix the ideas, let us suppose that $\sup({\mathcal V}-{\mathcal U})$ is positive. Then, for any $\ep>0$ small enough, 
$$
\sup_{(t,x)\in (0,T]\times \Pk} {\mathcal V}(t,m)- {\mathcal U}(t,m)+\ep\log(\frac{t}{T}) 
$$
is positive. Let $(\hat t, \hat m)$ be a maximum point. Note that $\hat t<T$ because ${\mathcal V}(T,\cdot)={\mathcal U}(T,\cdot)$. 
By optimality of $(\hat t, \hat m)$ and regularity of ${\mathcal V}$ and ${\mathcal U}$, we have:
$$
\partial_t {\mathcal V}(\hat t, \hat m)- \partial_t{\mathcal U}(\hat t, \hat m)+\frac{\ep}{\hat t}=0
\qquad {\rm 
and}
\qquad
\frac{\delta {\mathcal V}}{\delta m}(\hat t,\hat m,\cdot)= \frac{\delta {\mathcal U}}{\delta m}(\hat t,\hat m,\cdot),
$$
so that
$$
D_m{\mathcal V}(\hat t,\hat m,\cdot)= D_m{\mathcal U}(\hat t,\hat m,\cdot)\; {\rm and} \; 
\dive\left[D_m{\mathcal V}\right](\hat t,\hat m,\cdot)= \dive\left[D_m{\mathcal U}\right](\hat t,\hat m,\cdot). 
$$
Using the equation satisfied by ${\mathcal U}$ and ${\mathcal V}$ yields to $\frac{\ep}{\hat t}=0$, a contradiction. 
\end{proof}

\newpage

\section{MFG system with a common noise}
\label{se:common:noise}
The main purpose of the two next sections is to show that the same approach as the one developed in the 
previous section may be implemented in the case when the whole system 
is forced by a so-called `common noise'. Such a common noise 
is sometimes referred to as a 'systemic noise', see for instance Lions' lectures 
at the \textit{Coll\`ege de France}. 

Thinking of a game with a finite number of players, the 
common noise describes some noise that affects all the players
in the same way, so that the dynamics of one given particle 
reads\footnote{Equation
\eqref{eq:sec:3:X} 
is set on $\R^d$ but 
the solution may be canonically mapped onto
$\T^d$ since the coefficients are periodic: 
When the process $(X_{t})_{t \in [0,T]}$
is initialized with a probability 
measure on $\T^d$, 
the dynamics on the torus are independent of 
the representative in $\R^d$
of the initial condition.}
\begin{equation}
\label{eq:sec:3:X}
dX_{t} = - D_{p} H(X_{t},Du_{t}(X_{t})) dt + \sqrt{2} dB_{t} + \sqrt{2\beta} dW_{t}, \quad t \in [0,T],
\end{equation} 
where $\beta$ is a nonnegative parameter, $B$ and $W$ are two independent $d$-dimensional Wiener processes, 
$B$ standing for the same idiosyncratic noise as in the previous section and $W$ now standing for the so-called common noise. 
Throughout the section, 
we use the standard convention from the theory of stochastic processes that consists in 
indicating the time parameter as an index in random functions.

As we shall see next, the effect of the common noise is to randomize the MFG equilibria
so that, with the same notations as above, $(m_{t})_{t \geq 0}$ becomes a random flow of measures. Precisely, it reads 
as the flow of conditional marginal measures of $(X_{t})_{t \in [0,T]}$
given the realization of $W$. In order to distinguish things properly, we shall refer
the situation discussed in the previous section to as the `deterministic' or `first-order' case. 
In this way, we point out that, without common noise, equilibria are  
completely deterministic. Compared to the notation of the introduction or of section \ref{sec:prelim}, we let the level of common noise $\beta$ be equal to $1$ throughout the section: this is without loss of generality and simplifies (a little) the notation.

This section is specifically devoted to the analysis of the MFG system
 in the presence of the common noise (see 
 \eqref{e.MFGstoch}).
Using a \textit{continuation like} argument (instead of the classical strategy based on 
the Schauder fixed point theorem), we investigate existence and uniqueness 
of a solution. On the model of the first-order case, we also investigate the 
linearized system. The derivation of the master equation is deferred to the next section. 
The use of the continuation method 
in the analysis of MFG systems is a new point, which is directly inspired from 
the analysis of finite dimensional forward-backward systems: Its application is 
here made possible thanks to the monotonicity assumption required on $F$ and $G$. 

{\it As already mentioned, we assume without loss of generality that $\beta=1$ throughout this section.}

\subsection{Stochastic Fokker-Planck/Hamilton-Jacobi-Bellman System}

The major difficulty for handling MFG with a common noise 
is that the system made of the Fokker-Planck and Hamilton-Jacobi-Bellman
equations in 
\eqref{MFG}
becomes stochastic. 
Its general form has been already discussed in \cite{CaDe14}. 
Both the forward and the backward equations become stochastic as 
both the equilibrium $(m_{t})_{0 \leq t \leq T}$
and the value function $(u_{t})_{0 \leq t\leq T}$
depend upon the realization of the common noise $W$. 
Unfortunately, 
the stochastic system does not consist of a simple 
randomization of the coefficients: In order to ensure that 
the value function $u_{t}$ at time $t$
depends upon the past before $t$ in the realization of $(W_{s})_{0 \leq s \leq T}$,
the backward equation incorporates an additional 
correction term which is reminiscent of the theory of finite-dimensional 
backward stochastic differential equations. 

The Fokker-Planck equation
satisfied by $(m_{t})_{t \in [0,T]}$ reads
\begin{equation}
\label{eq:se:3:m}
d_{t} m_{t} = \bigl[  2 \Delta m_{t} + {\rm div} \bigl( m_{t} D_{p} H(m_{t},D u_{t}) 
\bigr) \bigr] dt - \sqrt{2} {\rm div} ( m_{t} dW_{t} \bigr), \quad t \in [0,T]. 
\end{equation}
The value function $u$ is sought as the solution of the stochastic HJB equation:
\begin{equation}
\label{eq:se:3:u}
 d_{t} u_{t} = \bigl\{ -  2 \Delta u_{t} + H(x,Du_{t}) - F(x,m_{t}) -  \sqrt{2} {\rm div}(v_{t}) \bigr\} dt+ v_{t} \cdot dW_{t},
\end{equation}
where, at any time $t$, $v_{t}$ is a 
random function of $x$ with values in $\R^d$. 
Once again, we emphasize that the term $v_{t} \cdot dW_{t}
= \sum_{i=1}^d v_{t}^i dW_{t}^i$ permits to guarantee that 
$(u_{t})_{0 \leq t \leq T}$ is adapted with respect to the filtration generated by the common noise. 
The extra term $- \sqrt{2} {\rm div}(v_{t})$
may be explained by the so-called It\^o-Wentzell formula, 
which is the chain rule for random fields applied to random processes,
see for instance \cite{Kunita1990}. 
It permits to cancel out the bracket that arises in the application of 
the It\^o-Wentzell formula\footnote{In the application of It\^o-Wentzell formula, 
$u_{t}$ is seen as a (random) 
periodic function from $\R^d$ to $\R$.} to $(u_{t}(X_{t}))_{t \in [0,T]}$, with $(X_{t})_{0 \leq t \leq T}$ as in 
\eqref{eq:sec:3:X}. Indeed, when expanding the infinitesimal variation 
of $(u_{t}(X_{t}))_{t \in [0,T]}$, the martingale term contained in $u_{t}$ conspires with 
the martingale term contained in $X$ and generates an additional bracket term.
This additional bracket term is precisely $\sqrt{2} {\rm div}(v_{t})(X_{t})$; it thus cancels out with the term $- \sqrt{2} {\rm div}(v_{t})(X_{t})$ that appears in the dynamics of $u_{t}$. For the sake of completeness, we provide a rough version of the computations that 
enter the definition of this additional bracket.  When 
expanding the difference
$u_{t+dt}(X_{t+dt}) - u_{t}(X_{t})$, for $t \in [0,T]$ and an infinitesimal 
variation $dt$, the martingale structure 
in \eqref{eq:se:3:u} induces a term of the form 
$v_{t}(X_{t+dt})(W_{t+dt}-W_{t})$. By standard It\^o's formula, 
it looks like 
\begin{equation}
\label{eq:partie2:ito:wentzell:rough}
\begin{split}
&v_{t}(X_{t+dt})\bigl(W_{t+dt}-W_{t}\bigr)
\\
&\hspace{15pt} = \sum_{i=1}^d v_{t}^i(X_{t+dt}) \bigl(W_{t+dt}^i - W_{t}^i 
\bigr)
= \sum_{i=1}^d v_{t}^i(X_{t}) dW_{t}^i 
+ \sqrt{2}
\sum_{i=1}^d \frac{\partial v_{t}^i}{\partial x_{i}}(X_{t}) dt,
\end{split}
\end{equation}
the last term matching precisely 
the divergence term (up to the sign) that appears in 
\eqref{eq:se:3:u}.
\vspace{5pt}

As in the deterministic case, our aim is to define 
$U$ by means of the same formula as in \eqref{defU}, that 
is $U(0,x,m_{0})$ is the value 
at point $x$ of the value function taken at time $0$
when the population is initialized with the distribution 
$m_{0}$. 

In order to proceed, the idea is to reduce the equations
by taking advantage of the additive structure of the common noise. 
The point is to make the (formal) change of variable
\begin{equation*}
\tilde{u}_{t}(x) = u_{t}(x+ \sqrt{2} W_{t}), 
\quad \tilde{m}_{t}(x) = m_{t}(x+ \sqrt{2} W_{t}), \quad x \in \T^d, \quad t \in [0,T]. 
\end{equation*}
The second definition makes sense when $m_{t}$ is a density, which is the case in the analysis
because of the smoothing effect of the noise. 
A more rigorous way to define $\tilde{m}_{t}$ is to let
it be the push-forward of $m_{t}$ by the shift $\T^d \ni 
x \mapsto x-\sqrt{2}W_{t} \in \T^d$. Pay attention that such a definition is completely licit 
as $m_{t}$ reads as a conditional measure given the common noise. 
As the conditioning consists in freezing the common noise, 
the shift $x \mapsto x-\sqrt{2}W_{t}$ may be seen as a `deterministic' mapping. 

The main feature is that $\tilde{m}_{t}$ is the conditional law
of the process $(X_{t}-\sqrt{2}W_{t})_{t \in [0,T]}$ given the common noise.
Since
\begin{equation*}
d \bigl( X_{t} -\sqrt{2} W_{t} \bigr)
= - D_{p} H \bigl( X_{t}- \sqrt{2} W_{t} + \sqrt{2} W_{t},Du_{t}
(X_{t}- \sqrt{2}W_{t}+\sqrt{2}W_{t}) \bigr) dt + \sqrt{2} dB_{t}, 
\quad t \in [0,T].
\end{equation*}
we get that 
$(\tilde{m}_{t})_{t \in [0,T]}$ should satisfy
\begin{equation}
\label{eq:tilde:m}
\begin{split}
d_{t} \tilde{m}_{t} &= \bigl\{  \Delta \tilde{m}_{t}
+{\rm div} \bigl( \tilde{m}_{t} D_{p} H(\cdot + \sqrt{2} W_{t} ,D \tilde{u}_{t})
\bigr) \bigr\} dt
\\
&= \bigl\{ \Delta \tilde{m}_{t}
+{\rm div} \bigl( \tilde{m}_{t} D_{p} \tilde{H}_{t}(\cdot,D \tilde{u}_{t})
\bigr) \bigr\} dt,
\end{split} 
\end{equation}
where we have denoted $\tilde{H}_{t}(x,p) = H(x+ \sqrt{2}W_{t},p)$. 
This reads as the standard Fokker-Planck equation but in a random medium. 
Such a computation may be recovered by applying the 
It\^o-Wentzell formula to $(m_{t}(x+\sqrt{2} W_{t}))_{t \in [0,T]}$, provided 
that each $m_{t}$ be smooth enough in space. 
Quite remarkably, $(\tilde{m})_{t \in [0,T]}$ is of absolutely continuous variation in time, 
which has a clear meaning when $(\tilde{m}_{t})_{t \in [0,T]}$
is seen as a process with values in a set of smooth functions; 
when $(\tilde{m}_{t})_{t \in [0,T]}$
is seen as a process with values in $\Pk$,
the process $(\langle \varphi , \tilde{m}_{t} \rangle)_{t \in [0,T]}$
($\langle \cdot,\cdot \rangle$ standing for the duality bracket)
is indeed of absolutely continuous variation.

Similarly, 
we can apply (at least formally) It\^o-Wentzell formula 
to $(u_{t}(x+\sqrt{2} W_{t}))_{t \in [0,T]}$ in order
to express the dynamics of $(\tilde{u}_{t})_{t \in [0,T]}$. 
\begin{equation}
\label{eq:tilde:u}
\begin{split}
d_{t} \tilde{u}_{t}
&=  \bigl\{ -   \Delta \tilde{u}_{t} + H\bigl(\cdot + \sqrt{2} W_{t},D\tilde{u}_{t}
\bigr) - F\bigl(\cdot+ \sqrt{2}W_{t},m_{t} \bigr) 
\bigr\} dt
+ \tilde{v}_{t} dW_{t},
\\
&= \bigl\{ -   \Delta \tilde{u}_{t} + \tilde{H}_{t}(\cdot,D\tilde{u}_{t}) - 
\tilde{F}_{t}(\cdot,m_{t}) 
\bigr\} dt
+ \tilde{v}_{t} dW_{t}, \quad t \in [0,T], 
\end{split}
\end{equation}
where $\tilde{F}_{t}(x,m) = F(x+\sqrt{2} W_{t},m)$, 
for a new representation term $\tilde{v}_{t}(x)=v_{t}(x+\sqrt{2} W_{t})$,
the boundary condition writing
$\tilde{u}_{T}(\cdot) = \tilde{G}(\cdot,m_{T})$
with $\tilde{G}(x,m) = G(x+\sqrt{2} W_{T},m)$.
In such a way, we completely avoid any discussion about the smoothness of 
$\tilde{v}$. Pay attention that there is no way to get rid of the stochastic integral
as it permits to ensure that $\tilde{u}_{t}$ remains 
adapted with respect to the observation up until time $t$. 

Below, we shall investigate the system \eqref{eq:tilde:m}--\eqref{eq:tilde:u}
directly. It is only in the next section, see 
Subsection \ref{subse:proof:c.sec5.MFGstoch}, that we make the connection 
with the original 
formulation 
\eqref{eq:se:3:m}--\eqref{eq:se:3:u}
and then complete the proof of 
Corollary \ref{c.sec5.MFGstoch}.  
The reason is that it suffices to define 
the solution of the master equation by letting 
$U(0,x,m_{0})$ be the value of $\tilde{u}_{0}(x)$
with $m_{0}$ as initial distribution. 
Notice indeed that $\tilde{u}_{0}(x)$ is expected to 
match $\tilde{u}_{0}(x) = u_{0}(x-\sqrt{2} W_{0})
=u_{0}(x)$. 
Of course, the same strategy may be applied at any time 
$t \in [0,T]$ by investigating $(\tilde{u}_{s}(x+\sqrt{2}(W_{s}-W_{t})))_{s \in [t,T]}$. 

With these notations, the monotonicity assumption takes the form:
\begin{Lemma}
\label{lem:partie2:monotonie}
Let $m$ and $m'$ be two elements of $\Pk$.
For some $t \in [0,T]$
and for 
some realization of the noise, denote by 
$\tilde{m}$ and $\tilde{m}'$ the push-forwards of $m$ and $m'$ 
by the mapping $\T^d \ni x \mapsto x - \sqrt{2} W_{t} \in \T^d$. Then, 
for the given realization of $(W_{s})_{s \in [0,T]}$,
\begin{equation*}
\int_{\T^d} 
\bigl( \tilde{F}_{t}(x,m) -\tilde{F}_{t}(x,m') \bigr) d (\tilde{m} - \tilde{m}') \geq 0,
\quad \int_{\T^d} 
\bigl( \tilde{G}(x,m) -\tilde{G}(x,m') \bigr) d (\tilde{m} - \tilde{m}') \geq 0.
\end{equation*}
\end{Lemma}

\begin{proof}
The proof consists of a straightforward change of variable. 
\end{proof}

\begin{Remark}
\label{rem:partie2:notation}
Below, we shall use quite systematically, without recalling it, the notation tilde
$\sim$ in order to denote the new coefficients 
and the new solutions after the random change of variable 
$x \mapsto x + \sqrt{2} W_{t}$. 
\end{Remark}

\subsection{Probabilistic Set-Up}

Throughout the section, we shall use the probabilistic 
space $(\Omega,{\mathcal A},\P)$ equipped with 
two independent $d$-dimensional Brownian motions 
$(B_{t})_{t \geq 0}$ and $(W_{t})_{t \geq 0}$. 
The probability space is assumed to be complete. 
We then denote by $({\mathcal F}_{t})_{t \geq 0}$
the completion of the 
filtration generated by $(W_{t})_{t \geq 0}$. 
When needed, we shall also use the 
filtration generated by $(B_{t})_{t \geq 0}$.

Given an initial distribution ${m}_{0} \in \Pk$, we consider
the system
\begin{equation}
\label{eq:se:3:tilde:HJB:FP}
\begin{split}
&d_{t} \tilde{m}_{t} = \bigl\{ \Delta \tilde{m}_{t}
+{\rm div} \bigl( \tilde{m}_{t} D_{p} \tilde{H}_{t}(\cdot,D \tilde{u}_{t})
\bigr) \bigr\} dt,  
\\
&d_{t} \tilde{u}_{t}
=  \bigl\{ -   \Delta \tilde{u}_{t} + \tilde{H}_{t}(\cdot,D\tilde{u}_{t}) - 
\tilde{F}_{t}(\cdot,m_{t}) 
\bigr\} dt
+ d\tilde{M}_{t},
\end{split}
\end{equation}
with 
the initial condition $\tilde{m}_{0} = m_{0}$
and
the terminal boundary condition $\tilde{u}_{T}
= \tilde{G}(\cdot,m_{T})$, 
with $\tilde{G}(x,m_{T}) = G(x+ \sqrt{2} W_{T},m_{T})$. 

The solution  
$(\tilde{u}_{t})_{t \in [0,T]}$ is
seen as an $({\mathcal F}_{t})_{t \in [0,T]}$-adapted process
with paths in the space ${\mathcal C}^0([0,T],{\mathcal C}^{n}({\mathbb T}^d))$, 
where $n$ is a large enough integer (see the precise statements below). 
The process $(\tilde{m}_{t})_{t \in [0,T]}$ reads 
as an $({\mathcal F}_{t})_{t \in [0,T]}$-adapted process
with paths in the space 
$\cC^0([0,T],{\mathcal P}({\mathbb T}^d))$, 
${\mathcal P}({\mathbb T}^d)$ being equipped with 
the
$1$-Wasserstein metric ${\mathbf d}_{1}$. 
We shall look for solutions satisfying 
\begin{equation}
\label{eq:se:3:bsde:1}
\sup_{t \in [0,T]} \bigl( \| \tilde{u}_{t} \|_{n+\alpha}
\bigr) \in L^\infty(\Omega,{\mathcal A},\P),
\end{equation}
for some $\alpha \in (0,1)$. 

The process $(\tilde{M}_{t})_{t \in [0,T]}$ is
seen as an $({\mathcal F}_{t})_{t \in [0,T]}$-adapted process
with paths in the space ${\mathcal C}^0([0,T],
{\mathcal C}^{n-2}({\mathbb T}^d))$, 
such that, for any $x \in {\mathbb T}^d$, 
$(\tilde{M}_{t}(x))_{t \in [0,T]}$
is an $({\mathcal F}_{t})_{t \in [0,T]}$ martingale. 
It is required to satisfy 
\begin{equation}
\label{eq:se:3:bsde:2}
\sup_{t \in [0,T]} \bigl( \| \tilde{M}_{t} \|_{n-2+\alpha}
\bigr) \in L^\infty(\Omega,{\mathcal A},\P).
\end{equation}
Notice that, for our purpose, there is no need to discuss of the representation 
of the martingale as a stochastic integral. 

\subsection{Solvability of the Stochastic FP/HJB System}
The objective is to discuss the existence and uniqueness of a classical solution 
to such the system 
\eqref{eq:se:3:tilde:HJB:FP}
under the same assumptions as in the deterministic case.
Theorem 
\ref{thm:partie:2:existence:uniqueness}
below covers
Theorem \ref{thm:partie:2:existence:uniquenessINTRO}
in Section \ref{sec:prelim}:
\begin{Theorem}
\label{thm:partie:2:existence:uniqueness}
Assume that
$F$, $G$ and $H$ satisfy 
\eqref{HypD2H}
and 
\eqref{e.monotoneF}
in Subsection \ref{subsec:hyp}. 
Assume moreover that, for some integer $n \geq 2$ and some\footnote{In most of the analysis, 
$\alpha$ is assumed to be (strictly) positive, except in this statement where it may be zero. 
Including the case $\alpha=0$ allows for a larger range of application of the uniqueness property.}
 $\alpha \in [0,1)$, 
{\bf (HF1(${\boldsymbol n}$-1))}
and 
{\bf (HG1(${\boldsymbol n}$))}
hold true.

Then,
there exists a unique solution $(\tilde{m}_{t},\tilde{u}_{t},\tilde{M}_{t})_{t \in [0,T]}$
to \eqref{eq:se:3:tilde:HJB:FP}, with the prescribed initial condition
$\tilde{m}_{0}=m_{0}$,
satisfying 
\eqref{eq:se:3:bsde:1}
and \eqref{eq:se:3:bsde:2}. It 
satisfies 
$\sup_{t \in [0,T]} ( 
\| \tilde{u}_{t} \|_{n+\alpha} +
\| \tilde{M}_{t} \|_{n+\alpha-2}) 
\in L^{\infty}(\Omega,{\mathcal A},\P)$. 

Moreover,
we can find a constant $C$ such that, 
 for any two initial conditions $m_{0}$ and $m_{0}'$ in ${\mathcal P}(\T^d)$, we have 
\begin{equation*}
\begin{split}
& \sup_{t \in [0,T]}
\bigl( 
{\mathbf d}_{1}^2(\tilde m_{t},\tilde m_{t}')
+ \| \tilde{u}_{t} - \tilde{u}_{t}' \|_{n+\alpha}^2
\bigr)
 \leq C
{\mathbf d}_{1}^2(m_{0},m_{0}') \qquad \P-{\rm a.e.},
\end{split}
\end{equation*}
where $(\tilde{m},\tilde{u},\tilde{M})$ and $(\tilde{m}',\tilde{u}',\tilde{M}')$ denote the 
solutions to 
\eqref{eq:se:3:tilde:HJB:FP}
with $m_{0}$ and $m_{0}'$ as initial conditions. 
\end{Theorem}
Theorem \ref{thm:partie:2:existence:uniqueness} is the 
analogue of Propositions 
\ref{prop:reguum} and 
\ref{prop:Ulip}
 in the deterministic setting, except that we do not discuss the time
 regularity of the solutions (which, as well known in the theory of 
 finite dimensional BSDEs, may be a rather difficult question).
\vspace{5pt}

The strategy of proof relies on the so-called continuation method.  
We emphasize that, differently from the standard argument that is used in the deterministic case, 
we will not make use of Schauder's theorem to establish the existence of a solution. The reason is that, in order to apply Schauder's theorem, we would need a compactness criterion on the space 
on which the equilibrium is defined, namely $L^\infty(\Omega,{\mathcal A},\P;{\mathcal C}^0([0,T],{\mathcal P}(\T^d)))$. 
As already noticed in the earlier paper \cite{CaDeLa},
this would ask for a careful (and certainly complicated) discussion on the choice of $\Omega$ and then 
 on the behavior of the solution to \eqref{eq:se:3:tilde:HJB:FP}
 with respect to the topology put on $\Omega$.

Here the idea is as follows. 
Given two parameters $(\vartheta,\varpi) \in [0,1]^2$, we shall first have a look at the 
parameterized system:
\begin{equation}
\label{eq:se:FP/HJB:vartheta}
\begin{split}
&d_{t} \tilde{m}_{t} = \bigl\{  \Delta \tilde{m}_{t}
+
{\rm div} \bigl[ \tilde{m}_{t} \bigl( 
\vartheta D_{p} \tilde H_{t}( \cdot,D \tilde{u}_{t}) + b_{t} \bigr)
\bigr] \bigr\} dt,
\\
&d_{t} \tilde{u}_{t}
=  \bigl\{ - \Delta \tilde{u}_{t} + 
\vartheta \tilde H_{t}(\cdot ,D\tilde{u}_{t}) - \varpi \tilde F_{t}(\cdot,m_{t}) 
+ f_{t}  \bigr\} dt
+ d\tilde{M}_{t},
\end{split}
\end{equation}
with the initial condition $\tilde{m}_{0}=m_{0}$ and the terminal boundary condition $\tilde{u}_{T}
= \varpi \tilde{G}(\cdot,m_{T}) + g_{T}$, 
where $((b_{t},f_{t})_{t \in [0,T]},g_{T})$ is some input. 

In the above equation, there are two extreme regimes: when 
$\vartheta=\varpi=0$ and the input is arbitrary, the equation is known to be explicitly 
solvable; when $\vartheta=\varpi=1$ and the input is set equal to $0$, the above equation 
fits the original one. This is our precise purpose to
prove first, by a standard contraction argument, that the equation is solvable when $\vartheta=1$ and $\varpi=0$
and then to
 propagate existence and uniqueness from the case $(\vartheta,\varpi)=(1,0)$ to the case $(\vartheta,\varpi)=(1,1)$ by means of a continuation argument. 
 \vspace{5pt}

Throughout the analysis, 
the assumption of Theorem 
\ref{thm:partie:2:existence:uniqueness}
is in force. 
Generally speaking,
the inputs $(b_{t})_{t \in [0,T]}$ and $(f_{t})_{t \in [0,T]}$
are $({\mathcal F}_{t})_{t \in [0,T]}$ adapted processes 
with paths in the space ${\mathcal C}^0([0,T],
[{\mathcal C}^{1}({\mathbb T}^d)]^d)$
and ${\mathcal C}^0([0,T],{\mathcal C}^{n-1}({\mathbb T}^d))$ respectively. Similarly, 
$g_{T}$ is an ${\mathcal F}_{T}$-measurable 
random variable with realizations in ${\mathcal C}^{n+\alpha}({\mathbb T}^d)$.
We shall require that 
\begin{equation*}
\sup_{t \in [0,T]} \|b_{t} \|_{1}, \quad 
\sup_{t \in [0,T]} \|f_{t} \|_{n-1+\alpha}, \quad \|g_{T}\|_{n+\alpha}
\end{equation*}
are bounded (in $L^{\infty}(\Omega,{\mathcal A},\P)$). 

It is worth mentioning that, whenever 
$\varphi : [0,T] \times \T^d \rightarrow \R$ is a continuous mapping such that 
$\varphi(t,\cdot) \in {\mathcal C}^{\alpha}(\T^d)$ for any $t \in [0,T]$, 
the mapping $[0,T] \ni t \mapsto \| \varphi(t,\cdot) \|_{\alpha}$
is lower semicontinuous and, thus, the mapping 
$[0,T] \ni t \mapsto \sup_{s \in [0,t]} 
\| \varphi(t,\cdot) \|_{\alpha}$
is continuous. 
In particular, 
whenever
$(f_{t})_{t \in [0,T]}$
is a process 
with paths in 
${\mathcal C}^0([0,T],{\mathcal C}^{k}({\mathbb T}^d))$, for some 
$k \geq 0$, the quantity $\sup_{t \in [0,T]} \|f_{t}\|_{k+\alpha}$
is a random variable, equal to 
$\sup_{t \in [0,T] \cap {\mathbb Q}} \|f_{t}\|_{k+\alpha}$,
 and the process 
$(\sup_{s \in [0,t]} \|f_{s}\|_{k+\alpha})_{t \in [0,T]}$
has continuous paths. As a byproduct, 
$$\essup_{\omega \in \Omega}
\sup_{t \in [0,T]} \|f_{t}\|_{k+\alpha}
= \sup_{t \in [0,T]} \essup_{\omega \in \Omega}
 \|f_{t}\|_{k+\alpha}.$$ 

\subsubsection{Case $\vartheta = \varpi = 0$}
We start with the following simple lemma:

\begin{Lemma}
\label{lem:cas:vartheta=0}
Assume that $\vartheta=\varpi=0$. 
Then, with the same type of inputs as above, 
\eqref{eq:se:FP/HJB:vartheta}
has a unique solution 
$(\tilde{m}_{t},\tilde{u}_{t},\tilde{M}_{t})_{t \in [0,T]}$, with the prescribed initial condition. 
It satisfies 
\eqref{eq:se:3:bsde:1}
and \eqref{eq:se:3:bsde:2}. 
Moreover, there exists a constant $C$, only depending on $n$ and $T$, 
such that
\begin{equation}
\label{eq:9:bis}
\essup_{\omega \in \Omega} \sup_{t \in [0,T]} \|\tilde u_{t}\|_{n+\alpha}
\leq C \bigl( \essup_{\omega \in \Omega} \|g_{T} \|_{n+\alpha}
+ \essup_{\omega \in \Omega} \sup_{t \in [0,T]} \|f_{t} \|_{n-1+\alpha}
\bigr),
\end{equation}
\end{Lemma}


\begin{proof}[Proof of Lemma \ref{lem:cas:vartheta=0}] 
When $\vartheta=\varpi=0$,
the forward equation simply reads
\begin{equation*}
\begin{split}
&d_{t} \tilde{m}_{t} = \bigl\{ \Delta \tilde{m}_{t}
+
{\rm div} \bigl[ \tilde{m}_{t} b_{t}\bigr] \bigr\} dt, \quad t \in [0,T]
\end{split}
\end{equation*}
with initial condition $m_0$. This is a standard Kolmogorov equation (with random coefficient) which is pathwise solvable. By standard estimates, we have  
$$
\essup_{\omega \in \Omega} \sup_{s\neq t} \frac{\dk(\tilde m_s, \tilde m_s)}{|s-t|^{\frac12}} \leq \essup_{\omega \in \Omega}   \|b\|_\infty .
$$
As $\vartheta=\varpi=0$, the backward equation in 
\eqref{eq:se:FP/HJB:vartheta}
has the form:
\begin{equation*}
d_{t} \tilde{u}_{t}
=  \bigl\{ -   \Delta \tilde{u}_{t}  
+ f_{t} \bigr\} dt
+ d\tilde{M}_{t}, \quad t \in [0,T],
\end{equation*}
with the terminal boundary condition $\tilde{u}_{T} = g_{T}$. 
Although the equation is infinite-dimensional, it may be solved in 
a quite straightforward way. Taking the conditional expectation given 
$s \in [0,T]$ in the above equation, 
we indeed get that any solution  should satisfy (provided we can exchange
differentiation and conditional expectation):
\begin{equation*}
d_{t} {\mathbb E} \bigl[ \tilde{u}_{t} \vert {\mathcal F}_{s}
\bigr]
=  \bigl\{ -  \Delta {\mathbb E} \bigl[ 
\tilde{u}_{t}  \vert {\mathcal F}_{s} \bigr]
+ {\mathbb E} \bigl[ f_{t} \vert {\mathcal F}_{s} \bigr] \bigr\} dt,
\quad t \in [s,T],
\end{equation*}
which suggests to let 
\begin{equation}
\label{eq:partie2:construction:backward}
\tilde{u}_{s}(x) = \E \bigl[ \bar{u}_{s}(x) \vert {\mathcal F}_{s}
\bigr], \quad 
\bar{u}_{s}(x) = P_{T-s} g_{T}(x) - \int_{s}^T P_{t-s} f_{t}(x) dt,
\quad s \in [0,T], \ x \in \T^d,
\end{equation}
where $P$ denotes the heat semigroup (but associated with the Laplace
operator $\Delta$ instead of $(1/2)\Delta$).
For any $s \in [0,T]$ and $x \in {\mathbb T}^d$,
the conditional expectation is uniquely defined 
up to a negligible event under $\P$. We claim 
that, for any $s \in [0,T]$, we can find a version of the conditional expectation
in such a way that the process 
$[0,T] \ni s \mapsto (\T^d \ni x \mapsto \tilde{u}_{s}(x))$
reads as a progressively-measurable random variable with values in $\cC^0([0,T],\cC^0(\T^d))$.
By the representation formula 
\eqref{eq:partie2:construction:backward}, 
we indeed have that, $\P$ almost surely, $\bar u$ is jointly continuous in time and space. 
Making use of Lemma \ref{lem:part:2:conditionnement} below, 
we deduce that the realizations of 
$[0,T] \ni s \mapsto (\T^d \ni x \mapsto \tilde{u}_{s}(x))$
belong to $\cC^0([0,T],\cC^0(\T^d))$, 
the mapping 
$[0,T] \times \Omega \ni (s,\omega) 
\mapsto (\T^d \ni x \mapsto (\tilde{u}_{s}(\omega))(x))$
being measurable with respect to the progressive $\sigma$-field
\begin{equation}
\label{eq:partie2:progressive}
{\mathscr P} = \bigl\{ A \in {\mathcal B}([0,T]) \otimes {\mathcal A} : 
\quad \forall t \in [0,T],
\ 
A \cap ([0,t] \times \Omega) \in {\mathcal B}([0,t]) \otimes 
{\mathcal F}_{t}\bigr\}.
\end{equation} 
By the maximum principle, we can find a constant 
$C$, depending on $T$ and $d$ only, such that 
\begin{equation*}
\essup_{\omega \in \Omega}
\sup_{s \in [0,T]} \|\tilde u_{s}\|_{0} \leq
\essup_{\omega \in \Omega}
\sup_{s \in [0,T]} \|\bar u_{s}\|_{0}
 \leq 
C \bigl( \essup_{\omega \in \Omega} \| g_{T} \|_{0} + 
\essup_{\omega \in \Omega} \sup_{0 \leq s \leq T} \| f_{s} \|_{0}
\bigr).
\end{equation*}
More generally, 
taking
the representation formula
\eqref{eq:partie2:construction:backward}
at two different $x,x' \in \T^d$
and then
making the difference, we get 
\begin{equation*}
\essup_{\omega \in \Omega}
\sup_{s \in [0,T]} \|\tilde u_{s}\|_{\alpha} 
 \leq 
C \bigl( \essup_{\omega \in \Omega} \| g_{T} \|_{\alpha} + 
\essup_{\omega \in \Omega} \sup_{s \in [0,T]} \| f_{s} \|_{\alpha}
\bigr).
\end{equation*}
We now proceed with the derivatives of higher order. 
Generally speaking, there are two ways to differentiate 
the representation formula \eqref{eq:partie2:construction:backward}. 
The first one is to say that, for any $k \in \{1,\dots,n-1\}$, 
\begin{equation}
\label{eq:part:2:representation:derivatives}
D_{x}^k \bar{u}_{s}(x) = P_{T-s} \bigl( D^k  
g_{T} \bigr) (x) - \int_{s}^T P_{t-s} \bigl( D_{x}^k f_{t} \bigr)(x) dt,
\quad (s,x) \in [0,T] \times \T^d, 
\end{equation}
which may be established by a standard induction argument. 
The second way is to make use of the regularization property of the 
heat kernel in order to go one step further, namely, 
for any $k \in \{1,\dots,n\}$,
\begin{equation}
\label{eq:part:2:representation:derivatives:2}
\begin{split}
D_{x}^k \bar{u}_{s}(x) &= P_{T-s} \bigl( D^k  
g_{T} \bigr) (x) - \int_{s}^T D P_{t-s} \bigl( D_{x}^{k-1} f_{t} \bigr)(x) dt,
\\
&=  P_{T-s} \bigl( D^k  
g_{T} \bigr) (x) - \int_{0}^{T-s} D P_{t} \bigl( D_{x}^{k-1} f_{t+s} \bigr)(x) dt, 
\quad (s,x) \in [0,T] \times \T^d, 
\end{split}
\end{equation}
where $DP_{t-s}$ stands for the derivative of the heat semigroup. 
Equation 
\eqref{eq:part:2:representation:derivatives:2}
is easily derived from 
\eqref{eq:part:2:representation:derivatives}. 
It permits to handle the fact that $f$ is $(n-1)$-times differentiable only. 

Recalling that $| DP_{t} \varphi | \leq c t^{-1/2} \| \varphi \|_\infty $
for any bounded Borel function $\varphi : \T^d \rightarrow \R$ and
 for some $c \geq 1$ independent of $\varphi$ and of 
 $t \in [0,T]$, 
we deduce that, for any $k \in \{1,\dots,n\}$, 
the mapping $[0,T] \times \T^d \ni (s,x) \mapsto D_{x}^k \bar{u}_{s}(x)$
is continuous. Moreover, we can find a constant $C$ such that, 
for any $s \in [0,T]$,
\begin{equation}
\label{eq:part:2:representation:derivatives:3}
\essup_{\omega \in \Omega}
\| \bar{u}_{s} \|_{k+\alpha}
\leq 
\essup_{\omega \in \Omega}
\| g_{T} \|_{k+\alpha} 
+ C \int_{s}^T \frac{1}{\sqrt{t-s}} \essup_{\omega \in \Omega}
\| f_{t} \|_{k+\alpha-1} dt. 
\end{equation} 
In particular, invoking once again
Lemma \ref{lem:part:2:conditionnement}
below, we 
can find a version of the conditional expectation 
in the representation formula $\tilde{u}_{s}(x)
=\E [ \bar{u}_{s}(x) \vert {\mathcal F}_{s}
]$ such that 
$\tilde{u}$ has paths in 
${\mathcal C}^0([0,T],{\mathcal C}^n(\T^d))$. 
For any 
$k \in \{1,\dots,n\}$, 
$D_{x}^k \tilde{u}$ is 
progressively-measurable and, for all $(s,x) \in [0,T] \times \T^d$,
it holds that 
$D_{x}^k \tilde{u}_{s}(x) = \E[ D_{x}^k \bar{u}_{s}(x) \vert {\mathcal F}_{s}]$.

Using \eqref{eq:part:2:representation:derivatives:3}, we have, 
for any $k \in \{1,\dots,n\}$,
\begin{equation*}
\essup_{\omega \in \Omega}
\sup_{s \in [0,T]} \|\tilde u_{s}\|_{k+\alpha} \leq 
C \bigl( \essup_{\omega \in \Omega} \| g_{T} \|_{k+\alpha} + 
\essup_{\omega \in \Omega} \sup_{s \in [0,T]} \| f_{s} \|_{k+\alpha-1}
\bigr).
\end{equation*}

Now that $\tilde{u}$ has been constructed, it remains to reconstruct the martingale 
part $(\tilde{M}_{t})_{0 \leq t \leq T}$ in 
the backward equation of the system 
\eqref{eq:se:FP/HJB:vartheta} (with $\vartheta = \varpi = 0$ therein). 
Since $\tilde{u}$ has trajectories in ${\mathcal C}^0([0,T],\cC^{n+\alpha}(\T^d))$, 
$n \geq 2$,
we can let:
\begin{equation*}
\tilde M_{t}(x) = 
\tilde{u}_{t}(x) -  \tilde{u}_{0}(x)  + \int_{0}^t  \Delta \tilde{u}_{s}(x) ds - \int_{0}^t f_{s}(x) ds, \quad t \in [0,T], \ x \in \T^d. 
\end{equation*}
It is then clear that $\tilde{M}$ has trajectories in ${\mathcal C}^0([0,T],\cC^{n-2}(\T^d))$
and that 
\begin{equation*}
\essup_{\omega \in \Omega} 
\sup_{t \in [0,T]}
\big\| \tilde{M}_{t} \big\|_{n+\alpha-2} < \infty. 
\end{equation*}
It thus remains to prove that, for each $x \in \T^d$, 
the process $(\tilde M_{t}(x))_{0 \leq t \leq T}$
is a martingale (starting from $0$). Clearly,
it has continuous and $({\mathcal F}_{t})_{0 \leq t \leq T}$-adapted 
paths. Moreover,
\begin{equation*}
\begin{split}
\tilde{M}_{T}(x) - \tilde{M}_{t}(x)
= g_{T}(x) - \tilde{u}_{t}(x) + \int_{t}^T \Delta \tilde{u}_{s}(x) ds 
- \int_{t}^T f_{s}(x) ds, \quad t \in [0,T], \ x \in \T^d. 
\end{split}
\end{equation*}
Now, recalling the relationship $ {\mathbb E}[ \Delta  \tilde{u}_{s}(x) \vert {\mathcal F}_{t}]
= {\mathbb E}[ \Delta  \bar{u}_{s}(x) \vert {\mathcal F}_{t}]$, 
we get 
$$
{\mathbb E}[ 
\int_{t}^T
\Delta \tilde{u}_{s}(x) 
ds
\vert {\mathcal F}_{t}
]
 =
{\mathbb E}[ 
\int_{t}^T
\Delta \bar{u}_{s}(x) 
ds
\vert {\mathcal F}_{t}
].
$$
Taking the conditional expectation given ${\mathcal F}_{t}$, we deduce that
\begin{equation*}
\begin{split}
\E \bigl[ \tilde{M}_{T}(x) - \tilde{M}_{t}(x) \vert {\mathcal F}_{t}
\bigr] 
&= 
\E \biggl[ g_{T}(x) - \bar{u}_{t}(x) 
- \int_{t}^T f_{s}(x) ds 
+ \int_{t}^T \Delta  \bar{u}_{s}(x) ds \big\vert {\mathcal F}_{t} \biggr] = 0,
\end{split}
\end{equation*}
the second equality following from 
\eqref{eq:partie2:construction:backward}. This shows 
that $\tilde{M}_{t}(x) = \E [ \tilde{M}_{T}(x) \vert {\mathcal F}_{t}]$,
so that the process $(\tilde{M}_{t}(x))_{0 \leq t \leq T}$ is a martingale, as required. 
\end{proof}

\begin{Remark}
Notice that, alternatively
to \eqref{eq:9:bis}, we also have, by Doob's inequality,
\begin{equation}
\label{eq:9}
{\mathbb E} \bigl[ \sup_{t \in [0,T]}
\| \tilde u_{t} \|_{n+\alpha}^2 \bigr] \leq C 
{\mathbb E} \bigl[ \|g_{T} \|_{n+\alpha}^2 + 
\sup_{t \in [0,T]} \| f_{t} \|_{n+\alpha-1}^2 \bigr]. 
\end{equation}
\end{Remark}

\begin{Lemma}
\label{lem:part:2:conditionnement}
Consider a random field ${\mathcal U} : [0,T] \times \T^d \rightarrow \R$, 
with continuous paths (in the variable $(t,x) \in [0,T] \times \T^d$), such that 
\begin{equation*}
\essup_{\omega \in \Omega} \| {\mathcal U} \|_{0} < \infty.
\end{equation*}
Then, we can find a version of the random field $[0,T] \times \T^d 
\ni (t,x) \mapsto {\mathbb E}[{\mathcal U}(t,x) \vert {\mathcal F}_{t}]$
such that $[0,T] \ni t \mapsto (\T^d \ni x \mapsto 
\E [ {\mathcal U}(t,x) \vert {\mathcal F}_{t}])$
is a progressively-measurable random variable with values in ${\mathcal C}^0([0,T],{\mathcal C}^0(\T^d))$, the progressive $\sigma$-field ${\mathscr P}$ being defined in 
\eqref{eq:partie2:progressive}.  

More generally, if, for some $k \geq 1$, the paths of ${\mathcal U}$
are $k$-times differentiable in the space variable,  
the derivatives up to the order $k$ having jointly continuous (in $(t,x)$) paths 
and satisfying
\begin{equation*}
\essup_{\omega \in \Omega}
\sup_{t \in [0,T]}
\|{\mathcal U}(t,\cdot) \|_{k} < \infty,
\end{equation*}
then we can find a version of the random field $[0,T] \times 
\T^d \ni (t,x) \mapsto \E[{\mathcal U}(t,x) \vert {\mathcal F}_{t}]$
that is progressively-measurable and that has paths 
in ${\mathcal C}^0([0,T],{\mathcal C}^k(\T^d))$, 
the derivative of order $i$ writing 
$[0,T] \times 
\T^d \ni (t,x) \mapsto \E[D^i_{x}
{\mathcal U}(t,x) \vert {\mathcal F}_{t}]$.
\end{Lemma}

\begin{proof}
\textit{First step.}
We first prove the first part of the statement
(existence of a progressively-measurable 
version with continuous paths). Existence of a differentiable version 
will be handled next. A key fact in the proof is that, the filtration $({\mathcal F}_{t})_{t \in [0,T]}$ being generated by $(W_{t})_{t \in [0,T]}$, any martingale with respect to 
$({\mathcal F}_{t})_{t \in [0,T]}$ admits a continuous version. 

Throughout the proof, we denote by $w$ the (pathwise) modulus of continuity of 
${\mathcal U}$ on the compact set $[0,T] \times \T^d$, namely:
\begin{equation*}
w(\delta) = \sup_{x,y \in \T^d : 
\vert x-y \vert \leq \delta} \sup_{s,t \in [0,T] : \vert t-s \vert \leq \delta}
\vert {\mathcal U}(s,x) - {\mathcal U}(t,y) \vert, \quad \delta >0.
\end{equation*}
Since $\essup_{\omega \in \Omega} \| {\mathcal U} \|_{0} < \infty$, we
have, for any $\delta >0$, 
\begin{equation*}
\essup_{\omega \in \Omega} w(\delta) < \infty.
\end{equation*}
By Doob's inequality, we have that, for any integer $p \geq 1$,
\begin{equation*}
\forall \varepsilon >0, 
\quad {\mathbb P}
\Bigl( \sup_{s \in [0,T]} {\mathbb E}\Bigl[w \Bigl( \frac{1}{p} \Bigr) \vert {\mathcal F}_{s}
\Bigr] 
\geq \varepsilon \Bigr) \leq \varepsilon^{-1} {\mathbb E}\Bigl[w \Bigl( \frac{1}{p} \Bigr) \Bigr],
\end{equation*}
the right-hand side converging to $0$ as $p$ tends to $\infty$, thanks to Lebesgue's dominated convergence theorem. 
Therefore, by a standard application of Borel-Cantelli Lemma, we can find an increasing sequence of integers $(a_{p})_{p \geq 1}$ such that
the sequence 
$(\sup_{s \in [0,T]} {\mathbb E}[w(1/a_{p}) \vert {\mathcal F}_{s}])_{p \geq 1}$ converges 
to $0$ with probability 1. 

We now come back to the original problem. For any $(t,x) \in [0,T] \times \T^d$, we let 
\begin{equation*}
{\mathcal V}(t,x) = {\mathbb E}[{\mathcal U}(t,x) \vert {\mathcal F}_{t}]. 
\end{equation*}
The difficulty comes from the fact that each ${\mathcal V}(t,x)$ is uniquely defined 
up to a negligible set. The objective is thus to choose each of these 
negligible sets in a relevant way. 

Denoting by
${\mathcal T}$ a dense countable subset of $[0,T]$ and by
 ${\mathcal X}$ a dense countable subset 
of $\T^d$, we can find a negligible event $N \subset {\mathcal A}$
such that, outside $N$, the process 
$[0,T] \ni s \mapsto {\mathbb E}[{\mathcal U}(t,x) \vert {\mathcal F}_{s}]$
has a continuous version for any $t \in {\mathcal T}$ and $x \in 
{\mathcal X}$. 
Modifying the set $N$ if necessary, 
we have, 
outside $N$,
for any integer $p \geq 1$, any $t,t' \in {\mathcal T}$ and 
$x,x' \in {\mathcal X}$, with $\vert t-t' \vert + \vert x-x' \vert \leq 1/a_p$,
\begin{equation*}
\sup_{s \in [0,T]}
\bigl\vert 
{\mathbb E}[{\mathcal U}(t,x) \vert {\mathcal F}_{s}]
- {\mathbb E}[{\mathcal U}(t',x') \vert {\mathcal F}_{s}]
\bigr\vert \leq \sup_{s \in [0,T]} {\mathbb E}
\bigl[ w \bigl( \frac{1}{a_p} \bigr) \vert {\mathcal F}_{s} \bigr],
\end{equation*}
the right-hand side converging to $0$
as $p$ tends to $\infty$. Therefore, by a uniform continuity extension argument, it is thus possible to extend continuously,
outside $N$, 
the mapping ${\mathcal T} \times {\mathcal X} \ni (t,x) \mapsto 
([0,T] \ni s \mapsto
{\mathbb E}[{\mathcal U}(t,x) \vert {\mathcal F}_{s}]) \in {\mathcal C}^0([0,T],\R)$ to the entire 
$[0,T] \times \T^d$. For any $(t,x) \in [0,T] \times \T^d$, the 
value of the extension is a version of the conditional expectation 
$ {\mathbb E}[{\mathcal U}(t,x) \vert {\mathcal F}_{s}]$.
Outside $N$, 
the slice $(s,x) \mapsto 
{\mathbb E}[{\mathcal U}(s,x) \vert {\mathcal F}_{s}]$ is obviously continuous. Moreover, 
it satisfies, for all $p \geq 1$,
\begin{equation*}
\forall x,x' \in \T^d, 
\quad \vert x-x' \vert \leq \frac{1}{a_{p}}
\Rightarrow
\sup_{s \in [0,T]}
\bigl\vert 
{\mathbb E}[{\mathcal U}(s,x) \vert {\mathcal F}_{s}]
- {\mathbb E}[{\mathcal U}(s,x') \vert {\mathcal F}_{s}]
\bigr\vert 
\leq \sup_{s \in [0,T]} {\mathbb E}
\bigl[ w \bigl( \frac{1}{a_p} \bigr) \vert {\mathcal F}_{s} \bigr],
\end{equation*}
which says that, for each realization outside $N$, the functions $(\T^d \ni x \mapsto 
{\mathbb E}[{\mathcal U}(s,x) \vert {\mathcal F}_{s}])_{s \in [0,T]}$ 
are equicontinuous. Together with the continuity in $s$, we deduce that, outside $N$, the function $[0,T] \ni s \mapsto 
(\T^d \ni x \mapsto {\mathbb E}[{\mathcal U}(s,x) \vert {\mathcal F}_{s}]) \in \cC^0(\T^d)$ is continuous. 
On $N$, we can arbitrarily let ${\mathcal V} \equiv 0$,
which is licit since $N$ has zero probability. 
Progressive-measurability is then easily checked
(the fact that ${\mathcal V}$ 
is arbitrarily defined on $N$ does not matter since
the filtration is complete).
\vspace{2pt}

\textit{Second step.}
We now handle the second part of the statement
(existence 
of a ${\mathcal C}^k$ version). 
By a straightforward induction argument, it suffices to
treat the case $k=1$. 
By the first step, we already know that the random field
$[0,T] \times \T^d \ni (t,x) \mapsto \E[D_{x} u(t,x) \vert {\mathcal F}_{t}]$
has a continuous version. In particular, 
for any unit vector $e \in \R^d$, it makes sense to consider the mapping 
\begin{equation*}
\T^d \times \R^* \ni (x,h) \mapsto 
\frac1h 
\Bigl(
\E 
\bigl[
{\mathcal U}(t,x+h e) \vert {\mathcal F}_{t}
\bigr]  
- 
\E 
\bigl[
{\mathcal U}(t,x) \vert {\mathcal F}_{t}
\bigr]  
\Bigr)
- \E \bigl[ \langle D_{x} {\mathcal U}(t,x), e \rangle \vert {\mathcal F}_{t}
\bigr].
\end{equation*}
Notice that we can find an event of probability 1, 
on which
\begin{equation}
\label{eq:differentiable:version:cond:expectation}
\begin{split}
&\Bigl\vert \frac1h 
\Bigl(
\E 
\bigl[
{\mathcal U}(t,x+h e) \vert {\mathcal F}_{t}
\bigr]  
- 
\E 
\bigl[
{\mathcal U}(t,x) \vert {\mathcal F}_{t}
\bigr]  
\Bigr)
- \E \bigl[ \langle D_{x} {\mathcal U}(t,x), e \rangle \vert {\mathcal F}_{t}
\bigr]
\Bigr\vert 
\\
&=
\biggl\vert
\E 
\biggl[
\int_{0}^1
\Bigl\langle
D_{x}
{\mathcal U}(t,x+ \lambda h e) 
- D_{x} {\mathcal U}(t,x) ,
e 
\Bigr\rangle 
d\lambda
\vert {\mathcal F}_{t}
\biggr] \biggr\vert
\\
&=
\biggl\vert
\int_{0}^1
\Bigl(
\E
\Bigl[
\bigl\langle
D_{x}
{\mathcal U}(t,x+ \lambda h e),e
\bigr\rangle
\vert {\mathcal F}_{t}
\Bigr]
-
\E 
\Bigl[ 
\bigl\langle
D_{x}
{\mathcal U}(t,x),e
\bigr\rangle
\vert {\mathcal F}_{t}
\Bigr]
\Bigr)
d\lambda
\biggr\vert,
\end{split}
\end{equation}
where we used the fact 
the mapping
$[0,T] \times \T^d \ni (t,x) \mapsto \E[D_{x} u(t,x) \vert {\mathcal F}_{t}]$
has continuous paths in order to guarantee
the integrability of the integrand in the third line. 
By continuity of the paths again, the right hand side tends to $0$
with $h$ (uniformly in $t$ and $x$).   
\end{proof}

Instead of \eqref{eq:9:bis}, we will sometimes make use of the 
following:
\begin{Lemma}
\label{lem:reg}
We can find a constant $C$ such that, whenever $\vartheta = \varpi = 0$, 
any solution to \eqref{eq:se:FP/HJB:vartheta} satisfies:
\begin{equation*}
\begin{split}
\forall k \in \{1,\dots,n\}, \quad 
&\int_{t}^T 
\frac{
\essup_{\omega \in \Omega} \| \tilde u_{s}\|_{k+\alpha}}{\sqrt{s-t}}
ds 
\\
&\hspace{15pt} \leq C 
\biggl(  
\essup_{\omega \in \Omega}  \|g_{T}\|_{k+\alpha} + \int_{t}^T 
 \essup_{\omega \in \Omega}  \| f_{s} \|_{k+\alpha-1}
 ds
\biggr). 
\end{split}
\end{equation*}
\end{Lemma}
\begin{proof}
Assume that we have a solution to 
\eqref{eq:se:FP/HJB:vartheta}. Then, 
making use of 
\eqref{eq:part:2:representation:derivatives:3}
in the proof of Lemma \ref{lem:cas:vartheta=0}, we have that, 
for all $k \in \{1,\dots,n\}$
and all $s \in [0,T]$,
\begin{equation}
\label{eq:11}
\begin{split}
&\essup_{\omega \in \Omega} \|  \tilde u_{s}\|_{k+\alpha}
 \leq C \biggl( 
\essup_{\omega \in \Omega}  \|g_{T}\|_{k+\alpha} + \int_{s}^T 
\frac{ 
 \essup_{\omega \in \Omega}  \| f_{r} \|_{k+\alpha-1}}
 {\sqrt{r-s}} dr \biggr).
\end{split}
\end{equation}
Dividing by 
$\sqrt{s-t}$ for a given $t \in [0,T]$, integrating from 
$t$ to $T$
and modifying the value of $C$ if necessary, we deduce that 
\begin{equation*}
\begin{split}
&\int_{t}^T 
 \frac{
\essup_{\omega \in \Omega} \|\tilde u_{s}\|_{k+\alpha}}{\sqrt{s-t}}
ds
\\
&\leq C \biggl(
\essup_{\omega \in \Omega}  \|g_{T}\|_{k+\alpha} + \int_{t}^T  
 ds \int_{s}^T 
 \frac{
 \essup_{\omega \in \Omega}  \| f_{r} \|_{k+\alpha-1}}
 {\sqrt{s-t}\sqrt{r-s}}
 dr
 \biggr)
 \\
 &=
 C \biggl[
\essup_{\omega \in \Omega}  \|g_{T}\|_{k+\alpha} + \int_{t}^T  
 \essup_{\omega \in \Omega}  \| f_{r} \|_{k+\alpha-1}
 \biggl(\int_{t}^r
 \frac{1}
 {\sqrt{s-t}\sqrt{r-s}}
 ds\biggr) dr
 \Biggr],
 \end{split}
\end{equation*}
the last line following from Fubini's theorem. 
The result easily follows.
\end{proof}

Following 
\eqref{eq:9}, we shall use the following variant of 
Lemma 
\ref{lem:reg}:
\begin{Lemma}
\label{lem:reg:cond}
For $p \in \{1,2\}$, we can find a constant $C$ such that, whenever $\vartheta = \varpi = 0$, 
any solution to \eqref{eq:se:FP/HJB:vartheta} satisfies, for all $t \in [0,T]$:
\begin{equation*}
\begin{split}
&\forall k \in \{1,\dots,n\}, \quad 
\E \biggl[ \int_{t}^T 
\frac{
\| \tilde u_{s}\|_{k+\alpha}^p}{\sqrt{s-t}}
ds \vert {\mathcal F}_{t} \biggr] \leq C 
{\mathbb E}\biggl[ 
\| g_{T} \|_{k+\alpha}^p
+ 
\int_{s}^T
\| f_{r} \|_{k+\alpha-1}^p
dr
\, \big\vert \, {\mathcal F}_{t} \biggr].
\end{split}
\end{equation*}
\end{Lemma}

\begin{proof} The proof goes along 
the same lines
as that of Lemma 
\ref{lem:reg}. We start with
the following variant 
of \eqref{eq:part:2:representation:derivatives:3}, 
that holds, 
for any $s \in [0,T]$,
\begin{equation}
\label{eq:label:supplementaire:tildeu}
\| \tilde{u}_{s} \|_{k+\alpha}^p
\leq C 
{\mathbb E}\biggl[ 
\| g_{T} \|_{k+\alpha}^p
+ \int_{s}^T \frac{\| f_{r} \|_{k+\alpha-1}^p}{\sqrt{r-s}} dr
\, \big\vert \, {\mathcal F}_{s} \biggr].
\end{equation}
Therefore, for any $0 \leq t \leq s \leq T$,
we get 
\begin{equation*}
{\mathbb E}
\bigl[
\| \tilde{u}_{s} \|_{k+\alpha}^p \vert {\mathcal F}_{t}
\bigr]
\leq C 
{\mathbb E}\biggl[ 
\| g_{T} \|_{k+\alpha}^p
+ 
\int_{s}^T
\frac{
\| f_{r} \|_{k+\alpha-1}^p}{
\sqrt{r-s}}
dr
\, \big\vert \, {\mathcal F}_{t} \biggr].
\end{equation*}
Dividing by $\sqrt{s-t}$ and integrating in $s$, we get
\begin{equation*}
{\mathbb E}
\biggl[
\biggl(
\int_{t}^T  \frac{\| \tilde{u}_{s} \|_{k+\alpha}^p}{\sqrt{s-t}} ds
\biggr)
\vert {\mathcal F}_{t}
\biggr]
\leq 
C 
{\mathbb E}\biggl[ 
\| g_{T} \|_{k+\alpha}^p
+
\int_{t}^T
\| f_{r} \|_{k+\alpha-1}^p
\biggl( \int_{t}^r
\frac{1}{\sqrt{r-s}\sqrt{s-t}}
ds \biggr) dr
\vert {\mathcal F}_{t}
\biggr]
\end{equation*}
Therefore,
\begin{equation*}
{\mathbb E}
\biggl[
\biggl(
\int_{t}^T  \frac{\| \tilde{u}_{s} \|_{k+\alpha}^p}{\sqrt{s-t}} ds
\biggr)
\vert {\mathcal F}_{t}
\biggr]
\leq 
C 
{\mathbb E}\biggl[ 
\| g_{T} \|_{k+\alpha}^p
+
\int_{t}^T
\| f_{r} \|_{k+\alpha-1}^p
dr
\vert {\mathcal F}_{t}
\biggr],
\end{equation*}
which completes the proof. 
\end{proof}

%
%
%

\subsubsection{\textit{A priori} estimates}
In the previous paragraph, we handled the case $\vartheta= \varpi = 0$. 
In order to handle 
the more general case when $(\vartheta,\varpi) \in [0,1]^2$, 
we shall use the following \textit{a priori} regularity estimate:
\begin{Lemma}
\label{lem:super:reg}
Let $(b_{t}^0)_{t \in [0,T]}$ and 
$(f_{t}^0)_{t \in [0,T]}$ be 
$({\mathcal F}_{t})_{t \in [0,T]}$ 
adapted 
processes
with paths in the space ${\mathcal C}^0([0,T],\cC^1(\T^d,\R^d))$
and ${\mathcal C}^0([0,T],\cC^{n-1}(\T^d))$ and 
$g_{T}$ be an ${\mathcal F}_{T}$-measurable 
random variable with values in ${\mathcal C}^n(\T^d)$, such that 
\begin{equation*}
\essup_{\omega \in \Omega} \sup_{t \in [0,T]} \| b_{t}^0 \|_{1}, 
\
\essup_{\omega \in \Omega}
\sup_{t \in [0,T]} \|f_{t}^0 \|_{n+\alpha-1}, \
\essup_{\omega \in \Omega} \|g_{T}^0\|_{n+\alpha} \leq C,
\end{equation*}
for some constant $C \geq 0$. 
Then, for any $k \in \{0,\dots,n\}$, we can find two constants 
$\lambda_{k}$ and $\Lambda_{k}$, depending upon $C$, such that, 
denoting by ${\mathcal B}$ the cylinder:
\begin{equation*}
{\mathcal B}
:= \Bigl\{ 
w \in {\mathcal C}^0([0,T],\cC^n(\T^d)) : 
\forall k \in \{0,\dots,n\}, \ \forall 
t \in [0,T], \ \| w_{t} \|_{k+\alpha}
\leq 
 \Lambda_{k} \exp \bigl( \lambda_{k} (T-t) \bigr) \Bigr\},
\end{equation*}
it holds that, for any integer $N \geq 1$, any
family of adapted processes
$(\tilde m^i,\tilde u^i)_{i=1,\dots,N}$ with paths in  ${\mathcal C}^0([0,T],{\mathcal P}(\T^d))
\times {\mathcal B}$, any families 
$(a^i)_{i=1,\dots,N} \in [0,1]^N$
and 
$(b^i)_{i=1,\dots,N} \in [0,1]^N$
 with 
$a^{1}+\dots+a^{N}
\leq 2$
and $b^1 + \dots + b^N \leq 2$,
 and any input  $(f_{t})_{t \in [0,T]}$ and $g_{T}$ of the form
\begin{equation*}
\begin{split}
f_{t} = 
\sum_{i=1}^N 
\bigl[ a^i \tilde{H}_{t}(\cdot,D \tilde{u}_{t}^i)
- b^i \tilde{F}_{t}(\cdot,\tilde{m}_{t}^i) \bigr] + f_{t}^0,
\quad g_{T}=
\sum_{i=1}^N 
b^i\tilde{G}(\cdot,\tilde{m}_{T}^i) + g_{T}^0,
\end{split}
\end{equation*}
any solution $(\tilde{m},\tilde{u})$
 to \eqref{eq:se:FP/HJB:vartheta} 
 for some $\vartheta,\varpi \in [0,1]$
has paths in $\cC^0([0,T],{\mathcal P}(\T^d)) \times {\mathcal B}$, that is
\begin{equation*}
\essup_{\omega \in \Omega} \| \tilde u_{t} \|_{k+\alpha} \leq \Lambda_{k} \exp \bigl( \lambda_{k} (T-t) \bigr), \quad t \in [0,T].
\end{equation*}
\end{Lemma}

\begin{proof}
Consider the source term in the backward equation in \eqref{eq:se:FP/HJB:vartheta}:
\begin{equation*}
\varphi_{t} := \vartheta \tilde H_{t}(\cdot,D \tilde{u}_{t}) 
- \varpi \tilde F_{t} (\cdot,\tilde m_{t}) + 
\sum_{i=1}^N
\bigl[ a^i 
\tilde H_t(\cdot ,D \tilde{u}_{t}^i) 
- b^i \tilde F_{t} (\cdot , \tilde m_{t}^i) \bigr] + f_{t}^0. 
\end{equation*}
Then, for any $k \in \{1,\dots,n\}$, we can find 
a constant $C_{k}$ and a continuous non-decreasing function $\Phi_{k}$, independent 
of $(\tilde{m}^i,\tilde{u}^i)$, $i=1,\dots,N$,
and of $(\tilde{m},\tilde{u})$ (but depending on the inputs $(b_{t}^0)_{t \in [0,T]}$, 
$(f_{t}^0)_{t \in [0,T]}$ and $g_{T}$), such that 
\begin{equation}
\label{eq:partie2:borne:varphi}
\| \varphi_{t} \|_{k+\alpha-1} \leq C_{k }
\Bigl[ 1 +
\Phi_{k} \Bigl( \|\tilde{u}_{t} \|_{k+\alpha-1} + \max_{i=1,\dots,N}
\|\tilde{u}_{t}^i \|_{k+\alpha-1}  \Bigr) +  \| \tilde{u}_{t} \|_{k+\alpha}
+  
\max_{i=1,\dots,N}
\|\tilde{u}_{t}^i \|_{k+\alpha}
 \Bigr].
\end{equation}
When $k=1$, the above bound 
holds true with $\Phi_{1} \equiv 0$: It then
follows from 
{\bf (HF1(${\boldsymbol n}$-1))}
and from 
the fact that 
$H$ (or equivalently $\tilde{H}_{t}$)
is globally Lipschitz in $(x,p)$ (uniformly in $t$ if dealing with 
$\tilde{H}_{t}$ instead of $H$). 
When $k \in \{2,\dots,n\}$, it follows from 
the standard Fa\`a di Bruno formula for the 
higher-order derivatives of the composition of two functions
(together with the fact that 
$D_{p}H$ is 
globally
bounded and that the higher-order derivatives
of $H$ are locally bounded). 
Fa\`a di Bruno's formula says that 
each $\Phi_{k}$ may be chosen as a polynomial function.  

Therefore, by 
\eqref{eq:partie2:borne:varphi}
and by
\eqref{eq:11}
in the proof of 
Lemma \ref{lem:reg}
(choosing the constant 
$C_{k}$
such that 
$\| g^0_{T} \|_{k+\alpha}
+ 
\sup_{m \in {\mathcal P}(\T^d)}
\| G(\cdot,m)\|_{k+\alpha} \leq 
C_{k}$
), we deduce that 
\begin{equation*}
\begin{split}
\essup_{\omega \in \Omega}
\| \tilde{u}_{t} \|_{k+\alpha}
\leq C_{k}
&\biggl[ 1 + \essup_{\omega \in \Omega} \sup_{s \in [0,T]} \Phi_{k} \Bigl( \|\tilde{u}_{s} \|_{k+\alpha-1} +
\max_{i=1,\dots,N} \|\tilde{u}_{s}^i \|_{k+\alpha-1}
\Bigr)
\\
&\hspace{15pt} + \int_{t}^T {\frac1{\sqrt{s-t}}}\bigl[  \essup_{\omega \in \Omega}\| \tilde{u}_{s} \|_{k+\alpha}
+  \essup_{\omega \in \Omega} \max_{i=1,\dots,N}
\| \tilde{u}_{s}^i \|_{k+\alpha} \bigr] ds \biggr].
\end{split}
\end{equation*}
Now (independently of the above bound), 
by 
\eqref{eq:partie2:borne:varphi}
and
Lemma \ref{lem:reg}, we can modify $C_{k}$ in such a way that 
\begin{equation*}
\begin{split}
\int_{t}^T \frac
{ \essup_{\omega \in \Omega}\| \tilde{u}_{s} \|_{k+\alpha}}{\sqrt{s-t}} ds
\leq C_{k}
&\biggl[ 1 + \essup_{\omega \in \Omega} \sup_{s \in [0,T]} \Phi_{k} \Bigl( \|\tilde{u}_{s} \|_{k+\alpha-1} +
\max_{i=1,\dots,N} \|\tilde{u}_{s}^i \|_{k+\alpha-1}
\Bigr)
\\
&\hspace{15pt} + \int_{t}^T  
\Bigl(
\essup_{\omega \in \Omega}
\| \tilde{u}_{s} \|_{k+\alpha}
+
\essup_{\omega \in \Omega} \max_{i=1,\dots,N}
\| \tilde{u}_{s}^i \|_{k+\alpha} \Bigr) ds \biggr],
\end{split}
\end{equation*}
so that, collecting the two last inequalities
(and allowing the constant $C_{k}$ to increase from line to line),
\begin{equation}
\label{eq:partie2:proof:stability:s-t:root:bis}
\begin{split}
&\essup_{\omega \in \Omega}
\| \tilde{u}_{t} \|_{k+\alpha}
\\
&\leq C_{k}
\biggl[ 1 + \essup_{\omega \in \Omega} \sup_{s \in [0,T]} \Phi_{k} \Bigl( \|\tilde{u}_{s} \|_{k+\alpha-1} +
\max_{i=1,\dots,N} \|\tilde{u}_{s}^i \|_{k+\alpha-1}
\Bigr)
\\
&\hspace{30pt} + \int_{t}^T
\Bigl(
\essup_{\omega \in \Omega}
\| \tilde{u}_{s} \|_{k+\alpha}
 +
\frac{  
 \essup_{\omega \in \Omega} \max_{i=1,\dots,N}
\| \tilde{u}_{s}^i \|_{k+\alpha}}{\sqrt{s-t}} \Bigr) ds \biggr]
\\
&\leq C_{k}
\biggl[ 1 + \essup_{\omega \in \Omega} \sup_{s \in [0,T]} \Phi_{k} \Bigl( \|\tilde{u}_{s} \|_{k+\alpha-1} +
\max_{i=1,\dots,N} \|\tilde{u}_{s}^i \|_{k+\alpha-1}
\Bigr)
\\
&\hspace{30pt} + \int_{t}^T
\Bigl(
\essup_{\omega \in \Omega}
\| \tilde{u}_{s} \|_{k+\alpha}
 +
\frac{  
 \essup_{\omega \in \Omega} \max_{i=1,\dots,N}
 \sup_{r \in [s,T]}
\| \tilde{u}_{r}^i \|_{k+\alpha}}{\sqrt{s-t}} \Bigr) ds \biggr].
\end{split}
\end{equation}
Now, notice that the last term in the above right-hand side may be rewritten
\begin{equation*}
\begin{split}
&\int_{t}^T
\frac{  
 \essup_{\omega \in \Omega} \max_{i=1,\dots,N}
 \sup_{r \in [s,T]}
\| \tilde{u}_{r}^i \|_{k+\alpha}}{\sqrt{s-t}}
 ds
 \\
&\hspace{15pt} = \int_{0}^{T-t} 
 \frac{  
 \essup_{\omega \in \Omega} \max_{i=1,\dots,N}
 \sup_{r \in [t+s,T]}
\| \tilde{u}_{r}^i \|_{k+\alpha}}{\sqrt{s}}
ds,
\end{split}
\end{equation*}
which is clearly non-increasing in $t$. 
Returning to 
\eqref{eq:partie2:proof:stability:s-t:root:bis}, this permits 
to apply Gronwall's lemma, from which we get:
\begin{equation}
\label{eq:partie2:proof:stability:s-t:root}
\begin{split}
\essup_{\omega \in \Omega}
\| \tilde{u}_{t} \|_{k+\alpha}
&\leq C_{k}
\biggl[ 1 + \essup_{\omega \in \Omega} \sup_{s \in [0,T]} \Phi_{k} \Bigl( \|\tilde{u}_{s} \|_{k+\alpha-1} +
\max_{i=1,\dots,N} \|\tilde{u}_{s}^i \|_{k+\alpha-1}
\Bigr)
\\
&\hspace{30pt} + \int_{t}^T
\frac{   
 \essup_{\omega \in \Omega} \max_{i=1,\dots,N}
 \sup_{r \in [s,T]}
\| \tilde{u}_{r}^i \|_{k+\alpha}}{\sqrt{s-t}} ds \biggr].
\end{split}
\end{equation}
In particular, if, for any $s \in [0,T]$
and any $i \in \{1,\dots,N\}$, 
$\essup_{\omega \in \Omega} 
\| \tilde{u}_{s}^i \|_{k+\alpha} \leq \Lambda_{k} \exp(\lambda_{k}(T-s))$, 
then, for all $t \in [0,T]$, 
\begin{equation}
\label{eq:partie2:proof:stability:s-t:root:2}
\begin{split}
\int_{t}^T  \frac{ 
 \essup_{\omega \in \Omega} \max_{i=1,\dots,N}
\sup_{r \in [s,T]}
\| \tilde{u}_{r}^i \|_{k+\alpha}}{\sqrt{s-t}} ds 
&\leq \Lambda_{k}
\int_{t}^T \frac{\exp(\lambda_{k}(T-s))}{\sqrt{s-t}} ds
\\
&\leq 
\Lambda_{k} \exp (\lambda_{k}(T-t)) 
\int_{0}^{T-t} \frac{\exp(- \lambda_{k} s)}{\sqrt{s}} ds,
\end{split}
\end{equation}
the passage from the first to the second line following from a change of variable. 
Write now
\begin{equation*}
\begin{split}
&\Lambda_{k} \exp (\lambda_{k}(T-t)) 
\int_{0}^{T-t} \frac{\exp(- \lambda_{k} s)}{\sqrt{s}} ds
\\
&= \Lambda_{k} \exp (\lambda_{k}(T-t)) 
\int_{0}^{\infty} \frac{\exp(- \lambda_{k} s)}{\sqrt{s}} ds
- \Lambda_{k} 
\int_{T-t}^{+\infty} \frac{\exp(- \lambda_{k} (s-(T-t))}{\sqrt{s}} ds
\\
&= \Lambda_{k} \exp (\lambda_{k}(T-t)) 
\int_{0}^{\infty} \frac{\exp(- \lambda_{k} s)}{\sqrt{s}} ds
-
\Lambda_{k}  
\int_{0}^{+\infty} \frac{\exp(- \lambda_{k} s)}{\sqrt{T-t+s}} ds
\\
&\leq
\Lambda_{k} \exp (\lambda_{k}(T-t)) 
\int_{0}^{\infty} \frac{\exp(- \lambda_{k} s)}{\sqrt{s}} ds
-
\Lambda_{k}  
\int_{0}^{+\infty} \frac{\exp(- \lambda_{k} s)}{\sqrt{T+s}} ds,
\end{split}
\end{equation*}
and deduce, 
from \eqref{eq:partie2:proof:stability:s-t:root}
and \eqref{eq:partie2:proof:stability:s-t:root:2}, that 
we can find two constants $\gamma_{1}(\lambda_{k})$
and $\gamma_{2}(\lambda_{k})$ that tend to $0$ as $\lambda_{k}$
tend to $+\infty$ such that 
\begin{equation*}
\begin{split}
&\essup_{\omega \in \Omega}
\| \tilde{u}_{t} \|_{k+\alpha}
\\
&\leq C_{k}
\biggl[ 1 + \essup_{\omega \in \Omega} \essup_{s \in [0,T]} \Phi_{k} \bigl( \|\tilde{u}_{s} \|_{k+\alpha-1} +
\max_{i=1,\dots,N}
\|\tilde{u}_{s}^i \|_{k+\alpha-1} 
 \bigr)
- \Lambda_{k} \gamma_{1}(\lambda_{k})
\\
&\hspace{15pt}
+  \gamma_{2}(\lambda_{k})
\Lambda_{k} \exp \bigl( \lambda_{k} (T-t) \bigr) \biggr].
\end{split}
\end{equation*}
Choosing $\lambda_{k}$ first such that $\gamma_{2}(\lambda_{k})
C_{k} \leq 1$
and then $\Lambda_{k}$ such that 
\begin{equation*}
1 + \essup_{\omega \in \Omega} \essup_{s \in [0,T]} \Phi_{k} \bigl( \|\tilde{u}_{s} \|_{k+\alpha-1} +
\max_{i=1,\dots,N}
\|\tilde{u}_{s}^i \|_{k+\alpha-1}  \bigr) \leq 
\gamma_{1}(\lambda_{k}) \Lambda_{k},
\end{equation*}
we finally get that 
\begin{equation*}
\begin{split}
\essup_{\omega \in \Omega}
\| \tilde{u}_{t} \|_{k+\alpha}
&\leq 
\Lambda_{k} \exp \bigl( \lambda_{k} (T-t) \bigr).
\end{split}
\end{equation*}
The proof is easily completed by induction.
\end{proof}

\subsubsection{Case $(\vartheta,\varpi)=(1,0)$}
Using a standard contraction argument, we are going to prove:
\begin{Proposition}
\label{prop:partie2:vartheta1:varpi0}
Given some adapted inputs $(b_{t})_{t \in [0,T]}$, $(f_{t})_{t \in [0,T]}$
and $g_{T}$ satisfying 
\begin{equation*}
\essup_{\omega \in \Omega} \sup_{t \in [0,T]} \|b_{t} \|_{1},
\
\essup_{\omega \in \Omega}
\sup_{t \in [0,T]} \|f_{t} \|_{n+\alpha-1}, \
\essup_{\omega \in \Omega} \|g_{T}\|_{n+\alpha} < \infty,
\end{equation*}
the system \eqref{eq:se:FP/HJB:vartheta}, with $\vartheta=1$ and $\varpi=0$, admits a 
unique adapted
solution $(\tilde{m}_{t},\tilde{u}_{t})_{t \in [0,T]}$, with paths in 
${\mathcal C}^0([0,T],{\mathcal P}(\T^d)) \times 
{\mathcal C}^0([0,T],\cC^n(\T^d))$. 
It satisfies 
\begin{equation*}
\essup_{\omega \in \Omega}
\sup_{t \in [0,T]} \|\tilde{u}_{t} \|_{n+\alpha} < \infty.
\end{equation*}
\end{Proposition}
\begin{proof}
Actually, the only difficulty is to solve the backward equation. 
Once the backward equation has been solved, the forward equation may
be solved by means of Lemma 
\ref{lem:cas:vartheta=0}.  

In order to solve the backward equation,
we make use of the Picard fixed point theorem. 
Given an $({\mathcal F}_{t})_{t \in [0,T]}$ adapted process
$(\tilde{u}_{t})_{t \in [0,T]}$, with paths in 
${\mathcal C}^0([0,T],\cC^{n}(\T^d))$
and
satisfying $\essup_{\omega \in \Omega}
\sup_{t \in [0,T]} \| \tilde{u}_{t} \|_{n+\alpha}
< \infty$, we 
denote by 
$(\tilde{u}_{t}')_{t \in [0,T]}$
the solution to the backward equation in \eqref{eq:se:FP/HJB:vartheta}, with $\vartheta=\varpi=0$
and with $(f_{t})_{t \in [0,T]}$ replaced 
by 
$(f_{t} + H_{t}(\cdot,D \tilde{u}_{t}))_{t \in [0,T]}$. 
By Lemma \ref{lem:cas:vartheta=0}, the process
$(\tilde{u}_{t}')_{t \in [0,T]}$ belongs to 
$ {\mathcal C}^0([0,T],\cC^n(\T^d))$
and satisfies 
$\essup_{\omega \in \Omega}
\sup_{t \in [0,T]} \| \tilde{u}_{t}' \|_{n+\alpha}
< \infty$. 
This defines a mapping (with obvious domain and codomain)
\begin{equation*}
\Psi : 
(\tilde{u}_{t})_{t \in [0,T]}
\mapsto 
(\tilde{u}_{t}')_{t \in [0,T]}.
\end{equation*}
The point is to exhibit a norm for which it is a contraction. 

Given two adapted datas 
$(\tilde{u}_{t}^i)_{t \in [0,T]}$, $i=1,2$,
with paths 
in ${\mathcal C}^0([0,T],\cC^n(\T^d))$
and with 
$\essup_{\omega \in \Omega}
\sup_{t \in [0,T]} \| \tilde{u}_{t}^i \|_{n+\alpha}
< \infty$, $i=1,2$, 
we call 
$(\tilde{u}_{t}^{\prime,i})_{t \in [0,T]}$, $i=1,2$, 
the images by $\Psi$. 
By Lemma \ref{lem:super:reg}
(with $N=1$, $a^1=1$ and $b^1=0$), we can find constants 
$(\lambda_{k},\Lambda_{k})_{k=1,\dots,n}$ such that the cylinder 
\begin{equation*}
{\mathcal B}
= \Bigl\{ 
w \in {\mathcal C}^0([0,T],\cC^n(\T^d)) : 
\forall k \in \{0,\dots,n\}, \ \forall 
t \in [0,T], \ \| w_{t} \|_{k+\alpha}
\leq 
 \Lambda_{k} \exp \bigl( \lambda_{k} (T-t) \bigr) \Bigr\},
\end{equation*}
is stable by $\Psi$. 
We shall prove that $\Psi$ is a contraction on ${\mathcal B}$. 

We let $\tilde{w}_{t} = \tilde{u}_{t}^1 - \tilde{u}_{t}^2$
and $\tilde{w}_{t}' = \tilde{u}_{t}^{\prime,1} - \tilde{u}_{t}^{\prime,2}$, 
for $t \in [0,T]$. We notice that 
\begin{equation*}
-d \tilde w_{t}' = \bigl[  
\Delta \tilde{w}_{t}'
-  \langle \tilde V_{t},D \tilde{w}_{t} \rangle
\bigr] dt - d \tilde{N}_{t},
\end{equation*}
with the terminal boundary condition 
$\tilde{w}_{T}' = 0$. 
Above,  $(\tilde{N}_{t})_{t \in [0, T]}$ is a process with paths in 
$\cC^0([0,T],{\mathcal C}^{n-2}(\T^d))$ and, for any $x \in \T^d$, 
$(\tilde{N}_{t}(x))_{t \in [0,T]}$ is a martingale. 
Moreover, $(\tilde V_{t})_{t \in [0,T]}$ is given by 
\begin{equation*}
\begin{split}
&\tilde V_{t}(x) = \int_0^1 D_p \tilde H_{t}\bigl (x, r D \tilde u^1(x) +(1-r)D \tilde u^2(x) \bigr) dr.
 \end{split}
\end{equation*}
We can find a constant 
$C$ such that, for any $\tilde{u}^1,\tilde{u}^2 \in {\mathcal B}$, 
\begin{equation*}
\sup_{t \in [0,T]}
\|
\tilde{V}_{t}
\|_{n+\alpha-1}
\leq C. 
\end{equation*}
Therefore, for any $\tilde{u}^1,\tilde{u}^2 \in {\mathcal B}$,
for any $k \in \{0,\dots,n-1\}$
\begin{equation*}
\forall t \in [0,T], \quad 
\|
\langle \tilde{V}_{t}, D \tilde{w}_{t} \rangle
\|_{k+\alpha}
\leq C \| \tilde{w}_{t} \|_{k+1+\alpha}, \quad \tilde{w} := \tilde{u}^1 - \tilde{u}^2.
\end{equation*}
Now, 
following 
\eqref{eq:11}, we deduce that, for any 
$k \in \{1,\dots,n\}$,
\begin{equation}
\label{eq:partie2:proof:existence:vartheta=1:varpi=0}
\begin{split}
&\essup_{\omega \in \Omega}  \| \tilde{w}_{t}' \|_{k+\alpha}
\leq C \int_{t}^T 
\frac{\essup_{\omega \in \Omega}
  \| \tilde{w}_{s} \|_{k+\alpha}}{\sqrt{s-t}}
 ds,
\end{split}
\end{equation}
so that, for any $\mu >0$,
\begin{equation*}
\begin{split}
\int_{0}^T 
\essup_{\omega \in \Omega}  \| \tilde{w}_{t}' \|_{k+\alpha}
\exp( \mu t) dt
&\leq C \int_{0}^T 
\essup_{\omega \in \Omega}  \| \tilde{w}_{s} \|_{k+\alpha}
\biggl( \int_{0}^s {\frac{\exp( \mu t)}{\sqrt{s-t}}} dt \biggr) ds
\\
&\leq \biggl( C \int_{0}^{+\infty}
\frac{\exp(- \mu s)}{\sqrt{s}} ds
\biggr) 
\int_{0}^T 
\essup_{\omega \in \Omega}  \| \tilde{w}_{s} \|_{k+\alpha}
\exp( \mu s) ds.
\end{split}
\end{equation*}
Choosing $\mu$ large enough,
we easily deduce that $\Psi$ has at most 
one fixed point in 
${\mathcal B}$. Moreover, letting 
$\tilde{u}^0 \equiv 0$
and defining by induction 
$\tilde{u}^{i+1} = \Psi(\tilde{u}^i)$, $i \in {\mathbb N}$, 
we easily deduce that, 
for $\mu$ large enough,
for any $i,j \in {\mathbb N}$,
\begin{equation*}
\int_{0}^T 
\essup_{\omega \in \Omega}  \| \tilde{u}^{i+j}_{t} - \tilde{u}^i_{t} \|_{n+\alpha}
\exp( \mu t) dt \leq \frac{C}{2^i},
\end{equation*}
so that (modifying the value of $C$)
\begin{equation*}
\int_{0}^T 
\essup_{\omega \in \Omega}  \| \tilde{u}^{i+j}_{t} - \tilde{u}^i_{t} \|_{n+\alpha}
dt \leq \frac{C}{2^i}.
\end{equation*}
Therefore, by definition of ${\mathcal B}$ and by
\eqref{eq:partie2:proof:existence:vartheta=1:varpi=0}, we deduce that, 
for any $\varepsilon >0$,
\begin{equation*}
\forall i \in {\mathbb N}, \quad \sup_{j \in {\mathbb N}}
\essup_{\omega \in \Omega} \sup_{t \in [0,T]}  \| \tilde{u}^{i+j}_{t} - \tilde{u}^i_{t} \|_{n+\alpha}
\leq C \sqrt{\varepsilon} + \frac{C}{2^i \sqrt{\varepsilon}}, 
\end{equation*}
from which we deduce that the sequence 
$(\tilde{u}^i)_{i \in {\mathbb N}}$
converges in $L^{\infty}(\Omega,{\mathcal C}^0([0,T],{\mathcal C}^n(\T^d)))$. 
The limit is in ${\mathcal B}$ and is a fixed point of $\Psi$. 

Actually, by Lemma \ref{lem:super:reg}
(with $N=1$ and $a^1=b^1=0$), any fixed point must be in ${\mathcal B}$, so that 
$\Psi$ has a unique fixed point in the whole space. 
\end{proof}

\subsubsection{Stability estimates}

\begin{Lemma}
\label{lem:partie:2:stability}
Consider two sets of inputs 
$(b,f,g)$ and $(b',f',g')$ to 
\eqref{eq:se:FP/HJB:vartheta}, when driven by 
two parameters $\vartheta,\varpi \in [0,1]$. 
Assume that 
$(\tilde{m},\tilde{u})$ and $(\tilde{m}',\tilde{u}')$
are associated solutions (with adapted paths that take 
values in ${\mathcal C}^0([0,T],{\mathcal P}(\T^d))
\times {\mathcal C}^0([0,T],\cC^n(\T^d))$) that 
satisfy the conclusions of Lemma 
\ref{lem:super:reg} with respect to some 
vectors of constants 
${\boldsymbol \Lambda}=(\Lambda_{1},\dots,\Lambda_{n})$
and ${\boldsymbol \lambda}=(\lambda_{1},\dots,\lambda_{n})$. 
Then, we can find a constant $C \geq 1$, depending on 
the inputs and the outputs through ${\boldsymbol \Lambda}$
and ${\boldsymbol \lambda}$ only, such that, provided that
\begin{equation*}
\essup_{\omega \in \Omega} \sup_{t \in [0,T]}
\| b_{t} \|_{1} \leq \frac{1}{C}
\end{equation*}
it holds that 
\begin{equation*}
\begin{split}
&{\mathbb E}
\bigl[ \sup_{t \in [0,T]} \| \tilde{u}_{t} - \tilde{u}_{t}' \|_{n+\alpha}^2
+{\mathbf d}_{1}^2(\tilde m_{t},\tilde m_{t}')
\bigr] 
\\
&\hspace{15pt} \leq C
\Bigl\{ 
{\mathbf d}_{1}^2(m_{0},m_{0}')
+  
{\mathbb E} \bigl[ \sup_{t \in [0,T]}
\| b_{t} - b_{t}' \|_{0}^2
+  \sup_{t \in [0,T]}
\| f_{t} - f_{t}' \|_{n+\alpha-1}^2
+  
\| g_{T} - g_{T}' \|_{n+\alpha}^2
\bigr] \Bigr\}.
\end{split}
\end{equation*}
\end{Lemma}

\begin{Remark}
The precise knowledge of ${\boldsymbol \Lambda}$
and ${\boldsymbol \lambda}$ is crucial in order to 
make use of the convexity assumption of the Hamiltonian.  
\end{Remark}

The proof relies on the following 
stochastic integration by parts formula:
\begin{Lemma}
\label{lem:partie2:ito:special}
Let $(m_{t})_{t \in [0,T]}$ be an adapted process 
with paths in ${\mathcal C}^0([0,T],{\mathcal P}(\T^d))$
such that, 
with $n$ as in the statement of Theorem 
\ref{thm:partie:2:existence:uniqueness},
for any smooth test function $\varphi \in \cC^n(\T^d)$,
$\P$ almost surely,
\begin{equation*}
d_{t} \biggl[ 
\int_{\T^d} \varphi(x) dm_{t}(x) \biggr]
= \biggl\{ 
\int_{\T^d} 
\bigl[ 
\Delta \varphi(x) - \langle \beta_{t}(x), D\varphi(x) \rangle \bigr] dm_{t}(x)
\biggr\} dt, 
\quad t \in [0,T],
\end{equation*}
for some adapted process $(\beta_{t})_{0 \leq t \leq T}$
with paths in ${\mathcal C}^0([0,T],[\cC^0(\T^d)]^d)$. 
(Notice, by separability of 
$\cC^n(\T^d)$, that the above holds true, $\P$ almost surely, 
for any smooth test function $\varphi \in \cC^n(\T^d)$.)

Let $(u_{t})_{t \in [0,T]}$ be an adapted process 
with paths in ${\mathcal C}^0([0,T],\cC^n(\T^d))$
such that, for any $x \in \T^d$,
\begin{equation*}
d_{t} u_{t}(x) = \gamma_{t}(x) dt + dM_{t}(x), 
 \quad t \in [0,T],
\end{equation*}
where $(\gamma_{t})_{t \in [0,T]}$ and $(M_{t})_{t \in [0,T]}$
are adapted 
processes with paths in ${\mathcal C}^0([0,T],\cC^0(\T^d))$
and, for any $x \in \T^d$, $(M_{t}(x))_{t \in [0,T]}$
is a martingale.  

Assume that
\begin{equation}
\label{eq:partie:2:assumption:ito:special}
\essup_{\omega \in \Omega} \sup_{0 \leq t \leq T}
\bigl( \| {u}_{t} \|_{n} + \| \beta_{t} \|_{0} + \| \gamma_{t} \|_{0}
+  \| M_{t} \|_{0} \bigr) < \infty. 
\end{equation}
Then, the process 
\begin{equation*}
\biggl( 
\int_{\T^d} u_{t}(x) dm_{t}(x)
-
\int_{0}^t \biggl\{ 
\int_{\T^d} 
\bigl[ 
\gamma_{s}(x) + \Delta u_{s}(x) - \langle \beta_{s}(x), D u_{s}(x) \rangle \bigr] dm_{s}(x)
\biggr\} ds
\biggr)_{t \in [0,T]}
\end{equation*}
is a continuous martingale. 
\end{Lemma}

\begin{proof}
Although slightly technical, the proof is quite standard. 
Given two reals $s<t$ in $[0,T]$, we consider 
 a mesh $s=r_{0}<r_{1}<\dots<r_{N}=t$ of the interval
$[s,t]$. Then, 
\begin{equation}
\label{eq:IPP:function:measure:proof}
\begin{split}
&\int_{\T^d} u_{t}(x) dm_{t}(x)
- \int_{\T^d} u_{s}(x) dm_{s}(x)
\\
&= \sum_{i=0}^{N-1}
\biggl[ \int_{\T^d} u_{r_{i+1}}(x) dm_{r_{i+1}}(x)
- \int_{\T^d} u_{r_{i}}(x) dm_{r_{i}}(x) \biggr]
\\
&= \sum_{i=0}^{N-1} \biggl[ \int_{\T^d} u_{r_{i+1}}(x)
dm_{r_{i+1}}(x) -  \int_{\T^d} u_{r_{i+1}}(x)
dm_{r_{i}}(x)
\biggr] 
\\
&\hspace{15pt} + \sum_{i=0}^{N-1} \biggl[
\int_{\T^d} u_{r_{i+1}}(x)
dm_{r_{i}}(x) - \int_{\T^d} u_{r_{i}}(x)
dm_{r_{i}}(x)
\biggr]
\\
&= \sum_{i=0}^{N-1} \int_{r_{i}}^{r_{i+1}}
\biggl\{ 
\int_{\T^d} 
\bigl[ 
\Delta u_{r_{i+1}}(x) - \langle \beta_{r}(x), D u_{r_{i+1}}(x) \rangle \bigr] dm_{r}(x)
\biggr\} dr
\\
&\hspace{15pt} + \sum_{i=0}^{N-1} \int_{\T^d} 
\biggl\{ \int_{r_{i}}^{r_{i+1}}
\gamma_{r}(x) dr + M_{r_{i+1}}(x) - M_{r_{i}}(x) 
\biggr\} dm_{r_{i}}(x). 
\end{split}
\end{equation}
By conditional Fubini's theorem and by 
\eqref{eq:partie:2:assumption:ito:special},
\begin{equation*}
\begin{split}
&{\mathbb E}
\biggl[ 
\sum_{i=0}^{N-1} \int_{\T^d} 
\bigl\{  M_{t_{i+1}}(x) - M_{t_{i}}(x) 
\bigr\} dm_{t_{i}}(x) \vert {\mathcal F}_{s} \biggr]
\\
&= 
\sum_{i=0}^{N-1}
 \int_{\T^d} 
\bigl\{
{\mathbb E}
\bigl[ 
  M_{t_{i+1}}(x) - M_{t_{i}}(x)
\vert {\mathcal F}_{s}  \bigr]
\bigr\} dm_{t_{i}}(x) =0,
\end{split}
\end{equation*}
so that
\begin{equation*}
\E \bigl[ S^N \vert {\mathcal F}_{s} \bigr] = 0,
\end{equation*}
where we have let
\begin{equation*}
\begin{split}
S^N &:=\int_{\T^d} u_{t}(x) dm_{t}(x)
- \int_{\T^d} u_{s}(x) dm_{s}(x)
\\
&\hspace{15pt} - \sum_{i=0}^{N-1}
 \int_{r_{i}}^{r_{i+1}}
\biggl\{ 
\int_{\T^d} 
\bigl[ 
\Delta u_{r_{i+1}}(x) - \langle \beta_{r}(x), D u_{r_{i+1}}(x) \rangle \bigr] dm_{r}(x)
\biggr\} dr
\\
&\hspace{15pt} - \sum_{i=0}^{N-1} \int_{\T^d} 
\biggl\{ \int_{r_{i}}^{r_{i+1}}
\gamma_{r}(x) dr 
\biggr\} dm_{r_{i}}(x).
\end{split}
\end{equation*}
Now, we notice that the sequence $(S^N)_{N \geq 1}$ converges pointwise to
\begin{equation*}
\begin{split}
S^\infty &:=\int_{\T^d} u_{t}(x) dm_{t}(x)
- \int_{\T^d} u_{s}(x) dm_{s}(x)
\\
&\hspace{15pt} - 
 \int_{s}^{t}
\biggl\{ 
\int_{\T^d} 
\bigl[ 
\Delta u_{r}(x) - \langle \beta_{r}(x), D u_{r}(x) \rangle
+ \gamma_{r}(x) \bigr] dm_{r}(x)
\biggr\} dr.
\end{split}
\end{equation*}
As the sequence $(S^N)_{N \geq 1}$ is bounded in 
$L^\infty(\Omega,{\mathcal A},\P)$, it is straightforward to deduce that, 
$\P$ almost surely, 
\begin{equation*}
{\mathbb E}
\bigl[ S^\infty \vert {\mathcal F}_{s} \bigr]
=\lim_{N \rightarrow \infty}
{\mathbb E} \bigl[ S^N \vert {\mathcal F}_{s} \bigr]
 = 0.
\end{equation*}
\end{proof}

We now switch to 
\begin{proof}[Proof of Lemma \ref{lem:partie:2:stability}]
Following the deterministic case, the idea is to use the monotonicity condition. Using the same duality argument as in the deterministic case, we thus compute
by means of Lemma \ref{lem:partie2:ito:special}:
\begin{equation*}
\begin{split}
&d_{t} \int_{{\mathbb T}^d}
\bigl( \tilde{u}^{\prime}_{t} - \tilde u_{t}
\bigr) d \bigl( 
\tilde{m}^{\prime}_{t} - \tilde m_{t}
\bigr)
\\
&=  \biggl\{  - \vartheta
\int_{{\mathbb T}^d}
\big\langle D \tilde{u}^{\prime}_{t}
- D \tilde{u}_{t},
D_{p} \tilde H_{t}(\cdot,D \tilde{u}_{t}^{\prime})
d \tilde{m}_{t}'
-D_{p} \tilde H_{t}(\cdot,D \tilde{u}_{t})
d \tilde{m}_{t}
\big\rangle
\\
&\hspace{15pt}- \int_{{\mathbb T}^d}
\bigl\langle 
 D \tilde{u}^{\prime}_{t}
- D \tilde{u}_{t}, b_{t}' d \tilde{m}_{t}' - b_{t} d \tilde{m}_{t} \big\rangle
+ \vartheta \int_{{\mathbb T}^d}
\bigl( \tilde H_{t}(\cdot,D \tilde{u}_{t}^{\prime})
- 
\tilde H_{t}(\cdot ,D \tilde{u}_{t})
\bigr) d (  \tilde{m}_{t}' - \tilde m_{t}) 
\\
&\hspace{15pt} - \varpi \int_{{\mathbb T}^d}
\bigl( \tilde F_{t}(\cdot, m_{t}')
- 
\tilde F_{t}(\cdot,m_{t})
\bigr) d ( \tilde{m}_{t}' - \tilde m_{t}) 
+ \int_{{\mathbb T}^d}
\bigl( 
f_{t}' - f_{t}
\bigr)d \bigl( \tilde{m}_{t}' - \tilde{m}_{t}
\bigr)
\biggr\} dt
\\
&+ dM_{t},
\end{split}
\end{equation*}
where $(M_{t})_{t \in [0,T]}$ is a martingale, with the terminal boundary condition
\begin{equation*}
\begin{split}
\int_{\T^d}
\bigl(
\tilde{u}_{T}' - \tilde{u}_{T}
\bigr) d ( \tilde{m}_{T}'
- \tilde{m}_{T}) &=
\varpi 
\int_{\T^d}
\bigl(
\tilde{G}(\cdot,m_{T}') - 
\tilde{G}(\cdot,m_{T}) \bigr)
 d ( \tilde{m}_{T}'
- \tilde{m}_{T})
\\
&\hspace{15pt} + \int_{\T^d}
\bigl( g_{T}' - g_{T} \bigr) 
d ( \tilde{m}_{T}'
- \tilde{m}_{T}).
\end{split} 
\end{equation*}

Making use of the convexity and monotonicity assumptions and taking the expectation, 
we can find a constant $c>0$, depending on the inputs and the outputs 
through ${\boldsymbol \Lambda}$ and ${\boldsymbol \lambda}$ only, such that 
\begin{equation}
\label{eq:partie2:convexity}
\begin{split}
&\vartheta c {\mathbb E}\int_{0}^T 
\biggl[ \int_{{\mathbb T}^d}
\vert D \tilde{u}_{t}' - D\tilde{u}_{t}
\vert^2 
d \bigl( \tilde{m}_{t} + \tilde{m}_{t}'
\bigr)\biggr] dt
\\
&\leq
\| u_{0}'-u_{0} \|_{1}
{\mathbf d}_{1}(\tilde{m}_{0},\tilde{m}'_{0}) +
 {\mathbb E}
 \bigl[
\| g_{T}' -g_{T} \|_{1}
{\mathbf d}_{1}(\tilde{m}_{T},\tilde{m}'_{T})
\bigr]
\\
&\hspace{5pt}+
 {\mathbb E}\int_{0}^T 
\| b_{t}' -b_{t} \|_{0}
\| \tilde{u}_{t}' -\tilde{u}_{t}
\|_{1} dt 
+
 {\mathbb E}\int_{0}^T 
 \bigl( 
\| \langle b_{t} , D \tilde{u}_{t}' - D \tilde{u}_{t} \rangle \|_{1}
+ \| f_{t}' -f_{t} \|_{1} \bigr)
{\mathbf d}_{1}(\tilde{m}_{t},\tilde{m}'_{t}) dt.
\end{split}
\end{equation}
We now implement the same strategy as in the
proof of Proposition \ref{prop:Ulip}
in the
deterministic case. 
Following 
\eqref{X1t-X2t}, we get
that there exists a constant $C$,
depending upon $T$, the Lipschitz constant 
of $D_{p} H$
and the parameters 
${\boldsymbol \Lambda}$ and ${\boldsymbol \lambda}$,
 such that
\begin{equation}
\label{eq:partie2:estimation:m:Du}
\begin{split}
&\sup_{t \in [0,T]} 
{\mathbf d}_{1}(\tilde{m}_{t}',\tilde{m}_{t})
 \\
 &\hspace{15pt} \leq C
\biggl( {\mathbf d}_{1}(\tilde{m}_{0}',\tilde{m}_{0})
+  \sup_{t \in [0,T]} \| b_{t}' - b_{t} \|_{0} 
+ \vartheta \int_{0}^T 
\biggl[ \int_{{\mathbb T}^d}
\vert D \tilde{u}_{s}' - D\tilde{u}_{s} \vert  
d
\bigl( \tilde{m}_{s} + \tilde{m}_{s}'
\bigr) \biggr] ds \biggr),
\end{split}
\end{equation}
which holds pathwise. 

Taking the square and the expectation and then plugging \eqref{eq:partie2:convexity}, 
we deduce that, for any small $\eta>0$ and for a possibly new value of $C$,
\begin{equation}
\label{eq:partie2:inegalite:principale:stabilite}
\begin{split}
{\mathbb E} \bigl[ 
\sup_{t} {\mathbf d}_{1}^2(\tilde{m}_{t},\tilde{m}'_{t}) \bigr]
&\leq C
\Bigl\{ \eta^{-1}
{\mathbf d}_{1}^2(m_{0},m_{0}')
+ \eta {\mathbb E}
\bigl[  \sup_{t \in [0,T]}
\| \tilde u_{t} - \tilde u_{t}' \|_{1}^2 \bigr]
\\
&\hspace{15pt} + \eta^{-1} \essup_{\omega \in \Omega} \sup_{t \in [0,T]} \|b_{t}\|_{1}
{\mathbb E}
\bigl[  \sup_{t \in [0,T]}
\| \tilde u_{t} - \tilde u_{t}' \|_{2}^2 \bigr]
\\
&\hspace{15pt}+  
\eta^{-1} {\mathbb E} \bigl[ \sup_{t \in [0,T]}
\| b_{t} - b_{t}' \|_{0}^2
+  \sup_{t \in [0,T]}
\| f_{t} - f_{t}' \|_{1}^2
+  
\| g_{T} - g_{T}' \|_{1}^2
\bigr] \Bigr\}.
\end{split}
\end{equation}
Following the deterministic case, we let $\tilde{w}_{t} = \tilde{u}_{t} - \tilde{u}_{t}'$, 
for $t \in [0,T]$, so that 
\begin{equation}
\label{eq:partie:2:linear:typical:method:1}
-d \tilde w_{t} = \bigl[  
\Delta \tilde{w}_{t}
- \vartheta \langle \tilde V_{t},D \tilde{w}_{t} \rangle
+ \varpi\tilde R^1_{t} - \bigl( f_{t}-f_{t}' \bigr) 
\bigr] dt - d \tilde{N}_{t},
\end{equation}
with the terminal boundary condition 
$\tilde{w}_{T} = \varpi \tilde R^T + g_{T}'-g_{T}$. 
Above,  $(\tilde{N}_{t})_{t \in [0, T]}$ is a process with paths in 
$\cC^0([0,T],{\mathcal C}^0(\T^d))$, 
with $\essup_{\omega \in \Omega} \sup_{t \in [0,T]}
\| \tilde{N}_{t}\|_{0} < \infty$, and, for any $x \in \T^d$, 
$(\tilde{N}_{t}(x))_{t \in [0,T]}$ is a martingale. 
Moreover, 
the coefficients $(\tilde V_{t})_{t \in [0,T]}$, $(\tilde R^1_{t})_{t \in [0,T]}$
and $\tilde R^T$ are given by 
\begin{equation*}
\begin{split}
&\tilde V_{t}(x) = \int_0^1 D_p \tilde H_{t}\bigl (x, r D \tilde u(x) +(1-r)D \tilde u'(x) \bigr) \, dr,
\\
&\tilde R^1_{t}(x) = \int_0^1 \frac{\delta \tilde F_{t}}{\delta m}\bigl(x, r \tilde m_{t}
+(1-r) \tilde m_{t}' \bigr) \bigl(\tilde m_{t} - \tilde m_{t}'
\bigr)
\, dr, 
\\
&\tilde R^T(x) = \int_{0}^1 
\frac{\delta G}{\delta m}(x, r \tilde m_{T}+ 
(1-r) \tilde m_{T}') \bigl( \tilde m_{T} - \tilde m_{T}' \bigr) dr. 
\end{split}
\end{equation*}
Following 
the deterministic case, 
we have 
\begin{equation}
\label{eq:partie2:majoration:R}
\begin{split}
\sup_{t \in [0,T]}  \| \tilde R_t^1 \|_{n+\alpha-1} 
+ \|\tilde R^T \|_{n+\alpha}
&\leq  
C \sup_{t \in [0,T]}
\dk(\tilde m_{t},\tilde m_{t}').
\end{split}
\end{equation}
Moreover, recalling that 
the outputs $\tilde{u}$ and $\tilde{u}'$ are assumed 
to satisfy the conclusion of Lemma 
\ref{lem:super:reg}, we deduce that
\begin{equation*}
\sup_{t \in [0,T]}
\|
\tilde{V}_{t}
\|_{n+\alpha-1}
\leq C. 
\end{equation*}
In particular, for any $k \in \{0,\dots,n-1\}$
\begin{equation*}
\forall t \in [0,T], \quad 
\|
\langle \tilde{V}_{t}, D \tilde{w}_{t} \rangle
\|_{k+\alpha}
\leq C \| \tilde{w}_{t} \|_{k+1+\alpha}. 
\end{equation*}
Now, following 
\eqref{eq:label:supplementaire:tildeu}
and implementing 
\eqref{eq:partie2:majoration:R}, we get, 
for any $t \in [0,T]$,
\begin{equation*}
\begin{split}
&\| \tilde{w}_{t} \|_{k+\alpha}
\\
&\leq {\mathbb E}\biggl[ \| g_{T} - g_{T}' \|_{k+\alpha}
+  \int_{t}^T \frac{ \| \tilde{w}_{s} \|_{k+\alpha}}{\sqrt{s-t}} 
ds + \sup_{s \in [0,T]}
\| f_{s} - f_{s}' \|_{k+\alpha-1}
+
\sup_{s \in [0,T]}
\dk(\tilde m_{s},\tilde m_{s}') 
\, \vert {\mathcal F}_{t}
\biggr]
\\
&\leq C {\mathbb E}\biggl[ \| g_{T} - g_{T}' \|_{k+\alpha}
+  \int_{t}^T \| \tilde{w}_{s} \|_{k+\alpha}
ds + \sup_{s \in [0,T]}
\| f_{s} - f_{s}' \|_{k+\alpha-1} 
+
\sup_{s \in [0,T]}
\dk(\tilde m_{s},\tilde m_{s}') 
\, \vert {\mathcal F}_{t}
\biggr],
\end{split}
\end{equation*}
the second line following from Lemma 
\ref{lem:reg:cond} (with $p=1$). 
By Doob's inequality, we deduce that 
\begin{equation*}
\begin{split}
&{\mathbb E}
\bigl[ \sup_{s \in [t,T]}
\| \tilde{w}_{s} \|_{k+\alpha}^2
\bigr]
\\
\hspace{15pt }&\leq {\mathbb E}\biggl[ \| g_{T} - g_{T}' \|_{k+\alpha}
+  \int_{t}^T \| \tilde{w}_{s} \|_{k+\alpha}^2
ds + \sup_{s \in [0,T]}
\| f_{s} - f_{s}' \|_{k+\alpha-1}^2 
+
\sup_{s \in [0,T]}
\dk^2(\tilde m_{s},\tilde m_{s}') 
\biggr].
\end{split}
\end{equation*}
By Gronwall's lemma, we deduce that, for any 
$k \in \{1,\dots,n\}$,
\begin{equation}
\label{eq:partie:2:linear:typical:method:2}
\begin{split}
&{\mathbb E}
\bigl[ \sup_{t \in [0,T]} \| \tilde{w}_{t} \|_{k+\alpha}^2
\bigr] 
\leq C \E \Bigl[ 
\| g_{T} - g_{T}' \|_{k+\alpha}^2
 +
 \sup_{t \in [0,T]}
\| f_{t} - f_{t}' \|_{k+\alpha-1}^2
+  
 \sup_{t \in [0,T]} {\mathbf d}_{1}^2(\tilde{m}_{t},\tilde{m}'_{t})
\Bigr].
\end{split}
\end{equation}
We finally go back to \eqref{eq:partie2:inegalite:principale:stabilite}. 
Choosing $\eta$ small enough
and assuming that 
$\essup_{\omega \in \Omega} \|b_{t}\|_{1}$
is also
small enough, we finally obtain (modifying the constant $C$):
\begin{equation*}
\begin{split}
&{\mathbb E}
\bigl[ \sup_{t \in [0,T]} \bigl( \| \tilde{u}_{t} - \tilde{u}_{t}' \|_{n+\alpha}^2
+{\mathbf d}_{1}^2(\tilde m_{t},\tilde m_{t}')
\bigr)
\bigr] 
\\
&\leq C
\Bigl\{ 
{\mathbf d}_{1}^2(m_{0},m_{0}')
+  
{\mathbb E} \bigl[ \sup_{t \in [0,T]}
\| b_{t} - b_{t}' \|_{0}^2
+  \sup_{t \in [0,T]}
\| f_{t} - f_{t}' \|_{n+\alpha-1}^2
+  
\| g_{T} - g_{T}' \|_{n+\alpha}^2
\bigr] \Bigr\},
\end{split}
\end{equation*}
which completes the proof.
\end{proof}
\subsubsection{Proof of Theorem \ref{thm:partie:2:existence:uniqueness}}
We now end up the proof of Theorem \ref{thm:partie:2:existence:uniqueness}. 

\textit{First step.}
We first  notice that the $L^2$ stability estimate in the statement is a direct consequence of 
Lemma 
\ref{lem:super:reg}
(in order to bound the solutions)
and
of Lemma
\ref{lem:partie:2:stability}
(in order to get the stability estimate itself),
provided that existence and uniqueness hold true. 
\vspace{5pt}

\textit{Second step (a).}
We now prove that, given an initial condition $m_{0} \in {\mathcal P}(\T^d)$,
the system \eqref{eq:se:3:tilde:HJB:FP} is uniquely solvable. 

The strategy consists in increasing 
inductively the value of $\varpi$, step by step, from $\varpi =0$ 
to $\varpi=1$, and to 
prove, at each step, that 
existence and uniqueness hold true. 
At each step of the induction, the strategy relies on a fixed point argument. It works as follows.
Given some $\varpi \in [0,1)$, we assume that, for any input $(f,g)$ in a certain class, 
we can (uniquely) solve (in the same sense as in the statement of
Theorem \ref{thm:partie:2:existence:uniqueness})
\begin{equation}
\label{eq:10}
\begin{split}
&d_{t} \tilde{m}_{t} = \bigl\{  \Delta \tilde{m}_{t}
+
{\rm div} \bigl[ \tilde{m}_{t}  
 D_{p} \tilde H_{t}( \cdot ,D \tilde{u}_{t}) 
 \bigr] \bigr\} dt,  
\\
&d_{t} \tilde{u}_{t}
=  \bigl\{ -  \Delta \tilde{u}_{t} + 
 \tilde H_{t}(\cdot ,D\tilde{u}_{t}) - \varpi \tilde F_{t}(\cdot ,m_{t}) 
+ f_{t} \bigr\} dt
+ d \tilde{M}_{t},
\end{split}
\end{equation}
with $\tilde{m}_{0}=m_{0}$ as initial condition and $\tilde{u}_{T} = \varpi \tilde G(\cdot,m_{T}) + g_{T}$ as boundary condition. 
Then, the objective is 
to prove that the same holds true
for $\varpi$ replaced by $\varpi+\epsilon$, 
for $\epsilon>0$ small enough (independent of $\varpi$). 
Freezing an input $(\bar f,\bar g)$ in the admissible class, 
the point is to show that the mapping 
\begin{equation*}
\Phi : 
(\tilde m_{t})_{t \in [0,T]}
\mapsto 
\left\{
\begin{array}{l}
\bigl(f_{t} = -\epsilon \tilde F_{t}(\cdot,{m}_{t}) + \bar f_{t} \bigr)_{t \in [0,T]}
\\
g_{T}=
\epsilon \tilde G(\cdot,{m}_{T}) + \bar g_{T}
\end{array}
\right\}
\mapsto (\tilde{m}_{t}')_{t \in [0,T]},
\end{equation*}
is a contraction on the space 
of adapted processes 
$(\tilde m_{t})_{t \in [0,T]}$
with paths in ${\mathcal C}^0([0,T],{\mathcal P}(\T^d))$, 
 where the 
last output is given as the forward component of the solution of the system 
\eqref{eq:10}. 

The value of $\varpi$ being given, we assume that the input
$(\bar f,\bar g)$ is of the form
\begin{equation}
\label{eq:partie:2:bar:f:g}
\begin{split}
&\bar f_{t} = 
- \sum_{i=1}^N b^i
 \tilde{F}_{t}(\cdot,{m}_{t}^i),
\quad
\bar g_{T} =
\sum_{i=1}^N b^i \tilde{G}(\cdot,{m}_{T}^i),
\end{split}
\end{equation}
where $N \geq 1$, $b^1,\dots,b^N \geq 0$, with 
$\epsilon+
b^1+\dots+b^N \leq 2$, and 
$(\tilde{m}^i)_{i=1,\dots,N}$ (or equivalently 
$({m}^i)_{i=1,\dots,N}$)
 is a family of $N$ adapted processes with paths in 
$\cC^0([0,T],{\mathcal P}(\T^d))$. 

The input $((\bar{f}_{t})_{t \in [0,T]},\bar{g}_{T})$
being given, 
we consider 
two adapted processes
$(\tilde{m}^{(1)}_{t})_{t \in [0,T]}$ and 
$(\tilde{m}^{(2)}_{t})_{t \in [0,T]}$
with paths in 
$\cC^0([0,T],{\mathcal P}(\T^d))$
(or equivalently 
$({m}^{(1)}_{t})_{t \in [0,T]}$ and 
$({m}^{(2)}_{t})_{t \in [0,T]}$
without the push-forwards by each of the mappings 
$(\T^d \ni x \mapsto x - \sqrt{2} W_{t} \in \T^d)_{t \in [0,T]}$, 
\textit{cf.}
Remark \ref{rem:partie2:notation}),
and we let
\begin{equation*}
f_{t}^{(i)}=  -\epsilon \tilde F_{t}\bigl(\cdot,{m}_{t}^{(i)}\bigr) + \bar f_{t}, \ 
t \in [0,T] ; \quad 
g_{T}^{(i)} = 
- \epsilon \tilde{G} \bigl( \cdot, m_{T}^{(i)} \bigr)
+ \bar{g}_{T} ;
 \quad i =1,2.
\end{equation*}
and
\begin{equation*}
\tilde{m}^{(i\prime)} = \Phi\bigl(\tilde{m}^{(i)}\bigr),
\quad
i=1,2.
\end{equation*}

\textit{Second step (b).}
By Lemma \ref{lem:super:reg}, 
we can find  
positive constants 
$(\lambda_{k})_{k=0,\dots,n}$
and $(\Lambda_{k})_{k=0,\dots,n}$
such that, whenever 
$(\tilde{m}_{t},\tilde{u}_{t})_{t \in [0,T]}$
solves \eqref{eq:10}
with respect to an input
$((\bar{f}_{t})_{t \in [0,T]},\bar{g}_{T})$
of the same type as in \eqref{eq:partie:2:bar:f:g},
then
\begin{equation*}
\forall k \in \{0,\dots,n\}, \ \forall t \in [0,T],
\quad
\essup_{\omega \in \Omega}
\| \tilde{u}_{t}
\|_{k+\alpha} \leq \Lambda_{k} \exp \bigl( \lambda_{k} (T-t) \bigr).
\end{equation*}
\vspace{5pt}
It is worth mentioning that the values of 
$(\lambda_{k})_{k=0,\dots,n}$
and $(\Lambda_{k})_{k=0,\dots,n}$
are somehow universal in the sense that they depend neither on 
$\varpi$ nor on the precise shape of the inputs 
$(\bar{f},\bar{g})$ when taken in the class \eqref{eq:partie:2:bar:f:g}. In particular, any output $(\tilde{m}_{t}')_{t \in [0,T]}$
of the mapping $\Phi$ 
must satisfy the same bound. 
\vspace{5pt}

\textit{Second step (c).}
We apply Lemma \ref{lem:partie:2:stability} with $b=b'=0$,
$(f_{t},f_{t}')_{0 \leq t \leq T}=(\bar{f}^{(1)}_{t},\bar{f}^{(2)}_{t})_{0 \leq t \leq T}$
and 
$(g_{T},g_{T}')=(\bar{g}^{(1)}_{T},\bar{g}^{(2)}_{T})$. We deduce that
\begin{equation*}
\begin{split}
&{\mathbb E}
\bigl[ \sup_{t \in [0,T]} 
{\mathbf d}_{1}^2(\tilde m_{t}^{(1\prime)},\tilde m_{t}^{(2\prime)})
\bigr] 
\\
&\hspace{15pt} \leq \epsilon^2 C
\Bigl\{ 
{\mathbb E} \bigl[ \sup_{t \in [0,T]}
\| \tilde F_{t}(\cdot,m_{t}^{(1)}) - \tilde F_{t}(\cdot,m_{t}^{(2)}) \|_{n+\alpha-1}^2
+  
\| \tilde G_{T}(\cdot,m_{T}^{(1)}) - \tilde G_{T}(\cdot,m_{T}^{(2)})  \|_{n+\alpha}^2
\bigr] \Bigr\},
\end{split}
\end{equation*}
the constant $C$ being independent of $\varpi$ and of the precise shape of 
the input $(\bar{f},\bar{g})$ in the class \eqref{eq:partie:2:bar:f:g}. 
Up to a modification of $C$, we deduce that 
\begin{equation*}
\begin{split}
&{\mathbb E}
\bigl[ \sup_{t \in [0,T]} 
{\mathbf d}_{1}^2(\tilde m_{t}^{(1\prime)},\tilde m_{t}^{(2\prime)})
\bigr] 
 \leq \epsilon^2 C
{\mathbb E}
\bigl[ \sup_{t \in [0,T]} 
{\mathbf d}_{1}^2(\tilde m_{t}^{(1)},\tilde m_{t}^{(2)})
\bigr],
\end{split}
\end{equation*}
which shows that $\Phi$ is a contraction
on the space $L^2(\Omega,{\mathcal A},\P;{\mathcal C}^0([0,T],{\mathcal P}(\T^d)))$,
 when $\epsilon$ is small enough (independently of $\varpi$ and of 
$(\bar{f},\bar{g})$ in the class \eqref{eq:partie:2:bar:f:g}).
By Picard fixed point theorem, we deduce that the system 
\eqref{eq:10} is solvable when 
$\varpi$ is replaced by $\varpi+\varepsilon$ (and for the same input 
$(\bar{f},\bar{g})$ in the class \eqref{eq:partie:2:bar:f:g}). 
By Lemma \ref{lem:super:reg}
and Proposition
\ref{lem:partie:2:stability}, the solution must be unique.
\vspace{5pt}

\textit{Third step.} We finally establish the $L^{\infty}$ version of the stability estimates. 
The trick is to derive the $L^\infty$ estimate from the $L^2$ version of the stability 
estimates, which seems rather surprising at first sight but which is quite standard in the theory 
of backward SDEs. 

The starting point is to notice that the expectation in the proof of 
the $L^2$ version permits to get rid of the martingale part when 
applying It\^o's formula in the proof of 
Lemma \ref{lem:partie:2:stability}
(see for instance \eqref{eq:partie2:convexity}). Actually, it would 
suffice to use the conditional expectation given 
${\mathcal F}_{0}$ in order to get rid of it, which means that 
the $L^2$ estimate may be written as 
\begin{equation*}
\begin{split}
&{\mathbb E}
\bigl[ \sup_{t \in [0,T]}
\bigl( 
{\mathbf d}_{1}^2(\tilde m_{t},\tilde m_{t}')
+ \| \tilde{u}_{t} - \tilde{u}_{t}' \|_{n+\alpha}^2
\bigr) \vert {\mathcal F}_{0}
\bigr] 
 \leq C
{\mathbf d}_{1}^2(m_{0},m_{0}'),
\end{split}
\end{equation*}
which holds $\P$ almost surely. Of course, when 
$m_{0}$ and $m_{0}'$ are deterministic the above conditional bound does not say 
anything 
more in comparison with the original one: When $m_{0}$ and $m_{0}'$ 
are deterministic,  the $\sigma$-field ${\mathcal F}_{0}$ contains no information and 
is almost surely trivial. Actually, the inequality is especially meaningful when 
the initial time $0$ is replaced by another time $t \in (0,T]$, in which case the initial conditions 
become $\tilde{m}_{t}$ and $\tilde{m}_{t}'$ and are thus random. 
The trick is thus to say that the same inequality as above holds with any time $t \in [0,T]$
as initial condition instead of $0$. This proves that 
\begin{equation*}
{\mathbb E}
\bigl[ \sup_{s \in [t,T]}
\bigl( 
{\mathbf d}_{1}^2(\tilde m_{s},\tilde m_{s}')
+ \| \tilde{u}_{s} - \tilde{u}_{s}' \|_{n+\alpha}^2
\bigr) \vert {\mathcal F}_{t}
\bigr] 
 \leq C
{\mathbf d}_{1}^2(m_{t},m_{t}').
\end{equation*}
Since $\| \tilde{u}_{t} - \tilde{u}_{t}' \|_{n+\alpha}$ is ${\mathcal F}_{t}$-measurable, 
we deduce that
\begin{equation*}
\| \tilde{u}_{t} - \tilde{u}_{t}' \|_{n+\alpha}
 \leq C
{\mathbf d}_{1}(m_{t},m_{t}').
\end{equation*}
Plugging the above bound in \eqref{eq:partie2:estimation:m:Du}, 
we deduce that (modifying $C$ if necessary)
\begin{equation*}
\sup_{t \in [0,T]}{\mathbf d}_{1}(m_{t},m_{t}')
 \leq C
{\mathbf d}_{1}(m_{0},m_{0}').
\end{equation*}
Collecting the two last bounds, the proof is easily completed. \qed
\subsection{Linearization}
\label{subse:par tie:2:linearization}
\textbf{Assumption.} Throughout the paragraph,
 $\alpha$ stands for a H\"older exponent in $(0,1)$.  
\vspace{5pt}

The purpose here is to follow Subsection \ref{subsec:LS} and to discuss the following linearized version 
of the system \eqref{eq:se:3:tilde:HJB:FP}: 
\begin{equation}
\label{eq:partie2:systeme:linearise}
\begin{split}
&d_{t} \tilde{z}_{t} = \bigl\{ -  \Delta \tilde{z}_{t} + \langle \tilde V_{t}(\cdot),D \tilde{z}_{t}
\rangle - 
\frac{\delta \tilde F_{t}}{\delta m}(\cdot,m_{t})(\rho_{t})
 + \tilde{f}_{t}^0 \bigr\} dt + d \tilde{M}_{t},
\\
&\partial_{t} \tilde{\rho}_{t} -  \Delta \tilde{\rho}_{t}
- \textrm{div}\bigl(\tilde{\rho}_{t} \tilde V_{t} 
\bigr) - \textrm{div} \bigl( \tilde{m}_{t} \Gamma_{t} D \tilde{z}_{t} + \tilde b_{t}^0
\bigr) = 0,
\end{split}
\end{equation}
with a boundary condition of the form 
\begin{equation*}
\tilde{z}_{T} = 
\frac{\delta \tilde G}{\delta m}(\cdot,m_{T})(\rho_{t}) + \tilde{g}_{T}^0,
\end{equation*}
where $(\tilde{M}_{t})_{t \in [0,T]}$
is the so-called martingale part of the backward equation, 
that is $(\tilde{M}_{t})_{t \in [0,T]}$ is
an $({\mathcal F}_{t})_{t \in [0,T]}$-adapted process
with paths in the space ${\mathcal C}^0([0,T],
{\mathcal C}^{0}({\mathbb T}^d))$, 
such that, for any $x \in {\mathbb T}^d$, 
$(\tilde{M}_{t}(x))_{t \in [0,T]}$
is an $({\mathcal F}_{t})_{t \in [0,T]}$ martingale. 

\begin{Remark}
\label{rem:partie2:notation:2}
Above, we used the same convention as in Remark \ref{rem:partie2:notation}. 
For $(\tilde{\rho}_{t})_{t \in [0,T]}$ with paths in 
${\mathcal C}^0([0,T],({\mathcal C}^k(\T^d))')$
for some $k \geq 0$, we let 
$({\rho}_{t})_{t \in [0,T]}$ be the distributional-valued
random function 
with paths in 
${\mathcal C}^0([0,T],({\mathcal C}^k(\T^d))')$
defined by 
\begin{equation*}
\langle
\varphi,\rho_{t} \rangle_{
{\mathcal C}^k(\T^d),
({\mathcal C}^k(\T^d))'}
= \langle
\varphi ( \cdot + \sqrt{2} W_{t} ),\tilde \rho_{t} 
\rangle_{
{\mathcal C}^k(\T^d),
({\mathcal C}^k(\T^d))'}.
\end{equation*}
\end{Remark}

Generally speaking, the framework 
is the same as that used in 
Subsection \ref{subsec:LS}, namely we can find a constant $C \geq 1$ such that:
\begin{enumerate}
\item The initial condition $\tilde{\rho}_{0}={\rho}_{0}$ takes values in 
$({\mathcal C}^{n+\alpha'}(\T^d))'$, for some $\alpha' \in (0,\alpha)$, and, unless it is explicitly stated, it is deterministic. 
\item $(\tilde{V}_{t})_{t \in [0,T]}$ is an adapted process 
with paths in ${\mathcal C}^0([0,T],{\mathcal C}^{n}(\T^d,\R^d))$,
with 
\begin{equation*}
\essup_{\omega \in \Omega}
\sup_{t \in [0,T]}
\|\tilde{V}_{t} \|_{n+\alpha} \leq C.
\end{equation*}
\item $(\tilde{m}_{t})_{t \in [0,T]}$ is an adapted process
with paths in ${\mathcal C}^0([0,T],{\mathcal P}(\T^d))$.

\item $({\Gamma}_{t})_{t \in [0,T]}$ is an adapted process
with paths in ${\mathcal C}^0([0,T],[{\mathcal C}^1(\T^d)]^{d \times d})$
such that, with probability 1,
\begin{equation*}
\begin{split}
&\sup_{t \in [0,T]} \| \Gamma_{t} \|_{1} \leq C,
\\
&\forall (t,x) \in [0,T] \times \T^d, \quad 
C^{-1} I_{d} \leq \Gamma_{t}(x) \leq C I_{d}.
\end{split}
\end{equation*}
\item $(\tilde{b}_{t}^0)_{t \in [0,T]}$
is an adapted process with paths in 
${\mathcal C}^0([0,T],
[({\mathcal C}^{n+\alpha-1}(\T^d))']^d)$,
and $(\tilde{f}_{t}^0)_{t \in [0,T]}$
is an adapted process with paths in 
${\mathcal C}^0([0,T],{\mathcal C}^{n}(\T^d))$, with
\begin{equation*}
\essup_{\omega \in \Omega}
\sup_{t \in [0,T]} \bigl( \| \tilde{b}_{t}^0 \|_{-(n+\alpha'-1)}
+ \| \tilde{f}_{t}^0 \|_{{n+\alpha}}
\bigr) < \infty.
\end{equation*}
\item $\tilde{g}_{T}^0$ is an 
${\mathcal F}_{T}$-measurable random variable
with values in ${\mathcal C}^{n+1}(\T^d)$, with
\begin{equation*}
\essup_{\omega \in \Omega}
\| \tilde{g}_{T}^0 \|_{n+1+\alpha} < \infty.
\end{equation*}
\end{enumerate}

Here is the analogue of Lemma \ref{lem:BasicEsti2}:

\begin{Theorem}
\label{thm:partie2:existence:uniqueness:linearise}
Under the assumption (1--6) right above
and 
{\bf (HF1(${\boldsymbol n}$))}
and 
{\bf (HG1(${\boldsymbol n}$+1))}, for $n \geq 2$
and $\beta \in (\alpha',\alpha)$, 
the system 
\eqref{eq:partie2:systeme:linearise}
admits a unique solution 
$(\tilde{\rho},\tilde{z},\tilde{M})$, adapted with respect to the filtration 
$({\mathcal F}_{t})_{t \in [0,T]}$, with paths in the space
${\mathcal C}^0([0,T],({\mathcal C}^{n+\beta}(\T^d))' 
\times {\mathcal C}^{n+1+\beta}(\T^d)
\times {\mathcal C}^{n+\beta}(\T^d)
)$ and with $\essup_{\omega} \sup_{t \in [0,T]}
(\| \tilde{\rho}_{t} \|_{-(n+\beta)} + \| \tilde{z}_{t} \|_{n+1+\beta}
+ \| \tilde{M}_{t} \|_{n-1+\beta})
< \infty$.
It satisfies
\begin{equation*}
\begin{split}
&
\essup_{\omega \in \Omega}
\sup_{t \in [0,T]}
\bigl( \|\tilde{\rho}_{t}\|_{-(n+\alpha')}
+  \| \tilde{z}_{t}\|_{n+1+\alpha} 
+  \| \tilde{M}_{t}\|_{n+\alpha-1} 
\bigr)
< \infty.
\end{split}
\end{equation*}
\end{Theorem}

The proof imitates that one of Theorem 
\ref{thm:partie:2:existence:uniqueness}
and relies on a continuation argument.
For a parameter $\vartheta \in [0,1]$, we consider the system 
\begin{equation}
\label{eq:partie2:systeme:linearise:continuation}
\begin{split}
&d_{t} \tilde{z}_{t} = \bigl\{ -  \Delta \tilde{z}_{t} +  \langle \tilde V_{t}(\cdot),D \tilde{z}_{t}
\rangle - 
\vartheta \frac{\delta \tilde F_{t}}{\delta m}(\cdot,m_{t})(\rho_{t})
 + \tilde{f}_{t}^0 \bigr\} dt + d \tilde{M}_{t},
\\
&\partial_{t} \tilde{\rho}_{t} -  \Delta \tilde{\rho}_{t}
-  \textrm{div}\bigl(\tilde{\rho}_{t} \tilde V_{t} 
\bigr) -  \textrm{div} \bigl( \vartheta  \tilde{m}_{t} \Gamma_{t} D \tilde{z}_{t} + \tilde b_{t}^0
\bigr) = 0,
\end{split}
\end{equation}
with the boundary conditions
\begin{equation}
\label{eq:partie2:systeme:linearise:continuation:bdary}
\tilde{\rho}_{0}=\rho_{0},
\quad
\tilde{z}_{T} = \vartheta
\frac{\delta \tilde G}{\delta m}(\cdot,m_{T})(\rho_{T}) + \tilde{g}_{T}^0. 
\end{equation}
As above the goal is to prove, by increasing step by step the value
of $\vartheta$, that the system \eqref{eq:partie2:systeme:linearise:continuation}, with 
the boundary condition
\eqref{eq:partie2:systeme:linearise:continuation:bdary},
has a unique solution for any $\vartheta \in [0,1]$. 

Following the discussion after 
Theorem \ref{thm:partie:2:existence:uniqueness}, 
notice that, whenever 
$(b_{t})_{t \in [0,T]}$
is a process 
with paths in 
${\mathcal C}^0([0,T],{\mathcal C}^{-(n+\beta)}({\mathbb T}^d))$, for some 
$\beta \in (\alpha',\alpha)$, the quantity $\sup_{t \in [0,T]} \|b_{t}\|_{-(n+\alpha')}$
is a random variable, equal to 
$\sup_{t \in [0,T] \cap {\mathbb Q}} \|b_{t}\|_{-(n+\alpha')}$. Moreover, 
$$\essup_{\omega \in \Omega}
\sup_{t \in [0,T]} \|b_{t}\|_{-(n+\alpha')}
= \sup_{t \in [0,T]} \essup_{\omega \in \Omega}
 \|b_{t}\|_{-(n+\alpha')}.$$ 

Below, we often omit the process
$(\tilde{M}_{t})_{t \in [0,T]}$
when denoting a solution, namely we often write 
$(\tilde{\rho}_{t},\tilde{z}_{t})_{t \in [0,T]}$
instead of
$(\tilde{\rho}_{t},\tilde{z}_{t},\tilde{M}_{t})_{t \in [0,T]}$
so that the backward component is understood implicitly.
We feel that the rule is quite clear now: In a systematic way, 
the martingale component 
has two degrees of regularity less than $(\tilde{z}_{t})_{t \in [0,T]}$. 

Throughout the subsection, we assume that
the assumption of 
Theorem 
\ref{thm:partie2:existence:uniqueness:linearise} 
is in force. 
\subsubsection{Case $\vartheta =0$}
We start with the case $\vartheta=0$:
\begin{Lemma}
\label{lem:partie2:estimate:linear:vartheta=0}
Assume that $\vartheta=0$ in 
the system \eqref{eq:partie2:systeme:linearise:continuation}
with the boundary condition \eqref{eq:partie2:systeme:linearise:continuation:bdary}. 
Then, 
for any $\beta \in (\alpha',\alpha)$,
there is a unique solution 
$(\tilde{\rho},\tilde{z})$, adapted with respect to 
$({\mathcal F}_{t})_{t \in [0,T]}$, 
with paths in 
${\mathcal C}^0([0,T],({\mathcal C}^{n+\beta}(\T^d))' \times {\mathcal C}^{n+1+\beta}(\T^d)))$ and  
 with $\essup_{\omega} \sup_{t \in [0,T]}
(\| \tilde{\rho}_{t} \|_{-(n+\beta)} + \| \tilde{z}_{t} \|_{n+1+\beta})
< \infty$. Moreover,
we can find a constant $C'$, only depending upon $C$, 
the bounds in {\bf (HF1(${\boldsymbol n}$))}
and 
{\bf (HG1(${\boldsymbol n}$+1))},
$T$ and $d$, such that 
\begin{equation*}
\begin{split}
&\essup_{\omega \in \Omega} \sup_{t \in [0,T]} \|\tilde{\rho}_{t}\|_{-(n+\alpha')}
\leq 
C'\bigl( \|{\rho}_{0} \|_{-(n+\alpha')} + 
 \essup_{\omega \in \Omega} \sup_{t \in [0,T]}
\| \tilde{b}_{t}^0 \|_{-(n+\alpha'-1)} \bigr),
\\
&\essup_{\omega \in \Omega} \sup_{t \in [0,T]} \| \tilde{z}_{t}\|_{n+1+\alpha} 
\leq 
C'\bigl(\essup_{\omega \in \Omega} \| \tilde{g}_{T}^0 \|_{n+1+\alpha} + 
 \essup_{\omega \in \Omega} \sup_{t \in [0,T]}
\| \tilde{f}_{t}^0 \|_{n+\alpha} \bigr).
\end{split}
\end{equation*}
\end{Lemma}

\begin{proof}
When 
$\vartheta=0$,  there is no nonlinearity in the equation
and  
it simply reads
\begin{equation}
\label{eq:partie2:systeme:linearise:continuation:2}
\begin{split}
(i) \; &d_{t} \tilde{z}_{t} = \bigl\{ -  \Delta \tilde{z}_{t} +  \langle \tilde V_{t}(\cdot),D \tilde{z}_{t}
\rangle 
 + \tilde{f}_{t}^0 \bigr\} dt + d \tilde{M}_{t},
\\
(ii) \; &\partial_{t} \tilde{\rho}_{t} -  \Delta \tilde{\rho}_{t}
-  \textrm{div}\bigl(\tilde{\rho}_{t} \tilde V_{t} 
\bigr) -  \textrm{div} \bigl(  \tilde b_{t}^0
\bigr) = 0,
\end{split}
\end{equation}
with the boundary condition $\tilde{\rho}_{0}={\rho}_{0}$ and
$\tilde{z}_{T} = \tilde{g}_{T}^0$. 

\textit{First step.} Let us first consider the forward equation (\ref{eq:partie2:systeme:linearise:continuation:2}-(ii)). We notice that,
whenever $\rho_{0}$ and $(\tilde{b}_{t}^0)_{t \in [0,T]}$
are smooth in the space variable, the forward equation may be solved pathwise in the classical 
sense. Then, by the same duality technique as in Lemma \ref{lem:BasicEsti2}
(with the restriction that 
the role played by $n$ in the statement of Lemma \ref{lem:BasicEsti2}
is now played by $n-1$ and that the coefficients $c$ and $b$ in the statement of 
Lemma \ref{lem:BasicEsti2} are now 
respectively
denoted by $\tilde{b}^0$
and $\tilde{f}^0$), 
for any $\beta \in [\alpha',\alpha]$,
it holds, $\P$ almost surely, that
\begin{equation}
\label{eq:linear:adapt:deterministic:5}
\begin{split}
&\sup_{t \in [0,T]} \| \tilde{\rho}_{t} \|_{-(n+\beta)} 
\leq C'
\bigl( \| \tilde{\rho}_{0} \|_{-(n+\beta)} + 
\sup_{t \in [0,T]}
\| \tilde{b}_{t}^0 \|_{-(n-1+\beta)} \bigr).
\end{split}
\end{equation}

Whenever $\rho_{0}$ and 
$(\tilde{b}_{t}^0)_{t \in [0,T]}$
are not smooth but
take values in $({\mathcal C}^{n+\alpha'}(\T^d))'$ 
and 
$({\mathcal C}^{n+\alpha'-1}(\T^d))'$
only, 
we can mollify them by a standard convolution argument.
Denoting 
the mollified sequences
by $(\rho_{0}^N)_{N \geq 1}$ and 
$((\tilde{b}_{t}^{0,N})_{t \in [0,T]})_{N \geq 1}$, it is standard to check that, 
for any $\beta \in (\alpha',\alpha)$,
$\P$ almost surely, 
\begin{equation}
\label{eq:partie:2:cv:mollification}
\lim_{N \rightarrow + \infty} 
\bigl( 
\| \rho_{0}^N -  \rho_{0} \|_{-(n+\beta)}
+ \sup_{t \in [0,T]}
\| \tilde{b}_{t}^{0,N} - \tilde{b}_{t} \|_{-(n-1+\beta)}
\bigr) = 0,
\end{equation}
from which, together with  
\eqref{eq:linear:adapt:deterministic:5}, we deduce
that, $\P$ almost surely, the sequence $((\tilde{\rho}_{t}^{N})_{t \in [0,T]})_{N \geq 1}$
is Cauchy in the space ${\mathcal C}([0,T],({\mathcal C}^{n+\beta}(\T^d))')$, 
where each $(\tilde{\rho}_{t}^{N})_{t \in [0,T]}$
denotes the solution of the forward equation (\ref{eq:partie2:systeme:linearise:continuation:2}-(ii))
with inputs $(\rho_{0}^N,(\tilde{b}_{t}^{0,N})_{t \in [0,T]})$. 
With probability $1$ under $\P$, the limit of the Cauchy sequence belongs to 
${\mathcal C}([0,T],({\mathcal C}^{n+\beta}(\T^d))')$
and satisfies \eqref{eq:linear:adapt:deterministic:5}. 
Pathwise, it solves the forward equation. 

Note that the duality techniques of Lemma \ref{lem:BasicEsti2} are valid for any 
solution $(\tilde{\rho}_{t})_{t \in [0,T]}$
of the forward equation in 
(\ref{eq:partie2:systeme:linearise:continuation:2}-(ii)), with paths in ${\mathcal C}^0([0,T],({\mathcal C}^{n+\beta}(\T^d))')$. 
This proves uniqueness to the forward equation. 

Finally, it is plain that the solution is adapted with respect to the filtration 
$({\mathcal F}_{t})_{t \in [0,T]}$. The reason is that
the solutions are constructed as limits of Cauchy sequences,
which may be shown to be adapted by means of a Duhamel type formula. 
\vspace{2pt}

\textit{Second step.}
For the backward component of \eqref{eq:partie2:systeme:linearise:continuation:2}, 
we can adapt Proposition 
\ref{prop:partie2:vartheta1:varpi0}: the solution is adapted, has paths in ${\mathcal C}^0([0,T],{\mathcal C}^{n+1+\beta}(\T^d))$, for any $\beta \in (\alpha',\alpha)$,
and, following \eqref{eq:9:bis}, it satisfies:
\begin{equation*}
\essup_{\omega \in \Omega} \sup_{0 \leq t \leq T} \| \tilde{z}_{t} \|_{n+1+\alpha} 
\leq C' \bigl( \essup_{\omega \in \Omega} \| \tilde g_{T}^0 \|_{n+1+\alpha}
+ \essup_{\omega \in \Omega} \sup_{t \in [0,T]} \|\tilde f_{t}^0 \|_{n+\alpha}
\bigr),
\end{equation*}
which completes the proof. 
\end{proof}

\subsubsection{Stability argument}
The purpose is now to increase $\vartheta$ step by step in order to 
prove that 
\eqref{eq:partie2:systeme:linearise:continuation}--\eqref{eq:partie2:systeme:linearise:continuation:bdary}
has a unique solution. 

We start with the following consequence of Lemma 
\ref{lem:partie2:estimate:linear:vartheta=0}:

\begin{Lemma}
Given some $\vartheta \in [0,1]$,
an initial condition $\tilde{\rho}_{0}$ 
in $({\mathcal C}^{n+\alpha'}(\T^d))'$,
a set of coefficients 
$(\tilde{V}_{t},\tilde{m}_{t},\Gamma_{t})_{t \in [0,T]}$
as in points 2, 3 and 4
of the introduction of 
Subsection \ref{subse:par tie:2:linearization}
 and a set of inputs 
$((\tilde b_{t}^0,\tilde{f}_{t}^0)_{t \in [0,T]},
\tilde{g}_{T}^0)$
as in points 5 and 6
of the introduction of 
Subsection \ref{subse:par tie:2:linearization},
consider 
a solution 
$(\tilde{\rho}_{t},\tilde{z}_{t})_{t \in [0,T]}$
of the system 
\eqref{eq:partie2:systeme:linearise:continuation}
with the boundary condition 
\eqref{eq:partie2:systeme:linearise:continuation:bdary},  
the solution being adapted with respect to the filtration $({\mathcal F}_{t})_{t \in [0,T]}$, having 
paths in the space
${\mathcal C}^0([0,T],({\mathcal C}^{n+\beta}(\T^d))') \times 
{\mathcal C}^0([0,T],{\mathcal C}^{n+1+\beta}(\T^d))$,
for some $\beta \in (\alpha',\alpha)$, 
and satisfying $\essup_{\omega \in \Omega} \sup_{t \in [0,T]}
(\| \tilde{\rho}_{t} \|_{-(n+\beta)} + \| \tilde{z}_{t} \|_{n+1+\beta})
< \infty$. 

Then,   
\begin{equation*}
\essup_{\omega \in \Omega}
\sup_{t \in [0,T]} \Bigl[ 
\| \tilde{\rho}_{t} \|_{-(n+\alpha')}
+ 
\| \tilde{z}_{t} \|_{n+1+\alpha}
\Bigr] < \infty.
\end{equation*}
\end{Lemma}

\begin{proof}
Given a solution $(\tilde{\rho}_{t},\tilde{z}_{t})_{t \in [0,T]}$
as in the statement, we let
\begin{equation*}
\begin{split}
\hat{b}^0_{t}
= \tilde{b}^0_{t}
+ \vartheta 
\tilde{m}_{t} \Gamma_{t} 
D \tilde{z}_{t}, 
\quad 
\hat{f}^0_{t}
= \tilde{f}^0_{t}
- \vartheta 
\frac{\delta \tilde{F}_{t}}{\delta m}(\cdot,m_{t})(\rho_{t}),
\quad
t \in [0,T] \ ; \quad 
\hat{g}^0_{T}
= \tilde{g}^0_{T} + \vartheta \frac{\delta \tilde{G}}{\delta m}(\cdot,m_{T})
(\rho_{T})
\end{split}
\end{equation*}
Taking benefit from the assumption {\bf (HF1(${\boldsymbol n}$))}, we can check that 
$(\hat{b}_{t}^0)_{t \in [0,T]}$, 
$(\hat{f}_{t}^0)_{t \in [0,T]}$
and
$\hat{g}_{T}^0$
satisfy the same assumptions as  
$(\tilde{b}_{t}^0)_{t \in [0,T]}$, 
$(\tilde{f}_{t}^0)_{t \in [0,T]}$
and
$\tilde{g}_{T}^0$ in the introduction of 
Subsection \ref{subse:par tie:2:linearization}. 
The
result then follows from Lemma 
\ref{lem:partie2:estimate:linear:vartheta=0}.
\end{proof}
\vspace{4pt}

The strategy now relies on a new stability argument, which is the analog of 
Lemma \ref{lem:partie:2:stability}:

\begin{Proposition}
\label{lem:partie2:stability:2}
Given some $\vartheta \in [0,1]$,
two initial conditions $\tilde{\rho}_{0}$ and $\tilde{\rho}_{0}'$
in $({\mathcal C}^{n+\alpha'}(\T^d))'$,
two sets of coefficients 
$(\tilde{V}_{t},\tilde{m}_{t},\Gamma_{t})_{t \in [0,T]}$
and 
$(\tilde{V}_{t}',\tilde{m}_{t}',\Gamma_{t}')_{t \in [0,T]}$
as in points 2, 3 and 4
of the introduction of 
Subsection \ref{subse:par tie:2:linearization}
 and two sets of inputs 
$((\tilde b_{t}^0,\tilde{f}_{t}^0)_{t \in [0,T]},
\tilde{g}_{T}^0)$ and 
$(
(\tilde{b}_{t}^{0\prime},\tilde{f}_{t}^{0\prime})_{t \in [0,T]},
\tilde{g}_{T}^{0\prime})$
as in points 5 and 6
of the introduction of 
Subsection \ref{subse:par tie:2:linearization},
consider 
two solutions $(\tilde{\rho}_{t},\tilde{z}_{t})_{t \in [0,T]}$
and $(\tilde{\rho}_{t}',\tilde{z}_{t}')_{t \in [0,T]}$
of the system 
\eqref{eq:partie2:systeme:linearise:continuation}
with the boundary condition 
\eqref{eq:partie2:systeme:linearise:continuation:bdary}, both 
being adapted with respect to the filtration $({\mathcal F}_{t})_{t \in [0,T]}$, having 
paths in the space
${\mathcal C}^0([0,T],({\mathcal C}^{n+\beta}(\T^d))') \times 
{\mathcal C}^0([0,T],{\mathcal C}^{n+1+\beta}(\T^d))$,
for some $\beta \in (\alpha',\alpha)$, 
and satisfying 
$$
\essup_{\omega \in \Omega} \sup_{t \in [0,T]}
\left(\| \tilde{\rho}_{t} \|_{-(n+\beta)} + \| \tilde{z}_{t} \|_{n+1+\beta}
+\| \tilde{\rho}_{t}' \|_{-(n+\beta)} + \| \tilde{z}_{t}' \|_{n+1+\beta}\right)
< \infty. 
$$
Then,   
it holds that
\begin{equation*}
\begin{split}
&{\mathbb E}
\biggl[ 
\sup_{t \in [0,T]} \| \tilde{z}_{t} - \tilde{z}_{t}' \|_{n+1+\alpha}^2
+
\sup_{t \in [0,T]} 
\| \tilde{\rho}_{t} - \tilde{\rho}_{t}' \|_{-(n+\alpha')}^2
\biggr] 
\\
&\hspace{5pt} \leq C'
\biggl\{ 
\| \tilde{\rho}_{0} - \tilde{\rho}_{0}' \|_{-(n+\alpha')}^2
\\
&\hspace{15pt}+  
{\mathbb E} \biggl[ \sup_{t \in [0,T]}
\| \tilde b_{t}^0 - \tilde b_{t}^{0\prime} \|_{-(n+\alpha'-1)}^2
+  \sup_{t \in [0,T]}
\| \tilde f_{t}^{0} - \tilde f_{t}^{0\prime} \|_{{n+\alpha}}^2
+  
\| \tilde g_{T}^{0} - \tilde g_{T}^{0\prime} \|_{n+1+\alpha}^2
\\
&\hspace{25pt}
+ \sup_{t \in [0,T]}
\Bigl\{
\bigl( \| \tilde{z}_{t}' \|_{n+1+\alpha}^2 + \| \tilde{\rho}_{t}'\|_{-(n+\alpha')}^2
\bigr) \bigl(
\| \tilde{V}_{t} - \tilde{V}_{t}' \|_{n+\alpha}^2
+ [{\mathbf d}_{1}(m_{t},m_{t}')]^2
+ 
\| \Gamma_{t} - \Gamma_{t}' \|_{0}^2
\bigr) \Bigr\}
\biggr] \biggr\},
\end{split}
\end{equation*}
the constant $C'$ only depending upon $C$ in the introduction of 
Subsection \ref{subse:par tie:2:linearization},
$T$, $d$, $\alpha$ and $\alpha'$. 
\end{Proposition}

 \begin{proof}
\textit{First step.} The first step is to make use of a duality argument. 

We start with the case when 
$\tilde \rho_{0}$, $\tilde \rho_{0}'$, 
$\tilde{b}^0$
and $\tilde{b}^{0\prime}$
are smooth. Letting 
$\hat{b}^0_{t} = \vartheta \tilde{m}_{t}
\Gamma_{t} D \tilde{z}_{t} + 
\tilde{b}^0_{t}$
and 
$\hat{b}^{0\prime}_{t} = \vartheta \tilde{m}_{t}'
\Gamma_{t}' D \tilde{z}_{t}' + 
\tilde{b}^{0\prime}_{t}$, 
for $t \in [0,T]$,
we notice that 
$(\tilde{\rho}_{t})_{t \in [0,T]}$
and 
$(\tilde{\rho}_{t}')_{t \in [0,T]}$
solve the linear equation \textit{(ii)} in 
\eqref{eq:partie2:systeme:linearise:continuation:2}
with $(\tilde{b}^0_{t})_{t \in [0,T]}$
and $(\tilde{b}^{0\prime}_{t})_{t \in [0,T]}$
replaced by $(\hat{b}^0_{t})_{t \in [0,T]}$
and $(\hat{b}^{0\prime}_{t})_{t \in [0,T]}$
respectively. By Lemma 
\ref{lem:partie2:estimate:linear:vartheta=0}
with 
$(\tilde{b}^0_{t})_{t \in [0,T]}$
in 
\eqref{eq:partie2:systeme:linearise:continuation:2}
equal to 
$(\hat b_{t}^0)_{t \in [0,T]}$ and 
with $n$ in the statement of Lemma 
\ref{lem:partie2:estimate:linear:vartheta=0} replaced by 
$n-1$, we deduce that 
$(\tilde{\rho}_{t})_{t \in [0,T]}$
and $(\tilde{\rho}_{t}')_{t \in [0,T]}$
have
bounded paths in
${\mathcal C}^0([0,T],({\mathcal C}^{n-1+\beta}(\T^d))')$, 
for the same $\beta \in (\alpha',\alpha)$ as in the statement of 
Proposition \ref{lem:partie2:stability:2}.

 With a suitable adaptation of Lemma 
 \ref{lem:partie2:ito:special}
   and with the same kind of notations 
 as in 
Subsection \ref{subsec:LS}, this permits to 
 expand the infinitesimal variation of the duality 
 bracket 
$\langle \tilde{z}_{t} - \tilde{z}_{t}',\tilde{\rho}_{t} 
 - \tilde{\rho}_{t}' \rangle_{X_{n},X_{n}'}$, 
 with $X_{n}={\mathcal C}^{n+\beta}(\T^d)$. 
We compute
 \begin{equation*}
\begin{split}
 &d_{t} \bigl\langle \tilde{z}_{t} - \tilde{z}_{t}',\tilde{\rho}_{t} 
 - \tilde{\rho}_{t}' \bigr\rangle_{X_{n},X_{n}'}
 \\
 &= \Bigl\{ - \Bigl\langle \tilde D ( \tilde{z}_{t} - \tilde{z}_{t}'), \tilde{\rho}_{t}' 
  \bigl( \tilde{V}_{t}
 -  \tilde{V}_{t}'\bigr) \Bigr\rangle_{X_{n},X_{n}'}  +
 \Bigl\langle 
 D \tilde{z}_{t}',
 \bigl( 
 \tilde{V}_{t}
 - \tilde{V}_{t}' \bigr) \bigl( \tilde{\rho}_{t}
- \tilde{\rho}_{t}'\bigr)
 \Bigr\rangle_{X_{n},X_{n}'} \Bigr\} dt
 \\ 
 &\hspace{5pt} + \Bigl\{ \Big\langle \tilde{f}_{t}^0 - \tilde{f}_{t}^{0\prime},
\tilde{\rho}_{t} 
 - \tilde{\rho}_{t}' \Big\rangle_{X_{n},X_{n}'} dt 
- \Big\langle D \bigl( \tilde{z}_{t}- \tilde{z}_{t}' \bigr), 
 \tilde{b}_{t}^0- \tilde{b}_{t}^{0\prime} \Big\rangle_{X_{n-1},X_{n-1}'} 
 \Bigr\} dt
 \\
&\hspace{5pt} - \vartheta \Bigl\{ \Big\langle \frac{\delta \tilde{F}_{t}}{\delta m}(\cdot,m_{t})
 \bigl( \rho_{t} - \rho_{t}' \bigr), \tilde \rho_{t} - \tilde \rho_t' \Big\rangle_{X_{n},X_{n}'}
 +
 \Big\langle \bigl( 
 \frac{\delta \tilde{F}_{t}}{\delta m}(\cdot,m_{t})
 - \frac{\delta \tilde{F}_{t}}{\delta m}(\cdot,m_{t}')
 \bigr)
 \bigl( \rho_{t}' \bigr), \tilde \rho_{t} - \tilde \rho_t' \Big\rangle_{X_{n},X_{n}'}
 \Bigr\} dt
 \\
&\hspace{5pt} - \vartheta 
\Bigl\{ 
\Big\langle D \bigl( \tilde{z}_{t}- \tilde{z}_{t}' \bigr),
\tilde{m}_{t} \Gamma_{t}
D \bigl( \tilde{z}_{t}- \tilde{z}_{t}' \bigr)
 \Big\rangle_{X_{n},X_{n}'} 
 +
\Big\langle D \bigl( \tilde{z}_{t}- \tilde{z}_{t}' \bigr),
\bigl( \tilde{m}_{t} \Gamma_{t}- \tilde{m}_{t}' \Gamma_{t}'
\bigr) 
D \tilde{z}_{t}' 
 \Big\rangle_{X_{n},X_{n}'}  
 \Bigr\}dt 
 \\
&\hspace{15pt} + d_t M_{t},
 \end{split} 
 \end{equation*}
 where $(M_{t})_{0 \leq t \leq T}$ is a martingale
 and where we applied Remark 
 \ref{rem:partie2:notation:2}
 to define $(\rho_{t})_{t \in [0,T]}$
 and $(\rho_{t}')_{t \in [0,T]}$. 
 An important fact in the proof is that 
 the martingale part 
 in 
 \eqref{eq:partie2:systeme:linearise:continuation:2}
  has
continuous paths in 
${\mathcal C}^0([0,T],\cC^{n-1+\beta}(\T^d))$, which 
permits to 
give a sense to the duality bracket (in $x$) with $(\tilde{\rho}_{t}-\tilde{\rho}_{t}')_{t \in [0,T]}$, 
since $(\tilde{\rho}_{t}-\tilde{\rho}_{t}')_{t \in [0,T]}$
is here assumed to have continuous 
paths
in
${\mathcal C}^0([0,T],(\cC^{n-1+\beta}(\T^d))')$. 
Similarly, the duality bracket of $(\tilde{z}_{t}-\tilde{z}_{t}')_{t \in [0,T]}$ 
with the Laplacian 
of $(\tilde{\rho}_{t}-\tilde{\rho}_{t}')_{t \in [0,T]}$ makes sense
and, conversely,  the duality bracket of $(\tilde{\rho}_{t})_{t \in [0,T]}$ 
with the Laplacian 
of $(\tilde{z}_{t}-\tilde{z}_{t}')_{t \in [0,T]}$ makes sense as well, the two of them 
canceling with one another. 

Of course, the goal is to relax the smoothness 
assumption made 
on $\tilde{\rho}_{0}$, 
$\tilde{\rho}_{0}'$, $\tilde{b}^0$ and $\tilde{b}^{0\prime}$. 
Although it was pretty straightforward to do in the deterministic case, 
it is more difficult here because of the additional martingale term. 
As already mentioned, the martingale term is defined as a duality 
bracket between a path with values in 
${\mathcal C}^0([0,T],\cC^{n-1+\beta}(\T^d))$
and a path with values in
${\mathcal C}^0([0,T],\cC^{-(n-1+\beta)}(\T^d))$. Of course, the problem is that
 this
 is no more true in the general case that
$(\tilde{\rho}_{t}-\tilde{\rho}_{t}')_{t \in [0,T]}$
has paths in 
${\mathcal C}^0([0,T],\cC^{-(n-1+\beta)}(\T^d))$. 
In order to circumvent the difficulty, a way is to take first the expectation 
in order to cancel the martingale part and then to relax the smoothness 
conditions. 
Taking the expectation in the above formula, we get
(in the mollified setting):
 \begin{align}
 &\frac{d}{dt} \E \bigl[ \bigl\langle \tilde{z}_{t} - \tilde{z}_{t}',\tilde{\rho}_{t} 
 - \tilde{\rho}_{t}' \bigr\rangle_{X_{n},X_{n}'}
 \bigr] \nonumber
 \\
 &= \Bigl\{ - \E  \Bigl[ 
 \Bigl\langle \tilde D ( \tilde{z}_{t} - \tilde{z}_{t}'), \tilde{\rho}_{t}' 
  \bigl( \tilde{V}_{t}
 -  \tilde{V}_{t}'\bigr) \Bigr\rangle_{X_{n},X_{n}'}
 \Bigr]  +
\E  \Bigl[ \Bigl\langle 
 D \tilde{z}_{t}',
 \bigl( 
 \tilde{V}_{t}
 - \tilde{V}_{t}' \bigr) \bigl( \tilde{\rho}_{t}
- \tilde{\rho}_{t}'\bigr)
 \Bigr\rangle_{X_{n},X_{n}'} \Bigr] \Bigr\} \nonumber
 \\ 
 &\hspace{5pt} + \Bigl\{ 
 \E \Bigl[
 \Big\langle \tilde{f}_{t}^0 - \tilde{f}_{t}^{0\prime},
\tilde{\rho}_{t} 
 - \tilde{\rho}_{t}' \Big\rangle_{X_{n},X_{n}'} 
 \Bigr] - \E \Bigl[ \Big\langle D \bigl( \tilde{z}_{t}- \tilde{z}_{t}' \bigr), 
 \tilde{b}_{t}^0- \tilde{b}_{t}^{0\prime} \Big\rangle_{X_{n-1},X_{n-1}'} 
 \Bigr]
 \Bigr\} \nonumber
  \\
&\hspace{5pt} - \vartheta \Bigl\{ 
\E \Bigl[ 
\Big\langle \frac{\delta \tilde{F}_{t}}{\delta m}(\cdot,m_{t})
 \bigl( \rho_{t} - \rho_{t}' \bigr), \tilde \rho_{t} - \tilde \rho_t' \Big\rangle_{X_{n},X_{n}'}
 \Bigr]
 \label{eq:duality:martingale}
\\
&\hspace{150pt}
 +
\E \Bigl[ \Big\langle \bigl( 
 \frac{\delta \tilde{F}_{t}}{\delta m}(\cdot,m_{t})
 - \frac{\delta \tilde{F}_{t}}{\delta m}(\cdot,m_{t}')
 \bigr)
 \bigl( \rho_{t}' \bigr), \tilde \rho_{t} - \tilde \rho_t' \Big\rangle_{X_{n},X_{n}'}
 \Bigr]
 \Bigr\} \nonumber
 \\
&\hspace{5pt} - \vartheta 
\Bigl\{ 
\E \Bigl[ \Big\langle D \bigl( \tilde{z}_{t}- \tilde{z}_{t}' \bigr),
\tilde{m}_{t} \Gamma_{t}
D \bigl( \tilde{z}_{t}- \tilde{z}_{t}' \bigr)
 \Big\rangle_{X_{n},X_{n}'} 
 \Bigr]
 +
 \E \Bigl[
\Big\langle D \bigl( \tilde{z}_{t}- \tilde{z}_{t}' \bigr),
\bigl( \tilde{m}_{t} \Gamma_{t}- \tilde{m}_{t}' \Gamma_{t}'
\bigr) 
D \tilde{z}_{t}' 
 \Big\rangle_{X_{n},X_{n}'}  
 \Bigr]
 \Bigr\}. \nonumber
 \end{align}
Whenever $\tilde{\rho}_{0}$, 
$\tilde{\rho}_{0}'$, $\tilde{b}^0$ and 
$\tilde{b}^{0\prime}$ are not 
smooth
(and thus just satisfy the assumption in the statement of 
Proposition 
\ref{lem:partie2:stability:2}), we can mollify them in the same 
way as in 
the first step of 
Lemma 
\ref{lem:partie2:estimate:linear:vartheta=0}. 
We call $(\tilde{\rho}_{0}^{N})_{p \geq 1}$, 
$(\tilde{\rho}_{0}^{\prime,N})_{p \geq 1}$,
$(\tilde{b}^{0,N}_{t})_{p \geq 1}$
and $(\tilde{b}^{0\prime,N}_{t})_{p \geq 1}$
the mollifying sequences. For any $\beta' \in (\alpha',\alpha)$
and $\P$ almost surely, 
the two sequences respectively converge to 
$\tilde{\rho}_{0}$ and $\tilde{\rho}_{0}'$
in norm $\| \cdot \|_{-(n+\beta')}$ and the two last ones 
respectively converge to $\tilde{b}^0_{t}$ and 
$\tilde{b}^{0\prime}_{t}$ in norm $\| \cdot \|_{-(n-1+\beta')}$,
uniformly in $t \in [0,T]$. With 
$(\tilde{\rho}_{t},\tilde{z}_{t})_{t \in [0,T]}$
and $(\tilde{\rho}_{t}',\tilde{z}_{t}')_{t \in [0,T]}$
the original solutions given by the statement of 
Proposition \ref{lem:partie2:stability:2}, 
we 
denote, for each $N \geq 1$, by $(\tilde{\rho}_{t}^{N},\tilde{z}_{t}^{N})_{t \in [0,T]}$
and
$(\tilde{\rho}_{t}^{\prime,N},\tilde{z}_{t}^{\prime,N})_{t \in [0,T]}$ 
the respective solutions to 
\eqref{eq:partie2:systeme:linearise:continuation:2}, but 
with 
$(\tilde{b}^0_{t},\tilde{f}_{t}^0,\tilde{g}^0_{T})_{t \in [0,T]}$
respectively replaced by 
\begin{equation*}
\begin{split}
&\Bigl(\hat{b}^{0,N}_{t} = \tilde{b}_{t}^{0,N} + \vartheta
\tilde{m}_{t} \Gamma_{t} D \tilde{z}_{t},
\hat{f}^{0}_{t} = \tilde{f}_{t}^{0} - \vartheta
\frac{\delta F}{\delta m}(\cdot,m_{t})(\rho_{t}),
\hat{g}^0_{T}= \tilde{g}_{T}^0 + \vartheta 
\frac{\delta G}{\delta m}(\cdot,m_{T})(\rho_{T})
\Bigr)_{t \in [0,T]},
\\
\textrm{\rm and} \  &\Bigl(\hat{b}^{0\prime,N}_{t} = \tilde{b}_{t}^{0\prime,N} + \vartheta
\tilde{m}_{t}' \Gamma_{t}' D \tilde{z}_{t}',
\hat{f}^{0\prime}_{t} = \tilde{f}_{t}^{0\prime} - \vartheta
\frac{\delta F}{\delta m}(\cdot,m_{t}')(\rho_{t}'),
\hat{g}^{0\prime}_{T} = \tilde{g}_{T}^{0\prime} + \vartheta
\frac{\delta G}{\delta m}(\cdot,m_{T}')(\rho_{T}')
\Bigr)_{t \in [0,T]}.
\end{split}
\end{equation*}
By linearity of 
\eqref{eq:partie2:systeme:linearise:continuation:2}
and by Lemma 
\ref{lem:partie2:estimate:linear:vartheta=0}, we have that 
$(\tilde{\rho}_{t}^{N})_{N \geq 1}$
and $(\tilde{\rho}_{t}^{\prime,N})_{N \geq 1}$
converge to $\tilde{\rho}_{t}$
and $\tilde{\rho}_t'$ in norm $\| \cdot \|_{-(n+\beta)}$, 
uniformly in $t \in [0,T]$, 
and that $(\tilde{z}_{t}^{N})_{N \geq 1}$
and $(\tilde{z}_{t}^{\prime,N})_{N \geq 1}$
converge to $\tilde{z}_{t}$
and $\tilde{z}_t'$ in norm $\| \cdot \|_{n+1+\beta}$, 
uniformly in $t \in [0,T]$.

Then, 
we may write down the analogue of 
\eqref{eq:duality:martingale}
for any mollified solution 
$(\tilde{\rho}_{t}^{N},
\tilde{z}_{t}^{N})_{N \geq 1}$
(pay attention that the formulation
of
\eqref{eq:duality:martingale}
for the mollified solutions
 is slightly different
since the mollified solutions only satisfy an approximating 
version of \eqref{eq:partie2:systeme:linearise:continuation}). 
Following
\eqref{eq:partie:2:cv:mollification}, we can 
pass to the limit under the symbol $\E$. 
By Lemma
\ref{lem:partie2:estimate:linear:vartheta=0}, 
we can easily exchange the almost sure convergence and the 
symbol $\E$, proving that the identity 
\eqref{eq:duality:martingale}
holds true 
under the standing assumption on $\tilde{\rho}_{0}$, 
$\tilde{\rho}_{0\prime}$, 
$(\tilde{b}^{0}_{t})_{t \in [0,T]}$
and  $(\tilde{b}^{0\prime}_{t})_{t \in [0,T]}$. 
 
Using the convexity of $\Gamma$ and 
the monotonicity of $\tilde{F}$, we deduce that
\begin{equation*}
\begin{split}
&{\mathbb E}
\left[ \bigl\langle \tilde{z}_{T} - \tilde{z}_{T}',\tilde{\rho}_{T} 
 - \tilde{\rho}_{T}' \bigr\rangle_{X_{n},(X_{n})'} \right]
 + C^{-1} \vartheta
 {\mathbb E}
 \biggl[ \int_{t}^T 
\biggl(
 \int_{\T^d}
 \big\vert 
 D \bigl(
 \tilde{z}_{s} - \tilde{z}_{s}' \bigr)
 \big\vert^2 d\tilde{m}_{s}
 \biggr)
 ds
 \biggr]
 \\
&\hspace{15pt} 
\leq 
{\mathbb E}
\bigl[ 
\bigl\langle \tilde{z}_{0} - \tilde{z}_{0}',\tilde{\rho}_{0} 
 - \tilde{\rho}_{0}' \bigr\rangle_{X_{n},X_{n}'}  \bigr]
+ C' {\mathbb E}
\biggl[ \int_{t}^T 
\Theta \Bigl( 
\| \tilde{\rho}_s - \tilde{\rho}_{s}' \|_{{-(n+\alpha')}} +
\| \tilde{z}_s - \tilde{z}_{s}' \|_{n+1+\alpha} \Bigr) ds
\biggr],
\end{split}
\end{equation*}
where
\begin{equation*}
\begin{split}
&\Theta :=
\| \tilde{\rho}_{0} - \tilde{\rho}_{0} \|_{-(n+\alpha')}
+
\| \tilde{g}_{T}^0 - \tilde{g}_{T}^{0\prime} \|_{n+1+\alpha}
\\
&\hspace{30pt}
+
\sup_{s \in [0,T]}
\Bigl[
\| \tilde{f}_{s}^0 - \tilde{f}_{s}^{0\prime} \|_{{n+\alpha}}
+
\| \tilde{b}_{s}^0 - \tilde{b}_{s}^{0\prime} \|_{-(n+\alpha'-1)}
\\
&\hspace{50pt}
+ 
\bigl( \| \tilde{z}_{s}' \|_{n+1+\alpha} + \| \tilde{\rho}_{s}'\|_{-(n+\alpha')}
\bigr) \bigl(
\| \tilde{V}_{s} - \tilde{V}_{s}' \|_{n+\alpha}
+ {\mathbf d}_{1}(m_{s},m_{s}')
+ 
\| \Gamma_{s} - \Gamma_{s}' \|_{0}
\bigr)
\Bigr].
\end{split}
\end{equation*}
Recalling that 
\begin{equation*}
\begin{split}
&\langle
\tilde{z}_{T} - \tilde{z}_{T}',\tilde{\rho}_{T} - \tilde{\rho}_{T}' \rangle_{X_{n},X_{n}'} 
\\
&= 
\vartheta \Big\langle 
\frac{\delta \tilde G}{\delta m}(\cdot,m_{T})(\rho_{T}-\rho_{T}'),
\tilde{\rho}_{T} - \tilde{\rho}_{T}' \Big\rangle_{X_{n},X_{n}'}
\\
&\hspace{15pt}
+ 
\vartheta \Big\langle 
\bigl(
\frac{\delta \tilde G}{\delta m}(\cdot,m_{T})
- \frac{\delta \tilde G}{\delta m}(\cdot,m_{T}')
\bigr)
(\rho_{T}'),
\tilde{\rho}_{T} - \tilde{\rho}_{T}' \Big\rangle_{X_{n},X_{n}'}
 + 
\langle
\tilde{g}_{T}^0 - \tilde{g}_{T}^{0\prime},\tilde{\rho}_{T}
-\tilde{\rho}_{T}'
\rangle_{X_{n},X_{n}'} 
\\
&\geq 
- C' \Theta
\| \tilde{\rho}_{T} - \tilde{\rho}_{T}' \|_{-(n+\alpha')},
\end{split}
\end{equation*}
where we have used the monotonicity of $G$ to deduce the second line,  
we thus get 
\begin{equation}
\label{eq:partie2:systeme:linearise:continuation:10}
\begin{split}
&\vartheta
 {\mathbb E}
 \biggl[ \int_{0}^T 
 \biggl( \int_{\T^d}
 \big\vert 
 D \bigl(
 \tilde{z}_{s} - \tilde{z}_{s}' \bigr)
 \big\vert^2
 d \tilde{m}_{s} \biggr)
 ds
 \biggr]
 \\
&\hspace{15pt} 
\leq 
C' {\mathbb E}
\biggl[ 
\Theta
\biggl( 
\| \tilde{z}_{0} - \tilde{z}_{0}' \|_{n+1+\alpha}
+
\| \tilde{\rho}_{T} - \tilde{\rho}_{T}' \|_{-(n+\alpha')}
\\
&\hspace{30pt}
+ \int_{0}^T \bigl( 
\| \tilde{\rho}_s - \tilde{\rho}_{s}' \|_{-(n+\alpha')} 
+ \| \tilde{z}_s - \tilde{z}_{s}' \|_{n+1+\alpha}
\bigr) ds
\biggr)
\biggr].
\end{split}
\end{equation}

\textit{Second step.}
As a second step, 
we follow the strategy used in the deterministic case in order to 
estimate 
$(\| \tilde\rho_{t} - \tilde{\rho}_{t}'\|_{-(n+\alpha')})_{t \in [0,T]}$
in terms of 
$\int_{0}^T 
(\int_{\T^d}
 \vert 
 D \bigl(
 \tilde{z}_{s} - \tilde{z}_{s}' \bigr)
 \vert^2
 d \tilde{m}_{s}) ds$
 in the left-hand side of \eqref{eq:partie2:systeme:linearise:continuation:10}. 

We use again a duality argument. 
Given $\xi \in {\mathcal C}^{n+\alpha}(\T^d)$ and $\tau \in [0,T]$, 
we consider the 
solution $(\tilde w_{t})_{t \in [0,\tau]}$,
with paths in ${\mathcal C}^{0}([0,\tau],{\mathcal C}^{n+\beta}(\T^d))$,
to the backward PDE:
\begin{equation}
\label{eq:partie:2:bounds:linear:regularization:2}
\partial_{t} \tilde{w}_{t} = \bigl\{ -  \Delta \tilde{w}_{t} 
+ \langle \tilde V_{t}(\cdot), D \tilde{w}_{t} \rangle \bigr\},
\end{equation}
with the terminal boundary condition $\tilde{w}_{\tau} = \xi$. 
Pay attention that the solution is not adapted. It satisfies (see the proof in the last step below), with probability 
$1$, 
\begin{equation}
\label{eq:partie:2:bounds:linear:regularization}
\begin{split}
&\forall t \in [0,\tau], \quad \|\tilde{w}_{t} \|_{n+\alpha'} \leq 
C'
\| \xi \|_{n+\alpha'},
\\
&\forall t \in [0,\tau), \quad \|\tilde{w}_{t} \|_{n+1+\alpha'} \leq \frac{C'}{\sqrt{\tau-t}} 
\| \xi \|_{n+\alpha'}.
\end{split} 
\end{equation}
Then, 
letting $X_{n}^- = {\mathcal C}^{n+\alpha'}(\T^d)$
and 
following the end of the proof of Lemma 
\ref{lem:BasicEsti2}, we have
\begin{equation*}
\begin{split}
&d_{t} \langle \tilde{w}_{t},\tilde{\rho}_{t}
-\tilde{\rho}_{t}'
\rangle_{X_{n}^-,(X_{n}^-)'}
\\
&= - \Big\langle D \tilde{w}_{t}, 
 \tilde{b}_{t}^0- \tilde{b}_{t}^{0\prime} \Big\rangle_{X_{n-1}^-,(X_{n-1}^-)'} dt \
{+} \
\Big\langle D \tilde{w}_{t}, (\tilde{V}_{t}'-\tilde{V}_{t}) 
\tilde{\rho}_{t}' \Big\rangle_{X_{n}^-,(X_{n}^-)'} dt 
\\
&\hspace{5pt} -\vartheta 
\Big\langle D \tilde{w}_{t} ,
\tilde{m}_{t} \Gamma_{t}
D \bigl( \tilde{z}_{t}- \tilde{z}_{t}' \bigr)
 \Big\rangle_{X_{n}^-,(X_{n}^-)'} dt
 - \vartheta
 \Big\langle D \tilde{w}_{t},
 (\tilde{m}_{t} \Gamma_{t}
 - \tilde{m}_{t}' \Gamma_{t}')
 D \tilde{z}_{t}'
 \Big\rangle_{X_{n}^-,(X_{n}^-)'} dt,
\end{split}
\end{equation*}
so that 
\begin{equation*}
\begin{split}
&\langle \xi ,\tilde{\rho}_{\tau}
-\tilde{\rho}_{\tau}'
\rangle_{X_{n}^-,(X_{n}^-)'}
\leq C' 
\| \xi \|_{n+\alpha'}
\biggl[
\Theta
+
\vartheta \int_{0}^\tau
\biggl( \int_{\T^d} 
\big| D (\tilde{z}_{s} - \tilde{z}_{s}') \big\vert^2 d \tilde{m}_{s}
\biggr)^{1/2} ds  \biggr].
\end{split}
\end{equation*}
Therefore,
\begin{equation}
\label{eq:partie2:systeme:linearise:continuation:11}
\begin{split}
&\| \tilde{\rho}_{\tau} - \tilde{\rho}_{\tau}' \|_{-(n+\alpha')}
\leq C'  
\biggl[ \Theta 
+ \vartheta  
\biggl( 
\int_{0}^T\int_{\R^d} 
\big| D (\tilde{z}_{s} - \tilde{z}_{s}') \big\vert^2 d \tilde{m}_{s}
\biggr)^{1/2} ds  \biggr].
\end{split}
\end{equation}
Plugging
\eqref{eq:partie2:systeme:linearise:continuation:11}
into \eqref{eq:partie2:systeme:linearise:continuation:10}, we obtain
\begin{equation}
\label{eq:partie2:systeme:linearise:continuation:10:b}
\begin{split}
\vartheta
 {\mathbb E}
 \biggl[ \int_{0}^T 
 \biggl( \int_{\T^d}
 \big\vert 
 D \bigl(
 \tilde{z}_{s} - \tilde{z}_{s}' \bigr)
 \big\vert^2
 d \tilde{m}_{s} \biggr)
 ds
 \biggr]
\leq 
C' {\mathbb E}
\Bigl[ 
\Theta
\Bigl( \Theta
+ 
  \sup_{t \in [0,T]} \| \tilde{z}_{t} - \tilde{z}_{t}' \|_{n+1+\alpha}
  \Bigr) \Bigr].
\end{split}
\end{equation}
Therefore,
\begin{equation}
\label{eq:partie2:systeme:linearise:continuation:10:c}
\begin{split}
&
\E \Bigl[
\sup_{t \in [0,T]} 
\| \tilde{\rho}_{t} - \tilde{\rho}_{t}' \|_{-(n+\alpha')}^2
\Bigr]
\leq C'  
\E \Bigl[ \Theta \Bigl( \Theta 
+  \sup_{t \in [0,T]} \| \tilde{z}_{t} - \tilde{z}_{t}' \|_{n+1+\alpha}
\Bigr)  
\Bigr].
\end{split}
\end{equation}

\textit{Third step.} We now combine the two first steps to get an estimate of
$(\| \tilde{z}_{t}- \tilde{z}_{t}'\|_{n+1+\alpha})_{t \in [0,T]}$. 
Following the proof of
\eqref{eq:partie:2:linear:typical:method:2}
on the linear equation 
\eqref{eq:partie:2:linear:typical:method:1}
and using the assumptions
{\bf (HF1(${\boldsymbol n}$))}
and 
{\bf (HG1(${\boldsymbol n}$+1))}, 
we get
 that 
\begin{equation}
\label{eq:partie2:systeme:linearise:continuation:61}
\begin{split}
{\mathbb E}
\bigl[\sup_{t \in [0,T]}
\| \tilde{z}_{t} - \tilde{z}_{t}'\|_{n+1+\alpha}^2
\bigr] \leq 
{\mathbb E} \biggl[ 
\Theta^2
+
\| \tilde{\rho}_{T} - \tilde{\rho}_{T}'\|_{-(n+\alpha')}^2
+ 
\int_{0}^T 
 \| \tilde{\rho}_{s} - \tilde{\rho}_{s}' \|_{-(n+\alpha')}^2
 ds
\biggr]. 
\end{split}
\end{equation}
By
\eqref{eq:partie2:systeme:linearise:continuation:10:c}, 
we easily complete the proof. 

It just remains to prove 
\eqref{eq:partie:2:bounds:linear:regularization}. 
The first line follows from Lemma \ref{l.estilinear}. 
The second line may be proved as follows. 
Following \eqref{eq:partie2:proof:existence:vartheta=1:varpi=0},
we have, with probability $1$,
\begin{equation}
\label{eq:partie2:duality:argument:100}
\forall t \in [0,\tau), \quad 
\| \tilde{w}_{t} \|_{n+1+\alpha'} \leq 
C' \biggl( 
\frac{
\| \xi \|_{n+\alpha'}
}{\sqrt{\tau -t}}
+ \int_{t}^{\tau} \frac{\| \tilde{w}_{s} \|_{n+1+\alpha'}}{\sqrt{s-t}}
ds
\biggr). 
\end{equation}
Integrating and allowing 
the constant $C'$ to increase from line to line,
we have, for all $t \in [0,\tau)$,
\begin{equation*}
\begin{split}
&\int_{t}^{\tau}
\frac{\| \tilde{w}_{s} \|_{n+1+\alpha'}}{\sqrt{s-t}} ds 
\\
&\leq 
C' \biggl[
\| \xi \|_{n+\alpha'}
\int_{t}^{\tau}
\frac{1}{\sqrt{\tau -s} \sqrt{s-t}} ds
+ \int_{t}^{\tau}
\| \tilde{w}_{r} \|_{n+1+\alpha'}
\biggl(
\int_{t}^{r}
 \frac{1}{\sqrt{r-s}\sqrt{s-t}}
ds
\biggr) dr \biggr]
\\
&\leq C' \biggl[
\| \xi \|_{n+\alpha'}
+ \int_{t}^{\tau}
\| \tilde{w}_{r} \|_{n+1+\alpha'}
 dr \biggr].
\end{split}
\end{equation*}
Plugging the above estimate into 
\eqref{eq:partie2:duality:argument:100}, we get that 
\begin{equation*}
\forall t \in [0,\tau), \quad 
\sqrt{\tau -t}
\| \tilde{w}_{t} \|_{n+1+\alpha'} \leq 
C' 
\biggl( 
\| \xi \|_{n+\alpha'} + 
\int_{t}^{\tau}
\sqrt{\tau -r}
\| \tilde{w}_{r} \|_{n+1+\alpha'}
dr
\biggr),
\end{equation*}
which yields, by Gronwall's lemma,
\begin{equation*}
\label{eq:partie2:duality:argument:101}
\forall t \in [0,\tau), \quad 
\| \tilde{w}_{t} \|_{n+1+\alpha'} \leq 
\frac{C'}{\sqrt{\tau-t}} \| \xi \|_{n+\alpha'},
\end{equation*}
which is the required bound. 
 \end{proof}

\subsubsection{A priori estimate}
 
A typical example of application
of Proposition
\ref{lem:partie2:stability:2}
 is to choose:
$\tilde{\rho}_{0}'=0$, $(\tilde b^{0\prime},\tilde f^{0\prime},\tilde g^{0\prime})
\equiv
(0,0,0)$, $\tilde{V} \equiv \tilde{V}'$, $\Gamma
\equiv \Gamma'$, in which case 
\begin{equation*}
\bigl( \tilde{\rho}',\tilde{z}' \bigr) \equiv (0,0).
\end{equation*}
Then, Proposition
\ref{lem:partie2:stability:2}
provides an \textit{a priori} $L^2$ estimate of the solutions 
to 
\eqref{eq:partie2:systeme:linearise}. 
(Pay attention that the constant $C$ in the statement 
depends upon the 
smoothness assumptions satisfied by 
$\tilde{V}$.)
The following corollary shows that the 
$L^2$ bound can be turned into 
an $L^\infty$ bound. It reads as extension of Lemma 
\ref{lem:partie2:estimate:linear:vartheta=0} to the case when 
$\vartheta$ may be non zero:
\begin{Corollary}
\label{cor:partie2:systeme:linearise:linfty}
Given $\vartheta \in [0,1]$, an initial condition $\tilde{\rho}_{0}$ in $({\mathcal C}^{n+\alpha'}(\T^d))'$
and a set of inputs $((\tilde b_{t}^0,\tilde{f}_{t}^0)_{t \in [0,T]},
\tilde{g}_{T}^0)$
as in points 1--6
in the introduction of 
Subsection \ref{subse:par tie:2:linearization}, 
consider 
an adapted solution $(\tilde{\rho}_{t},\tilde{z}_{t})_{t \in [0,T]}$
of the system 
\eqref{eq:partie2:systeme:linearise:continuation}--\eqref{eq:partie2:systeme:linearise:continuation:bdary}, 
with paths in the space 
${\mathcal C}^0([0,T],({\mathcal C}^{n+\beta}(\T^d))') \times 
{\mathcal C}^0([0,T],{\mathcal C^{n+1+\beta}}(\T^d))$
for some $\beta \in (\alpha',\alpha)$, 
such that 
$$\essup_{\omega \in \Omega}
\sup_{t \in [0,T]} 
\Bigl( \| \tilde{z}_{t}  \|_{n+1+\beta}
+
\| \tilde{\rho}_{t} \|_{-(n+\beta)} 
\Bigr) < \infty.$$
Then, we can find a constant $C'$, only depending upon $C$, $T$, $d$, 
$\alpha$ and $\alpha'$, such that 
\begin{equation}
\label{eq:cor:partie2:systeme:linearise:linfty:1}
\begin{split}
&
\essup_{\omega \in \Omega}
\sup_{t \in [0,T]}
\Bigl( \| \tilde{z}_{t}  \|_{n+1+\alpha}
+
\| \tilde{\rho}_{t} \|_{-(n+\alpha')}
\Bigr)
\\
&\hspace{5pt} \leq C'
\biggl( 
\| \tilde{\rho}_{0}\|_{-(n+\alpha')}
+  
\essup_{\omega \in \Omega}
\Bigl[
\| \tilde g_{T}^{0}  \|_{n+1+\alpha}
+
\sup_{t \in [0,T]}
\Bigl(
\| \tilde f_{t}^{0}  \|_{{n+\alpha}}
+  
\| \tilde b_{t}^0 \|_{-(n+\alpha'-1)}
\Bigr)
\Bigr]
\biggr).
\end{split}
\end{equation}
For another initial condition $\tilde{\rho}_{0}'$ in $({\mathcal C}^{n+\alpha'}(\T^d))'$
and another set of inputs $((\tilde b_{t}^{0\prime},\tilde{f}_{t}^{0\prime})_{t \in [0,T]},
\tilde{g}_{T}^{0,\prime})$
as in points 1--6
in the introduction of 
Subsection \ref{subse:par tie:2:linearization}, 
consider 
an adapted solution $(\tilde{\rho}_{t}',\tilde{z}_{t}')_{t \in [0,T]}$
of the system 
\eqref{eq:partie2:systeme:linearise:continuation}--\eqref{eq:partie2:systeme:linearise:continuation:bdary}, 
with paths in the space 
${\mathcal C}^0([0,T],({\mathcal C}^{n+\beta}(\T^d))') \times 
{\mathcal C}^0([0,T],{\mathcal C^{n+1+\beta}}(\T^d))$
for the same $\beta \in (\alpha',\alpha)$ as above, 
such that,  
$$\essup_{\omega \in \Omega}
\sup_{t \in [0,T]} 
\Bigl( \| \tilde{z}_{t}'  \|_{n+1+\beta}
+
\| \tilde{\rho}_{t}' \|_{-(n+\beta)} 
\Bigr) < \infty.$$
Then, we can find a constant $C'$, only depending upon $C$, $T$, $d$, 
$\alpha$ and $\alpha'$ and on 
\begin{equation*}
\begin{split}
&\| \tilde{\rho}_{0}\|_{-(n+\alpha')}
+
\| \tilde{\rho}_{0}'\|_{-(n+\alpha')}
+
\essup_{\omega \in \Omega}
\bigl[
\| \tilde g_{T}^{0}  \|_{n+1+\alpha}
+
\| \tilde g_{T}^{0 \prime}  \|_{n+1+\alpha}
\bigr]
\\
&\hspace{15pt} +  
\essup_{\omega \in \Omega}
\sup_{t \in [0,T]}
\Bigl[
\| \tilde f_{t}^{0}  \|_{{n+\alpha}}
+\| \tilde f_{t}^{0 \prime}  \|_{{n+\alpha}}
+  
\| \tilde b_{t}^0 \|_{-(n+\alpha'-1)}
+
\| \tilde b_{t}^{0\prime} \|_{-(n+\alpha'-1)}
\Bigr],
\end{split}
\end{equation*}
such that 
\begin{equation}
\label{eq:cor:partie2:systeme:linearise:linfty:2}
\begin{split}
&\essup_{\omega \in \Omega} 
\sup_{t \in [0,T]} 
\Bigl[ 
\| \tilde{z}_{t} - \tilde{z}_{t}' \|_{n+1+\alpha}^2
+
\| \tilde{\rho}_{t} - \tilde{\rho}_{t}' \|_{-(n+\alpha')}^2
\Bigr]
\\
&\hspace{5pt} \leq C'
\biggl\{ 
\| \tilde{\rho}_{0} - \tilde{\rho}_{0}' \|_{-(n+\alpha')}^2
\\
&\hspace{15pt}
+ \essup_{\omega \in \Omega}
\biggl(
\| \tilde g_{T}^{0} - \tilde g_{T}^{0\prime} \|_{n+1+\alpha}^2
+\sup_{t \in [0,T]}
\Bigl[
\| \tilde b_{t}^0 - \tilde b_{t}^{0\prime} \|_{-(n+\alpha'-1)}^2
+ 
\| \tilde f_{t}^{0} - \tilde f_{t}^{0\prime} \|_{{n+\alpha}}^2
\Bigr]
\\
&\hspace{15pt}+  
\essup_{\omega \in \Omega}
\sup_{t \in [0,T]}
\Bigl[
\| \tilde{V}_{t} - \tilde{V}_{t}' \|_{n+\alpha}^2
+ {\mathbf d}_{1}^2(m_{t},m_{t}')
+
\| \Gamma_{t} - \Gamma_{t}' \|_{0}^2
\Bigr] \biggr) \biggr\}.
\end{split}
\end{equation}

\end{Corollary}

\begin{proof}
We start with the proof of 
\eqref{eq:cor:partie2:systeme:linearise:linfty:1}. 
\vspace{2pt}

\textit{First step.}
The proof relies on the
same trick as that used in the third step of the proof of 
Theorem \ref{thm:partie:2:existence:uniqueness}. In the statement of
Proposition \ref{lem:partie2:stability:2}, 
the initial conditions
$\tilde{\rho}_{0}$ and $\tilde{\rho}_{0}'$ are assumed to be deterministic. 
It can be checked that the same argument holds
when both are random and the expectation is replaced
by a conditional expectation given 
the initial condition. 
More generally, given some time 
$t \in [0,T]$, we may see the 
pair $(\tilde{\rho}_{s},\tilde{z}_{s})_{s \in [t,T]}$
as the solution 
of
the system 
\eqref{eq:partie2:systeme:linearise:continuation}
with the boundary condition 
\eqref{eq:partie2:systeme:linearise:continuation:bdary}, but
on the interval $[t,T]$ instead of $[0,T]$. In particular, 
when
$\tilde{\rho}_{0}'=0$,
$(\tilde b^{0\prime},\tilde f^{0\prime},\tilde g^{0\prime})
\equiv
(0,0,0)$, $\tilde{V} \equiv \tilde{V}'$, $\Gamma
\equiv \Gamma'$
(in which case $(\tilde{\rho}',\tilde{z}') \equiv (0,0)$), 
we get 
\begin{equation*}
\begin{split}
&{\mathbb E}
\Bigl[ 
\sup_{s \in [t,T]} 
\Bigl( \| \tilde{z}_{s}  \|_{n+1+\alpha}^2
+
\| \tilde{\rho}_{s}  \|_{-(n+\alpha')}^2
\Bigr)
\big\vert {\mathcal F}_{t}
\biggr] 
 \leq C'
\Bigl[ 
\| \tilde{\rho}_{t}  \|_{-(n+\alpha')}^2
+  
{\mathbb E} \bigl[ 
\Theta^2
\vert {\mathcal F}_{t}
\bigr] \Bigr],
\end{split}
\end{equation*}
where we have let
\begin{equation*}
\Theta =
\sup_{s \in [t,T]}
\| \tilde b_{s}^0  \|_{-(n+\alpha'-1)}
+  \sup_{s \in [t,T]}
\| \tilde f_{s}^{0}  \|_{n+\alpha}
+  
\| \tilde g_{T}^{0} \|_{n+1+\alpha}.
\end{equation*}

\textit{Second step.}
We now prove the estimate on 
$\tilde{\rho}$. 
From the first step, we deduce that 
\begin{equation}
\label{eq:partie2:linearisation:linfty:2}
\| \tilde{z}_{t} \|^2_{n+1+\alpha} \leq 
C' \Bigl[ 
\| \tilde{\rho}_{t} \|^2_{-(n+\alpha')}
+ {\mathbb E}
\bigl[ \Theta^2
\vert {\mathcal F}_{t} \bigr] 
\Bigr]
\leq C' \Bigl[  
\| \tilde{\rho}_{t} \|^2_{-(n+\alpha')}
+ \essup_{\omega \in \Omega} \Theta^2
\Bigr]
\end{equation}
The above inequality holds 
true for any $t \in [0,T]$, $\P$ almost surely. 
By continuity of both sides, we can exchange the 
`$\P$ almost sure' and the `for all $t \in [0,T]$'.  
Now we can use the same duality trick as in
the proof of 
Proposition
\ref{lem:partie2:stability:2}. 
With the same notations as in 
\eqref{eq:partie:2:bounds:linear:regularization:2}
and 
\eqref{eq:partie:2:bounds:linear:regularization},
we have 
\begin{equation*}
\begin{split}
&\forall t \in [0,\tau], 
\quad \| \tilde{w}_{t} \|_{n+\alpha'} \leq 
C' \| \xi \|_{n+\alpha'}. 
\end{split}
\end{equation*}
Then, we have
\begin{equation*}
\begin{split}
\bigl\langle \tilde{w}_{\tau},
\tilde{\rho}_{\tau}
\bigr\rangle_{X_{n}^-,(X_{n}^-)'}
&\leq 
\bigl\langle \tilde{w}_{0},
\tilde{\rho}_{0}
\bigr\rangle_{X_{n}^-,(X_{n}^-)'}
+ \int_{0}^{\tau}
\| D \tilde{w}_{s} \|_{n+\alpha'-1}
\bigl( \| \tilde{b}_{s}^0 \|_{-(n+\alpha'-1)}
+ \| \tilde{z}_{s} \|_{n+\alpha} \bigr) ds
\\
&\leq C' \| \xi \|_{n+\alpha'}
\biggl( \| \tilde{\rho}_{0} \|_{-(n+\alpha')}
+ \int_{0}^{\tau} 
\Bigl[ 
\| \tilde{\rho}_{s} \|_{-(n+\alpha')}
+ \essup_{\omega \in \Omega} \Theta
\Bigr]
ds \biggr),
\end{split}
\end{equation*}
from which we deduce, by Gronwall's lemma, that 
\begin{equation*}
\begin{split}
\| \tilde{\rho}_{\tau}
\|_{-(n+\alpha')}
&\leq 
C' 
\Bigl( \| \tilde{\rho}_{0} \|_{-(n+\alpha')}
+ \sup_{t \in [0,T]}
 \essup_{\omega \in \Omega} \Theta
 \Bigr), 
\end{split}
\end{equation*}
and thus 
\begin{equation}
\label{eq:partie2:linearisation:linfty}
\essup_{\omega \in \Omega}
\sup_{t \in [0,T]}
\| \tilde{\rho}_{t}
\|_{-(n+\alpha')}
\leq 
C' 
\Bigl(
 \| \tilde{\rho}_{0} \|_{-(n+\alpha')}
+ \essup_{\omega \in \Omega} \Theta
 \Bigr).
\end{equation}
%

By 
\eqref{eq:partie2:linearisation:linfty:2}
and 
\eqref{eq:partie2:linearisation:linfty}, 
we easily get a bound for 
$\tilde{z}$. 
\vspace{5pt}

\textit{Last step.}
It then remains to prove 
\eqref{eq:cor:partie2:systeme:linearise:linfty:2}. 
By means of the first step, we have bounds for
$$\essup_{\omega \in \Omega}
\sup_{t \in [0,T]} 
\Bigl( 
\| \tilde{z}_{t}  \|_{n+1+\alpha}
+\| \tilde{z}_{t}'  \|_{n+1+\alpha}
+
\| \tilde{\rho}_{t} \|_{-(n+\alpha')} 
+
\| \tilde{\rho}_{t}' \|_{-(n+\alpha')} 
\Bigr).$$
Plugging the bound into 
the stability estimate in Proposition 
\ref{lem:partie2:stability:2}, we may proceed in the same way 
as in the two first steps in order to complete the proof. 
\end{proof}

\subsubsection{Proof of Theorem 
\ref{thm:partie2:existence:uniqueness:linearise}}

We now complete the proof of Theorem 
\ref{thm:partie2:existence:uniqueness:linearise}. 
It suffices to prove
\begin{Proposition}
\label{prop:partie2:argument:continuation}
There is an $\varepsilon >0$ 
such that
if, for  some $\vartheta \in [0,1)$ and 
$\beta \in (\alpha',\alpha)$,
for any initial condition 
$\tilde{\rho}_{0}$ in $({\mathcal C}^{n+\alpha'}(\T^d))'$
and any input $((\tilde{b}_{t}^0)_{t \in [0,T]},
(\tilde{f}_{t}^0)_{t \in [0,T]},
\tilde{g}_{T}^0)$ as in the introduction of Subsection 
\ref{subse:par tie:2:linearization},
the system 
\eqref{eq:partie2:systeme:linearise:continuation}--\eqref{eq:partie2:systeme:linearise:continuation:bdary}
has a unique solution
$(\tilde{\rho}_{t},\tilde{z}_{t})_{t \in [0,T]}$ 
with paths in ${\mathcal C}^0([0,T],(\cC^{n+\beta}(\T^d))')
\times 
{\mathcal C}^0([0,T],\cC^{n+1+\beta}(\T^d))$
such that $\essup_{\omega} \sup_{t \in [0,T]}
(\| \tilde{\rho}_{t} \|_{-(n+\beta)} + \| \tilde{z}_{t} \|_{n+1+\beta})
< \infty$, 
$(\tilde{\rho}_{t},\tilde{z}_{t})_{t \in [0,T]}$ 
also satisfying 
$\essup_{\omega} \sup_{t \in [0,T]}
(\| \tilde{\rho}_{t} \|_{-(n+\alpha')} + \| \tilde{z}_{t} \|_{n+1+\alpha})
< \infty$,
then 
unique solvability also holds with 
$\vartheta$ replaced by $\vartheta+\varepsilon$, 
for the same class of initial conditions and of inputs
and in the same space; moreover, solutions also
lie (almost surely) in a bounded subset of the space 
$L^\infty([0,T],(\cC^{(n+\alpha')}(\T^d))')
\times 
L^\infty([0,T],\cC^{n+1+\alpha}(\T^d))$. 
\end{Proposition}

\begin{proof}
Given $\vartheta \in [0,1)$ in the statement,
$\varepsilon >0$,
an initial condition
$\tilde{\rho}_{0} \in ({\mathcal C}^{n+\alpha'}(\T^d))'$, 
an input $((\tilde{b}_{t}^0)_{t \in [0,T]},
(\tilde{f}_{t}^0)_{t \in [0,T]},
\tilde{g}_{T}^0)$
satisfying the prescription 
described in the introduction of Subsection
\ref{subse:par tie:2:linearization}
and 
an adapted process $(\tilde{\rho}_t,\tilde{z}_{t})_{t \in [0,T]}$
($\tilde{\rho}$ having $\tilde{\rho}_{0}$ as initial condition)
with paths in ${\mathcal C}^0([0,T],({\mathcal C}^{n+\beta}(\T^d))')
\times {\mathcal C}^0([0,T],{\mathcal C}^{n+1+\beta}(\T^d))$
such that 
\begin{equation}
\label{eq:partie2:linearisation:linfty:2:c}
\essup_{\omega \in \Omega} \sup_{t \in [0,T]}
\bigl( \| \tilde{\rho}_{t} \|_{-(n+\alpha')}
+  \|\tilde{z}_{t} \|_{n+1+\alpha} \bigr) < \infty,
\end{equation}
we call $\Phi_{\varepsilon}(\tilde{\rho},\tilde{z})$ the pair 
$(\tilde{\rho}_t',\tilde{z}_{t}')_{0 \leq t \leq T}$
solving the system 
\eqref{eq:partie2:systeme:linearise:continuation} 
with respect to the initial condition
$\tilde{\rho}_{0}$
and to the input:
\begin{equation*}
\begin{split}
&\tilde b_{t}^{0\prime} = \varepsilon  \tilde{m}_{t} \Gamma_{t} D \tilde{z}_{t} + \tilde b_{t}^0,
\\
&\tilde f_{t}^{0\prime} =
-\varepsilon \frac{\delta \tilde F_{t}}{\delta m}(\cdot,m_{t})(\rho_{t})
+ \tilde{f}_{t}^0,
\\
& \tilde{g}_{T}^{0\prime}
=
\varepsilon
\frac{\delta \tilde G}{\delta m}(\cdot,m_{T})(\rho_{T}) + \tilde{g}_{T}^0.
\end{split}
\end{equation*}
By assumption, it satisfies
\begin{equation*}
\essup_{\omega \in \Omega} \sup_{t \in [0,T]}
\bigl( \| \tilde{\rho}_{t}' \|_{-(n+\alpha')}
+  \|\tilde{z}_{t}' \|_{n+1+\alpha} \bigr) < \infty,
\end{equation*}
By Corollary \ref{cor:partie2:systeme:linearise:linfty},
\begin{equation*}
\begin{split}
&
\essup_{\omega \in \Omega}
\sup_{t \in [0,T]}
\bigl(
 \| \tilde{z}_{t}'  \|_{n+1+\alpha}
+
\| \tilde{\rho}_{t}' \|_{-(n+\alpha')}
\bigr)
\\
&\leq C'
\biggl[ 
\| \tilde{\rho}_{0}\|_{-(n+\alpha')}
+ c \, \varepsilon  \,
\essup_{\omega \in \Omega}
\sup_{t \in [0,T]}
\bigl( 
 \| \tilde{\rho}_{t}  \|_{-(n+\alpha')}
+ 
\| \tilde{z}_{t}  \|_{n+1+\alpha}
\bigr)
\\
&\hspace{15pt}
+ \essup_{\omega \in \Omega}
\Bigl[
\sup_{t \in [0,T]}
\bigl(
\| \tilde b_{t}^0 \|_{-(n+\alpha'-1)}
+ 
\| \tilde f_{t}^{0}  \|_{n+\alpha}
\bigr)
+  
\| \tilde g_{T}^{0}  \|_{n+1+\alpha}
\Bigr]
\biggr],
\end{split}
\end{equation*}
where $c$ is a constant, 
which
 only depends on the constant $C$ 
 appearing in points 1--6 in introduction 
 of Subsection 
 \ref{subse:par tie:2:linearization}
 and on the bounds appearing in 
 {\bf (HF1(${\boldsymbol n}$))}
and 
{\bf (HG1(${\boldsymbol n}$+1))}. 

In particular, if 
\begin{align}
\label{eq:partie2:linearisation:linfty:2:b}
&\essup_{\omega \in \Omega}
\sup_{t \in [0,T]}
\bigl( 
\| \tilde{z}_{t}  \|_{n+1+\alpha}
+
\| \tilde{\rho}_{t} \|_{-(n+\alpha')}
\bigr)
\\
&\hspace{10pt} \leq 2C'
\biggl( 
\| \tilde{\rho}_{0}\|_{-(n+\alpha')}
+ \essup_{\omega \in \Omega}
\Bigl[
\| \tilde g_{T}^{0}  \|_{n+1+\alpha}
+
\sup_{t \in [0,T]}
\bigl(
\| \tilde b_{t}^0 \|_{-(n+\alpha'-1)}
+ 
\| \tilde f_{t}^{0}  \|_{n+\alpha}
\bigr)
\Bigr]
\biggr), \nonumber
\end{align}
and $2C' c \varepsilon \leq 1$, then
\begin{equation*}
\begin{split}
&\essup_{\omega \in \Omega}
\sup_{t \in [0,T]}
\bigl( 
\| \tilde{z}_{t}'  \|_{n+1+\alpha}
+
\| \tilde{\rho}_{t}' \|_{-(n+\alpha')}
\bigr)
\\
&\hspace{10pt} 
\leq 2C'
\biggl( 
\| \tilde{\rho}_{0}\|_{-(n+\alpha')}
+ \essup_{\omega \in \Omega}
\Bigl[
\| \tilde g_{T}^{0}  \|_{n+1+\alpha}
+
\sup_{t \in [0,T]}
\bigl(
\| \tilde b_{t}^0 \|_{-(n+\alpha'-1)}
+ 
\| \tilde f_{t}^{0}  \|_{n+\alpha}
\bigr)
\Bigr]
\biggr),
\end{split}
\end{equation*}
so that the set of pairs $(\tilde{\rho},\tilde{z})$ that satisfy 
\eqref{eq:partie2:linearisation:linfty:2:c}
and
\eqref{eq:partie2:linearisation:linfty:2:b}
is stable by $\Phi_{\varepsilon}$ for $\varepsilon$ small enough. 

Now, given two pairs 
$(\tilde{\rho}_{t}^1,\tilde{z}_{t}^1)_{t \in [0,T]}$
and 
$(\tilde{\rho}_{t}^2,\tilde{z}_{t}^2)_{t \in [0,T]}$
satisfying 
\eqref{eq:partie2:linearisation:linfty:2:b},
we let 
$(\tilde{\rho}_{t}^{1\prime},\tilde{z}_{t}^{1\prime})_{t \in [0,T]}$
and 
$(\tilde{\rho}_{t}^{2\prime},\tilde{z}_{t}^{2\prime})_{t \in [0,T]}$
be their respective images by $\Phi_{\varepsilon}$. 
We deduce from 
Proposition  \ref{lem:partie2:stability:2} that
\begin{equation*}
\begin{split}
&{\mathbb E}
\Bigl[ 
\sup_{t \in [0,T]} \| \tilde{z}_{t}^{1\prime} - \tilde{z}_{t}^{2\prime}\|_{n+1+\alpha}^2
+
\sup_{t \in [0,T]} 
\| \tilde{\rho}_{t}^{1\prime} - \tilde{\rho}_{t}^{2\prime} \|_{-(n+\alpha')}^2
\Bigr] 
\\
&\hspace{15pt} \leq C'
\varepsilon^2
{\mathbb E}
\Bigl[ 
\sup_{t \in [0,T]} \| \tilde{z}_{t}^1 - \tilde{z}_{t}^2 \|_{n+1+\alpha}^2
+
\sup_{t \in [0,T]} 
\| \tilde{\rho}_{t}^1 - \tilde{\rho}_{t}^2 \|_{-(n+\alpha')}^2
\Bigr],
\end{split}
\end{equation*}
for a possibly new value of the constant $C'$, but still independent of $\vartheta$ and $\varepsilon$. 
Therefore, for $C' \varepsilon^2 <1$ and 
$2C' c \varepsilon \leq 1$, $\Phi_{\varepsilon}$ 
is a contraction on the set of adapted processes $(\tilde{\rho}_{t},\tilde{z}_{t})_{t \in [0,T]}$ 
having paths 
in ${\mathcal C}^{0}([0,T],({\mathcal C}^{n+\beta}(\T^d))')
\times
{\mathcal C}^{0}([0,T],{\mathcal C}^{n+1+\beta}(\T^d))$
and satisfying 
\eqref{eq:partie2:linearisation:linfty:2:b}
(and thus \eqref{eq:partie2:linearisation:linfty:2:c} as well), which forms a closed subset of 
the Banach space
${\mathcal C}^0([0,T],({\mathcal C}^{n+\beta}(\T^d))')
\times 
{\mathcal C}^0([0,T],{\mathcal C}^{n+\beta}(\T^d))$. 
By Picard fixed point theorem, we 
deduce that $\Phi_{\varepsilon}$ has a unique fixed point 
satisfying \eqref{eq:partie2:linearisation:linfty:2:b}.
The fixed point 
solves \eqref{eq:partie2:systeme:linearise:continuation}--\eqref{eq:partie2:systeme:linearise:continuation:bdary},
with $\vartheta$ replaced by 
$\vartheta+\varepsilon$.

Consider now
another solution
to \eqref{eq:partie2:systeme:linearise:continuation}--\eqref{eq:partie2:systeme:linearise:continuation:bdary}
with $\vartheta$ replaced by 
$\vartheta+\varepsilon$,
with paths in a bounded subset of ${\mathcal C}^0([0,T],(\cC^{n+\beta}(\T^d))')
\times 
{\mathcal C}^0([0,T],\cC^{n+1+\beta}(\T^d))$.
By Proposition 
\ref{lem:partie2:stability:2}, it must coincide with the solution we just constructed.
\end{proof}

\newpage

\section{The second-order master equation}
\label{subse:partie:2:first:order:derivative}

Taking benefit of the analysis performed in the previous section on 
the unique solvability of the MFG system, we are now ready to define and investigate the 
solution of the master equation. The principle is the same as in the first-order case: 
the forward component of the MFG system has to be seen as 
the characteristics of the master equation. The regularity of 
the solution of the master equation is then investigated  
through the 
\textit{tangent process} that solves the linearized MFG system.

{\it As in the previous section, the level of common noise $\beta$ is set to $1$ throughout this section. This is without loss of generality and  this makes the notation a little bit simpler. }

\subsection{Construction of the Solution}
\label{subse:construction:solution:master}

\textbf{Assumption.} Throughout the paragraph, we assume that 
the assumption of Theorem \ref{thm:partie:2:existence:uniqueness}
is in force, with $\alpha \in (0,1)$. 
\vspace{5pt}

For any initial distribution $m_{0} \in {\mathcal P}(\T^d)$,
the system 
\eqref{eq:se:3:tilde:HJB:FP}
admits a unique solution so that, following 
the analysis performed in the deterministic setting, 
we may let 
\begin{equation*}
U(0,x,m_{0}) =  \tilde{u}_{0}(x), \quad x \in \T^d.
\end{equation*}
The initialization is here performed at time $0$, but, of course, there is no difficulty in replacing $0$ by any arbitrary time $t_{0} \in [0,T]$, in which case 
the system \eqref{eq:se:3:tilde:HJB:FP}
rewrites 
\begin{equation}
\label{eq:se:3:tilde:HJB:FP:t0}
\begin{split}
&d_{t} \tilde{m}_{t} = \bigl\{ \Delta \tilde{m}_{t}
+{\rm div} \bigl( \tilde{m}_{t} D_{p} \tilde{H}_{t_{0},t}(\cdot,D \tilde{u}_{t})
\bigr) \bigr\} dt,  
\\
&d_{t} \tilde{u}_{t}
=  \bigl\{ -   \Delta \tilde{u}_{t} + \tilde{H}_{t_{0},t}(\cdot,D\tilde{u}_{t}) - 
\tilde{F}_{t_{0},t}(\cdot,m_{t_{0},t}) 
\bigr\} dt
+ d\tilde{M}_{t},
\end{split}
\end{equation}
with 
the initial condition $\tilde{m}_{t_{0}} = m_{0}$
and
the terminal boundary condition $\tilde{u}_{T}
= \tilde{G}_{t_{0}}(\cdot,m_{t_{0},T})$, 
under the prescription that 
\begin{equation}
\label{eq:se:3:tilde:HJB:FP:prescription}
\begin{split}
&m_{t_{0},t} = \bigl( id + \sqrt{2} (W_{t}- W_{t_{0}}) \bigr) \sharp \tilde{m}_{t},
\\
&\tilde{F}_{t_{0},t}(x,\mu) = F \bigl( x + \sqrt{2} (W_{t}-W_{t_{0}}),\mu \bigr),
\\
&\tilde{G}_{t_{0}}(x,\mu) = G \bigl( x + \sqrt{2} (W_{T}-W_{t_{0}}),\mu \bigr),
\\
&\tilde{H}_{t_{0},t}(x,p) =  H \bigl( x + \sqrt{2} (W_{t}-W_{t_{0}}),p
\bigr), \quad x \in \T^d, \ p \in \R^d, \ \mu \in {\mathcal P}(\T^d). 
\end{split}
\end{equation}

It is then possible to let
\begin{equation*}
U(t_{0},x,m_{0}) = \tilde{u}_{t_{0}}(x), \quad x \in \T^d. 
\end{equation*}

We shall often use the following important fact:
\begin{Lemma}
\label{lem:partie2:master:equation}
Given an initial condition $(t_{0},m_{0}) \in [0,T] \times {\mathcal P}(\T^d)$, denote by $(\tilde{m}_{t},\tilde{u}_{t})_{t \in [t_{0},T]}$ the solution of 
\eqref{eq:se:3:tilde:HJB:FP:t0}
with the prescription  
\eqref{eq:se:3:tilde:HJB:FP:prescription}
and with
 $\tilde{m}_{t_{0}} = m_{0}$
as initial condition. Call $m_{t_{0},t}$ the image of 
$\tilde{m}_{t}$ by the random mapping 
$\T^d \ni x \mapsto x + \sqrt{2}(W_{t}-W_{t_{0}})$
that is ${m}_{t_{0},t} = [id + 
\sqrt{2}(W_{t}-W_{t_{0}})] \sharp 
\tilde{m}_{t}$. Then,
for any $t_{0}+h \in [t_{0},T]$,
$\P$ almost surely,
\begin{equation*}
\tilde{u}_{t_{0}+h}(x) =
U\bigl(t_{0}+h,x + \sqrt{2}(W_{t_{0}+h}-W_{t_{0}}),m_{t_{0},t_{0}+h}\bigr), \quad x \in \T^d. 
\end{equation*}
\end{Lemma} 

\begin{proof}
Given $t_{0}$ and $h$ as above, we let
\begin{equation*}
\bar{m}_{t} = \bigl[ 
id + \sqrt{2} \bigl( W_{t_{0}+h} -W_{t_{0}}
\bigr) \bigr] \sharp \tilde{m}_{t},
\
\bar{u}_{t}(x) = \tilde{u}_{t} \bigl[ x - \sqrt{2} 
\bigl( W_{t_{0}+h} - W_{t_{0}} \bigr) \bigr], 
\ t \in [t_{0}+h,T], \ x \in \T^d. 
\end{equation*}
We claim that 
$(\bar{m}_{t},\bar{u}_{t})_{t \in [t_{0}+h,T]}$
is a solution 
of \eqref{eq:se:3:tilde:HJB:FP:t0}--\eqref{eq:se:3:tilde:HJB:FP:prescription}, with $t_{0}$
replaced by $t_{0}+h$ and with $m_{t_{0},t_{0}+h}$
as initial condition. 

The proof is as follows. 
We start with a preliminary remark. For $t \in [t_{0}+h,T]$,
\begin{equation}
\label{eq:partie2:change:variable:t0}
\bigl[ 
id + \sqrt{2} \bigl( W_{t} -W_{t_{0+h}}
\bigr) \bigr]
\sharp 
\bar{m}_{t} = \bigl[ 
id + \sqrt{2} \bigl( W_{t} -W_{t_{0}}
\bigr) \bigr] \sharp \tilde{m}_{t} = m_{t_{0},t}.
\end{equation}

We now prove that 
the pair $(\bar{m}_{t},\bar{u}_{t})_{t_{0}+h \leq t \leq T}$ solves the forward equation in \eqref{eq:se:3:tilde:HJB:FP:t0}. To this end, 
denote by $(X_{t_{0},t})_{t \in [t_{0},T]}$ the solution 
of the SDE 
\begin{equation*}
dX_{t_{0},t} = - D_{p} \tilde{H}_{t_{0},t}\bigl(X_{t_{0},t},D \tilde{u}_{t}(X_{t_{0},t})
\bigr) dt + \sqrt{2} dB_{t}, \quad 
t \in [t_{0},T],
\end{equation*}
the initial condition $X_{t_{0},t_{0}}$ having $m_{0}$ as distribution.
(Notice that the equation is well-posed as $D\tilde{u}$ is known to be Lipschitz in space.) 
Then, the process $(\tilde{X}_{t} = X_{t_{0},t} + \sqrt{2}(W_{t_{0}+h}- W_{t_{0}}))_{t \in [t_{0}+h,T]}$
has $(\bar{m}_{t} = (id + \sqrt{2}(W_{t_{0}+h}-W_{t_{0}}))\sharp \tilde{m}_{t})_{t \in [t_{0}+h,T]}$ as marginal conditional distributions (given $(W_{t})_{t \in [0,T]}$). 
The process satisfies the SDE
\begin{equation*}
\begin{split}
d \tilde{X}_{t} 
&= - D_{p} \tilde{H}_{t_{0},t}
\Bigl(
\tilde{X}_{t} - \sqrt{2}(W_{t_{0}+h}- W_{t_{0}}), 
D \tilde{u}_{t}
\bigl( \tilde{X}_{t} -
\sqrt{2}(W_{t_{0}+h}- W_{t_{0}}) \bigr)
\Bigr) dt 
 + \sqrt{2} dB_{t} 
\\
&=- D_{p} \tilde{H}_{t_{0}+h,t}
\Bigl(
\tilde{X}_{t}, 
D \bar{u}_{t}
\bigl( \tilde{X}_{t} \bigr)
\Bigr) dt  + \sqrt{2} dB_{t}, 
\end{split}
\end{equation*}
which is enough to check that the forward equation holds true, 
with $\bar{m}_{t_{0}+h} = m_{t_{0},t_{0}+h}$
as initial condition, see 
\eqref{eq:partie2:change:variable:t0}. 

We now have
\begin{equation*}
\begin{split}
d_{t} \bar{u}_{t}
&=  \bigl[ -   \Delta \bar{u}_{t} + \bigl\{  \tilde{H}_{t_{0},t}(\cdot,D\tilde{u}_{t}) - 
\tilde{F}_{t_{0},t}(\cdot,m_{t_{0},t})
\bigr\}\bigl(\cdot -  \sqrt{2}( W_{t_{0}+h} - W_{t_{0}}) \bigr) \bigr] dt
 \\
&\hspace{15pt} + d \tilde{M}_{t}\bigl(\cdot -  \sqrt{2}( W_{t_{0}+h} - W_{t_{0}}) \bigr)
\\
&= \bigl[ -   \Delta \bar{u}_{t} + \bigl\{  \tilde{H}_{t_{0}+h,t}(\cdot,D\bar{u}_{t}) - 
\tilde{F}_{t_{0}+h,t}(\cdot,m_{t_{0},t})
\bigr\} \bigr] dt  + d \tilde{M}_{t}\bigl(\cdot -  \sqrt{2}( W_{t_{0}+h} - W_{t_{0}}) \bigr).
\end{split} 
\end{equation*}
Now, \eqref{eq:partie2:change:variable:t0}
says that $m_{t_{0},t}$ reads 
$[ 
id + \sqrt{2} \bigl( W_{t} -W_{t_{0}+h}
\bigr)]
\sharp 
\bar{m}_{t}$, where $(\bar{m}_{t})_{t_{0}+h \leq t \leq T}$ is the 
current
forward component. This matches exactly the prescription on the backward equation 
in \eqref{eq:se:3:tilde:HJB:FP:t0}
and
\eqref{eq:se:3:tilde:HJB:FP:prescription}.

If $m_{t_{0},t_{0}+h}$ was deterministic, we would have, by  
definition of $U$, $U(t_{0}+h,x,m_{t_{0},t_{0}+h})=
\bar{u}_{t_{0}+h}(x)$, $x \in \T^d$, and thus, by definition 
of $\bar{u}_{t_{0}+h}$, 
\begin{equation}
\label{eq:Partie2:U:t0:t0+h} 
\tilde{u}_{t_{0}+h}(x) = U
\bigl(t_{0}+h,x+ \sqrt{2}(W_{t_{0}+h} - W_{t_{0}})
,m_{t_{0},t_{0}+h}\bigr), \quad x \in \T^d. 
\end{equation}
Although the result is indeed correct, the argument is false as 
$m_{t_{0},t_{0}+h}$ is random.

To 
prove
\eqref{eq:Partie2:U:t0:t0+h}, we proceed as follows.
By compactness of
${\mathcal P}(\T^d)$, we can find, for any 
$\varepsilon$, a family 
of $N$ disjoint Borel subsets 
$A^{1},\dots,A^{N} \subset {\mathcal P}(\T^d)$,
each of them being of diameter less than $\varepsilon$,
that covers 
${\mathcal P}(\T^d)$.

For each 
$i \in \{1,\dots,N\}$, we may find $\mu_{i} \in
A^{i}$. We then denote by 
$(\hat{m}_{t}^i,\hat{u}_{t}^i)_{t \in [t_{0}+h,T]}$
the solution of \eqref{eq:se:3:tilde:HJB:FP:t0}--\eqref{eq:se:3:tilde:HJB:FP:prescription}, with $t_{0}$
replaced by $t_{0}+h$ and with $\mu_{i}$
as initial condition. We let
\begin{equation*}
\begin{split}
&\hat{m}_{t} 
= \sum_{i=1}^N \hat{m}_{t}^i {\mathbf 1}_{A^{i}}
\bigl( m_{t_{0},t_{0}+h} \bigr) ,
\\
&\hat{u}_{t} 
= \sum_{i=1}^N \hat{u}_{t}^i {\mathbf 1}_{A^{i}}
\bigl( m_{t_{0},t_{0}+h} \bigr).
\end{split}
\end{equation*}

Since the events $\{m_{t_{0},t_{0}+h} \in A^i\}$, for
each $i =1,\dots,N$, are independent of 
the Brownian motion $(W_{t}-W_{t_{0}+h})_{t \in [t_{0}+h,T]}$, the process $(\hat{m}_{t},\hat{u}_{t})_{t \in [t_{0}+h,T]}$
is a solution 
of \eqref{eq:se:3:tilde:HJB:FP:t0}--\eqref{eq:se:3:tilde:HJB:FP:prescription}, with $t_{0}$
replaced by $t_{0}+h$ and with $\hat{m}_{t_{0},t_{0}+h}$
as initial condition. 
With an obvious generalization of Theorem 
\ref{thm:partie:2:existence:uniqueness}
to cases when the initial conditions are random, we deduce that 
\begin{equation*}
{\mathbb E} \bigl[ \| \bar{u}_{t_{0+h}} - \hat{u}_{t_{0}+h}
\|_{n+\alpha}^2 \bigr] \leq C 
{\mathbb E} \bigl[ {\mathbf d}_{1}^2(\bar{m}_{t_{0}+h},
\hat{m}_{t_{0}+h})
\bigr] =  
C 
\sum_{i=1}^N {\mathbb E} \bigl[
{\mathbf 1}_{A^{i}}
\bigl( m_{t_{0},t_{0}+h} \bigr)
 {\mathbf d}_{1}^2(m_{t_{0},t_{0}+h},
\mu^i)
\bigr].
\end{equation*}
Obviously, the right-hand side is less than $C \varepsilon^2$. 
The
trick is then to say that 
$\hat{u}_{t_{0}+h}^i$ reads $U(t_{0}+h,\cdot,\mu_{i})$. 
Therefore,
\begin{equation*}
\sum_{i=1}^N 
{\mathbb E} \bigl[
{\mathbf 1}_{A^{i}}
\bigl( m_{t_{0},t_{0}+h} \bigr)
 \| \bar{u}_{t_{0+h}} - 
 U(t_{0}+h,\cdot,\mu^i)
\|_{n+\alpha}^2 \bigr] \leq C \varepsilon^2. 
\end{equation*}
Using the Lipschitz property of $U(t_{0}+h,\cdot,\cdot)$
in the measure argument (see Theorem 
\ref{thm:partie:2:existence:uniqueness}), we deduce that 
\begin{equation*}
{\mathbb E} \Bigl[
 \bigl\| \bar{u}_{t_{0+h}} - 
 U\bigl(t_{0}+h,\cdot,
 m_{t_{0},t_{0}+h}
 \bigr)
\bigr\|_{n+\alpha}^2 \Bigr] \leq C \varepsilon^2. 
\end{equation*}
Letting $\varepsilon$ tend to $0$, we complete the proof.
\end{proof}

\begin{Corollary}
\label{cor:partie:2:reg:temps}
For any $\alpha' \in (0,\alpha)$,
we can find a constant $C$ such that, 
for any $t_{0} \in [0,T]$, $h \in [0,T-t_{0}]$, and $m_{0} \in {\mathcal P}(\T^d)$,
\begin{equation*}
\bigl\|
U(t_{0}+h,\cdot,m_{0}) 
- 
U(t_{0},\cdot,m_{0}) 
\bigr\|_{n+\alpha'} \leq C h^{(\alpha-\alpha')/2}.
\end{equation*}
\end{Corollary}
\begin{proof}
Using the backward equation in 
\eqref{eq:se:3:tilde:HJB:FP:t0}, we have that 
\begin{equation*}
\tilde{u}_{t_{0}}(\cdot) =
{\mathbb E}
\biggl[ P_{h} \tilde{u}_{t_{0}+h}(\cdot)
- \int_{t_{0}}^{t_{0}+h}
P_{s-t_{0}} \bigl( \tilde{H}_{t_{0},t}(\cdot,D 
\tilde{u}_{s}) - \tilde{F}_{t_{0},t}(\cdot,m_{t_{0},s})
\bigr) ds \biggr].  
\end{equation*}
Therefore,
\begin{equation*}
\tilde{u}_{t_{0}}(\cdot) - 
{\mathbb E} \bigl( 
\tilde{u}_{t_{0}+h}(\cdot) \bigr) =
{\mathbb E}
\biggl[ \bigl( P_{h} - \textit{id} \bigr) \tilde{u}_{t_{0}+h}(\cdot)
- \int_{t_{0}}^{t_{0}+h}
P_{s-t_{0}} \bigl( \tilde{H}_{t_{0},t}(\cdot,D 
\tilde{u}_{s}) - \tilde{F}_{t_{0},t}(\cdot,m_{t_{0},s})
\bigr) ds \biggr].  
\end{equation*}
So that 
\begin{equation*}
\begin{split}
\| \tilde{u}_{t_{0}} - {\mathbb E} \bigl(
 \tilde{u}_{t_{0}+h} \bigr)
\|_{n+\alpha'} &\leq  {\mathbb E} \Bigl[ 
\bigl\|
\bigl( P_{h}
- \textit{id} \bigr) \tilde{u}_{t_{0}+h}  \bigr\|_{n+\alpha'}
\Bigr]
\\
&\hspace{15pt} + C
\int_{t_{0}}^{t_{0}+h}
(s-t_{0})^{-1/2}
\bigl\| \tilde{H}_{t_{0},t}(\cdot,D 
\tilde{u}_{s}) - \tilde{F}_{t_{0},t}(\cdot,m_{t_{0},s})
\bigr\|_{n+\alpha'-1} ds.  
\end{split}
\end{equation*}
It is well checked that 
\begin{equation*}
\begin{split}
\E \Bigl[
\bigl\| \bigl(
P_{h} - \textit{id} \bigr) \tilde{u}_{t_{0}+h}
\bigr\|_{n+\alpha'}\Bigl]
&\leq 
C 
h^{(\alpha-\alpha')/2}
\E \Bigl[
\bigl\| 
 \tilde{u}_{t_{0}+h}
\bigr\|_{n+\alpha} \Bigr]
\\
&\leq 
C h^{(\alpha-\alpha')/2},
\end{split}
\end{equation*}
the last line following from Lemma 
\ref{lem:super:reg}. 

Now, by Lemma \ref{lem:partie2:master:equation},
\begin{equation*}
\begin{split}
{\mathbb E} \bigl[ \tilde{u}_{t_{0}+h}
\bigr] 
&= 
{\mathbb E} \bigl[ U\bigl(t_{0}+h,\cdot + \sqrt{2} ( 
W_{t_{0}+h}-W_{t_{0}}),m_{t_{0},t_{0+h}}\bigr)
\bigr] 
\\
&=
{\mathbb E} \bigl[ U\bigl(t_{0}+h,\cdot + \sqrt{2} ( 
W_{t_{0}+h}-W_{t_{0}}),m_{t_{0},t_{0+h}}\bigr)
- 
 U\bigl(t_{0}+h,\cdot ,m_{0}\bigr)
\bigr]  +  U\bigl(t_{0}+h,\cdot ,m_{0}\bigr),
\end{split}
\end{equation*}
where, by Theorem 
\ref{thm:partie:2:existence:uniqueness}, it holds that
\begin{equation*}
\begin{split}
&\Bigl\| {\mathbb E} \bigl[ U\bigl(t_{0}+h,\cdot + \sqrt{2} ( 
W_{t_{0}+h}-W_{t_{0}}),m_{t_{0},t_{0+h}}\bigr)
- 
 U\bigl(t_{0}+h,\cdot ,m_{0}\bigr)
\bigr] \Bigr\|_{n+\alpha'}
\\
&\hspace{15pt} \leq C {\mathbb E} 
\bigl[ \vert  \dk (m_{t_{0},t_{0}+h},m_{0})
\vert 
\bigr] 
\\
&\hspace{30pt} + 
 {\mathbb E} \Bigl[
 \bigl\| U\bigl(t_{0}+h,\cdot
 + \sqrt{2} ( 
W_{t_{0}+h}-W_{t_{0}}),m_{0}\bigr)
- 
 U\bigl(t_{0}+h,\cdot ,m_{0}\bigr)
 \bigr\|_{n+\alpha'} \Bigr],
\end{split}
\end{equation*}
which is less than $C h^{(\alpha-\alpha')/2}$. 

\end{proof}

\subsection{First-order Differentiability}
\label{subsubse:partie:2:1st:order}

\textbf{Assumption.} Throughout the paragraph, we assume that
$F$, $G$ and $H$ satisfy 
\eqref{HypD2H}
and 
\eqref{e.monotoneF}
in Subsection \ref{subsec:hyp}
and 
that, for some integer $n \geq 2$
and some $\alpha \in (0,1)$, 
{\bf (HF1(${\boldsymbol n}$))}
and 
{\bf (HG1(${\boldsymbol n}$+1))}
hold true.
\vspace{5pt}

The purpose is here to follow Subsection
\ref{subse:partie:1:differentiability}
in order to establish the differentiability of 
$U$ with respect to the argument 
$m_{0}$. 
The analysis is performed at $t_{0}$ fixed, so that, 
without any loss of generality, $t_{0}$ can be 
chosen as $t_{0}=0$.

The initial distribution $m_{0} \in {\mathcal P}(\T^d)$
being given, we call $(\tilde{m},\tilde{u})$
the solution of the system 
\eqref{eq:se:3:tilde:HJB:FP}
with $m_{0}$ as initial distribution. 
Following 
\eqref{eq:vmubis}, the strategy is to investigate the linearized 
system (of the same type as \eqref{eq:partie2:systeme:linearise}):
\begin{equation}
\label{eq:partie2:systeme:derivative}
\begin{split}
&d_{t} \tilde{z}_{t} = \bigl\{ -  \Delta \tilde{z}_{t} + \langle 
D_{p} \tilde{H}_{t}(\cdot,D \tilde{u}_{t}),D \tilde{z}_{t}
\rangle - 
\frac{\delta \tilde F_{t}}{\delta m}(\cdot,m_{t})(\rho_{t})
  \bigr\} dt + d \tilde{M}_{t},
\\
&\partial_{t} \tilde{\rho}_{t} -  \Delta \tilde{\rho}_{t}
- \textrm{div}\bigl(\tilde{\rho}_{t} D_{p}
\tilde{H}_{t}(\cdot,D \tilde{u}_{t}) 
\bigr) - \textrm{div} \bigl( \tilde{m}_{t} 
D^2_{pp} \tilde{H}_{t}(\cdot,D \tilde{u}_{t})
 D \tilde{z}_{t}
\bigr) = 0,
\end{split}
\end{equation}
with a boundary condition of the form 
\begin{equation*}
\tilde{z}_{T} = 
\frac{\delta \tilde G}{\delta m}(\cdot,m_{T})(\rho_{T}).
\end{equation*}
As explained later on, the initial condition of the forward equation will be chosen in an appropriate way. 
In that framework, we shall repeatedly apply the results from 
Subsection \ref{subse:par tie:2:linearization} with
\begin{equation}
\label{eq:partie:2:V:Gamma}
\tilde{V}_{t}(\cdot) = D_{p} \tilde{H}_{t}(\cdot,D \tilde{u}_{t}),
\quad \Gamma_{t} = D^2_{pp} \tilde{H}_{t}(\cdot,D \tilde{u}_{t}),
\quad t \in [0,T],
\end{equation}
which motivates the following lemma:
\begin{Lemma}
\label{lem:partie2:V:Gamma}
There exists a constant $C$ such that, for any initial condition 
$m_{0} \in {\mathcal P}(\T^d)$, 
the processes $(\tilde{V}_{t})_{t \in [0,T]}$
and $(\Gamma_t)_{t \in [0,T]}$ in 
\eqref{eq:partie:2:V:Gamma}
satisfy points 2 and 4 in the introduction 
of Subsection \ref{subse:par tie:2:linearization}. 
\end{Lemma}
\begin{proof}
By Theorem \ref{thm:partie:2:existence:uniqueness}
and Lemma \ref{lem:super:reg}, 
we can find a constant $C$ such that
any solution $(\tilde{m}_{t},\tilde{u}_{t})_{t \in [0,T]}$
to \eqref{eq:se:3:tilde:HJB:FP} satisfies, 
independently of the initial condition 
$m_{0}$,
\begin{equation*}
\essup_{\omega \in \Omega} \sup_{t \in [0,T]} \| \tilde{u}_{t} \|_{n+1+\alpha} \leq C. 
\end{equation*}
In particular, allowing the constant $C$ to increase from line to line, it must hold that 
\begin{equation*}
\essup_{\omega \in \Omega} \sup_{t \in [0,T]} 
\bigl\| D_{p} \tilde{H}_{t} (\cdot, D \tilde{u}_{t} \bigr) \bigr\|_{n+\alpha} \leq C. 
\end{equation*}
Moreover, implementing the local 
coercivity condition 
\eqref{HypD2H}, we deduce that (assuming $C \geq 1$), with probability 
$1$, for all $t \in [0,T]$,
\begin{equation*}
\| \Gamma_{t} \|_{1} \leq C \ ; \quad \forall x \in \T^d, \quad 
C^{-1} I_{d} \leq \Gamma_{t}(x) \leq C I_{d},
\end{equation*}
which completes the proof.
\end{proof}

Given $y \in \T^d$
and a $d$-tuple $\ell \in \{0,\dots,n\}^d$
such that $\vert \ell \vert = \sum_{i=1}^n \ell_{i}
\leq n$, we call 
$\T^d \ni x \mapsto v^{(\ell)}(x,m_{0},y) \in \R$
the value at time $0$ of the backward component of the solution to \eqref{eq:partie2:systeme:derivative}
when the forward component is initialized with the distribution 
$(-1)^{\vert \ell \vert} D^{\ell} \delta_{y}$.
Clearly, $D^\ell \delta_{y} \in ({\mathcal C}^{n+\alpha'}(\T^d))'$
for any $\alpha' \in (0,1)$, so that,  
by Theorem \ref{thm:partie2:existence:uniqueness:linearise}, 
$v^{(\ell)}(\cdot,m_{0},y)$ belongs to ${\mathcal C}^{n+\alpha}(\T^d)$. 
(Recall that, for a
test function $\varphi \in {\mathcal C}^{n}(\T^d)$,  
$(D^{\ell} \delta_{y}) \varphi = (-1)^{\vert \ell \vert} D^{\ell}_{y_{1}^{\ell_{1}} \dots y_{d}^{\ell_{d}}} \varphi(y)$.) 
%
Similarly, 
we may denote by 
$(\tilde{\rho}_{t}^{\ell,y},\tilde{z}_{t}^{\ell,y})_{t \in [0,T]} 
$
the solution of 
\eqref{eq:partie2:systeme:derivative} with $\tilde{\rho}_{0}^{\ell,y}
= (-1)^{\vert \ell \vert}D^{\ell} \delta_{y}$
as initial condition. For simplicity, we omit $m_{0}$ in the notation. 
We then have
\begin{equation}
\label{eq:representation:z0:ell:y}
\tilde{z}_{0}^{\ell,y} = v^{(\ell)}(\cdot,m_{0},y).
\end{equation}

We then claim

\begin{Lemma}
\label{lem:partie2:kernel:derivative}
Let $m_{0} \in {\mathcal P}(\T^d)$. Then, with the same notation 
as above, we have, for any 
$\alpha' \in (0,\alpha)$
and 
any $d$-tuple
$\ell \in \{0,\dots,n\}^d$
such that $\vert \ell \vert \leq n$,
\begin{equation}
\label{eq:partie2:kernel:derivative:1}
\begin{split}
&\lim_{\T^d \ni h \rightarrow 0}
\essup_{\omega \in \Omega}
\sup_{t \in [0,T]}
\Bigl( 
\bigl\| \tilde{\rho}_{t}^{\ell,y+h}
- \tilde{\rho}_{t}^{\ell,y} \bigr\|_{-(n+\alpha')}
+
\bigl\| \tilde{z}_{t}^{\ell,y+h}
- \tilde{z}_{t}^{\ell,y} \bigr\|_{n+1+\alpha}
\Bigr) = 0.
\end{split}
\end{equation}
Moreover, for any $\ell \in \{0,\dots,n-1\}^d$
with $\vert \ell \vert \leq n-1$
and any $i \in \{1,\dots,d\}$,
\begin{equation*}
\begin{array}{l}
\lim_{\R \backslash \{0\}\ni h \rightarrow 0}
\essup_{\omega \in \Omega}
\sup_{t \in [0,T]}
\Bigl( 
\bigl\|
\frac{1}{h}
\bigl(
 \tilde{\rho}_{t}^{\ell,y+h e_{i}}
- \tilde{\rho}_{t}^{\ell,y}
\bigr)
- 
\tilde{\rho}_{t}^{\ell+e_{i},y}
 \bigr\|_{-(n+\alpha')} 
 \\
\hspace{150pt}+
\bigl\|
\frac{1}{h}
\bigl( \tilde{z}_{t}^{\ell,y+h e_{i}}
- \tilde{z}_{t}^{\ell,y}
\bigr)
- \tilde{z}_{t}^{\ell+e_{i},y} 
 \bigr\|_{n+1+\alpha}
\Bigr) = 0, 
\end{array}
\end{equation*}
where $e_{i}$ denotes the $i^{\textrm{th}}$ vector of the canonical basis
and $\ell+e_{i}$ is understood as $(\ell+e_{i})_{j} = \ell_{j}
+\delta_{i}^j$, for $j \in \{1,\dots,d\}$, 
$\delta_{i}^j$
denoting the Kronecker symbol.

In particular, 
the function 
$[\T^d]^2 \ni (x,y) \mapsto v^{(0)}(x,m_{0},y)$
is $n$-times differentiable with respect to 
$y$
and,
for any $\ell \in \{0,\dots,n\}^d$ with $| \ell | \leq n$, 
the derivative 
$D^{\ell}_{y} v^{(0)}(\cdot,m_{0},y) 
: \T^d \ni x \mapsto 
D^{\ell}_{y} v^{(0)}(x,m_{0},y)$
belongs to ${\mathcal C}^{n+1+\alpha}(\T^d)$
and writes
\begin{equation*}
D^{\ell}_{y} v^{(0)}(x,m_{0},y) = v^{(\ell)}(x,m_{0},y), 
\quad (x,y) \in \T^d. 
\end{equation*}
Moreover, 
\begin{equation*}
\sup_{m_{0} \in {\mathcal P}(\T^d)}
\sup_{y \in \T^d} \| D^{\ell}_{y} v^{(0)}(\cdot,m_{0},y) \|_{n+1+\alpha} < \infty. 
\end{equation*}
\end{Lemma}
\begin{proof}
By
Corollary 
\ref{cor:partie2:systeme:linearise:linfty} (with $\alpha=\alpha$
and $\alpha'=\alpha'$ for some $\alpha' \in (0,\alpha)$), 
we can find a constant $C$ such that, 
for all $y \in \T^d$,
for all $m_{0} \in {\mathcal P}(\T^d)$
and all $\ell \in \{0,\dots,n\}^d$ with $\vert \ell \vert \leq n$, 
\begin{equation*}
\essup_{\omega \in \Omega}
\sup_{t \in [0,T]}
\bigl( 
\| \tilde{z}_{t}^{\ell,y} \|_{n+1+\alpha}
+
\| \tilde{\rho}_{t}^{\ell,y} \|_{-(n+\alpha')}
\bigr) \leq C. 
\end{equation*}
In particular, 
\begin{equation*}
\| v^{(\ell)}(\cdot,m_{0},y) \|_{n+1+\alpha} \leq C. 
\end{equation*}
Now, we make use of Proposition 
\ref{lem:partie2:stability:2}. 
We know that, for any $\alpha' \in (0,1)$, 
\begin{equation*}
\lim_{h \rightarrow 0}
\bigl\|
D^{\ell} \delta_{y+h} - 
D^{\ell} \delta_{y}
\bigr\|_{-(n+\alpha')} = 0.
\end{equation*}
Therefore, for $\alpha' < \alpha$,
Corollary 
\ref{cor:partie2:systeme:linearise:linfty}
(with $\alpha'=\alpha' < \alpha$
and $\alpha=\alpha$)
gives 
\eqref{eq:partie2:kernel:derivative:1}. 
This
yields
\begin{equation*}
\lim_{h \rightarrow 0}
\bigl\|
v^{(\ell)}
(\cdot,m_{0},y+h) 
 - 
v^{(\ell)}
(\cdot,m_{0},y) 
\bigr\|_{n+1+\alpha} = 0,
\end{equation*}
proving that the mapping 
$\T^d \ni y \mapsto v^{(\ell)}(\cdot,m_{0},y) \in 
{\mathcal C}^{n+1+\alpha}(\T^d)$ is continuous. 

Similarly, for
$\vert \ell \vert \leq n-1$
and $i \in \{1,\dots,d\}$,
\begin{equation*}
\lim_{\R \setminus  \{0\} \ni h \rightarrow 0}
\bigl\|
\frac{1}{h}
\bigl(
D^{\ell} \delta_{y+h e_{i}} - 
D^{\ell} \delta_{y}
\bigr) + D^{\ell+e_{i}} \delta_{y}
\bigr\|_{-(n+\alpha')} = 0,
\end{equation*}
or equivalently,
\begin{equation*}
\lim_{\R \setminus  \{0\} \ni h \rightarrow 0}
\bigl\|
\frac{1}{h}
\bigl(
(-1)^{\vert \ell \vert}
D^{\ell} \delta_{y+h e_{i}} - 
(-1)^{\vert \ell \vert}
D^{\ell} \delta_{y}
\bigr)  - 
(-1)^{\vert \ell + e_{i} \vert}
 D^{\ell+e_{i}} \delta_{y}
\bigr\|_{-(n+\alpha')} = 0,
\end{equation*}
As a byproduct, we get
\begin{equation*}
\lim_{\R \setminus \{0\} \ni h \rightarrow 0}
\Bigl\|
\frac{1}{h}
\bigl[
v^{(\ell)}
(\cdot,m_{0},y+h e_{i}) 
 - 
v^{(\ell)}
(\cdot,m_{0},y) 
\bigr] 
-
v^{(\ell+e_{i})}
(\cdot,m_{0},y) 
\Bigr\|_{n+1+\alpha} = 0,
\end{equation*}
which proves, by induction, that 
\begin{equation*}
D^{\ell}_{y} v^{(0)}(x,m_{0},y) = v^{(\ell)}(x,m_{0},y), \quad 
x,y \in \T^d.
\end{equation*}
This completes the proof. 
\end{proof}

Now, we prove
\begin{Lemma}
\label{lem:partie2:representation:derivative}
Given a finite signed measure $\mu$ on $\T^d$, 
the solution $\tilde{z}$ to 
\eqref{eq:partie2:systeme:derivative}
with $\mu$ as initial condition reads, when taken 
at time $0$,
\begin{equation*}
\tilde{z}_{0} : \R^d 
\ni x  \mapsto 
\tilde{z}_{0}(x) = \int_{\T^d} v^{(0)}(x,m_{0},y) d\mu(y). 
\end{equation*}
\end{Lemma}

\begin{proof}
By compactness of the torus,
we can find, for a given $\varepsilon >0$, a covering 
$(U_{i})_{1 \leq i \leq N}$ of $\T^d$, made of disjoint Borel 
subsets, such that each $U_{i}$, $i=1,\dots,N$, 
has a diameter less than $\varepsilon$.
Choosing, for each $i \in \{1,\dots,N\}$,
$y_{i} \in U_{i}$,
we then let
\begin{equation*}
\mu^{\varepsilon} = \sum_{i=1}^N \mu \bigl( U_{i} )
\delta_{y_{i}}. 
\end{equation*} 
Then, for any $\varphi \in {\mathcal C}^1(\T^d)$, 
with $\| \varphi\|_{1} \leq 1$,
we have
\begin{equation*}
\begin{split}
\biggl\vert \int_{\T^d} \varphi (y) 
d \bigl( \mu - \mu^{\varepsilon} 
\bigr)(y) \biggr\vert &=
\biggl\vert \sum_{i=1}^N 
\int_{U_{i}} \bigl( \varphi(y) 
- \varphi(y_{i}) \bigr) d \mu(y) 
\biggr\vert 
\leq C \| \mu \| \varepsilon,
\end{split}
\end{equation*}
where we have denoted by $\| \mu\|$ the total mass of $\mu$. 

Therefore, by Proposition 
\ref{lem:partie2:stability:2},
\begin{equation*}
\biggl\| 
\tilde{z}_{0} - \sum_{i=1}^N 
\int_{U_{i}}
v^{(0)}(\cdot,m_{0},y_{i}) d\mu(y)
\biggr\|_{n+1+\alpha} \leq C \| \mu \| \varepsilon,
\end{equation*}
where we have used the fact that, by linearity, the 
value at time $0$ of the backward component of the 
solution to 
\eqref{eq:partie2:systeme:derivative}, when the forward component is initialized with $\mu^{\varepsilon}$, reads
\begin{equation*}
\sum_{i=1}^N
\mu(U_{i})
v^{(0)}(\cdot,m_{0},y_{i})
=
\sum_{i=1}^N 
\int_{U_{i}}
v^{(0)}(\cdot,m_{0},y_{i}) d\mu(y).
\end{equation*}
By smoothness of $v^{(0)}$ in $y$, we easily deduce that 
\begin{equation*}
\biggl\| 
\tilde{z}_{0} - 
\int_{\T^d}
v^{(0)}(\cdot,m_{0},y) d\mu(y)
\biggr\|_{n+1+\alpha} \leq C \| \mu \| \varepsilon.
\end{equation*}
The result follows by letting $\varepsilon$ tend to $0$. 
\end{proof}

On the model of Corollary \ref{cor:diff0}, we now claim
\begin{Proposition}
\label{prop:partie2:U:differentiability}
Given two initial conditions $m_{0},m_{0}' \in {\mathcal P}(\T^d)$, 
we denote by 
$(\tilde{m}_{t},\tilde{u}_{t})_{t \in [0,T]}$
and 
$(\tilde{m}_{t}',\tilde{u}_{t}')_{t \in [0, T]}$ the respective solutions
of \eqref{eq:se:3:tilde:HJB:FP} with $m_{0}$
and $m_{0}'$ as initial conditions and 
by $(\tilde{\rho}_{t},\tilde{z}_{t})_{t \in [0,T]}$
the solution of 
\eqref{eq:partie2:systeme:derivative}
with $m_{0}'-m_{0}$ as initial condition, 
so that 
we can let 
\begin{equation*}
\delta \tilde{\rho}_{t}
= \tilde{m}_{t}' - 
\tilde{m}_{t} - \tilde{\rho}_{t},
\quad
\delta \tilde{z}_{t} = 
\tilde{u}_{t}' - \tilde{u}_{t} - \tilde{z}_{t},
\quad t \in [0,T].
\end{equation*}
Then, for any $\alpha' \in (0,\alpha)$,
we can find a constant $C$, independent of 
$m_{0}$ and $m_{0}'$, such that 
\begin{equation*}
\essup_{\omega \in \Omega}
\sup_{0 \leq t \leq T}
\bigl( \| \delta \tilde{\rho}_{t} \|_{-(n+\alpha')}
+ \| \delta \tilde{z}_{t} \|_{n+1+\alpha}
\bigr)
\leq 
C {\mathbf d}_{1}^2(m_{0},m_{0}').
\end{equation*}
In particular,
\begin{equation*}
\biggl\|
U(0,\cdot,m_{0}')
- U(0,\cdot,m_{0})
- \int_{\T^d} v^{(0)}(x,m_{0},y)
d \bigl(m_{0}' - m_{0})(y) 
\biggr\|_{n+1+\alpha}
\leq C {\mathbf d}_{1}^2(m_{0},m_{0}'),
\end{equation*}
and, thus, for any $x \in \T^d$, the mapping 
${\mathcal P}(\T^d) \ni m \mapsto U(0,x,m)$
is differentiable with respect to $m$ and the derivative reads, 
for any $m \in {\mathcal P}(\T^d)$,
\begin{equation*}
\frac{\delta U}{\delta m}(0,x,m,y) = v^{(0)}(x,m,y), \quad y \in \T^d. 
\end{equation*}
\end{Proposition}

The normalization condition holds: 
$$
\inte v^{(0)}(x,m,y)dm(y)=0.
$$
The proof is the same as in the deterministic case (see Remark \ref{rem:normalization}).  

\begin{proof}
%
We have
\begin{equation*}
\begin{split}
&d_{t} 
\bigl(\delta \tilde{z}_{t}\bigr) = \bigl\{ -  \Delta
\bigl( \delta \tilde{z}_{t} \bigr) 
+ \langle 
D_{p} \tilde{H}_{t}(\cdot,D \tilde{u}_{t}),D 
\bigl (\delta \tilde{z}_{t} \bigr)
\rangle - 
\frac{\delta \tilde F_{t}}{\delta m}(\cdot,m_{t})
\bigl(\delta \rho_{t} \bigr)
+ \tilde{f}_{t}
  \bigr\} dt + d \tilde{M}_{t},
\\
&\partial_{t} \bigl( \delta \tilde{\rho}_{t} \bigr) -  \Delta
\bigl( \delta \tilde{\rho}_{t} \bigr)
- \textrm{div} \bigl[ \bigl(\delta \tilde{\rho}_{t}\bigr) D_{p}
\tilde{H}_{t}(\cdot,D \tilde{u}_{t}) 
\bigr] - \textrm{div} \bigl[ \tilde{m}_{t} 
D^2_{pp} \tilde{H}_{t}(\cdot,D \tilde{u}_{t})
 \bigl( D \delta \tilde{z}_{t} \bigr)
 + \tilde{b}_{t}
\bigr] = 0,
\end{split}
\end{equation*}
with a boundary condition of the form 
\begin{equation*}
\delta \tilde{z}_{T} = 
\frac{\delta \tilde G}{\delta m}(\cdot,m_{T})\bigl(
\delta \rho_{T} \bigr)
+ \tilde{g}_{T}, 
\end{equation*}
where
\begin{equation*}
\begin{split}
\tilde{b}_{t} &= 
\tilde{m}_{t}' \bigl( 
D_{p} \tilde{H}_{t}( \cdot,D \tilde{u}_{t}')
- D_{p} \tilde{H}_{t}(\cdot, D \tilde{u}_{t})
\bigr) 
-
\tilde{m}_{t}  D^2_{pp} \tilde{H}_{t}(\cdot,D \tilde{u}_{t})
\bigl(
D \tilde{u}_{t}' - D \tilde{u}_{t}
\bigr)
\\
\tilde{f}_{t} &= 
\tilde{H}_{t}(\cdot,D \tilde{u}_{t}')
- 
\tilde{H}_{t}(\cdot,D \tilde{u}_{t})
-
\bigl\langle D_{p} \tilde{H}_{t}(\cdot,D \tilde{u}_{t}),
D \tilde{u}_{t}' - D \tilde{u}_{t} \bigr\rangle 
\\
&\hspace{15pt} - \Bigl( 
\tilde{F}_{t}(\cdot,m_{t}')
-
\tilde{F}_{t}
(\cdot,m_{t})
- \frac{\delta \tilde{F}_{t}}{\delta m}(\cdot,m_{t})
\bigl( m_{t}' - m_{t} \bigr)
 \Bigr),
\\
\tilde{g}_{T}
&= \tilde{G}(\cdot,m_{T}') - 
\tilde{G}(\cdot,m_{T}) - \frac{\delta \tilde{G}}{\delta m}(\cdot,m_{T}) \bigl( m_{T}' - m_{T} \bigr). 
\end{split}
\end{equation*}
Now,
\begin{equation*}
\begin{split}
\tilde{b}_{t} &= 
\bigl( \tilde{m}_{t}' - \tilde{m}_{t} \bigr) \bigl( 
D_{p} \tilde{H}_{t}( \cdot,D \tilde{u}_{t}')
- D_{p} \tilde{H}_{t}(\cdot, D \tilde{u}_{t})
\bigr) 
\\
&\hspace{15pt}
+ 
\tilde{m}_{t} \int_{0}^1 \Bigl[ 
D^2_{pp} \tilde{H}_{t}
\bigl(\cdot, \lambda D \tilde{u}_{t}' + (1-\lambda) 
D \tilde{u}_{t} \bigr)
- D^2_{pp} \tilde{H}_{t}(\cdot,D \tilde{u}_{t})
\Bigr]
\bigl(
D \tilde{u}_{t}' - D \tilde{u}_{t}
\bigr) d\lambda
\\
&= \bigl( \tilde{m}_{t}' - \tilde{m}_{t} \bigr) \int_{0}^1 
D^2_{pp}
\tilde{H}_{t}\bigl( \cdot, \lambda D \tilde{u}_{t}'
+
(1-\lambda) D \tilde{u}_{t}
\bigr)
\bigl( D \tilde{u}_{t}' - D \tilde{u}_{t} \bigr) 
d \lambda
\\
&\hspace{15pt}
+ 
\tilde{m}_{t} \int_{0}^1 \int_{0}^1 
\lambda D^3_{ppp} \tilde{H}_{t}
\bigl(\cdot, \lambda s D \tilde{u}_{t}' + (1-\lambda + \lambda (1-s)) 
D \tilde{u}_{t} \bigr)
\bigl(
D \tilde{u}_{t}' - D \tilde{u}_{t}
\bigr)^{\otimes 2}
 d\lambda
ds,
\\
\tilde{f}_{t} &=
\int_{0}^1 \bigl\langle 
D_{p} \tilde{H}_{t}(\cdot,\lambda D \tilde{u}_{t}' + 
(1-\lambda) D \tilde{u}_{t} )
-
D_{p} \tilde{H}_{t}(\cdot,D \tilde{u}_{t} ),
D \tilde{u}_{t}' - 
D \tilde{u}_{t} \bigr\rangle d\lambda
\\
&\hspace{15pt}
- \int_{0}^1 
\Bigl( 
\frac{\delta \tilde{F}_{t}}{\delta m}
\bigl( \cdot,\lambda m_{t}' + (1-\lambda) m_{t}
\bigr) - 
\frac{\delta \tilde{F}_{t}}{\delta m}
( \cdot,m_{t} )
\Bigr) \bigl( {m}_{t}' - {m}_{t} \bigr) d\lambda
\\
&=
\int_{0}^1 \int_{0}^1 
\lambda \bigl\langle 
D^2_{pp} \tilde{H}_{t}\bigl(\cdot,\lambda s D \tilde{u}_{t}' + 
(1-\lambda + \lambda (1-s)) D \tilde{u}_{t} \bigr)
\bigl( D \tilde{u}_{t}' - 
D \tilde{u}_{t} \bigr) ,
D \tilde{u}_{t}' - 
D \tilde{u}_{t} \bigr\rangle d\lambda ds
\\
&\hspace{15pt}
- \int_{0}^1 
\Bigl( 
\frac{\delta \tilde{F}_{t}}{\delta m}
\bigl( \cdot,\lambda m_{t}' + (1-\lambda) m_{t}
\bigr) - 
\frac{\delta \tilde{F}_{t}}{\delta m}
( \cdot,m_{t} )
\Bigr) \bigl( {m}_{t}' - {m}_{t} \bigr) d\lambda,
\\
\tilde{g}_{T}
&=
\int_{0}^1 
\Bigl( 
\frac{\delta \tilde{G}}{\delta m}
\bigl( \cdot,\lambda m_{T}' + (1-\lambda) m_{T}
\bigr) - 
\frac{\delta \tilde{G}}{\delta m}
( \cdot,m_{T} )
\Bigr) \bigl( {m}_{T}' - {m}_{T} \bigr) d\lambda.
\end{split}
\end{equation*}
By Lemma \ref{lem:super:reg}, we have a universal bound for 
\begin{equation*}
\essup_{\omega \in \Omega}
\sup_{t \in [0,T]}
\bigl( \| \tilde{u}_{t} \|_{n+1+\alpha}
+ \| \tilde{u}_{t}' \|_{n+1+\alpha} \bigr). 
\end{equation*}
We deduce that
\begin{equation*}
\begin{split}
&\| \tilde{b}_{t} \|_{-1} \leq C
\Bigl( 
 {\mathbf d}_{1}
\bigl(\tilde{m}_{t}',\tilde{m}_{t}
\bigr) \| \tilde{u}_{t}' - \tilde{u}_{t} \|_{2}
+ \| \tilde{u}_{t}' - \tilde{u}_{t} \|_{1}^2 \Bigr),
\\
&\| \tilde{f}_{t} \|_{n+\alpha} \leq C 
\Bigl(
\| \tilde{u}_{t}' - \tilde{u}_{t} \|^2_{n+1+\alpha}
+  
{\mathbf d}_{1}^2
\bigl(\tilde{m}_{t}',\tilde{m}_{t}
\bigr)  
\Bigr),
\\
&\| \tilde{g}_{T} \|_{n+1+\alpha} \leq C 
 {\mathbf d}_{1}^2
\bigl(\tilde{m}_{T}',\tilde{m}_{T}
\bigr).
\end{split}
\end{equation*}
Therefore, by 
Theorem \ref{thm:partie:2:existence:uniqueness}, we deduce that 
\begin{equation*}
\begin{split}
&\essup_{\omega \in \Omega}
\sup_{0 \leq t \leq T}
\| \tilde{b}_{t} \|_{-1} 
+ \essup_{\omega \in \Omega}
\sup_{0 \leq t \leq T}
\| \tilde{f}_{t} \|_{n+\alpha}
+ 
\essup_{\omega \in \Omega}
\| \tilde{g}_{T} \|_{n+1+\alpha}
\leq C
 {\mathbf d}_{1}^2
\bigl({m}_{0}',{m}_{0}
\bigr).
\end{split}
\end{equation*}
By Corollary 
\ref{cor:partie2:systeme:linearise:linfty}, we get the first of the two inequalities in the statement. 
We deduce that 
\begin{equation*}
\bigl\|
U(0,\cdot,m_{0}')
- U(0,\cdot,m_{0})
- \tilde{z}_{0}
\bigr\|_{n+1+\alpha}
\leq C  {\mathbf d}_{1}^2(m_{0},m_{0}'). 
\end{equation*}
By 
Lemma
\ref{lem:partie2:representation:derivative}, we complete the proof. 
\end{proof}

\begin{Proposition}
\label{prop:partie2:U:differentiability:2}
For any $\alpha' \in (0,\alpha)$,
we can find a constant $C$
such that, for any $m_{0},m_{0}' \in {\mathcal P}(\T^d)$, 
any $y,y' \in \T^d$
and any 
index $\ell \in \{0,\dots,n\}^d$
with $\vert \ell \vert \leq n$,
denoting by 
$(\tilde{m}_{t},\tilde{u}_{t})_{t \in [0,T]}$
and $(\tilde{m}_{t}',\tilde{u}_{t}')_{t \in [0,T]}$
the respective solutions of 
\eqref{eq:se:3:tilde:HJB:FP},
and then   
$(\tilde{\rho}_{t},\tilde{z}_{t})_{t \in [0,T]}$
and $(\tilde{\rho}_{t}',\tilde{z}_{t}')_{t \in [0,T]}$
the corresponding 
solutions of
\eqref{eq:partie2:systeme:derivative}
 when driven by two initial conditions
 $(-1)^{\vert \ell \vert}D^{\ell} \delta_{y}$ and 
 $(-1)^{\vert \ell \vert}D^\ell \delta_{y'}$, 
 it holds that  
\begin{equation*}
\begin{split}
&\essup_{\omega \in \Omega}
\biggl[ 
\sup_{t \in [0,T]} \| \tilde{z}_{t} - \tilde{z}_{t}' \|_{n+1+\alpha}
+
\sup_{t \in [0,T]} 
\| \tilde{\rho}_{t} - \tilde{\rho}_{t}' \|_{-(n+\alpha')}
\biggr] 
 \leq C
\Bigl(  {\mathbf d}_{1}(m_{0},m_{0}') + \vert y - y' \vert^{\alpha'}
\Bigr).
\end{split}
\end{equation*}
In particular,
\begin{equation*}
\forall y,y' \in \T^d, \quad
\biggl\|
D_{y}^{\ell} \frac{\delta U}{\delta m}(0,\cdot,m_{0},y)
- 
D_{y}^{\ell}
\frac{\delta U}{\delta m}(0,\cdot,m_{0}',y')
\biggr\|_{n+1+\alpha}
\leq C
\Bigl(  {\mathbf d}_{1}(m_{0},m_{0}') + \vert y - y' \vert^{\alpha'}
\Bigr). 
\end{equation*}
\end{Proposition}

\begin{proof}
Given two initial conditions $m_0$ and $m_0'$, we call 
$(\tilde{m}_{t},\tilde{u}_{t})_{t \in [0,T]}$
and $(\tilde{m}_{t}',\tilde{u}_{t}')_{t \in [0,T]}$
the respective solutions of 
\eqref{eq:se:3:tilde:HJB:FP}. 
With $(\tilde{m}_{t},\tilde{u}_{t})_{t \in [0,T]}$
and $(\tilde{m}_{t}',\tilde{u}_{t}')_{t \in [0,T]}$, 
we associate 
the solutions 
$(\tilde{\rho}_{t},\tilde{z}_{t})_{t \in [0,T]}$
and $(\tilde{\rho}_{t}',\tilde{z}_{t}')_{t \in [0,T]}$
of
\eqref{eq:partie2:systeme:derivative}
 when driven by two initial conditions
 $(-1)^{\vert \ell \vert}
 D^{\ell} \delta_{y}$ and 
 $(-1)^{\vert \ell \vert} D^\ell \delta_{y'}$. 
 Since $\vert \ell \vert \leq n$, we have
 \begin{equation*}
 \bigl\| 
 D^\ell \delta_{y}
 - D^\ell \delta_{y'}
 \bigr\|_{-(n+\alpha')} \leq \vert y - y' \vert^{\alpha'}.
 \end{equation*}
In order to prove the first estimate, we can apply 
Corollary \ref{cor:partie2:systeme:linearise:linfty}
with 
\begin{equation*}
\begin{split}
&\tilde{V}_{t} = D_{p}\tilde{H}(\cdot,D \tilde{u}_{t}), 
\quad 
\tilde{V}_{t}' = D_{p}\tilde{H}(\cdot,D \tilde{u}_{t}'), 
\\
&
\Gamma_{t} = 
D^2_{pp} \tilde{H}_{t}( \cdot,D \tilde{u}_{t}),
\quad 
\Gamma_{t}'=
D^2_{pp} \tilde{H}_{t}( \cdot,D \tilde{u}_{t}'), 
\end{split}
\end{equation*}
so that, following 
the proof of 
Proposition
\ref{prop:partie2:U:differentiability},
\begin{equation*}
\begin{split}
&\| \tilde{V}_{t} - \tilde{V}_{t}' \|_{n+\alpha} 
+ \| \Gamma_{t} - \Gamma_{t}' \|_{0} \leq C 
\| \tilde{u}_{t} - \tilde{u}_{t}' \|_{n+1+\alpha}.  
\end{split}
\end{equation*}
Now, the first estimate in the statement follows from 
the combination of  
Theorem 
\ref{thm:partie:2:existence:uniqueness}
and Corollary \ref{cor:partie2:systeme:linearise:linfty}.

The second estimate is a straightforward consequence of the first one. 
\end{proof}

\begin{Proposition}
\label{prop:partie2:regularity:time:first:order}
Propositions
\ref{prop:partie2:U:differentiability}
and 
\ref{prop:partie2:U:differentiability:2}
easily extend to any initial time $t_{0} \in [0,T]$. 
Then,  for any
$\alpha' \in (0,\alpha)$,  
any 
$t_{0} \in [0,T]$
and
$m_{0} \in \Pk$
\begin{equation*}
\lim_{h \rightarrow 0}
\sup_{\ell \in \{0,\dots,n\}^d, \vert \ell \vert \leq n}
\Bigl\|
D_{y}^{\ell} \frac{\delta U}{\delta m}(t_{0}+h,\cdot,m_{0},\cdot)
- 
D_{y}^{\ell} \frac{\delta U}{\delta m}(t_{0},\cdot,m_{0},\cdot)
\Bigr\|_{n+1+\alpha',\alpha'}
=0.
\end{equation*}

\end{Proposition}

\begin{proof}
Given two probability measures $m,m' \in {\mathcal P}(\T^d)$, we know
from Proposition 
\ref{prop:partie2:U:differentiability}
 that, for any 
$t \in [0,T]$, 
\begin{equation}
\label{eq:partie2:continuity:derivatives:measures}
U\bigl(t,\cdot,m'
\bigr)
- U \bigl( t,\cdot,m \bigr)
  = 
  \int_{\T^d}
  \frac{\delta U}{\delta m}\bigl(t,\cdot,
  m,y \bigr)
  d \bigl( m' - m \bigr)(y) + 
  O \bigl( {\mathbf d}_{1}^2(m,m') \bigr),
\end{equation}
the equality holding true in ${\mathcal C}^{n+1+\alpha}(\T^d)$
and the Landau notation $O(\cdot)$ 
being uniform in $t_{0}$ and $m$ (the constant
$C$ in the statement of 
Proposition \ref{prop:partie2:U:differentiability} being explicitly 
quantified by means of Proposition 
\ref{lem:partie2:stability:2}, related to the stability of solutions to the linear equation). 

By Proposition
\ref{prop:partie2:U:differentiability:2}, the set of functions 
$([\T^d]^2 \ni (x,y) 
\mapsto (\delta U/\delta m)(t,x,m,y))_{t \in [0,T]}$
is relatively compact in ${\mathcal C}^{n+1+\alpha'}(\T^d)
\times 
{\mathcal C}^{n+\alpha'}(\T^d)
$, for any $\alpha' \in (0,\alpha)$. Any limit $\Phi : [\T^d]^2 \rightarrow \R$ 
obtained by letting $t$ tend to $t_{0}$ in 
\eqref{eq:partie2:continuity:derivatives:measures}
must satisfy 
(use 
Corollary
\ref{cor:partie:2:reg:temps} to pass to the limit in the left-hand side):
\begin{equation*}
U\bigl(t_{0},\cdot,m'
\bigr)
- U \bigl( t_{0},\cdot,m \bigr)
  = 
  \int_{\T^d}
 \Phi\bigl(\cdot,y \bigr)
  d \bigl( m' - m \bigr)(y) + 
  O \bigl( {\mathbf d}_{1}^2(m,m') \bigr),
\end{equation*}
the equality holding true in ${\mathcal C}^0(\T^d)$. This proves that, 
for any $x \in \T^d$, 
\begin{equation*}
\int_{\T^d}
  \frac{\delta U}{\delta m}\bigl(t_{0},x,
  m,y \bigr)
  d \bigl( m' - m \bigr)(y) 
  = 
  \int_{\T^d}
 \Phi\bigl(x,y \bigr)
  d \bigl( m' - m \bigr)(y).
\end{equation*}
Choosing $m'$ as the solution at time $h$ of the Fokker-Planck equation 
\begin{equation*}
\partial_{t} m_{t} = - \textrm{\rm div}(b m_{t}), \quad t \geq 0,
\end{equation*}
for a smooth field $b$ and with $m_{0}=m$ as initial condition,
and then letting $h$ tend to $0$, 
we deduce that 
\begin{equation*}
\int_{\T^d}
D_mU\bigl(t_{0},x,
  m,y \bigr) \cdot b(y)
dm(y)
  = 
  \int_{\T^d}
 D_{y}
 \Phi\bigl(x,y \bigr)
\cdot  b(y)
 dm(y).
\end{equation*}
When $m$ has full support, this proves that 
\begin{equation*}
\Phi(x,y) =  \frac{\delta U}{\delta m}\bigl(t_{0},x,
  m,y \bigr) + c(x), \quad x,y \in \T^d. 
\end{equation*}
Since both sides have a zero integral in $y$ with respect to $m$, $c(x)$ must be zero. 

When the support of $m$ does not cover $\T^d$, we can 
approximate $m$ by a sequence $(m_{n})_{n \geq 1}$
of measures with full supports. By Proposition 
\ref{prop:partie2:U:differentiability:2}, we know that,
for any $\alpha' \in (0,\alpha)$,
\begin{equation*}
\lim_{n \rightarrow \infty}\sup_{t \in [0,T] }\Bigl\| 
\frac{\delta U}{\delta m}\bigl(t,\cdot,
m_{n},\cdot \bigr) 
- 
\frac{\delta U}{\delta m}\bigl(t,\cdot,
m,\cdot \bigr) 
\Bigr\|_{n+1+\alpha',\alpha'} =0,
\end{equation*}
so that, in 
${\mathcal C}^{n+1+\alpha'}(\T^d) 
\times {\mathcal C}^{\alpha'}(\T^d)$, 
\begin{equation*}
\lim_{t \rightarrow t_{0}}
\frac{\delta U}{\delta m}\bigl(t,\cdot,
m,\cdot \bigr) 
= 
\lim_{n \rightarrow \infty}
\lim_{t \rightarrow t_{0}}
\frac{\delta U}{\delta m}\bigl(t,\cdot,
m_{n},\cdot \bigr) 
= \frac{\delta U}{\delta m}\bigl(t_{0},\cdot,
m,\cdot \bigr). 
\end{equation*}
We easily complete the proof
when $\vert \ell \vert =0$. 
Since  the
set of functions 
$([\T^d]^2 \ni (x,y) 
\mapsto (D_{y}^{\ell}
\delta U/\delta m)(t,x,m,y))_{t \in [0,T]}$
is relatively compact in 
 ${\mathcal C}^{n+1+\alpha'}(\T^d)
\times
{\mathcal C}^{\alpha'}(\T^d)$, any limit 
as $t$ tends to $t_{0}$ must coincide
with the 
derivative of index $\ell$ in $y$ of the limit of 
$[\T^d]^2 \ni (x,y) 
\mapsto 
[\delta U/\delta m](t,x,m,y)$
as $t$ tends to $t_{0}$. 
\end{proof}

\subsection{Second-order Differentiability}
\label{subse:second-order differentiability:proof}

\textbf{Assumption.} Throughout the paragraph, we assume that
$F$, $G$ and $H$ satisfy 
\eqref{HypD2H}
and 
\eqref{e.monotoneF}
in Subsection \ref{subsec:hyp}
and 
that, for some integer $n \geq 2$
and some $\alpha \in (0,1)$, 
{\bf (HF2(${\boldsymbol n}$))}
and 
{\bf (HG2(${\boldsymbol n}$+1))}
hold true.
\vspace{5pt}
%

In order to complete the analysis of the master equation, 
we need to investigate the second-order differentiability
in the direction of the measure, on the
same model as for the first-order derivatives. 

As for the first order, the idea is to write the second-order
derivative of $U$ in the direction $m$ as the initial 
value of the backward component of a linearized system 
of the type \eqref{eq:partie2:systeme:linearise}, 
which is referred next to as the \textit{second-order linearized system}. 
Basically, the \textit{second-order linearized system}
is obtained by differentiating 
one step more 
the \textit{first-order linearized system}
\eqref{eq:partie2:systeme:derivative}. 
Recalling that 
\eqref{eq:partie2:systeme:derivative}
has the form 
\begin{equation}
\label{eq:partie2:systeme:derivative:bis}
\begin{split}
&d_{t} \tilde{z}_{t} = \bigl\{ -  \Delta \tilde{z}_{t} + \langle 
D_{p} \tilde{H}_{t}(\cdot,D \tilde{u}_{t}),D \tilde{z}_{t}
\rangle - 
\frac{\delta \tilde F_{t}}{\delta m}(\cdot,m_{t})(\rho_{t})
  \bigr\} dt + d \tilde{M}_{t},
\\
&\partial_{t} \tilde{\rho}_{t} -  \Delta \tilde{\rho}_{t}
- \textrm{div}\bigl(\tilde{\rho}_{t} D_{p}
\tilde{H}_{t}(\cdot,D \tilde{u}_{t}) 
\bigr) - \textrm{div} \bigl( \tilde{m}_{t} 
D^2_{pp} \tilde{H}_{t}(\cdot,D \tilde{u}_{t})
 D \tilde{z}_{t}
\bigr) = 0,
\end{split}
\end{equation}
with the boundary condition
\begin{equation*}
\tilde{z}_{T} = 
\frac{\delta \tilde G}{\delta m}(\cdot,m_{T})(\rho_{T}),
\end{equation*}
the procedure is to 
differentiate the pair 
$(\tilde{\rho}_{t},\tilde{z}_{t})_{t \in [0,T]}$ 
with respect to the initial condition $m_{0}$ of $(\tilde{m}_{t},\tilde{u}_{t})_{t \in [0,T]}$,
the initial condition of 
$(\tilde{\rho}_{t},\tilde{z}_{t})_{t \in [0,T]}$ 
being kept frozen. 

Above, $(\tilde{m}_{t},\tilde{u}_{t})_{0 \leq t \leq T}$
is indeed chosen as the solution of the 
system
\eqref{eq:se:3:tilde:HJB:FP}, 
for a given initial distribution 
$m_{0} \in {\mathcal P}(\T^d)$,
and
$(\tilde{\rho}_{t},\tilde{z}_{t})_{t \in [0,T]}$
as the solution of the system 
\eqref{eq:partie2:systeme:derivative:bis}
with an initial condition $\rho_{0} \in ({\mathcal C}^{n+\alpha'}(\T^d))'$, 
for some $\alpha'<\alpha$.
Implicitly, 
the initial condition
$\rho_{0}$ is understood as some $m_{0}'-m_{0}$ for another
$m_{0}' \in {\mathcal P}(\T^d)$, in which case 
we know from Proposition \ref{prop:partie2:U:differentiability}
that 
$(\tilde{\rho}_{t},\tilde{z}_{t})_{t \in [0,T]}$ 
reads as the derivative, at $\varepsilon =0$, of the solution to 
\eqref{eq:se:3:tilde:HJB:FP}
when initialized with 
the measure
$m_{0} + \varepsilon (m_{0}' - m_{0})$. 
However, following the strategy used in the analysis 
of the first-order derivatives of $U$, 
it is much more convenient,
in order to investigate the second-order derivatives of $U$, 
 to distinguish the initial condition of 
$(\tilde \rho_{t})_{t \in [0,T]}$ from the direction $m_{0}'-m_{0}$  
used to differentiate the system 
\eqref{eq:se:3:tilde:HJB:FP}. 
This says that, in 
\eqref{eq:partie2:systeme:derivative:bis},
 we should allow 
$(\tilde{\rho}_{t},\tilde{z}_{t})_{t \in [0,T]}$
to be driven by an arbitrary initial condition 
$\rho_{0} \in ({\mathcal C}^{n+\alpha'}(\T^d))'$. 

Now, when 
\eqref{eq:partie2:systeme:derivative:bis}
is driven by an
arbitrary initial condition
$\rho_{0}$ and 
$m_{0}$ is perturbed in the direction $m_{0}'-m_{0}$
 for another $m_{0}' \in {\mathcal P}(\T^d)$
 (that is 
$m_{0}$ is changed into $m_{0} + \varepsilon (m_{0}'-m_{0})$
for some small $\varepsilon$), 
the system
obtained by differentiating
\eqref{eq:partie2:systeme:derivative:bis} 
 (at $\varepsilon = 0$) 
takes the form
\begin{equation}
\label{eq:partie2:systeme:derivative:2}
\begin{split}
&d_{t} \tilde{z}_{t}^{(2)} = \Bigl\{ -  \Delta \tilde{z}_{t}^{(2)} + \bigl\langle 
D_{p} \tilde{H}_{t}(\cdot,D \tilde{u}_{t}),D \tilde{z}_{t}^{(2)}
\bigr\rangle 
- \frac{\delta \tilde F_{t}}{\delta m}(\cdot,m_{t})(\rho_{t}^{(2)})
\\
&\hspace{50pt} + \bigl\langle 
D^2_{pp} \tilde{H}_{t}(\cdot,D \tilde{u}_{t}),
D \tilde{z}_{t}
\otimes
D \partial_{m} \tilde{u}_{t}
\bigr\rangle
- 
\frac{\delta^2 \tilde F_{t}}{\delta m^2}(\cdot,m_{t})(\rho_{t},\partial_{m}
m_{t})
  \Bigr\} dt + d \tilde{M}_{t},
\\
&\partial_{t} \tilde{\rho}_{t}^{(2)} -  \Delta \tilde{\rho}_{t}^{(2)}
- \textrm{div}\Bigl(\tilde{\rho}_{t}^{(2)} D_{p}
\tilde{H}_{t}(\cdot,D \tilde{u}_{t}) 
\Bigr) - \textrm{div} \Bigl( \tilde{m}_{t} 
D^2_{pp} \tilde{H}_{t}(\cdot,D \tilde{u}_{t})
 D \tilde{z}_{t}^{(2)}
\Bigr) 
\\
&\hspace{50pt}
- \textrm{div}\Bigl(\tilde{\rho}_{t}
 D_{pp}^2
\tilde{H}_{t}(\cdot,D \tilde{u}_{t}) 
D \partial_{m}\tilde{u}_{t}
\Bigr)
- \textrm{div}\Bigl(\partial_{m} \tilde m_{t}
 D_{pp}^2
\tilde{H}_{t}(\cdot,D \tilde{u}_{t}) 
D \tilde{z}_{t}
\Bigr)
\\
&\hspace{50pt}
- \textrm{div} \Bigl( \tilde{m}_{t} 
D^3_{ppp} \tilde{H}_{t}(\cdot,D \tilde{u}_{t})
 D \tilde{z}_{t}
\otimes 
 D \partial_{m}\tilde{u}_{t}\Bigr) 
=0,
\end{split}
\end{equation}
with a terminal boundary condition of the form
\begin{equation*}
\tilde{z}_{T}^{(2)}
= 
\frac{\delta \tilde{G}}{\delta m}
(\cdot,m_{T})\bigl( \rho_{T}^{(2)} \bigr)
+
\frac{\delta^2 \tilde{G}}{\delta m^2}
(\cdot,m_{T})(\rho_{T},\partial_{m} m_{T}),
\end{equation*}
where we have denoted by 
$(\partial_{m} \tilde{m}_{t},\partial_{m} \tilde{u}_{t})_{t \in [0,T]}$
the derivative of 
$(\tilde{m}_{t},\tilde{u}_{t})_{t \in [0,T]}$
when the initial condition is differentiated 
in the direction $m_{0}'-m_{0}$ at point $m_{0}$, for another 
$m_{0}' \in {\mathcal P}(\T^d)$.  
In \eqref{eq:partie2:systeme:derivative:2}, 
the pair $(\tilde{\rho}^{(2)}_{t},\tilde{z}^{(2)}_{t})_{t \in [0,T]}$
is then understood as the derivative 
of the solution
$(\tilde{\rho}_{t},\tilde{z}_{t})_{t \in [0,T]}$
to 
\eqref{eq:partie2:systeme:derivative:bis}. 

Now, using the same philosophy as in the analysis of the first-order derivatives,
we can choose freely the initial condition $\rho_{0}$.
Generally speaking, we will choose 
$\rho_{0} = 
(-1)^{\vert \ell \vert}
D^{\ell} \delta_{y}$, for some multi-index $\ell \in \{0,\dots,n-1\}^d$ 
with $\vert \ell \vert \leq n-1$ and some $y \in \T^d$. 
Since $\rho_{0}$ is expected to be insensitive to any 
perturbation that could apply to $m_{0}$, it then makes sense to let 
$\rho^{(2)}_{0}=0$. As said above, the initial condition  
$\partial_{m} m_{0}$ of $(\partial_{m} \tilde{m}_{t})_{0 \leq t \leq T}$
is expected to have the form $m_{0}'-m_{0}$ for another probability 
measure $m_{0}' \in {\mathcal P}(\T^d)$. Anyhow, by the same linearity argument 
as in the analysis of the first-order derivative, we can start with the case when
$\partial_{m} m_{0}$ is the derivative of a Dirac mass, namely
$\partial_{m} m_{0} = 
(-1)^{\vert k \vert}
D^{k} \delta_{\zeta}$, for another 
multi-index $k \in \{0,\dots,n-1\}^d$, and 
another $\zeta \in \T^d$, in which case  
$(\partial_{m} \tilde{m}_{t},\partial_{m} \tilde{u}_{t})_{0 \leq t \leq T}$
is another solution to 
\eqref{eq:partie2:systeme:derivative:bis},
but with $\partial_{m} m_{0} =
(-1)^{\vert k \vert}
D^{k} \delta_{\zeta}$ as initial condition. 
Given these initial conditions, 
 we then let 
\begin{equation*}
v^{(\ell,k)}\bigl(\cdot,m_{0},y,\zeta 
\bigr) = \tilde{z}_{0}^{(2)}, 
\end{equation*}
provided that 
\eqref{eq:partie2:systeme:derivative:2} has a unique solution.

In order to check that existence and uniqueness hold true, we may proceed as follows. 
The system 
\eqref{eq:partie2:systeme:derivative:2}
is of the type
\eqref{eq:partie2:systeme:linearise}, with
\begin{equation}
\label{eq:partie2:def:second-order:coeff}
\begin{split}
&\tilde{V}_{t}
= D_{p} \tilde{H}_{t}(\cdot,D \tilde{u}_{t}),
\quad {\Gamma}_{t}
= D^2_{pp} \tilde{H}_{t}(\cdot,D\tilde{u}_{t}),
\\
&\tilde{b}_{t}^0
= \tilde{\rho}_{t} D^{2}_{pp} \tilde{H}_{t}
(\cdot,D\tilde{u}_{t}) D \partial_{m}\tilde{u}_{t}
+  \partial_{m}\tilde{m}_{t} D^{2}_{pp} \tilde{H}_{t}
(\cdot,D\tilde{u}_{t}) D \tilde{z}_{t}
+ \tilde{m}_{t} D^3_{ppp} \tilde{H}_{t}(\cdot,D\tilde{u}_{t})
D \tilde{z}_{t} \otimes D \partial_{m} \tilde{u}_{t},
\\
&\tilde{f}_{t}^0 = 
\bigl\langle 
D^2_{pp} \tilde{H}_{t}(\cdot,D \tilde{u}_{t}),
D \tilde{z}_{t}
\otimes
D \partial_{m}\tilde{u}_{t}
\bigr\rangle
- 
\frac{\delta^2 \tilde F_{t}}{\delta m^2}(\cdot,m_{t})(\rho_{t},\partial_{m} 
m_{t}), 
\\
&\tilde{g}_{T}^0 
=
\frac{\delta^2 \tilde{G}}{\delta m^2}
(\cdot,m_{T})(\rho_{T},\partial_{m} m_{T}).
\end{split}
\end{equation}
Recall from 
Theorem \ref{thm:partie:2:existence:uniqueness}
and Lemma \ref{lem:super:reg}
on the one hand 
and from 
Corollary 
\ref{cor:partie2:systeme:linearise:linfty}
on the other hand
that 
we can find a constant $C$ (the value of which 
is allowed to increase from line to line),
independent of $m_{0}$, $y$, $\zeta$,
$\ell$
and 
$k$, such that 
\begin{equation}
\label{eq:partie2:second:order:linfini:first:order}
\begin{split}
&\essup_{\omega \in \Omega} \sup_{t \in [0,T]} \| \tilde{u}_{t} \|_{n+1+\alpha} \leq C,
\\
&\essup_{\omega \in \Omega}
\Bigl[ 
\sup_{t \in [0,T]} 
\bigl( 
\| \tilde{z}_{t}  \|_{n+1+\alpha}
+
\| \partial_{m}\tilde{u}_{t}  \|_{n+1+\alpha}
+
\| \tilde{\rho}_{t} \|_{-(n+\alpha')} 
+
\| \partial_{m}\tilde{m}_{t}  \|_{-(n+\alpha')}
\bigr) 
\Bigr]
\leq C.
\end{split}
\end{equation}
Since $\vert \ell \vert, \vert k \vert \leq n-1$, 
we 
can apply 
Corollary 
\ref{cor:partie2:systeme:linearise:linfty}
with $n$ replaced by $n-1$ (notice that 
$n-1$ satisfies the assumption of 
\S
\ref{subsubse:partie:2:1st:order}), so that
\begin{equation}
\label{eq:partie2:second:order:linfini:first:order:n-1}
\essup_{\omega \in \Omega}
\Bigl[ 
\sup_{t \in [0,T]} 
\bigl( 
\| \tilde{\rho}_{t} \|_{-(n+\alpha'-1)}
+ 
\| \partial_{m}\tilde{m}_{t} \|_{-(n+\alpha'-1)}
\bigr)
\Bigr]
 \leq C. 
\end{equation}
Therefore, we deduce that 
\begin{equation*}
\essup_{\omega \in \Omega}
\sup_{t \in [0,T]} 
\| \tilde{b}_{t}^0 \|_{-(n+\alpha'-1)} \leq C. 
\end{equation*}
Similarly, 
\begin{equation*}
\essup_{\omega \in \Omega}
\sup_{t \in [0,T]} 
\| \tilde{f}_{t}^0 \|_{n+\alpha} 
+ 
\essup_{\omega \in \Omega}
\sup_{t \in [0,T]} 
\| \tilde{g}_{t}^0 \|_{n+1+\alpha} 
\leq C.  
\end{equation*}
From Theorem \ref{thm:partie2:existence:uniqueness:linearise}, we deduce that, with the prescribed initial conditions, 
\eqref{eq:partie2:systeme:derivative:2}
has a unique solution. 
Moreover, by Corollary
\ref{cor:partie2:systeme:linearise:linfty}, 
\begin{equation}
\label{eq:partie2:infinity:bound:solution}
\essup_{\omega \in \Omega}
\sup_{t \in [0,T]} \| \tilde{z}_{t}^{(2)}  \|_{n+1+\alpha}
+
\essup_{\omega \in \Omega}
\sup_{t \in [0,T]} 
\| \tilde{\rho}_{t}^{(2)} \|_{-(n+\alpha')}
\leq C. 
\end{equation}

On the model 
of Lemma \ref{lem:partie2:kernel:derivative}, we claim:
\begin{Lemma}
\label{lem:partie2:kernel:derivative:2}
The function 
$$[\T^d]^2 \ni (x,y,\zeta) \mapsto v^{(0,0)}(x,m_{0},y,\zeta)$$
admits continuous crossed 
derivatives in $(y,\zeta)$,
up to the order $n-1$
in 
$y$ and to the order $n-1$ in $\zeta$, the derivative 
$$D^{\ell}_{y} D^{k}_{\zeta} v^{(0,0)}(\cdot,m_{0},y,\zeta) 
: \T^d \ni x \mapsto 
D^{\ell}_{y}
 D^{k}_{\zeta} 
v^{(0,0)}(x,m_{0},y,\zeta),$$ 
for $\vert \ell \vert,  \vert k \vert \leq n-1$,
belonging to 
${\mathcal C}^{n+1+\alpha}(\T^d)$
and writing
\begin{equation*}
D^{\ell}_{y} 
 D^{k}_{\zeta} 
v^{(0,0)}(x,m_{0},y,\zeta) = 
v^{(\ell,k)}(x,m_{0},y,\zeta), 
\quad x,y,\zeta \in \T^d. 
\end{equation*}
Moreover, for $\alpha' \in (0,\alpha)$, there exists a constant $C$ such that, for any multi-indices 
$\ell,k$ with $\vert \ell \vert, \vert k \vert \leq n-1$, any $y,y',\zeta,\zeta' \in \T^d$ and any $m_{0} \in {\mathcal P}(\T^d)$,
\begin{equation*}
\begin{split}
&\bigl\| v^{(\ell,k)}(\cdot,m_{0},y,
\zeta) \bigr\|_{n+1+\alpha} \leq C,
\\
&\bigl\| v^{(\ell,k)}(\cdot,m_{0},y,
\zeta)
-
v^{(\ell,k)}(\cdot,m_{0},y',
\zeta')
 \bigr\|_{n+1+\alpha} \leq C \bigl( \vert y - y' \vert^{\alpha'} 
 + \vert \zeta - \zeta' \vert^{\alpha'}
 \bigr).  
\end{split}
\end{equation*}
\end{Lemma}

\begin{proof}
With the same notations as in Lemma
\ref{lem:partie2:kernel:derivative}, 
we denote by 
$(\tilde{\rho}_{t}^{k,\zeta},\tilde{z}_{t}^{k,\zeta})_{t \in [0,T]}$
the solution to \eqref{eq:partie2:systeme:derivative} with 
$(-1)^{\vert k \vert}D^{k} \delta_{\zeta}$ as initial condition
and by 
$(\tilde{\rho}_{t}^{\ell,y},\tilde{z}_{t}^{\ell,y})_{t \in [0,T]}$
the solution to \eqref{eq:partie2:systeme:derivative} with 
$(-1)^{\vert \ell \vert}D^{\ell}\delta_{y}$ as initial condition. 

By Proposition 
\ref{prop:partie2:U:differentiability:2}
(applied with both $n-1$ and $n$), we have, for any 
$y,y' \in \T^d$ and any $\zeta,\zeta' \in \T^d$,
\begin{equation}
\label{eq:partie2:2nd:order:alpha prime:holder:regularity}
\begin{split}
&\essup_{\omega \in \Omega}
\biggl[ 
\sup_{t \in [0,T]} \| \tilde{z}_{t}^{k,\zeta} - \tilde{z}_{t}^{k,\zeta'} \|_{n+1+\alpha}
+
\sup_{t \in [0,T]} 
\| \tilde{\rho}_{t}^{k,\zeta} - \tilde{\rho}_{t}^{k,\zeta'} \|_{-(n+\alpha'-1)}
\biggr] 
 \leq C
\vert \zeta - \zeta' \vert^{\alpha'},
\\
&\essup_{\omega \in \Omega}
\biggl[ 
\sup_{t \in [0,T]} \| \tilde{z}_{t}^{\ell,y} - \tilde{z}_{t}^{\ell,y'} \|_{n+1+\alpha}
+
\sup_{t \in [0,T]} 
\| \tilde{\rho}_{t}^{\ell,y} - \tilde{\rho}_{t}^{\ell,y'} \|_{-(n+\alpha'-1)}
\biggr] 
 \leq C\vert y - y' \vert^{\alpha'}.
\end{split}
\end{equation}
Denote now by $(\tilde{b}^{\ell,k,y,\zeta}_{t})_{t \in [0,T]}$
the process $(\tilde{b}_{t}^0)_{t \in [0,T]}$
in \eqref{eq:partie2:def:second-order:coeff}
when $(\tilde{\rho}_{t},\tilde{z}_{t})_{t \in [0,T]}$
stands for the process $(\tilde{\rho}_{t}^{\ell,y},\tilde{z}_{t}^{\ell,y})_{
t \in [0,T]}$
and
$(\partial_{m} \tilde{m}_{t},\partial_{m} \tilde{u}_{t})_{t \in [0,T]}$
is replaced by $(\tilde{\rho}_{t}^{k,\zeta},\tilde{z}_{t}^{k,\zeta})_{
t \in [0,T]}$. Define in a similar way
$(\tilde{f}^{\ell,k,y,\zeta}_{t})_{t \in [0,T]}$
and 
$\tilde{g}_{T}^{\ell,k,y,\zeta}$. 
Then, combining 
\eqref{eq:partie2:2nd:order:alpha prime:holder:regularity}
 with 
\eqref{eq:partie2:second:order:linfini:first:order}
and 
\eqref{eq:partie2:second:order:linfini:first:order:n-1}
\begin{equation*}
\begin{split}
&\essup_{\omega}
\Bigl[
\bigl\| 
\tilde{g}_{T}^{\ell,k,y',\zeta'}
- 
\tilde{g}_{T}^{\ell,k,y,\zeta}
\bigr\|_{n+1+\alpha}
\\
&\hspace{15pt} +
\sup_{t \in [0,T]}
\Bigl( 
\bigl\| 
\tilde{b}_{t}^{\ell,k,y',\zeta'}
- 
\tilde{b}_{t}^{\ell,k,y,\zeta}
\bigr\|_{-(n+\alpha'-1)}
+
\bigl\| 
\tilde{f}_{t}^{\ell,k,y',\zeta'}
- 
\tilde{f}_{t}^{\ell,k,y,\zeta}
\bigr\|_{n+\alpha}
\Bigr)
\Bigr] 
\\
&\leq C 
\bigl( 
\vert y - y' \vert^{\alpha'}
+
\vert z - z' \vert^{\alpha'}
\bigr). 
\end{split}
\end{equation*}
By Proposition \ref{lem:partie2:stability:2}, we deduce that 
\begin{equation}
\label{eq:partie:2:2nd:order:reg:y:zeta}
\Bigl\| v^{(\ell,k)}(\cdot,m_{0},y,
\zeta)
-
v^{(\ell,k)}(\cdot,m_{0},y',
\zeta')
 \Bigr\|_{n+1+\alpha} \leq C \bigl( \vert y - y' \vert^{\alpha'} 
 + \vert \zeta - \zeta' \vert^{\alpha'}
 \bigr),
\end{equation}
 which provides the last claim in the statement (the $L^{\infty}$
 bound following from 
\eqref{eq:partie2:infinity:bound:solution}).  

Now, by Lemma 
\ref{lem:partie2:kernel:derivative}
(applied with both $n$ and $n-1$), we know that, for $\vert k \vert \leq n-2$
and $j \in \{1,\dots,d\}$,
\begin{equation*}
\begin{split}
&\lim_{\R \setminus \{0\} \ni h \rightarrow 0}
\essup_{\omega \in \Omega}
\Bigl[ 
\sup_{t \in [0,T]}
\Bigl(
\bigl\|
\frac{1}{h}
\bigl(
 \tilde{\rho}_{t}^{\zeta +h e_{j},k}
- 
\tilde{\rho}_{t}^{\zeta,k}
\bigr)
- \tilde{\rho}_{t}^{\zeta,k+e_{j}}  
\bigr
\|_{-(n+\alpha'-1)}
\\
&\hspace{100pt}+  
\bigl\|
\frac{1}h 
\bigl(
 \tilde{z}_{t}^{\zeta +h e_{j},k}
- 
\tilde{z}_{t}^{\zeta,k}
\bigr)
- \tilde{z}_{t}^{\zeta,k+e_{j}} 
\bigr
\|_{n+1+\alpha}
\Bigr)
\Bigr] = 0,
\end{split}
\end{equation*}
where $e_{j}$ denotes the $j^{\textrm{th}}$ vector of the canonical basis of 
$\R^d$.  
Therefore, by \eqref{eq:partie2:second:order:linfini:first:order}, 
\begin{equation*}
\begin{split}
\lim_{\R \setminus \{0\} \ni h \rightarrow 0}
\essup_{\omega \in \Omega}
\Bigl[ 
&\sup_{t \in [0,T]}
\Bigl(
\bigl\| 
\frac{1}{h}
\bigl(
\tilde{b}_{t}^{\ell,k,y,\zeta+he_{j}}
- 
\tilde{b}_{t}^{\ell,k,y,\zeta}
\bigr)
- \tilde{b}_{t}^{\ell,k+e_{j},y,\zeta}  
\bigr
\|_{-(n+\alpha'-1)}
\\
&\hspace{15pt}+  
\bigl\|
\frac{1}{h}
\bigl(
 \tilde{f}_{t}^{\ell,k,y,\zeta +h e_{j}}
- 
\tilde{f}_{t}^{\ell,k,y,\zeta}
\bigr)
- \tilde{f}_{t}^{\ell,k+e_{j},y,\zeta}
\bigr
\|_{n+\alpha} \Bigr)
\\
&\hspace{15pt} +
\bigl\| 
\frac{1}{h}
\bigl(
\tilde{g}_{T}^{\ell,k,y,\zeta +h e_{j}}
- 
\tilde{g}_{T}^{\ell,k,y,\zeta}
\bigr)
- \tilde{g}_{T}^{\ell,k+e_{j},y,\zeta}
\bigr
\|_{n+1+\alpha}
\Bigr) 
\Bigr] = 0. 
\end{split}
\end{equation*}
By Proposition \ref{lem:partie2:stability:2}, 
\begin{equation*} 
\lim_{h \rightarrow 0}
\Bigl\| \frac{1}{h}
\bigl( v^{(\ell,k)}
(\cdot,m_{0},y,\zeta+h e_{j})
- v^{(\ell,k)}
(\cdot,m_{0},y,\zeta)
\bigr)
-
v^{(\ell,k+e_{j})}
(\cdot,m_{0},y,\zeta) 
 \Bigr\|_{n+1+\alpha} = 0,
\end{equation*}
which proves, by induction, that 
\begin{equation*}
D^{k}_{\zeta}
v^{(\ell,0)}
(x,m_{0},y,\zeta)
= v^{(\ell,k)}
(x,m_{0},y,\zeta), \quad x,y,\zeta \in \T^d. 
\end{equation*}
Similarly, we can prove that 
\begin{equation*}
D^\ell_{y}
v^{(0,k)}
(x,m_{0},y,\zeta)
= v^{(\ell,k)}
(x,m_{0},y,\zeta), \quad x,y,\zeta \in \T^d. 
\end{equation*}
Together with the continuity property 
\eqref{eq:partie:2:2nd:order:reg:y:zeta}, we complete the proof. 
\end{proof}

We claim that 

\begin{Proposition}
\label{prop:partie2:U:differentiability:3}
We can find a constant $C$
such that, for any $m_{0},m_{0}' \in {\mathcal P}(\T^d)$,
any $y,y',\zeta \in \T^d$, 
any multi-indices $\ell,k$ with $\vert \ell\vert, \vert k \vert \leq n-1$,
\begin{equation*}
\Bigl\|
v^{(\ell,k)}(\cdot,m_{0},y,\zeta)
- v^{(\ell,k)}(\cdot,m_{0}',y,\zeta)
\Bigr\|_{n+1+\alpha}
\leq C 
{\mathbf d}_{1}(m_{0},m_{0}'). 
\end{equation*}
\end{Proposition}

\begin{proof}
The proof consists of a new application of Proposition 
\ref{lem:partie2:stability:2}. 
Given 
\begin{itemize}
\item the solutions 
$(\tilde{m}_{t},\tilde{u}_{t})_{t \in [0,T]}$
and $(\tilde{m}_{t}',\tilde{u}_{t}')_{t \in [0,T]}$
to \eqref{eq:se:3:tilde:HJB:FP} with $\tilde{m}_{0}=m_{0}$
and $\tilde{m}_{0}'=m_{0}'$ 
as respective initial conditions,
\item the solutions
$(\partial_{m}\tilde{m}_{t},\partial_{m}\tilde{u}_{t})_{t \in [0,T]}$
and
$(\partial_{m}\tilde{m}_{t}',\partial_{m}\tilde{u}_{t}')_{t \in [0,T]}$
to \eqref{eq:partie2:systeme:derivative:bis}, 
with
$(\tilde{m}_{t},\tilde{u}_{t})_{t \in [0,T]}$
and 
$(\tilde{m}_{t}',\tilde{u}_{t}')_{t \in [0,T]}$
as respective input and
$\partial_{m}\tilde{m}_{0}
=\partial_{m}\tilde{m}_{0}' = (-1)^{\vert k\vert} D^k \delta_{\zeta}$
as initial condition,
for some multi-index $k$ with $\vert k \vert \leq n-1$ and for some $\zeta \in \T^d$,
\item the solutions 
$(\tilde{\rho}_{t},\tilde{z}_{t})_{t \in [0,T]}$
and 
$(\tilde{\rho}_{t}',\tilde{z}_{t}')_{t \in [0,T]}$
to \eqref{eq:partie2:systeme:derivative:bis}, 
with 
$(\tilde{m}_{t},\tilde{u}_{t})_{t \in [0,T]}$
and 
$(\tilde{m}_{t}',\tilde{u}_{t}')_{t \in [0,T]}$
as respective input and
$(-1)^{\vert \ell \vert} D^{\ell} \delta_{y}$ as initial condition, 
for some multi-index $\ell$ with $\vert \ell \vert \leq n-1$ and for some $y \in \T^d$,
\item the
solutions 
$(\tilde{\rho}_{t}^{(2)},\tilde{z}_{t}^{(2)})_{t \in [0,T]}$
and 
$(\tilde{\rho}_{t}^{(2)\prime},\tilde{z}_{t}^{(2)\prime})_{t \in [0,T]}$
to the second-order linearized system \eqref{eq:partie2:systeme:derivative:2}
with 
$(\tilde{m}_{t},\tilde{u}_{t},\tilde{\rho}_{t},\tilde{z}_{t}, 
\partial_{m} \tilde{m}_{t}, 
\partial_{m} \tilde{u}_{t}
)_{t \in [0,T]}$
and 
$(\tilde{m}_{t}',\tilde{u}_{t}',\tilde{\rho}_{t}',\tilde{z}_{t}',
\partial_{m} \tilde{m}_{t}', 
\partial_{m} \tilde{u}_{t}'
)_{t \in [0,T]}$
as respective input and with $0$ as initial condition.
\end{itemize} 

Notice
from 
\eqref{eq:representation:z0:ell:y}
 that 
 $\tilde{z}_{0} = v^{(\ell)}(\cdot,m_{0},y)$
 and 
$\tilde{z}_{0}' = v^{(\ell)}(\cdot,m_{0}',y)$.
\vspace{4pt}

With each of 
$(\tilde{m}_{t},\tilde{u}_{t},\tilde{\rho}_{t},\tilde{z}_{t},\partial_{m} \tilde{m}_{t}, 
\partial_{m} \tilde{u}_{t})_{t \in [0,T]}$
and 
$(\tilde{m}_{t}',\tilde{u}_{t}',\tilde{\rho}_{t}',\tilde{z}_{t}',\partial_{m} \tilde{m}_{t}', 
\partial_{m} \tilde{u}_{t}'
)_{t \in [0,T]}$, 
we can associate the same 
coefficients as in 
\eqref{eq:partie2:def:second-order:coeff},
labeling with a prime 
the coefficients associated with the input
$(\tilde{m}_{t}',\tilde{u}_{t}',\tilde{\rho}_{t}',\tilde{z}_{t}',
\partial_{m} \tilde{m}_{t}', 
\partial_{m} \tilde{u}_{t}'
)_{t \in [0,T]}$. 
Combining with 
\eqref{eq:partie2:second:order:linfini:first:order}
and 
\eqref{eq:partie2:second:order:linfini:first:order:n-1},
we obtain:
\begin{equation*}
\begin{split}
&\| \tilde{V}_{t}
- \tilde{V}_{t}' \|_{n+\alpha} + \| \Gamma_{t} - \Gamma_{t}' \|_{0}
+ \| b_{t}^0 - b_{t}^{0'} \|_{-(n+\alpha'-1)}
+ \| f_{t}^{0} - f_{t}^{0\prime} \|_{n+\alpha}
+ \| g_{T}^{0} - g_{T}^{0\prime} \|_{n+1+\alpha}
\\
&\leq C 
\Bigl(
\| \tilde{u}_{t} - \tilde{u}_{t}' \|_{n+1+\alpha}
+
\| \tilde{z}_{t} - \tilde{z}_{t}' \|_{n+1+\alpha}
+
\| \partial_{m} \tilde{u}_{t} - \partial_{m}\tilde{u}_{t}' \|_{n+1+\alpha}
\\
&\hspace{100pt}+ 
 {\mathbf d}_{1}(\tilde{m}_{t},\tilde{m_{t}}')
+
\| \tilde{\rho}_{t} - \tilde{\rho}_{t}'\|_{-(n+\alpha'-1)}
+
\| \partial_{m}\tilde{m}_{t} - \partial_{m}\tilde{m}_{t}'\|_{-(n+\alpha'-1)}
\Bigr).
\end{split}
\end{equation*}
%
By 
Propositions 
\ref{lem:partie2:stability:2} 
and \ref{prop:partie2:U:differentiability:2}
(applied with both $n$ and $n-1$), 
we complete the proof. 
\end{proof}

On the model of 
Lemma \ref{lem:partie2:representation:derivative}, 
we have
\begin{Lemma}
\label{lem:partie2:representation:derivative:2}
Given a finite measure $\mu$ on $\T^d$, 
the solution $\tilde{z}^{(2)}$ to 
\eqref{eq:partie2:systeme:derivative},
with $0$ as initial condition, 
when $(\tilde{m}_{t})_{0 \leq t \leq T}$
is initialized with $m_{0}$,
$(\tilde{\rho}_{t})_{0 \leq t \leq T}$
is initialized with $(-1)^{\vert \ell \vert}
D^{\ell} \delta_{y}$,
for $\vert \ell \vert \leq n-1$ and $y \in \T^d$,
and $(\partial_{m}\tilde{m}_{t})_{0 \leq t \leq T}$
is initialized with $\mu$, 
reads,
when taken 
at time $0$,
\begin{equation*}
\tilde{z}_{0}^{(2)} : \R^d 
\ni x  \mapsto 
\tilde{z}_{0}^{(2)}(x) = \int_{\T^d} v^{(\ell,0)}(x,m_{0},y,\zeta) d\mu(\zeta). 
\end{equation*}
\end{Lemma}

Now,

\begin{Proposition}
\label{prop:partie2:U:differentiability:2:b}
We can find a constant $C$
such that, 
for any multi-index $\ell$ with $\vert \ell\vert \leq n-1$,
any $m_{0},m_{0}' \in {\mathcal P}(\T^d)$
and any $y \in \T^d$,
\begin{equation*}
\biggl\|
v^{(\ell)}(\cdot,m_{0}',y)
- 
v^{(\ell)}(\cdot,m_{0},y)
- \int_{\T^d} v^{(\ell,0)}(\cdot,m_{0},y,\zeta)
d \bigl(m_{0}' - m_{0})(\zeta) 
\biggr\|_{n+1+\alpha}
\leq C {\mathbf d}_{1}^2(m_{0},m_{0}'). 
\end{equation*}
\end{Proposition}
\begin{proof}
We follow the lines of the proof of 
Proposition \ref{prop:partie2:U:differentiability}. 
Given two initial conditions 
$m_{0},m_{0}' \in {\mathcal P}(\T^d)$,
we consider
\begin{itemize}
\item the solutions 
$(\tilde{m}_{t},\tilde{u}_{t})_{t \in [0,T]}$
and $(\tilde{m}_{t}',\tilde{u}_{t}')_{t \in [0,T]}$
to \eqref{eq:se:3:tilde:HJB:FP} with $\tilde{m}_{0}=m_{0}$
and $\tilde{m}_{0}'=m_{0}'$ 
as respective initial conditions,
\item the solution
$(\partial_{m}\tilde{m}_{t},\partial_{m}\tilde{u}_{t})_{t \in [0,T]}$
to \eqref{eq:partie2:systeme:derivative}, 
when driven by the input 
$(\tilde{m}_{t},\tilde{u}_{t})_{t \in [0,T]}$
and by the initial condition 
$\partial_{m}\tilde{m}_{0} = m_{0}' - m_{0}$,
\item the solutions 
$(\tilde{\rho}_{t},\tilde{z}_{t})_{t \in [0,T]}$
and 
$(\tilde{\rho}_{t}',\tilde{z}_{t}')_{t \in [0,T]}$
to \eqref{eq:partie2:systeme:derivative:bis}, 
with 
$(\tilde{m}_{t},\tilde{u}_{t})_{t \in [0,T]}$
and 
$(\tilde{m}_{t}',\tilde{u}_{t}')_{t \in [0,T]}$
as respective input and
$(-1)^{\vert \ell \vert} D^{\ell} \delta_{y}$ as initial condition, 
for some multi-index $\ell$ with $\vert \ell \vert \leq n-1$ and for some $y \in \T^d$,
\item the solution  
$(\tilde{\rho}^{(2)}_{t},\tilde{z}_{t}^{(2)})_{t \in [0,T]}$
to \eqref{eq:partie2:systeme:derivative:2}
with 
$(\tilde{m}_{t},\tilde{u}_{t},
\tilde{\rho}_{t},\tilde{z}_{t},
\partial_{m}\tilde{m}_{t},\partial_{m}\tilde{u}_{t}
)_{t \in [0,T]}$
as input and $0$ as initial condition. 
\end{itemize}

Then, we let
\begin{equation*}
\delta \tilde{\rho}_{t}^{(2)}
= \tilde{\rho}_{t}' - \tilde{\rho}_{t}
- \tilde{\rho}_{t}^{(2)},
\quad
\delta \tilde{z}_{t}^{(2)}
= \tilde{z}_{t}' - \tilde{z}_{t}
- \tilde{z}_{t}^{(2)}, 
\quad t \in [0,T]. 
\end{equation*}
We have
\begin{equation*}
\begin{split}
&d_{t} 
\bigl(\delta \tilde{z}_{t}^{(2)}\bigr) = \bigl\{ -  \Delta
\bigl( \delta \tilde{z}_{t}^{(2)} \bigr) 
+ \langle 
D_{p} \tilde{H}_{t}(\cdot,D \tilde{u}_{t}),D 
\bigl (\delta \tilde{z}_{t}^{(2)} \bigr)
\rangle - 
\frac{\delta \tilde F_{t}}{\delta m}(\cdot,m_{t})
\bigl(\delta \rho_{t}^{(2)} \bigr)
+ \tilde{f}_{t}
  \bigr\} dt + d \tilde{M}_{t},
\\
&\partial_{t} \bigl( \delta \tilde{\rho}_{t}^{(2)} \bigr) -  \Delta
\bigl( \delta \tilde{\rho}_{t}^{(2)} \bigr)
- \textrm{div} \bigl[ \bigl(\delta \tilde{\rho}_{t}^{(2)}\bigr) D_{p}
\tilde{H}_{t}(\cdot,D \tilde{u}_{t}) 
\bigr] - \textrm{div} \bigl[ \tilde{m}_{t} 
D^2_{pp} \tilde{H}_{t}(\cdot,D \tilde{u}_{t})
 \bigl( D \delta \tilde{z}_{t}^{(2)} \bigr)
 + \tilde{b}_{t}
\bigr] = 0,
\end{split}
\end{equation*}
with a boundary condition of the form 
\begin{equation*}
\delta \tilde{z}_{T}^{(2)} = 
\frac{\delta \tilde G}{\delta m}(\cdot,m_{T})\bigl(
\delta \rho_{T}^{(2)} \bigr)
+ \tilde{g}_{T}, 
\end{equation*}
where
\begin{equation*}
\begin{split}
\tilde{b}_{t} &= 
\tilde{\rho}_{t}'
\Bigl( 
D_{p} \tilde{H}_{t}(\cdot,D \tilde{u}_{t}')
 - 
D_{p} \tilde{H}_{t}(\cdot,D \tilde{u}_{t})
\Bigr)
+
\Bigl(  \tilde{m}_{t}' 
D^2_{pp} \tilde{H}_{t}(\cdot,D \tilde{u}_{t}')
 - 
   \tilde{m}_{t} 
D^2_{pp} \tilde{H}_{t}(\cdot,D \tilde{u}_{t})
\Bigr)
 D \tilde{z}_{t}'
\\
&\hspace{15pt} -
\partial_{m}\tilde{m}_{t} D^{2}_{pp} \tilde{H}_{t}
(\cdot,D\tilde{u}_{t}) D \tilde{z}_{t}
- \tilde{\rho}_{t} D^{2}_{pp} \tilde{H}_{t}
(\cdot,D\tilde{u}_{t}) D \partial_{m}\tilde{u}_{t}
- \tilde{m}_{t} D^3_{ppp} \tilde{H}_{t}(\cdot,D\tilde{u}_{t})
D \tilde{z}_{t} \otimes D \partial_{m} \tilde{u}_{t},
\\
\tilde{f}_{t} &= 
\bigl\langle 
D_{p} \tilde{H}_{t}(\cdot,D \tilde{u}_{t}')
-D_{p} \tilde{H}_{t}(\cdot,D \tilde{u}_{t}),
D \tilde{z}_{t}'
\bigr\rangle 
- 
\bigl\langle 
D^2_{pp} \tilde{H}_{t}(\cdot,D \tilde{u}_{t}),
D \tilde{z}_{t}
\otimes 
D 
\partial_{m}
\tilde{u}_{t}
\bigr\rangle 
\\
&\hspace{15pt} -
\Bigl( 
\frac{\delta \tilde{F}_{t}}{\delta m}(\cdot,m_{t}')
({\rho}_{t}')
-
\frac{\delta \tilde{F}_{t}}{\delta m}(\cdot,m_{t})
({\rho}_{t}')
-\frac{\delta^2 \tilde{F}_{t}}{\delta m^2}(\cdot,m_{t})
({\rho}_{t},
\partial_{m} m_{t})
 \Bigr),
\\
\tilde{g}_{T}
&= 
\frac{\delta \tilde{G}}{\delta m}(\cdot,m_{T}')
({\rho}_{T}')
-
\frac{\delta \tilde{G}}{\delta m}(\cdot,m_{T})
({\rho}_{T}')
-\frac{\delta^2 \tilde{G}}{\delta m^2}(\cdot,m_{T})
({\rho}_{T},\partial_{m} m_{T}),
\end{split}
\end{equation*}
and where $(\tilde{M}_{t})_{t \in [0,T]}$
is a square integrable martingale as in 
\eqref{eq:partie2:systeme:linearise}. 

Therefore,
\begin{equation*}
\begin{split}
\tilde{b}_{t} &=
\bigl(
\tilde{\rho}_{t}'
-
\tilde{\rho}_{t}
\bigr)
\Bigl( 
D_{p} \tilde{H}_{t}(\cdot,D \tilde{u}_{t}')
 - 
D_{p} \tilde{H}_{t}(\cdot,D \tilde{u}_{t})
\Bigr)
\\
&\hspace{5pt} + \tilde{\rho}_{t}
\Bigl( D_{p} \tilde{H}_{t}\bigl(\cdot,
D \tilde{u}_{t}'
\bigr)
- D_{p} \tilde{H}_{t}\bigl(\cdot,
D \tilde{u}_{t}
\bigr) 
-
\langle D^2_{pp} \tilde{H}_{t}(\cdot,
D \tilde{u}_{t}
), D \partial_{m} \tilde{u}_{t}
\bigr\rangle \Bigr)
\\
&\hspace{5pt}
+
\Bigl(  \tilde{m}_{t}' 
D^2_{pp} \tilde{H}_{t}(\cdot,D \tilde{u}_{t}')
 - 
   \tilde{m}_{t} 
D^2_{pp} \tilde{H}_{t}(\cdot,D \tilde{u}_{t})
\Bigr)
\bigl( D \tilde{z}_{t}' - D \tilde{z}_{t} \bigr)
\\
&\hspace{5pt} + 
\Bigl(  \tilde{m}_{t}' - \tilde{m}_{t} \Bigr) 
\Bigl(
D^2_{pp} \tilde{H}_{t}(\cdot,D \tilde{u}_{t}')
-
D^2_{pp} \tilde{H}_{t}(\cdot,D \tilde{u}_{t})
\Bigr) 
 D \tilde{z}_{t}
\\
&\hspace{5pt} +
\Bigl( \tilde{m}_{t}'
 - \tilde{m}_{t}
 - \partial_{m}\tilde{m}_{t}
 \Bigr)
 D^2_{pp} \tilde{H}_{t}
(\cdot, D \tilde{u}_{t}) D \tilde{z}_{t}
  \\
 &\hspace{5pt}
 +
\tilde{m}_{t} 
\Bigl( 
D^2_{pp} \tilde{H}_{t}(\cdot,D \tilde{u}_{t}')
- 
D^2_{pp} \tilde{H}_{t}(\cdot,D \tilde{u}_{t})
- 
D^{3}_{ppp}
\tilde{H}_{t}(\cdot,D \tilde{u}_{t}) 
{D \partial_{m} \tilde{u}_{t}
\Bigr) D \tilde{z}_{t}},
\end{split}
\end{equation*}
and
\begin{equation*}
\begin{split}
\tilde{f}_{t} &=
\Bigl\langle 
D_{p} \tilde{H}_{t}\bigl(\cdot,
D \tilde{u}_{t}'
\bigr)
- D_{p} \tilde{H}_{t}\bigl(\cdot,
D \tilde{u}_{t}
\bigr), D \tilde{z}_{t}' - D \tilde{z}_{t}
\Bigr\rangle
\\
&\hspace{15pt} + 
\Bigl\langle 
D_{p} \tilde{H}_{t}\bigl(\cdot,
D \tilde{u}_{t}'
\bigr)
- D_{p} \tilde{H}_{t}\bigl(\cdot,
D \tilde{u}_{t}
\bigr)
- D^2_{pp} \tilde{H}_{t}(\cdot,
D \tilde{u}_{t}
)
D \partial_{m} \tilde{u}_{t}
, D \tilde{z}_{t}
\Bigr\rangle
\\
&\hspace{15pt}
+ \Bigl( 
\frac{\delta \tilde{F}_{t}}{\delta m}(\cdot,m_{t}')
-
\frac{\delta \tilde{F}_{t}}{\delta m}(\cdot,m_{t})
\Bigr)
\bigl({\rho}_{t}' - {\rho}_{t}\bigr)
\\
&\hspace{15pt}
+ \Bigl( 
\frac{\delta \tilde{F}_{t}}{\delta m}(\cdot,m_{t}')({\rho}_{t})
-
\frac{\delta \tilde{F}_{t}}{\delta m}(\cdot,m_{t})({\rho}_{t})
-\frac{\delta^2 \tilde{F}_{t}}{\delta m^2}(\cdot,m_{t})
({\rho}_{t},\partial_{m} m_{t})
 \Bigr),
\end{split}
\end{equation*}
Similarly, 
\begin{equation*}
\begin{split}
\tilde{g}_{T} &=
\Bigl( 
\frac{\delta \tilde{G}}{\delta m}(\cdot,m_{T}')
-
\frac{\delta \tilde{G}}{\delta m}(\cdot,m_{T})
\Bigr)
\bigl({\rho}_{T}' - {\rho}_{T}\bigr)
\\
&\hspace{15pt}
+ \Bigl( 
\frac{\delta \tilde{G}}{\delta m}(\cdot,m_{t}')({\rho}_{T})
-
\frac{\delta \tilde{G}}{\delta m}(\cdot,m_{t})({\rho}_{T})
-\frac{\delta^2 \tilde{G}}{\delta m^2}(\cdot,m_{T})
({\rho}_{T},\partial_{m} m_{T})
 \Bigr).
\end{split}
\end{equation*}
Applying
Theorem 
\ref{thm:partie:2:existence:uniqueness},
Lemma 
\ref{lem:super:reg},
Propositions \ref{prop:partie2:U:differentiability}
and 
 \ref{prop:partie2:U:differentiability:2}
 and 
 \eqref{eq:partie2:second:order:linfini:first:order}
 and 
 \eqref{eq:partie2:second:order:linfini:first:order:n-1} 
 and using the same kind of Taylor expansion
 as in the proof of Proposition
 \ref{prop:partie2:U:differentiability}, we deduce that : 
\begin{equation*}
\begin{split}
\essup_{\omega \in \Omega}
\sup_{t \in [0,T]}
\Bigl[ 
\| \tilde{b}_{t} \|_{-(n+\alpha'-1)}
+ 
\| \tilde{f}_{t} \|_{n+\alpha}
+
\| \tilde{g}_{T} \|_{n+1+\alpha}
\Bigr]
\leq C {\mathbf d}_{1}^2(m_{0},m_{0}'). 
\end{split}
\end{equation*}
By Proposition 
\ref{lem:partie2:stability:2}, 
we complete the proof.
\end{proof}

We thus deduce:

\begin{Proposition}
\label{prop:partie2:regularite:partial2Um}
For any $x \in \T^d$, the function 
${\mathcal P}(\T^d) \ni m \mapsto U(0,x,m)$
is twice differentiable in the direction $m$
and the second-order derivatives read, for any 
$m \in {\mathcal P}(\T^d)$
\begin{equation*} 
\frac{\delta^2 U}{\delta m^2}(0,x,m,y,y') = 
v^{(0,0)}(x,m,y,y'), \quad y,y' \in \T^d. 
\end{equation*}
In particular, 
for any $\alpha' \in (0,\alpha)$,
$t \in [0,T]$
and $m \in {\mathcal P}(\T^d)$, 
the function 
$[\delta^2 U/\delta m^2](0,\cdot,m,\cdot,\cdot)$
belongs to ${\mathcal C}^{n+1+\alpha'}(\T^d)
\times 
{\mathcal C}^{n-1+\alpha'}(\T^d) 
\times
{\mathcal C}^{n-1+\alpha'}(\T^d)$
and the mapping
\begin{equation*}
{\mathcal P}(\T^d) \ni 
m \mapsto \frac{\delta^2 U}{\delta m^2}(0,\cdot,m,\cdot,\cdot)
\in {\mathcal C}^{n+1+\alpha'}(\T^d)
\times 
{\mathcal C}^{n-1+\alpha'}(\T^d)
\times 
{\mathcal C}^{n-1+\alpha'}(\T^d)
\end{equation*}
is continuous (with respect to ${\mathbf d}_{1}$). 
The derivatives in $y$ and $y'$ read:
\begin{equation*} 
D_{y}^\ell D_{y'}^k \frac{\delta^2 U}{\delta m^2}(0,x,m,y,y') = 
v^{(\ell,k)}(x,m,y,y'), \quad y,y' \in \T^d, 
\quad \vert k\vert, \vert \ell \vert \leq n-1.  
\end{equation*}
\end{Proposition}

\begin{proof}
By Proposition 
\ref{prop:partie2:U:differentiability:2:b}, we indeed
know that, for any multi-index
$\ell$ with $\vert \ell \vert \leq n-1$
and any $x,y \in \T^d$, the mapping 
${\mathcal P}(\T^d) \ni m \mapsto D^{\ell}_{y}[\delta U/\delta m](0,x,m,y)$
is differentiable with respect to $m$, the derivative 
writing, for any $m \in {\mathcal P}(\T^d)$, 
\begin{equation*}
\frac{\delta}{\delta m}
\Bigl[
D^{\ell}_{y}
\frac{\delta U}{\delta m}
\Bigr]
(0,x,m,y,y') = v^{(\ell,0)}(0,x,m,y,y'),
\quad y,y' \in \T^d. 
\end{equation*}
By Lemma \ref{lem:partie2:kernel:derivative:2}, 
$[\delta / \delta m][D^{\ell}_{y} [\delta U/\delta m]](0,x,m,y,y')$ 
is $n-1$ times differentiable with respect to 
$y'$ and, together with
Proposition 
\ref{prop:partie2:U:differentiability:3},
 the derivatives are continuous 
 in all the parameters. Making use of 
  Schwarz' Lemma 
  \ref{lem:schwarz}, the proof is easily completed. 
\end{proof}

Following Proposition 
\ref{prop:partie2:regularity:time:first:order}, we
finally claim:
\begin{Proposition}
\label{prop:partie2:regularity:time:2nd:order}
Proposition
\ref{prop:partie2:regularite:partial2Um}
easily extend to any initial time $t_{0} \in [0,T]$. 
Then,  for any
$\alpha' \in (0,\alpha)$,  
any 
$t_{0} \in [0,T]$
and
$m_{0} \in \Pk$
\begin{equation*}
\lim_{h \rightarrow 0}
\sup_{\vert k\vert \leq n-1}
\sup_{\vert \ell \vert \leq n-1}
\Bigl\|
D_{y}^{\ell} D_{y'}^k
\frac{\delta^{2} U}{\delta m^{2}}(t_{0}+h,\cdot,m_{0},\cdot)
- 
D_{y}^{\ell} D_{y'}^k \frac{\delta^{2} U}{\delta m^{2}}(t_{0},\cdot,m_{0},\cdot)
\Bigr\|_{n+1+\alpha',\alpha',\alpha'}
=0.
\end{equation*}
\end{Proposition}

\subsection{Proof of 
Theorem \ref{theo:2nd-order:master:equation}}
We now prove 
Theorem
\ref{theo:2nd-order:master:equation}. Of course 
the key point is to prove that $U$, as constructed in the
previous, subsection is 
a solution of the master equation
\eqref{eq:master:equation:2nd:order:def}.

\subsubsection{Regularity Properties of the Solution} 
The regularity properties 
of $U$ follow
from 
Subsections
\ref{subse:construction:solution:master}, 
\ref{subsubse:partie:2:1st:order}
and
\ref{subse:second-order differentiability:proof},
see in particular 
Propositions
\ref{prop:partie2:regularite:partial2Um} 
and 
\ref{prop:partie2:regularity:time:2nd:order}
(pay attention that,
in the statements 
of 
Theorem
\ref{theo:2nd-order:master:equation}
and
of Proposition
\ref{prop:partie2:regularite:partial2Um},
the indices of regularity
in $y$ and $y'$ 
are not exactly the same).

\subsubsection{Derivation of the Master Equation}
We now have all the necessary ingredients in order to derive the 
master equation satisfied by $U$. The first point is to recall 
that, whenever 
the forward component $(\tilde{m}_{t})_{t \in [t_{0},T]}$
in 
\eqref{eq:se:3:tilde:HJB:FP:t0}
is initialized with $m_{0} \in {\mathcal P}(\T^d)$ at time $t_{0}$, then 
\begin{equation*}
U(t_{0},x,m_{0}) = \tilde{u}_{t_{0}}(x), \quad x \in \T^d, 
\end{equation*}
$(\tilde{u}_{t})_{t \in [t_{0},T]}$
denoting the backward component in 
\eqref{eq:se:3:tilde:HJB:FP:t0}. 
Moreover, by 
Lemma 
\ref{lem:partie2:master:equation},
for any $h \in [T-t_{0}]$, 
\begin{equation*}
\tilde{u}_{t_{0}+h}(x) =
U\bigl(t_{0}+h,x + \sqrt{2}(W_{t_{0}+h}-W_{t_{0}}),m_{t_{0},t_{0}+h}\bigr), \quad x \in \T^d, 
\end{equation*}
where $m_{t_{0},t}$ the image of 
$\tilde{m}_{t}$ by the random mapping 
$\T^d \ni x \mapsto x + \sqrt{2}(W_{t}-W_{t_{0}})$
that is ${m}_{t_{0},t} = [id + 
\sqrt{2}(W_{t}-W_{t_{0}})] \sharp 
\tilde{m}_{t}$. 
In particular, 
we can write
\begin{equation}
\label{eq:partie2:proof:master:equation}
\begin{split}
&\frac{U(t_{0}+h,x,m_{0})-U(t_{0},x,m_{0})}{h}
\\
&=  \frac{
  {\mathbb E}\bigl[U \bigl(t_{0}+h,x
+ \sqrt2  (W_{t_{0}+h}-W_{t_{0}})  
  ,m_{t_{0},t_{0+h}} \bigr) \bigr]
  - U(t_{0},x,m_{0}) }{h}
\\
&\hspace{15pt} + \frac{U(t_{0}+h,x,m_{0}) - {\mathbb E}
\bigl[U \bigl(t_{0}+h,x
+ \sqrt2  (W_{t_{0}+h}-W_{t_{0}})
,m_{t_{0},t_{0}+h} \bigr) \bigr]}{h}
\\
&=  
\frac{ 
\E [\tilde{u}_{t_{0}+h}(x)]
 -
\tilde{u}_{t_{0}}(x)}{h}
\\
&\hspace{15pt} + \frac{U(t_{0}+h,x,m_{0}) - 
{\mathbb E}
\bigl[U \bigl(t_{0}+h,x
+ \sqrt2  (W_{t_{0}+h}-W_{t_{0}})
,m_{t_{0},t_{0}+h} \bigr) \bigr]}{h}.
\end{split}
\end{equation}
We start with the first term in the right-hand side 
of  
\eqref{eq:partie2:proof:master:equation}. 
Following 
\eqref{eq:partie2:construction:backward},
we deduce from the backward equation in 
\eqref{eq:se:3:tilde:HJB:FP:t0}
that, for any $x \in \T^d$,
\begin{equation*}
\begin{split}
d_{t}
\bigl[ {\mathbb E}
\bigl( \tilde{u}_{t}\bigl(x
\bigr)
\bigr) \bigr]
=  {\mathbb E}
\Bigl[ \bigl\{ -   \Delta \tilde{u}_{t} + \tilde{H}_{t_{0},t}(\cdot,D\tilde{u}_{t}) - 
\tilde{F}_{t_{0},t}(\cdot,m_{t_{0},t}) 
\bigr\}(x)
\Bigr] dt,
\end{split}
\end{equation*}
where the coefficients 
$\tilde{F}_{t_{0},t}$
and 
$\tilde{H}_{t_{0},t}$
are given
by 
\eqref{eq:se:3:tilde:HJB:FP:prescription}. 
In particular, thanks to the regularity property 
in Corollary
\ref{cor:partie:2:reg:temps}, we deduce that 
\begin{equation}
\label{eq:partie2:proof:master:equation:20}
\begin{split}
&\lim_{h \searrow 0}
\frac{ 
\E [\tilde{u}_{t_{0}+h}(x)]
 -
\tilde{u}_{t_{0}}(x)
}{h}
= - \Delta_{x} U(t_{0},m_{0},x) 
+ H\bigl(x,D_{x} U(t_{0},m_{0},x) \bigr)
- F \bigl(x,m_{0} \bigr). 
\end{split}
\end{equation}
In order to pass to the limit in the last term in
\eqref{eq:partie2:proof:master:equation}, we need a specific
form of It\^o's formula. The precise version 
is given in Lemma 
\ref{lem:ito:local} below. 
Applied to the current setting,  
with
\begin{equation*}
\beta_{t}(\cdot) = D_{p} H\bigl(\cdot,D_{x} U(t,\cdot,m_{t_{0},t})\bigr),
\end{equation*}
it says that 
\begin{equation}
\label{eq:partie2:proof:master:equation:101}
\begin{split}
&
\lim_{h \searrow 0} \frac{1}{h}
\E \Bigl[
U\bigl(t_{0}+h,x+\sqrt{2}(W_{t_{0}+h}-W_{t_{0}}),m_{t_{0},t_{0}+h}
\bigr)
-U\bigl(t_{0}+h,x,m_{{0}}\bigr)\Bigr]
\\
&=  \Delta_{x} U(t_{0},x,m_{0}) 
\\
&\hspace{15pt} + 2 \int_{\T^d}
\textrm{div}_{y} 
\bigl[
D_{m} U
\bigr]
 \bigl( t_{0},x,m_{0},y
 \bigr) d {m}_{0}(y)
 \\
&\hspace{15pt}
- \int_{\T^d}
D_{m} 
U
 \bigl( t_{0},x,m_{0},y
 \bigr) 
 D_{p} H\bigl(y, D U(t_{0},y,m_{0}) 
 \bigr)
 d m_{0}(y)
 \\
&\hspace{15pt} + 2 
\int_{\T^d}
\textrm{div}_{x}
\bigl[
D_{m} 
U
\bigr]
\bigl( t_{0}, x,
 m_{0},y
\bigr) dm_{0}(y)
\\
&\hspace{15pt}
+
\int_{[\T^d]^2}
{\rm Tr} \Bigl[ D^2_{mm}
U
\bigl( t_{0}, x ,
 m_{0},y,y'\bigr) 
\Bigr] d m_{0}
(y)
d m_{0}
(y').
\end{split}
\end{equation}
From 
\eqref{eq:partie2:proof:master:equation:20}
and
\eqref{eq:partie2:proof:master:equation:101}, we deduce that, for any 
$(x,m_{0}) \in \T^d \times {\mathcal P}(\T^d)$,
the mapping 
$[0,T] \ni t \mapsto U(t,x,m_{0})$ is right-differentiable 
and, for any $t_{0} \in [0,T)$,
\begin{equation*}
\begin{split}
&\lim_{h \searrow 0}
\frac{U(t_{0}+h,x,m_{0}) - U(t_{0},x,m_{0}) }{h}
\\
&=- 2 \Delta_{x} U(t_{0},x,m_{0}) 
+ H\bigl(x,D_{x} U(t_{0},x,m_{0}) \bigr)
- F \bigl(x,m_{0} \bigr)
\\
&\hspace{15pt} -2 \int_{\T^d}
\textrm{div}_{y} 
\bigl[
D_{m} U
\bigr]
 \bigl( t_{0},x,m_{0},y
 \bigr) d {m}_{0}(y)
 \\
 &\hspace{15pt}
+
\int_{\T^d}
D_{m} 
U
 \bigl( t_{0},x,m_{0},y
 \bigr) 
 D_{p} H\bigl(y, D U(t_{0},y,m_{0}) 
 \bigr)
 d m_{0}(y)
 \\
&\hspace{15pt} - 2 
\int_{\T^d}
\textrm{div}_{x}
\bigl[
D_{m} 
U
\bigr]
\bigl( t_{0}, x,
 m_{0},y
\bigr) dm_{0}(y)
-
\int_{[\T^d]^2}
{\rm Tr} \Bigl[ D^2_{mm}
U
\bigl( t_{0}, x ,
 m_{0},y,y'\bigr) 
\Bigr] d m_{0}
(y)
d m_{0}
(y').
\end{split}
\end{equation*}
Since the right-hand side is continuous in $(t_{0},x,m_{0})$, we deduce that 
$U$ is continuously differentiable in time and satisfies the master 
equation \eqref{eq:master:equation:2nd:order:def}.

\subsubsection{Uniqueness}
\label{subsub:uniqueness.sec4}
It now remains to prove uniqueness. 
Considering a solution $V$ to the master 
equation \eqref{eq:master:equation:2nd:order:def}
along the lines of Definition \ref{def:master:eq:2nd:order}, 
the strategy is to expand
\begin{equation*}
\tilde u_{t}' = V\bigl(t,x+ \sqrt{2}W_{t},
m_{t}' \bigr),
\quad t \in [0,T],
\end{equation*}
where, for a given initial condition $m_{0} \in {\mathcal P}(\T^d)$,
$m_{t}'$ is the image of $\tilde{m}_{t}'$
by the mapping 
$\T^d \ni x \mapsto x + \sqrt{2} W_{t}$, 
$(\tilde{m}_{t}')_{t \in [0,T]}$ denoting the solution of the Fokker-Planck equation 
\begin{equation*}
d_{t} \tilde{m}_{t} '= \Bigl\{ \Delta \tilde{m}_{t}'
+{\rm div} \bigl( \tilde{m}_{t}' D_{p} \tilde{H}_{t}\bigl(\cdot,
D_{x} V(t,x+\sqrt{2}W_{t}, \tilde{m}_{t}')  \bigr) \bigr) \Bigr\} dt,
\end{equation*}
which reads, for almost every realization of $(W_{t})_{t \in [0,T]}$, as the flow of conditional marginal distributions (given $(W_{t})_{t \in [0,T]}$) of the McKean-Vlasov process
\begin{equation}
\label{eq:EDS:MKV:V}
dX_{t} = -D_{p} \tilde{H}_{t}\bigl(X_{t},
D_{x} V(t,x+ \sqrt{2} W_{t}, {\mathcal L}(X_t|W) )\bigr) dt + \sqrt{2} dB_{t}, \quad t \in [0,T],
\end{equation}
$X_{0}$ having $m_{0}$ as distribution.  
Notice that the above equation is uniquely solvable since 
$D_{x} V$ is Lipschitz continuous in the space and measure arguments 
(by the simple fact that $D^2_{x} V$ and $D_{m} D_{x} V$ 
are continuous functions on a compact set).
We refer to \cite{Szn}
for standard solvability results for McKean-Vlasov SDEs
(which may be easily extended to the current setting).

Of course, the key point is to prove that 
the pair
$(\tilde{m}_{t}',\tilde{u}_{t}')_{t \in [0,T]}$
solves the same forward-backward system 
\eqref{eq:se:3:tilde:HJB:FP}
as $(\tilde{m}_{t},\tilde{u}_{t})_{t \in [0,T]}$, in which 
case it will follow that $V(0,x,m_{0})=\tilde{u}_{0}' = \tilde{u}_{0}
= U(0,x,m_{0})$. (The same argument may be repeated for any other initial 
condition with another initial time.)
 
The strategy consists of a suitable application of Lemma 
\ref{lem:ito:local} below. 
Given $0 \leq t \leq t +h \leq T$, we have to expand the difference
\begin{equation}
\label{eq:2nd:order:uniqueness}
\begin{split}
&{\mathbb E} \bigl[ V\bigl(t+h,x+\sqrt{2}W_{t+h},m_{t+h}' \bigr)
\vert {\mathcal F}_{t} \bigr]
-
V\bigl(t,x+\sqrt{2}W_{t},m_{t}' \bigr) 
\\
&= {\mathbb E} \bigl[ 
V\bigl(t+h,x+\sqrt{2}W_{t+h},m_{t+h}' \bigr)\vert {\mathcal F}_{t} \bigr]
-
V\bigl(t+h,x+\sqrt{2}W_{t},m_{t}' \bigr)
\\
&\hspace{15pt} + 
V\bigl(t+h,x+\sqrt{2}W_{t},m_{t}' \bigr)
- 
V\bigl(t,x+\sqrt{2}W_{t},m_{t}' \bigr)
\\
&=S^1_{t,h} + S^2_{t,h}. 
\end{split}
\end{equation}
By Lemma \ref{lem:ito:local} below, with 
\begin{equation*}
\beta_{t}(\cdot) = D_{p} H\bigl(\cdot,D_{x} V(t,\cdot,m_{t}')\bigr), 
\quad t \in [0,T],
\end{equation*}
it holds that 
\begin{equation}
\label{eq:2nd:order:uniqueness:2}
\begin{split}
S^1_{t,h} 
&= h 
\biggl[ \Delta_{x} V\bigl(t,x + \sqrt{2}
W_{t},m_{t}'
\bigr)  + 2 \int_{\T^d}
\textrm{\rm div}_{y} 
\bigl[
D_{m} V
\bigr]
 \bigl( t,x+ \sqrt{2} W_{t},m_{t}',y
 \bigr) d {m}_{t}'(y)
 \\
&\hspace{15pt}-  \int_{\T^d}
D_{m} 
V
 \bigl( t,x+ \sqrt{2}W_{t},m_{t}',y
 \bigr) 
 \cdot D_{p} H\bigl(y,D_{x} V(t,y,m_{t}')
 \bigr)
 d m_{t}'(y)
 \\
&\hspace{15pt} + 2 
\int_{\T^d}
\textrm{\rm div}_{x}
\bigl[
D_{m} 
V
\bigr]
\bigl( t, x+ \sqrt{2} W_{t},
 m_{t}',y
\bigr) dm_{t}'(y)
\\
&\hspace{15pt} +
\int_{[\T^d]^2}
{\rm Tr} \Bigl[ D^2_{mm}
V\Bigr]
\bigl( t, x + \sqrt{2} W_{t},
 m_{t}',y,y'\bigr) 
 d m_{t}'
(y)
d m_{t}'
(y') + \varepsilon_{t,t+h}
\biggr],
\end{split}
\end{equation}
where $(\varepsilon_{s,t})_{s,t \in [t_{0},T]  :  s \leq t}$ 
is
a family of real-valued random variables
such that 
\begin{equation*}
\lim_{h \searrow 0} \sup_{s,t \in [t_{0},T] : \vert s -t \vert \leq h}  {\mathbb E} \bigl[ \vert \varepsilon_{s,t} \vert \bigr] = 0.
\end{equation*}
Expand now 
$S^{2}_{t,h}$
in \eqref{eq:2nd:order:uniqueness}
at the first order in $t$ and use the fact that $\partial_{t} V$ is uniformly continuous 
on the compact set $[0,T] \times \T^d \times {\mathcal P}_{2}(\T^d)$.
Combining 
\eqref{eq:2nd:order:uniqueness},
\eqref{eq:2nd:order:uniqueness:2}
and the master PDE 
\eqref{eq:master:equation:2nd:order:def}
satisfied by $V$, we deduce that 
\begin{equation*}
\begin{split}
&{\mathbb E} \bigl[ V\bigl(t+h,x+\sqrt{2}W_{t+h},m_{t+h}' \bigr)
\vert {\mathcal F}_{t} \bigr]
-
V\bigl(t,x+\sqrt{2}W_{t},m_{t}' \bigr) 
\\
&= - h \Bigl[ 
\Delta_{x} V\bigl(t,x+\sqrt{2} W_{t},m_{t}' \bigr)
- H\bigl(x+\sqrt{2} W_{t}, D_{x} V(t,x + \sqrt{2} W_{t},m_{t}') \bigr)
\\
&\hspace{15pt}+ F \bigl( x + \sqrt{2} W_{t},m_{t}' \bigr) 
+ \varepsilon_{t,t+h}
\Bigr].
\end{split}
\end{equation*}
Considering a partition $t=t_{0}<t_{1}<\dots<t_{N}=T$ of $[t,T]$ of step size $h$,
the above identity yields
\begin{equation*}
\begin{split}
&{\mathbb E} \bigl[ G\bigl(x+\sqrt{2}W_{T},m_{T}' \bigr)
- V\bigl(t,x+\sqrt{2}W_{t},m_{t}' \bigr)
\vert {\mathcal F}_{t} \bigr] 
\\
&= - h \sum_{i=0}^{N-1}\Bigl[ 
\Delta_{x} V\bigl(t_{i},x+\sqrt{2} W_{t_{i}},m_{t_{i}}' \bigr)
- H\bigl(x+\sqrt{2} W_{t_{i}}, D_{x} V(t_{i},x + \sqrt{2} W_{t_{i}},m_{t_{i}}') \bigr)
\\
&\hspace{60pt}+ F \bigl( x + \sqrt{2} W_{t_{i}},m_{t_{i}}' \bigr) \Bigr]
\\ 
&\hspace{15pt}+ h \sum_{i=0}^{N-1} 
\E \bigl[ \varepsilon_{t_{i},t_{i}+h} \vert {\mathcal F}_{t} \bigr].
\end{split}
\end{equation*}
Since
\begin{equation*}
\limsup_{h \searrow 0}
\sup_{r,s \in [0,T] : \vert r-s \vert \leq h}
\E 
\Bigl[ \bigl\vert \E \bigl[ \varepsilon_{r,s} \vert {\mathcal F}_{t} \bigr]
\bigr\vert \Bigr]
\leq 
\limsup_{h \searrow 0}
\sup_{r,s \in [0,T] : \vert r-s \vert \leq h}
\E 
\bigl[ \vert   \varepsilon_{r,s} 
\vert \bigr] = 0,
\end{equation*}
we can easily replace each 
$\E [ \varepsilon_{t_{i},t_{i}+h} \vert {\mathcal F}_{t}]$
by $\varepsilon_{t_{i},t_{i+h}}$ itself, allowing 
for a modification of $\varepsilon_{t_{i},t_{i+h}}$.
Moreover, 
here and below (\textit{cf.} the proof of Lemma 
\ref{lem:ito:local}), we use the fact that, for a random process $(\gamma_{t})_{t \in [{0},T]}$,
with paths in ${\mathcal C}^0([0,T],\R)$,
satisfying
\begin{equation}
\label{eq:partie2:trick:2}
\begin{split}
&\essup_{\omega \in \Omega}\sup_{t \in [0,T]} | \gamma_{t} | < \infty,
\end{split}
\end{equation}
it must hold that 
\begin{equation}
\label{eq:partie2:trick}
\lim_{h \searrow 0} 
\sup_{s,t \in [0,T] : \vert s-t \vert \leq h}
\E \bigl[ \vert \eta_{s,t} \vert \bigr] 
= 0, \quad 
\eta_{s,t} = \frac{1}{\vert s-t \vert}
\int_{s}^t (\gamma_{r}-\gamma_{s}) dr,
\end{equation}
the proof just consisting in bounding 
$\vert \eta_{s,t} \vert$ by $w_{\gamma}(h)$, where
$w_{\gamma}$ stands for the pathwise modulus of continuity 
of $(\gamma_{t})_{t \in [0,T]}$, which satisfies, thanks to 
\eqref{eq:partie2:trick:2}
and 
Lebesgue's dominated convergence theorem,
\begin{equation*}
\lim_{h \searrow 0} \E[ w_{\gamma}(h) ] = 0. 
\end{equation*}
Therefore, 
allowing for a modification of the random variables 
$\varepsilon_{t_{i},t_{i}+h}$, 
for $i=0,\dots,N-1$, we deduce that 
\begin{equation*}
\begin{split}
&{\mathbb E} \bigl[ G\bigl(x+\sqrt{2}W_{T},m_{T}' \bigr)
- V\bigl(t,x+\sqrt{2}W_{t},m_{t}' \bigr)
\vert {\mathcal F}_{t} \bigr] 
\\
&=
-
 \int_{t}^T \Bigl[ 
\Delta_{x} V\bigl(s,x+\sqrt{2} W_{s},m_{s}' \bigr)
- H\bigl(x+\sqrt{2} W_{s}, D_{x} V(s,x + \sqrt{2} W_{s},m_{s}') \bigr)
\\
&\hspace{60pt}+ F \bigl( x + \sqrt{2} W_{s},m_{s}' \bigr) \Bigr] ds
\\ 
&\hspace{15pt}+ h \sum_{i=0}^{N-1} \varepsilon_{t_{i},t_{i}+h}.
\end{split}
\end{equation*}
Letting, for any $x \in \T^d$, 
\begin{equation*}
\begin{split}
\tilde M_{t}'(x) &= 
V\bigl(t,x+\sqrt{2}W_{t},m_{t}' \bigr) 
\\
&\hspace{15pt} + \int_{0}^t 
\Bigl[ 
\Delta_{x} V\bigl(s,x+\sqrt{2} W_{s},m_{s}' \bigr)
- H\bigl(x+\sqrt{2} W_{s}, D_{x} V(s,x + \sqrt{2} W_{s},m_{s}') \bigr)
\\
&\hspace{60pt}+ F \bigl( x + \sqrt{2} W_{s},m_{s}' \bigr)
\Bigr] ds,
\end{split}
\end{equation*}
we deduce that 
\begin{equation*}
{\mathbb E} \bigl[ 
\tilde M_{T}'(x) -\tilde M_{t}'(x)
\vert {\mathcal F}_{t} \bigr] = h \sum_{i=0}^{N-1} \varepsilon_{t_{i},t_{i}+h}.
\end{equation*}
Now, letting 
$h$ tend to $0$, we deduce that 
$(\tilde M_{t}'(x))_{t \in [0,T]}$ is a martingale. 
Thanks to the regularity properties of $V$ and its derivatives, it is bounded.

Letting 
\begin{equation*}
\tilde{v}_{t}(x) = V(t,x+\sqrt{2} W_{t},m_{t}'), \quad t \in [0,T],
\end{equation*}
we finally notice that 
\begin{equation*}
\tilde{v}_{t}(x) = \tilde G_{T}(x,m_{T}')+
\int_{t}^T 
\bigl[ 
\Delta_{x} \tilde{v}_{s}(x)
- \tilde H_{s}\bigl(x, D \tilde{v}_{s}(x)\bigr)
+ \tilde F ( x,m_{s}' )
\bigr] ds - \bigl( \tilde M_{T}'- \tilde M_{t}' \bigr)(x), \quad t \in [0,T],
\end{equation*} 
which proves that $(\tilde{m}_{t}',\tilde{v}_{t},\tilde M_{t}')_{t \in [0,T]}$
solves \eqref{eq:se:3:tilde:HJB:FP}. 

\subsubsection{Tailor-made It\^o's Formula}
\label{subsub:Tailor-madeItoFormula}

Let 
$U$ be a function satisfying 
the same assumption as in 
Definition \ref{def:master:eq:2nd:order}
and, for a given $t_{0} \in [0,T]$, 
$(\tilde m_{t})_{t \in [t_{0},T]}$ be an adapted process 
with paths in ${\mathcal C}^0([t_{0},T],{\mathcal P}(\T^d))$
such that, 
with probability 1,
for any smooth test function $\varphi \in \cC^n(\T^d)$,
\begin{equation}
\label{eq:local:FKP}
\begin{split}
&d_{t} \biggl[ 
\int_{\T^d} \varphi(x) d \tilde m_{t}(x) \biggr]
\\
&\hspace{15pt}= \biggl\{ 
\int_{\T^d} 
\bigl[ 
\Delta \varphi(x) - \langle \beta_{t}\bigl(x + \sqrt{2} (W_{t}-W_{t_{0}})\bigr), D\varphi(x) \rangle \bigr] d \tilde m_{t}(x)
\biggr\} dt, 
\quad t \in [t_{0},T],
\end{split}
\end{equation}
for some adapted process $(\beta_{t})_{t \in [t_{0},T]}$,
with paths in ${\mathcal C}^0([t_{0},T],[\cC^0(\T^d)]^d)$,
such that 
\begin{equation*}
\begin{split}
&\essup_{\omega \in \Omega}\sup_{t \in [t_{0},T]} \| \beta_{t} \|_{0} < \infty,
\end{split}
\end{equation*}
so that, by Lebesgue's dominated convergence theorem,
\begin{equation*}
\lim_{h \rightarrow 0} 
\E \bigl[
\sup_{s,t \in [0,T], \vert t-s \vert \leq h} 
\| \beta_{s} - \beta_{t} \|_{0} \bigr] =0.
\end{equation*}
In other words, $(\tilde{m}_{t})_{t \in [t_{0},T]}$ stands for the flow of conditional 
marginal laws of $(X_{t})_{t \in [t_{0},T]}$ given ${\mathcal F}_{T}$,
where $(X_{t})_{t \in [t_{0},T]}$ solves the stochastic differential equation:
\begin{equation*}
dX_{t} = 
-
\beta_{t}\bigl(X_{t}
+ \sqrt{2}(W_{t}-W_{t_{0}})
\bigr) dt + \sqrt{2}dB_{t}, \quad t \in [t_{0},T],
\end{equation*}
$X_{t_{0}}$ being distributed according to $m_{t_{0}}$
conditional on ${\mathcal F}_{T}$.  In particular, 
there exists a deterministic constant $C$ such that,
with probability 1, 
 for all $t_{0} \leq t \leq t+h \leq T$,
\begin{equation*}
{\mathbf d}_{1}(\tilde m_{t+h},\tilde{m}_{t}) \leq C \sqrt{h}.
\end{equation*} 
Given some $t \in [t_{0},T]$, we denote by 
$m_{t}=(\cdot \mapsto \cdot + \sqrt{2}(W_{t}-W_{t_{0}})) \sharp 
\tilde m_{t}$ 
the push-forward of $\tilde m_{t}$ by the application 
$\T^d \ni x \mapsto x + W_{t} - W_{t_{0}} \in \T^d$ (so that 
$m_{t_{0}} = \tilde{m}_{t_{0}}$).  
\vspace{15pt}

We then have the local It\^o-Taylor expansion:

\begin{Lemma}
\label{lem:ito:local}
Under the above assumption, we can find a family 
of real-valued random variables $(\varepsilon_{s,t})_{s,t \in [t_{0},T]  :  s \leq t}$ such that 
\begin{equation*}
\lim_{h \searrow 0} \sup_{s,t \in [t_{0},T] : \vert s -t \vert \leq h} {\mathbb E} \bigl[  \vert \varepsilon_{s,t} \vert \bigr] = 0,
\end{equation*}
and, for any $t \in [t_{0},T]$, 
\begin{equation*}
\begin{split}
&
\frac1h
\Bigl[
\E \bigl[ 
U\bigl(t+h,x+\sqrt{2} ( W_{t+h}-W_{t_{0}}),m_{t+h}
\bigr)
-U\bigl(t+h,x + \sqrt{2} ( W_{t+h}-W_{t_{0}}),m_{t}\bigr) \vert {\mathcal F}_{t}
\bigr]
\Bigr]
\\
&=  \Delta_{x} U\bigl(t,x + \sqrt{2}
( W_{t}-W_{t_{0}}),m_{t}
\bigr)  + 2 \int_{\T^d}
\textrm{\rm div}_{y} 
\bigl[
D_{m} U
\bigr]
 \bigl( t,x+ \sqrt{2}( W_{t}-W_{t_{0}}),m_{t},y
 \bigr) d {m}_{t}(y)
 \\
&\hspace{15pt}-  \int_{\T^d}
D_{m} 
U
 \bigl( t,x+ \sqrt{2} ( W_{t}-W_{t_{0}}),m_{t},y
 \bigr) 
 \cdot \beta_{t}(y)
 d m_{t}(y)
 \\
&\hspace{15pt} + 2 
\int_{\T^d}
\textrm{\rm div}_{x}
\bigl[
D_{m} 
U
\bigr]
\bigl( t, x+ \sqrt{2} ( W_{t}-W_{t_{0}}),
 m_{t},y
\bigr) dm_{t}(y)
\\
&\hspace{15pt} +
\int_{[\T^d]^2}
{\rm Tr} \Bigl[ D^2_{mm}
U\Bigr]
\bigl( t, x + \sqrt{2}( W_{t}-W_{t_{0}}),
 m_{t},y,y'\bigr) 
 d m_{t}
(y)
d m_{t}
(y') + 
\varepsilon_{t,t+h}.
\end{split}
\end{equation*}
\end{Lemma}

\begin{proof}
Without any loss of generality, we assume that $t_{0}=0$.
Moreover, throughout the analysis, we shall use 
the following variant of \eqref{eq:partie2:trick}:
For two random 
processes $(\gamma_{t})_{t \in [{0},T]}$
and $(\gamma_{t}')_{t \in [{0},T]}$,
with paths in 
${\mathcal C}^0([0,T],\cC^0(E))$
and ${\mathcal C}^0([0,T],F)$ respectively,
where $E$ is a compact metric space (the distance being denoted by $d_{E}$)
and $F$ is a metric space (the distance being denoted by $d_{F}$),
satisfying
\begin{equation*}
\begin{split}
&
\essup_{\omega \in \Omega}\sup_{t \in [0,T]} \| \gamma_{t} \|_{0}
< \infty,
\end{split}
\end{equation*}
it must hold that 
\begin{equation}
\label{eq:partie2:trick:bb}
\lim_{h \searrow 0} 
\sup_{s,t \in [0,T] : \vert s-t \vert \leq h}
\E \bigl[ \vert \eta_{s,t} \vert \bigr]
=0, \quad 
\eta_{s,t} = 
\sup_{r \in [s,t]}
\sup_{x,y \in E : d_{E}(x,y) \leq \sup_{r \in [s,t]} d_{F}(\gamma_{r}',\gamma_{s}')}
\big\vert \gamma_{r}(y) -\gamma_{s}(x) \bigr\vert.
\end{equation}

Now, for given $t \in [0,T)$ and $h \in (0,T-t]$, 
we let
$\delta_{h} W_{t} = W_{t+h} - W_{t}$ 
and
$\delta_{h} m_{t} = m_{t,t+h} - m_{t}$. By
Taylor-Lagrange's formula, we can find some random variable $\lambda$ with values 
in $[0,1]$
\footnote{
The fact that $\lambda$ is a random variable may be justified as follows. 
Given a continuous mapping 
$\varphi$ from $\T^d \times {\mathcal P}(\T^d)$ into $\R$
and two random variables $(X,m)$ and $(X',m')$ with values in 
$(\R^d,{\mathcal P}(\T^d))$ such that 
the mapping $[0,1] \ni c \mapsto \varphi(cX'+(1-c)X,cm'+(1-c)m)$
vanishes at least once, the quantity 
$\lambda = \inf \{c \in [ 0,1] :
\varphi(cX'+(1-c)X,cm'+(1-c)m) =0 \}$ defines a random variable
since $\{ \lambda > c \} = 
\cap_{n \in {\mathbb N} \setminus \{0\}}
\cap_{c' \in {\mathbb Q} \in [0,c]}
\{ \varphi(c' X'+(1-c')X,c'm'+(1-c')m) \varphi( X,m) >1/n \}$.
} such that 
\begin{equation}
\label{eq:partie2:proof:master:equation:3}
\begin{split}
&U\bigl(t+h,x+\sqrt{2} W_{t+h},m_{t+h}
\bigr)
-U\bigl(t+h,x+\sqrt{2} W_{t},m_{t}\bigr)
\\
&= \sqrt{2} D_{x} U \bigl( t+h,x+\sqrt{2} W_{t},m_{t}
\bigr) \cdot \delta_{h} W_{t}
+ \int_{\T^d}
 \frac{\delta U}{\delta m}
 \bigl( t+h,x+\sqrt{2} W_{t},m_t,y
 \bigr) d \bigl( \delta_{h} m_{t} \bigr)(y) 
 \\
 &\hspace{1pt}+ D^2_{x} U \bigl( t+h, x  
 +\sqrt{2} W_{t}
 +
 \sqrt{2}
 \lambda 
 \delta_{h} W_{t},
 m_{t} + \lambda 
\delta_{h} m_{t}
\bigr) 
\cdot (\delta_{h} W_{t})^{\otimes 2}
\\
&\hspace{1pt} +
\sqrt{2}
\int_{\T^d}
D_{x} \frac{\delta U}{\delta m}
\bigl( t+h, x 
+\sqrt{2} W_{t}
+ \sqrt{2}
 \lambda 
 \delta_{h} W_{t},
 m_{t} + \lambda 
\delta_{h} m_{t},y
\bigr)
\cdot  
 \delta_{h} W_{t}
d \bigl( 
\delta_{h} m_{t}
\bigr)
(y)
\\
&\hspace{1pt}
+ \frac12
\int_{[\T^d]^2}
\frac{\delta^2 U}{\delta m^2}
\bigl( t+h, x
+\sqrt{2} W_{t}
 + \sqrt{2}
 \lambda 
 \delta_{h} W_{t},
 m_{t} + \lambda 
\delta_{h} m_{t},y,y'
\bigr) 
d \bigl( 
\delta_{h} m_{t}
\bigr)
(y)
d \bigl( 
\delta_{h} m_{t}
\bigr)(y')
\\
&= T^1_{h} + T^{2}_{h} + T^3_{h} + T^4_{h} + T^{5}_{h}, 
 \end{split}
 \end{equation}
 where we used the dot ``$\cdot$'' to denote the inner product in Euclidean spaces. 
 Part of the analysis relies on the following decomposition. 
 Given a bounded and Borel measurable 
 function $\varphi : \T^d \rightarrow \R$, 
 it holds that
 \begin{equation}
 \label{eq:partie2:proof:master:equation:10}
 \begin{split}
 &\int_{\T^d} \varphi(y) d \bigl( \delta_{h} m_{t}
 \bigr) (y)
 \\
  &=
 \int_{\T^d}
 \varphi(y) d m_{t+h}(y) 
 - \int_{\T^d}
 \varphi(y) d m_{t}(y) 
 \\
&= \int_{\T^d} \varphi \bigl( y +
\sqrt{2} W_{t+h} 
\bigr) d \tilde{m}_{t+h}(y) 
- \int_{\T^d}
 \varphi(y+
\sqrt{2} W_{t}) d \tilde{m}_{t}(y) 
 \\
 &= \int_{\T^d}
 \varphi \bigl( y 
 +
\sqrt{2} W_{t+h}
\bigr) d \bigl( 
\tilde{m}_{t+h}
- \tilde{m}_{t}\bigr)
(y) 
+ \int_{\T^d}
 \Bigl[
 \varphi \bigl( y 
+
\sqrt{2} W_{t+h} 
\bigr)
-
 \varphi( y+
\sqrt{2} W_{t})
\Bigr] d \tilde{m}_{t}(y)
\\
&= \int_{\T^d}
 \varphi \bigl( y 
 +
\sqrt{2} W_{t+h}
\bigr) d \bigl( 
\tilde{m}_{t+h}
- \tilde{m}_{t}\bigr)
(y) 
+ \int_{\T^d}
 \Bigl[
 \varphi \bigl( y 
+
\sqrt{2} \delta_{h}
W_{t} 
\bigr)
-
 \varphi( y)
\Bigr] d {m}_{t}(y). 
 \end{split}
 \end{equation}
 In particular,
 whenever $\varphi$ is a bounded Borel 
 measurable mapping from $[\T^d]^2$ into $\R$, 
it holds that
  \begin{align}
 &\int_{[\T^d]^2}
\varphi(y,y') d  
 \bigl( \delta_{h} m_{t}
 \bigr)(y) d \bigl( \delta_{h} m_{t}
 \bigr)(y') \nonumber
 \\
 &= \int_{[\T^d]^2}
 \varphi \bigl( y 
 +
\sqrt{2} W_{t+h},y'\bigr) 
 d \bigl( \tilde{m}_{t+h}
 - \tilde{m}_{t}
 \bigr)(y) d \bigl( \delta_{h} m_{t}
 \bigr)(y') \nonumber
 \\
 &\hspace{15pt}
 + \int_{[\T^d]^2}
\Bigl[
 \varphi \bigl( y 
 +
\sqrt{2}\delta_{h} W_{t},y'\bigr) 
 -
 \varphi( y,y') 
\Bigr] 
 d m_{t}(y) d \bigl( \delta_{h} m_{t}
 \bigr)(y') \nonumber
 \\
 &= \int_{[\T^d]^2}
 \varphi \bigl( y 
 +
\sqrt{2} W_{t+h},y'+
\sqrt{2} W_{t+h}
\bigr) 
 d \bigl( \tilde{m}_{t+h}
 - \tilde{m}_{t}
 \bigr)(y) 
d \bigl( \tilde{m}_{t+h}
 - \tilde{m}_{t}
 \bigr)(y')   \label{eq:partie2:proof:master:equation:11}
 \\
 &\hspace{15pt} +  
 \int_{[\T^d]^2}
\Bigl[
 \varphi \bigl( y+
\sqrt{2} W_{t+h},y'
 +
\sqrt{2} \delta_{h}W_{t}
 \bigr) 
 -  \varphi \bigl( y
+
\sqrt{2} W_{t+h},y'
 \bigr) 
 \Bigr]
 d \bigl( \tilde{m}_{t+h}
 - \tilde{m}_{t}
 \bigr)(y) d  m_{t}(y') \nonumber
 \\
 &\hspace{15pt} +  
 \int_{[\T^d]^2}
\Bigl[
 \varphi \bigl( y 
 +
\sqrt{2}\delta_{h} W_{t},y'+
\sqrt{2} W_{t+h}
 \bigr) 
 -  \varphi \bigl( y,y'+
\sqrt{2} W_{t+h}
 \bigr) 
 \Bigr]
 d  m_{t}(y)
 d \bigl( \tilde{m}_{t+h}
 - \tilde{m}_{t}
 \bigr)(y') \nonumber
 \\
 &\hspace{15pt}
 + \int_{[\T^d]^2}
\Bigl[
 \varphi \bigl( y+
\sqrt{2} \delta_{h} W_{t},y'
 +
\sqrt{2} \delta_{h}W_{t}
 \bigr) 
 -
  \varphi \bigl( y +
\sqrt{2} \delta_{h} W_{t},y' \bigr)  \nonumber
 \\
&\hspace{60pt}
 -
  \varphi \bigl( y,y'+
\sqrt{2} \delta_{h} W_{t}\bigr) 
+ \varphi ( y,y') 
\Bigr] 
 d m_{t}(y) d  m_{t}(y'). \nonumber
 \end{align}
 We now proceed with the analysis of 
 \eqref{eq:partie2:proof:master:equation:3}.  
 We start with $T^{1}_{h}$. It is pretty clear
that 
 \begin{equation}
 \label{eq:partie2:proof:master:equation:41} 
 \E \bigl[ T^1_{h} \vert {\mathcal F}_{t} \bigr] = 0. 
 \end{equation}
Look at now the term $T^2_{h}$. 
Following 
\eqref{eq:partie2:proof:master:equation:10}, write it
\begin{equation}
\label{eq:partie2:proof:master:equation:2}
\begin{split}
T^2_{h} &= \int_{\T^d}
 \frac{\delta U}{\delta m}
 \bigl( t+h,x
 +\sqrt{2} W_{t},m_{t},y
+ \sqrt{2} 
W_{t+h}
 \bigr) d \bigl( \tilde{m}_{t+h}
  -\tilde{m}_{t}\bigr)(y)
  \\
&\hspace{15pt} + 
\int_{\T^d}
 \Bigl[
 \frac{\delta U}{\delta m}
 \bigl( t+h,x
  +\sqrt{2} W_{t},m_{t},y
+ \sqrt{2}\delta_{h}
W_{t}
 \bigr)
 \\
&\hspace{100pt} -
 \frac{\delta U}{\delta m}
 \bigl( t+h,x +\sqrt{2} W_{t},m_{t},y
 \bigr)
 \Bigr]
  d  {m}_{t} (y)
  \\
  &= T^{2,1}_{h} + T^{2,2}_{h}. 
\end{split}
\end{equation}
By the PDE satisfied by $(\tilde{m}_{t})_{t \in [t_{0},T]}$, 
we have
\begin{equation}
\label{eq:T21:h}
\begin{split}
T^{2,1}_{h} &=
\int_{t}^{t+h}
ds \int_{\T^d}
\Delta_{y} \frac{\delta U}{\delta m}
 \bigl( t+h,x +\sqrt{2} W_{t},m_{t},y
+ \sqrt{2} 
W_{t+h}
 \bigr) d \tilde{m}_{s}(y)
 \\
 &\hspace{10pt}
 - 
 \int_{t}^{t+h}
ds \int_{\T^d}
D_{y} \frac{\delta U}{\delta m}
 \bigl( t+h,x +\sqrt{2} W_{t},m_{t},y
+ \sqrt{2} 
W_{t+h}
 \bigr) \cdot
\beta_{s}\bigl(y+\sqrt{2} W_{s} \bigr) 
 d \tilde{m}_{s}(y).
\end{split}
\end{equation}
Therefore, taking the conditional expectation, dividing by $h$
and using the fact that 
$m_{t}$ is the push-forward of 
$\tilde m_{t}$ by the mapping 
$\T^d \ni x \mapsto x + \sqrt{2} W_{t}$ (pay attention that the
measures below are $m_{t}$ and not $\tilde{m}_{t}$), we can write
\begin{equation*}
\begin{split}
\frac{1}{h}\E \bigl[ T^{2,1}_{h} \vert {\mathcal F}_{t}
\bigr]
&=
\int_{\T^d}
\Delta_{y} \frac{\delta U}{\delta m}
 \bigl( t,x +\sqrt{2} W_{t},m_{t},y
 \bigr) d {m}_{t}(y)
\\
&\hspace{15pt} -  \int_{\T^d}
D_{m} U
 \bigl( t,x +\sqrt{2} W_{t},m_{t},y
 \bigr) 
\cdot \beta_{t}(y)
 d m_{t}(y) + \varepsilon_{t,t+h},
\end{split}
\end{equation*}
where, as in the statement, $(\varepsilon_{s,t})_{0 \leq s \leq t \leq T}$ is a generic notation for denoting 
a family of random variables that satisfies
\begin{equation}
\label{eq:partie2:trick:3}
\lim_{h \searrow 0}
\sup_{\vert t-s \vert \leq h}
\E \bigl[  \vert \varepsilon_{s,t} \vert
\bigr] = 0.
\end{equation}
Here we used the same trick as in 
\eqref{eq:partie2:trick:bb}
to prove \eqref{eq:partie2:trick:3}
(see also \eqref{eq:partie2:trick}). Indeed, 
by a first application of \eqref{eq:partie2:trick:bb}, 
we can write 
\begin{equation*}
\begin{split}
T^{2,1}_{h} 
&=
\int_{t}^{t+h}
ds \int_{\T^d}
\Delta_{y} \frac{\delta U}{\delta m}
 \bigl(s,x +\sqrt{2} W_{s},m_{s},y
+ \sqrt{2} 
W_{s}
 \bigr) d \tilde{m}_{s}(y)
 \\
 &\hspace{15pt}
 - 
 \int_{t}^{t+h}
ds \int_{\T^d}
D_{y} \frac{\delta U}{\delta m}
 \bigl( s,x +\sqrt{2} W_{s},m_{s},y
+ \sqrt{2} 
W_{s}
 \bigr) \cdot
\beta_{s}\bigl(y+\sqrt{2} W_{s} \bigr) 
 d \tilde{m}_{s}(y) + h \varepsilon_{t,t+h}.
\end{split}
\end{equation*}
Then, we can apply \eqref{eq:partie2:trick:bb} once again 
with 
\begin{equation*}
\begin{split}
\gamma_{s}(x) &= 
\int_{\T^d}
\Delta_{y} \frac{\delta U}{\delta m}
 \bigl(s,x +\sqrt{2} W_{s},m_{s},y
+ \sqrt{2} 
W_{s}
 \bigr) d \tilde{m}_{s}(y)
 \\
 &\hspace{15pt}
 - 
\int_{\T^d}
D_{y} \frac{\delta U}{\delta m}
 \bigl( s,x +\sqrt{2} W_{s},m_{s},y
+ \sqrt{2} 
W_{s}
 \bigr) \cdot
\beta_{s}\bigl(y+\sqrt{2} W_{s} \bigr) 
 d \tilde{m}_{s}(y).
\end{split}
\end{equation*}

Using It\^o's formula to handle the second term in 
\eqref{eq:partie2:proof:master:equation:2}, we get
in a similar way 
\begin{equation}
\label{eq:partie2:proof:master:equation:4} 
\begin{split}
\frac{1}{h}
\E \bigl[ T^{2}_{h} \vert {\mathcal F}_{t}\bigr]
&=
2 \int_{\T^d}
\Delta_{y} \frac{\delta U}{\delta m}
 \bigl( t,x +\sqrt{2} W_{t},m_{t},y
 \bigr) d {m}_{t}(y)
\\
&\hspace{15pt} -  \int_{\T^d}
D_{m} U
 \bigl( t,x +\sqrt{2} W_{t},m_{t},y
 \bigr) 
\cdot \beta_{t}(y)
 d m_{t}(y)
  + \varepsilon_{t,t+h}.
\end{split}
\end{equation}
Turn now to 
$T^3_{h}$
in 
\eqref{eq:partie2:proof:master:equation:3}. 
Using again 
\eqref{eq:partie2:trick:bb}, it is quite clear that 
\begin{equation}
\label{eq:partie2:proof:master:equation:42}
\frac{1}{h}
\E 
\bigl[ T^3_{h}
\vert {\mathcal F}_{t}
\bigr]
= \Delta_{x} U(t,x +\sqrt{2} W_{t},m_{t})   + \varepsilon_{t,t+h}. 
\end{equation}
We now handle $T^4_{h}$. 
Following 
\eqref{eq:partie2:proof:master:equation:10}, we write
\begin{equation*}
\begin{split}
&T^4_{h}
\\
&= \sqrt{2}
\int_{\T^d}
\Bigl[
D_{x} \frac{\delta U}{\delta m}
\bigl( t +h, x + \sqrt{2} W_{t} + \sqrt{2}
 \lambda 
 \delta_{h} W_{t},
 m_{t} + \lambda 
\delta_{h} m_{t},y
+ \sqrt{2}
W_{t+h}
\bigr)
\\
&\hspace{300pt}
 \cdot
 \delta_{h} W_{t}
\Bigr]
d \bigl( 
\tilde{m}_{t+h}
- \tilde{m}_{t}
\bigr)
(y)
\\
&\hspace{15pt}
+\sqrt{2}
\int_{\T^d}
\Bigl[
D_{x} \frac{\delta U}{\delta m}
\bigl( t+h , x + \sqrt{2} W_{t} + \sqrt{2}
 \lambda 
 \delta_{h} W_{t},
 m_{t} + \lambda 
\delta_{h} m_{t},y
+ \sqrt{2}
\delta_{h}W_{t}
\bigr) 
\\
&\hspace{50pt}
- 
D_{x} \frac{\delta U}{\delta m}
\bigl( t+h, x + \sqrt{2} W_{t} + \sqrt{2}
 \lambda 
 \delta_{h} W_{t},
 m_{t} + \lambda 
\delta_{h} m_{t},y
\bigr) 
\Bigr] \cdot
 \delta_{h} W_{t}
d {m}_{t}
(y)
\\
&= T^{4,1}_{h} + T^{4,2}_{h}. 
\end{split}
\end{equation*}
Making use of the forward Fokker-Planck equation 
for $(\tilde{m}_{t})_{t \in [t_{0},T]}$ 
as in the proof of \eqref{eq:partie2:proof:master:equation:4}, we get 
that
\begin{equation*}
\frac{1}{h} \E
\bigl[
T^{4,1}_{h}
\vert {\mathcal F}_{t}
\bigr] = \varepsilon_{t,t+h}.
\end{equation*}

Now, by Taylor-Lagrange's formula, we can find another 
$[0,1]$-valued
random variable $\lambda'$ 
such that
\begin{equation*}
\begin{split}
T^{4,2}_{h}
&= 2
\int_{\T^d}
\Bigl[
D_{y} 
D_{x} \frac{\delta U}{\delta m}
\bigl( t+h, x + \sqrt{2} W_{t} + \sqrt{2}
 \lambda 
 \delta_{h} W_{t},
 m_{t} + \lambda 
\delta_{h} m_{t},y
+ \sqrt{2}
\lambda'
\delta_{h}
W_{t}
\bigr)
\\ 
&\hspace{330pt}
\cdot
( \delta_{h} W_{t}
)^{\otimes 2}
\Bigr]
d {m}_{t}
(y).
\end{split}
\end{equation*}
And, then,
\begin{equation}
\label{eq:partie2:proof:master:equation:44}
\begin{split}
\frac{1}{h}
\E \bigl[ T^{4}_{h}
\vert {\mathcal F}_{t}
\bigr]
&=  \frac{1}{h}
\E \bigl[ T^{4,2}_{h}
\vert {\mathcal F}_{t}
\bigr] + \varepsilon_{t,t+h}
\\
&= 2 
\int_{\T^d}
\textrm{div}_{y}
\bigl[
D_{x} \frac{\delta U}{\delta m}
\bigr]
\bigl( t, x +\sqrt{2} W_{t},
 m_{t},y
\bigr) dm_{t}(y)
+ \varepsilon_{t,t+h}
\\
&= 
2 
\int_{\T^d}
\textrm{div}_{x}
\bigl[
D_{y} \frac{\delta U}{\delta m}
\bigr]
\bigl( t, x +\sqrt{2} W_{t},
 m_{t},y
\bigr) dm_{t}(y)
+ \varepsilon_{t,t+h}. 
\end{split}
\end{equation}
It finally remains to handle $T^5_{h}$. 
Thanks to 
\eqref{eq:partie2:proof:master:equation:11}, we write
\begin{align}
&T^5_{h} \nonumber
\\
&=
\frac12
\int_{[\T^d]^2}
\frac{\delta^2 U}{\delta m^2}
\bigl( t+h, x + \sqrt{2} W_{t} + \sqrt{2}
 \lambda 
 \delta_{h} W_{t},
 m_{t} + \lambda 
\delta_{h} m_{t},y 
+ 
\sqrt{2}  W_{t+h}
,y'
+ 
\sqrt{2}  W_{t+h}
\bigr)  \nonumber
\\
&\hspace{200pt} d \bigl( 
\tilde{m}_{t+h}
- \tilde{m}_{t}
\bigr)
(y)
d \bigl( 
\tilde{m}_{t+h}
- \tilde{m}_{t}
\bigr)
(y') \nonumber
\\
&\hspace{5pt}
+ \frac12
\int_{[\T^d]^2}
\biggl[
\frac{\delta^2 U}{\delta m^2}
\bigl( t+h, x + \sqrt{2} W_{t}
+
\sqrt{2}
 \lambda 
 \delta_{h} W_{t},
 m_{t} + \lambda 
\delta_{h} m_{t},y 
+ 
\sqrt{2} \delta_{h}W_{t}
,y'
+ 
\sqrt{2}  W_{t+h}
\bigr)  \nonumber
\\
&\hspace{20pt} 
-
\frac{\delta^2 U}{\delta m^2}
\bigl( t+h, x + \sqrt{2} W_{t} + \sqrt{2}
 \lambda 
 \delta_{h} W_{t},
 m_{t} + \lambda 
\delta_{h} m_{t},y,y'
+ 
\sqrt{2}  W_{t+h}
\bigr) 
\biggl]
d {m}_{t}
(y) \nonumber
\\
&\hspace{350pt}d \bigl( 
\tilde{m}_{t+h}
- \tilde{m}_{t}
\bigr)
(y') \nonumber
\\
&\hspace{5pt}
+ \frac12
\int_{[\T^d]^2}
\biggl[
\frac{\delta^2 U}{\delta m^2}
\bigl( t+h, x + \sqrt{2} W_{t} + \sqrt{2}
 \lambda 
 \delta_{h} W_{t},
 m_{t} + \lambda 
\delta_{h} m_{t},y 
+ 
\sqrt{2}  W_{t+h}
,y'
+ 
\sqrt{2} \delta_{h} W_{t}
\bigr) \nonumber
\\
&\hspace{20pt} 
-
\frac{\delta^2 U}{\delta m^2}
\bigl( t+h, x + \sqrt{2} W_{t} + \sqrt{2}
 \lambda 
 \delta_{h} W_{t},
 m_{t} + \lambda 
\delta_{h} m_{t},y 
+ 
\sqrt{2}  W_{t+h},y' \bigr) 
\biggl]
d \bigl( 
\tilde{m}_{t+h}
- \tilde{m}_{t}
\bigr)
(y) \nonumber
\\
&\hspace{350pt}
d {m}_{t}
(y')\nonumber
\\
&\hspace{5pt} + \frac12
\int_{[\T^d]^2}
\biggl[
\frac{\delta^2 U}{\delta m^2}
\bigl( t+h, x + \sqrt{2} W_{t} + \sqrt{2}
 \lambda 
 \delta_{h} W_{t},
 m_{t} + \lambda 
\delta_{h} m_{t},y 
+ 
\sqrt{2} \delta_{h} W_{t}
,y'
+ 
\sqrt{2} \delta_{h}  W_{t}
\bigr) \nonumber
\\
&\hspace{20pt} 
-
\frac{\delta^2 U}{\delta m^2}
\bigl( t+h, x +
\sqrt{2} W_{t}+
 \sqrt{2}
 \lambda 
 \delta_{h} W_{t},
 m_{t} + \lambda 
\delta_{h} m_{t},y + 
\sqrt{2} \delta_{h} W_{t},y'
\bigr) \nonumber
\\
&\hspace{20pt} 
-
\frac{\delta^2 U}{\delta m^2}
\bigl( t+h, x + \sqrt{2}
 \lambda 
 \delta_{h} W_{t},
 m_{t} + \lambda 
\delta_{h} m_{t},y,y' + \sqrt{2}
\delta_{h}W_{t}
\bigr) \nonumber
\\
&\hspace{20pt}+
\frac{\delta^2 U}{\delta m^2}
\bigl( t+h, x + \sqrt{2} W_{t} + \sqrt{2}
 \lambda 
 \delta_{h} W_{t},
 m_{t} + \lambda 
\delta_{h} m_{t},y,y'\bigr) 
\biggl]
d {m}_{t}
(y)
d {m}_{t}
(y') \nonumber
\\
&= \frac12 \bigl( T^{5,1}_{h} + T^{5,2}_{h}
+ T^{5,3}_{h} + T^{5,4}_{h} \bigr). 
\label{eq:T5:h}
\end{align}
Making use of the Fokker-Planck 
equation satisfied by $(\tilde{m}_{t})_{t \in [t_{0},T]}$
together with 
the regularity 
assumptions of 
$\delta^2 U/\delta m^2$ in Definition \ref{def:master:eq:2nd:order}, 
it is readily seen that 
\begin{equation}
\label{eq:partie2:proof:master:equation:13} 
\frac{1}{h}
{\mathbb E}
\bigl[ T^{5,1}_{h}
+
T^{5,2}_{h}
+
T^{5,3}_{h}
\vert {\mathcal F}_{t}
\bigr]
= \varepsilon_{t,t+h}. 
\end{equation}
Focus now on $T^{5,4}_{h}$. With obvious notation, 
write it under the form 
\begin{equation}
\label{eq:partie2:proof:master:equation:12} 
T^{5,4}_{h}
= T^{5,4,1}_{h}
- T^{5,4,2}_{h} 
- T^{5,4,3}_{h}
+ T^{5,4,4}_{h}. 
\end{equation}
Performing a second-order Taylor expansion, we get
\begin{equation*}
\begin{split}
&T^{5,4,1}_{h}
\\
&=
\int_{[\T^d]^2}
\frac{\delta^2 U}{\delta m^2}
\bigl( t+h, x + \sqrt{2}
W_{t} + \sqrt{2}
 \lambda 
 \delta_{h} W_{t},
 m_{t} + \lambda 
\delta_{h} m_{t},y,y'\bigr) d {m}_{t}
(y)
d {m}_{t}
(y')
\\
&\hspace{1pt} + 
\int_{[\T^d]^2}
\sqrt{2}
D_{y}\frac{\delta^2 U}{\delta m^2}
\bigl( t+h, x 
+ \sqrt{2} W_{t}
+ \sqrt{2}
 \lambda 
 \delta_{h} W_{t},
 m_{t} + \lambda 
\delta_{h} m_{t},y ,y'\bigr) 
\cdot
 \delta_{h} W_{t}
 d {m}_{t}
(y)
d {m}_{t}
(y')
 \\
&\hspace{1pt} + 
\int_{[\T^d]^2}
\sqrt{2}
D_{y'}\frac{\delta^2 U}{\delta m^2}
\bigl( t+h, x + \sqrt{2}
W_{t} + \sqrt{2}
 \lambda 
 \delta_{h} W_{t},
 m_{t} + \lambda 
\delta_{h} m_{t},y,y'\bigr) 
\cdot
 \delta_{h} W_{t}
 d {m}_{t}
(y)
d {m}_{t}
(y')
\\
&\hspace{1pt}
+
\int_{[\T^d]^2}
D_{y}^2\frac{\delta^2 U}{\delta m^2}
\bigl( t +h, x + \sqrt{2}
W_{t} + \sqrt{2}
 \lambda 
 \delta_{h} W_{t},
 m_{t} + \lambda 
\delta_{h} m_{t},y,y'\bigr) 
\cdot
\bigl(
 \delta_{h} W_{t} \bigr)^{\otimes 2}
 d {m}_{t}
(y)
d {m}_{t}
(y')
\\
&\hspace{1pt}
+
\int_{[\T^d]^2}
D_{y'}^2\frac{\delta^2 U}{\delta m^2}
\bigl( t+h, x + \sqrt{2}
W_{t} + \sqrt{2}
 \lambda 
 \delta_{h} W_{t},
 m_{t} + \lambda 
\delta_{h} m_{t},y,y'\bigr) 
\cdot
\bigl(
 \delta_{h} W_{t} \bigr)^{\otimes 2}
d {m}_{t}
(y)
d {m}_{t}
(y')
 \\
 &\hspace{1pt}
+
\int_{[\T^d]^2}
2 D_{y} D_{y'}\frac{\delta^2 U}{\delta m^2}
\bigl( t+h, x + \sqrt{2}
W_{t} + \sqrt{2}
 \lambda 
 \delta_{h} W_{t},
 m_{t} + \lambda 
\delta_{h} m_{t},y,y'\bigr) 
\cdot
\bigl(
 \delta_{h} W_{t} \bigr)^{\otimes 2}
 d {m}_{t}
(y)
d {m}_{t}
(y')
\\
&\hspace{15pt} + \varepsilon_{t,t+h}
 \\
 &= T^{5,4,4}_{h} + I^1_{h} + I^2_{h} + J^{1}_{h} +  J^{2}_{h}
 + J^{1,2}_{h} 
 + h \varepsilon_{t,t+h}.
\end{split}
\end{equation*}
Similarly, we get 
\begin{equation*}
\begin{split}
&T^{5,4,2}_{h}
= T^{5,4,4}_{h} + I^1_{h} + J^{1}_{h} 
 + h \varepsilon_{t,t+h},
\\
&T^{5,4,3}_{h}
=
T^{5,4,4}_{h} + I^2_{h} + J^{2}_{h} +
h
 \varepsilon_{t,t+h}, 
\end{split}
\end{equation*}
from which, together 
with \eqref{eq:partie2:proof:master:equation:12}, 
we deduce that 
\begin{equation}
\label{T:5,4:h}
T^{5,4}_{h} = J^{1,2}_{h} + 
h
\varepsilon_{t,t+h},
\end{equation}
and then, with 
\eqref{eq:partie2:proof:master:equation:13}, 
\begin{equation}
\label{eq:partie2:proof:master:equation:14} 
\begin{split}
\frac{1}{h}
\E \bigl[ T^{5}_{h}
\vert {\mathcal F}_{t}
\bigr]
&=
\frac{1}{2h}
\E \bigl[ T^{5,4}_{h}
\vert {\mathcal F}_{t}
\bigr] + \varepsilon_{t,t+h}
\\
&= \int_{[\T^d]^2}
{\rm Tr} \Bigl[ D_{y} D_{y'}\frac{\delta^2 U}{\delta m^2}
\bigl( t, x+\sqrt{2} W_{t} ,
 m_{t},y,y'\bigr) 
\Bigr] d m_{t}
(y)
d m_{t}
(y') 
+ \varepsilon_{t,t+h}.
\end{split}
\end{equation}
From 
\eqref{eq:partie2:proof:master:equation:3},
\eqref{eq:partie2:proof:master:equation:41},
\eqref{eq:partie2:proof:master:equation:4},  
\eqref{eq:partie2:proof:master:equation:42},
\eqref{eq:partie2:proof:master:equation:44}
and 
\eqref{eq:partie2:proof:master:equation:14}, 
we deduce that, 
\begin{equation*}
\begin{split}
&
\frac1h 
\Bigl[
\E \bigl[ 
U\bigl(t+h,x+\sqrt{2} W_{t+h},m_{t}
\bigr)
-U\bigl(t+h,x+\sqrt{2} W_{t},m_{t}\bigr) \vert {\mathcal F}_{t}
\bigr]
\Bigr]
\\
&= \Delta_{x} U(t,x+\sqrt{2} W_{t},m_{t})  + 2 \int_{\T^d}
\textrm{div}_{y} 
\bigl[
D_{m} U
\bigr]
 \bigl( t,x+\sqrt{2}W_{t},m_{t},y
 \bigr) d {m}_{t}(y)
 \\
&\hspace{15pt}-  \int_{\T^d}
D_{m} 
U
 \bigl( t,x+\sqrt{2}W_{t},m_{t},y
 \bigr) 
 \cdot \beta_{t}(y)
 d m_{t}(y)
 \\
&\hspace{15pt} + 2 
\int_{\T^d}
\textrm{div}_{x}
\bigl[
D_{m} 
U
\bigr]
\bigl( t, x+\sqrt{2}W_{t},
 m_{t},y
\bigr) dm_{t}(y)
\\
&\hspace{15pt} +
\int_{[\T^d]^2}
{\rm Tr} \Bigl[ D^2_{mm}
U
\bigl( t, x +\sqrt{2}W_{t},
 m_{t},y,y'\bigr) 
\Bigr] d m_{t}
(y)
d m_{t}
(y') + \varepsilon_{t,t+h},
\end{split}
\end{equation*}
which completes the proof. 
\end{proof}

\subsection{Proof of Corollary \ref{c.sec5.MFGstoch}}
\label{subse:proof:c.sec5.MFGstoch}

We are now ready to come back to the well-posedness of the stochastic MFG system 
\be
\label{e.MFGstocsyst.sec5}
\left\{
\begin{array}{l}
d_{t} u_{t} = \bigl\{ -  2 \Delta u_{t} + H(x,Du_{t}) - F(x,m_{t}) -  \sqrt{2} {\rm div}(v_{t}) \bigr\} dt+ v_{t} \cdot \sqrt{2} dW_{t},\\
d_{t} m_{t} = \bigl[  2 \Delta m_{t} + {\rm div} \bigl( m_{t} D_{p} H(x,D u_{t}) 
\bigr) \bigr] dt - \sqrt{2} {\rm div} ( m_{t} dW_{t} \bigr), \qquad {\rm in}\; [t_0,T]\times \T^d,
\\
 m_{t_0}=m_0, \; u_T(x)= G(x, m_T) \qquad {\rm in}\; \T^d.
\end{array}\right.
\ee
For simplicity of notation, we prove the existence and uniqueness of the solution for $t_0=0$. 
\vspace{4pt}

\textit{First step. Existence of a solution.}
We start with the solution $(\tilde{u}_t,\tilde{m}_t,\tilde{M}_t)_{t \in [0,T]}$ to the system 
\begin{equation}
\label{eq:StochNashTransfo}
\left\{\begin{array}{l}
d_{t} \tilde{m}_{t} = \bigl\{ \Delta \tilde{m}_{t}
+{\rm div} \bigl( \tilde{m}_{t} D_{p} \tilde{H}_{t}(\cdot,D \tilde{u}_{t})
\bigr) \bigr\} dt,  
\\
d_{t} \tilde{u}_{t}
=  \bigl\{ -   \Delta \tilde{u}_{t} + \tilde{H}_{t}(\cdot,D\tilde{u}_{t}) - 
\tilde{F}_{t}(\cdot,m_{t}) 
\bigr\} dt
+ d\tilde{M}_{t},\\
 \tilde{m}_{0}=m_0, \; \tilde{u}_T(x)= \tilde{G}(x, m_T) \qquad {\rm in}\; \T^d.
\end{array}\right.
\end{equation}
where $\tilde{H}_{t}(x,p) = H(x+ \sqrt{2}W_{t},p)$, $\tilde{F}_{t}(x,m) = F(x+\sqrt{2} W_{t},m)$ and $\tilde{G}(x,m) = G(x+\sqrt{2} W_{T},m)$. The existence and uniqueness of 
a solution $(\tilde{u}_t,\tilde{m}_t,\tilde{M}_t)_{t \in [0,T]}$ 
to 
\eqref{eq:StochNashTransfo}
is ensured by Theorem \ref{thm:partie:2:existence:uniqueness}. 
Given such a solution, we let
$$
u_{t}(x) = \tilde{u}_{t}(x- \sqrt{2} W_{t}), 
\quad x \in \T^d \ ; 
\quad
\quad m_{t} = ( \textit{id} +\sqrt{2} W_{t})\sharp \tilde m_{t}, \quad t \in [0,T],
$$
and claim that the pair $(u_t,m_t)_{t \in [0,T]}$ thus defined satisfies \eqref{e.MFGstocsyst.sec5} (for a suitable $(v_{t})_{t \in [0,T]}$). 

The dynamics satisfied by $(m_t)_{t \in [0,T]}$
are given by the so-called
It\^o-Wentzell formula for distributed-valued processes,
see \cite[Theorem 1.1]{Kry11}, 
the proof of which works as follows:
for any test function $\phi\in \cC^{3}(\T^d)$
and any $z \in \R^d$, we have 
$\inte \phi(x)dm_t(x)= \inte \phi(x+\sqrt{2} W_{t}) d\tilde{m}_t(x)$; 
expanding the 
variation of
$(\int_{\T^d} \phi(x+z) d\tilde{m}_{t}(x))_{t \in [0,T]}$
by means of the Fokker-Planck equation satisfied by 
$(\tilde{m}_{t})_{t \in [0,T]}$
and then replacing 
$z$ by $\sqrt{2} W_{t}$, 
we then obtain the semi-martingale expansion of 
$(\int_{\T^d} \phi(x+
\sqrt{2} W_{t}) d\tilde{m}_{t}(x))_{t \in [0,T]}$
by applying the standard
It\^o-Wentzell formula. Once again we refer to 
\cite[Theorem 1.1]{Kry11}
for a complete account.

Applying \cite[Theorem 1.1]{Kry11}
to our framework
(with the formal 
writing
$(m_{t}(x) = \tilde{m}_{t}(x-\sqrt{2}W_{t}))_{t\in [0,T]}$), this shows exactly that
$(m_t)_{t \in [0,T]}$ solves 
\begin{equation}
\label{eq:m:proof:corollary:MFG:with:noise}
\begin{split}
d_tm_t&= \Bigl\{ 2\Delta m_{t} + \dive \Bigl(
 D_p{H}_{t}\bigl(x - \sqrt{2}W_{t},D \tilde{u}_{t}(x- \sqrt2 W_{t}) \bigr)
 \Bigr)
 \Bigr\}dt -
\sqrt{2} \dive(m_tdW_t)
\\
&= \Bigl\{ 2\Delta m_{t} + \dive \Bigl(
 D_p{H}\bigl(x,D {u}_{t}(x) \bigr)
 \Bigr)
 \Bigr\}dt -
\sqrt{2} \dive(m_tdW_t).
\end{split}
\end{equation}

Next we consider the equation satisfied by $(u_t)_{t \in [0,T]}$. 
Generally speaking, the strategy is similar. Intuitively, it consists in applying It\^o-Wenztell formula again, but
to $(u_{t}(x) = \tilde{u}_{t}(x-\sqrt{2} W_{t}))_{t \in [0,T]}$. 
Anyhow, 
in order to apply It\^o-Wentzell formula, 
we need first to identify the martingale part in $(\tilde{u}_{t}(x))_{t \in [0,T]}$
(namely $(\tilde{M}_{t}(x))_{t \in [0,T]}$). 
Recalling
from 
Lemma \ref{lem:partie2:master:equation}
 the formula 
\begin{equation*}
\tilde u_{t}(x) = U\bigl(t,x+ \sqrt{2}W_{t},
m_{t} \bigr),
\quad t \in [0,T],
\end{equation*}
we understand that the martingale part of $(\tilde{u}_{t}(x))_{t \in [0,T]}$
should be given by the first-order expansion of the above right-hand side (using an appropriate version of 
It\^o's formula for functionals defined on $[0,T] \times \T^d \times {\mathcal P}(\T^d)$).  
For our purpose, it is simpler to 
express $u_{t}(x)$ in terms of $U$ directly: 
\begin{equation*}
u_{t}(x) = U\bigl(t,x,m_{t} \bigr),
\quad t \in [0,T]. 
\end{equation*}
The trick is then to expand the above right-hand side by taking benefit from 
the master equation satisfied by $U$ and from the tailor-made It\^o's formula
given in Lemma \ref{lem:ito:local}. 

In order to apply Lemma \ref{lem:ito:local}, we observe that, 
in $(U(t,x,m_{t}))_{t \in [0,T]}$, the $x$-dynamics are entirely frozen so that 
we are led back to the case when $U$ is independent of $x$. 
With the same notation as in Lemma \ref{lem:ito:local}, we then
get
\begin{equation}
\label{eq:ito:measure:drift}
\begin{split}
&{\mathbb E} \bigl[ U(t+h,x,m_{t+h}) - U(t+h,x,m_{t})
\vert {\mathcal F}_{t}
\bigr]
\\
&= 2 \int_{\T^d} \textrm{\rm div}_{y}
\bigl[ D_{m} U \bigr](t,x,m_{t},y) dm_{t}(y)
\\
&\hspace{15pt}
- \int_{\T^d} D_{m} U(t,x,m_{t},y) \cdot D_{p} H\bigl(y,Du_{t}(y)\bigr) dm_{t}(y) 
\\
&\hspace{15pt}
+ \int_{[\T^d]^2}
\textrm{Tr} \bigl[ D^2_{mm} U
\bigr](t,x,m_{t},y,y') d m_{t}(y) dm_{t}(y') + 
\varepsilon_{t,t+h}. 
\end{split}
\end{equation}
Of course, this gives the absolutely continuous part only in the semi-martingale 
expansion of $(U(t,x,m_{t}))_{t \in [0,T]}$. In order to compute the martingale 
part, one must revisit the proof of Lemma 
\ref{lem:ito:local}. 
Going back 
to \eqref{eq:partie2:proof:master:equation:3}, we know that, in our 
case, 
$T^1_{h}$,
$T^3_{h}$ and $T^4_{h}$ are zero (as everything works as if $U$ was independent 
of $x$). 

Now, denoting by 
$(\eta_{s,t})_{s,t \in [0,T] : s \leq  t }$
a family of random variables satisfying
\begin{equation}
\label{eq:eta:t,t+h}
\lim_{h \searrow 0} \frac1h \sup_{s,t \in [0,T] :  \vert s-t \vert \leq h}
{\mathbb E} \bigl[ \vert \eta_{s,t} \vert^2 \bigr] = 0,  
\end{equation}
we can write, by 
\eqref{eq:partie2:proof:master:equation:2}
and \eqref{eq:T21:h}:
\begin{equation*}
T^2_{h}
= \sqrt{2} \biggl( \int_{\T^d} D_{y} \frac{\delta U}{\delta m}(t,x,m_{t},y) dm_{t}(y)
\biggr) \cdot \delta_{h} W_{t} + \eta_{t,t+h}
\end{equation*}
Moreover, 
by
\eqref{eq:T5:h}
and 
\eqref{T:5,4:h}
\begin{equation*}
T^5_{h}
= \eta_{t,t+h},
\end{equation*}
proving that 
\begin{equation*}
\begin{split}
&
U(t+h,x,m_{t+h}) 
-
{\mathbb E} \bigl[ U(t+h,x,m_{t+h}) 
\vert {\mathcal F}_{t}
\bigr]
\\
&\hspace{15pt} = 
\sqrt{2}
 \biggl( \int_{\T^d} D_{y} \frac{\delta U}{\delta m}(t,x,m_{t},y) dm_{t}(y)
\biggr) \cdot \delta_{h} W_{t} + \eta_{t,t+h},
\end{split}
\end{equation*}
for some family $(\eta_{s,t})_{s,t \in [0,T] :  s \leq t}$ that 
must satisfy \eqref{eq:eta:t,t+h}.
With such a decomposition, it holds that $\E[ \eta_{t,t+h} \vert {\mathcal F}_{t}]=0$.
Therefore, for any $t \in [0,T]$ and any partition 
$0=r_{0}<r_{1}<r_{2}<\dots<r_{N}=t$, we have
\begin{equation*}
\begin{split}
&\sum_{i=0}^{N-1}
\Bigl(
U(r_{i+1},x,m_{r_{i+1}}) 
-
{\mathbb E} \bigl[ U(r_{i+1},x,m_{r_{i+1}}) 
\vert {\mathcal F}_{r_{i}}
\bigr]
\Bigr)
\\
&= 
\sum_{i=0}^{N-1}
\biggl[
\sqrt{2}
 \biggl( \int_{\T^d} D_{y} \frac{\delta U}{\delta m}(r_{i},x,m_{r_{i}},y) dm_{r_{i}}(y)
\biggr) \cdot \bigl( W_{r_{i+1}} - W_{r_{i}} \bigr) + \eta_{r_{i},r_{i+1}}
\biggr],
\end{split}
\end{equation*}
with the property that 
\begin{equation*}
{\mathbb E}
\bigl[ \eta_{r_{i},r_{i+1}} \vert {\mathcal F}_{r_{i}}
\bigr] = 0, \quad
{\mathbb E}
\bigl[ \vert \eta_{r_{i},r_{i+1}} \vert^2
\bigr] \leq  \pi_{r_{i},r_{i+1}} \vert r_{i+1}-r_{i} \vert,
\end{equation*}
where $\lim_{h \searrow 0} \sup_{(s,t) \in [0,T]^2 : \vert s -t \vert \leq h}
\pi_{s,t} = 0$. By a standard computation of conditional expectation, 
we have that 
\begin{equation*}
\lim_{\delta \rightarrow 0}
\E
\biggl[ 
\Bigl\vert 
\sum_{i=0}^{N-1} \eta_{r_{i},r_{i+1}}
\Bigr\vert^2 
\biggr] = 0,
\end{equation*}
where $\delta$ stands for the mesh of the partition 
$r_{0},r_{1},\dots,r_{N}$. As a consequence,
the following limit holds true in $L^2$:
\begin{equation*}
\begin{split}
&\lim_{\delta \searrow 0}
\sum_{i=0}^{N-1}
\Bigl(
U(r_{i+1},x,m_{r_{i+1}}) 
-
{\mathbb E} \bigl[ U(r_{i+1},x,m_{r_{i+1}}) 
\vert {\mathcal F}_{r_{i}}
\bigr]
\Bigr)
= 
\sqrt{2}
\int_{0}^t D_{m} U(s,x,m_{s},y) \cdot dW_{s}.
\end{split}
\end{equation*}
Together with \eqref{eq:ito:measure:drift}, we deduce that 
\begin{equation*}
\begin{split}
d_t u_{t}(x) 
&=  
\biggl\{ \partial_t U\bigl(t,x,m_{t} \bigr) 
\\
&\hspace{15pt} + \inte \Bigl[ 2\dive_y \bigl[D_mU\bigr](t,x,m_t,y)
- D_m U (t,x,m_t,y)\cdot  D_p{H}_{t}\bigl(y,D {u}_{t}(y)\bigr)\Bigr] dm_t(y) 
\\
&\hspace{15pt} + 
\inte\inte {\rm Tr}\bigl[D^2_{mm} U\bigr](t,x,m_t,y,y')dm_t(y)dm_t(y')\biggr\}dt
\\
&\hspace{15pt} + \sqrt{2} \biggl(
\inte D_mU (t,x,m_t,y)dm_t(y)\biggr)\cdot dW_t .
\end{split}
\end{equation*}
Letting 
\begin{equation*}
v_t(x) = \inte D_mU (t,x,m_t,y)dm_t(y), \quad t \in [0,T], \ x \in \T^d,
\end{equation*}
and using the master equation satisfied by $U$, we obtain therefore 
$$
\begin{array}{l}
\ds d_t u_{t}(x)  
=  \ds \left\{ -2\Delta u_t(x)+H\bigl(x,Du_t(x)\bigr)-F(x,m_t)-\sqrt{2}\dive \bigl(v_t(x)
\bigr)\right\}dt
+ v_t(x)\cdot \sqrt{2}dW_t.
\end{array}
$$
Together with 
\eqref{eq:m:proof:corollary:MFG:with:noise}, this completes the proof of the existence 
of a solution to 
\eqref{e.MFGstocsyst.sec5}.
\vspace{4pt}

\textit{Second step. Uniqueness of the solution.}
We now prove uniqueness of the solution to 
\eqref{e.MFGstocsyst.sec5}.
Given a solution 
$(u_{t},m_{t})_{t \in [0,T]}$
(with some $(v_{t})_{t \in [0,T]}$)
to \eqref{e.MFGstocsyst.sec5}, we let 
\begin{equation*}
\tilde{u}_{t}(x) = u_{t}(x+ \sqrt{2} W_{t}), \quad 
x \in \T^d, \quad 
\tilde{m}_{t} = (\textit{id} - \sqrt{2} W_{t}) \sharp m_{t},
\quad t \in [0,T].
\end{equation*}
In order to prove uniqueness, it suffices to show that 
$(\tilde{u}_{t},\tilde{m}_{t})_{t \in [0,T]}$
is a solution to 
\eqref{eq:StochNashTransfo} (for some martingale $(\tilde{M}_{t})_{t \in [0,T]}$).

We first investigate the dynamics of $(\tilde{m}_{t})_{t \in [0,T]}$. 
As in the first step (\textit{existence of a solution}), 
we may apply 
It\^o-Wenztell formula for distribution-valued processes. 
Indeed, thanks to \cite[Theorem 1.1]{Kry11}
(with the formal writing 
$(\tilde{m}_{t}(x)=m_{t}(x+\sqrt{2} W_{t}))_{t \in [0,T]}$), 
we get 
exactly that 
$(\tilde{m}_{t})_{t \in [0,T]}$
satisfy the first equation in 
\eqref{eq:StochNashTransfo}. 

In order to prove the second equation in \eqref{eq:StochNashTransfo}, 
we apply It\^o-Wentzell formula 
for real-valued processes
 to $(\tilde{u}_{t}(x)
= u_{t}(x+ \sqrt{2} W_{t}))_{t \in [0,T]}$, 
see \cite[Theorem 3.1]{Kry11}.

\newpage
\section{Convergence of the Nash system}
\label{sec.convergence}

In this section, we consider,
for an integer $N \geq 2$, a classical solution $(v^{N,i})_{i \in \{1,\dots,N\}}$ of the Nash system with a common noise:  
\be\label{Nash0}
\left\{ \begin{array}{l}
\ds - \partial_t v^{N,i}(t,\bx) -  \sum_{j} \Delta_{x_j}v^{N,i}(t,\bx) - \beta  \sum_{j,k} {\rm Tr} D^2_{x_j,x_k} v^{N,i}(t,\bx) + H\bigl(x_i,  D_{x_i}v^{N,i}(t,\bx)
\bigr) \\
\ds \qquad \qquad + \sum_{j\neq i}  D_pH \bigl(x_j, D_{x_j}v^{N,j}(t,\bx)\bigr)
\cdot  D_{x_j}v^{N,i}(t,\bx)=
F(x_i, m^{N,i}_{\bx})\qquad {\rm in }\; [0,T]\times (\T^{d})^N,
\\
\ds v^{N,i}(T,\bx)= G(x_i,m^{N,i}_{\bx})\qquad {\rm in }\;  (\T^{d})^N,
\end{array}\right.
\ee
where we set, for $\ds {\bx}=(x_1, \dots, x_N)\in (\T^{d})^N$, $\ds m^{N,i}_{\bx}=\frac{1}{N-1}\sum_{j\neq i} \delta_{x_j}$.
Our aim is to prove Theorem \ref{thm:mainCV}, which says that the solution $(v^{N,i})_{i \in \{1,\dots,N\}}$ converges, in a suitable sense, to the solution of the second order master equation and Theorem \ref{thm:CvMFG}, which claims that the optimal trajectories also converge. 

Throughout this part we assume that $H$, $F$ and $G$ satisfy the assumption 
of Theorem 
\ref{theo:2nd-order:master:equation}
with $n \geq 2$.
This allows us to define $U=U(t,x,m)$ the solution of the second order master equation
\be\label{MasterCN}
\begin{array}{l}
\left\{\begin{array}{l} 
\ds - \partial_t U  - (1+\beta)\Delta_x U +H(x,D_xU) -(1+\beta)\inte \dive_y \left[D_m U\right]\ d m(y)\\
\ds    \;  + \inte  D_m U\cdot   D_pH(y,D_xU)\ dm(y) \\
\ds   \; -2 \beta  \inte  \dive_x\left[D_mU\right]dm(y) -\beta \int_{[\T^d]^2} {\rm Tr} D^2_{mm}U\ dm(y)dm(z) =F(x,m)
\vspace{2pt}
\\
 \ds \qquad \qquad  \qquad {\rm in }\; (0,T)\times \T^d\times \Pw,
\vspace{2pt} 
 \\
 U(T,x,m)= G(x,m) \qquad  {\rm in }\; \T^d\times \Pw,\\
\end{array}\right.
\end{array}
\ee
where $\beta\geq 0$ is a parameter for the common noise. 
For $\alpha'\in (0,\alpha)$, we have for any $(t,x)\in [0,T]\times \T^d$, $m,m'\in \Pk$
\begin{equation}
\label{Cond:Sol}
\begin{split}
&\|U(t,\cdot,m)\|_{n+2+\alpha'}+ \left\|\frac{\delta U}{\delta m}(t,\cdot, m, \cdot)\right\|_{(n+2+\alpha',n+1+\alpha')}+
  \left\|\frac{\delta^2 U}{\delta m^2}(t,\cdot, m, \cdot,\cdot)\right\|_{(n+2+\alpha',n+\alpha',n+\alpha')} 
\\
&\leq \ C_0,
\end{split}
\end{equation}
and that the mapping 
\be\label{Cond:Sol2}
[0,T] \times {\mathcal P}(\T^d) \ni (t,m) \mapsto 
\frac{\delta^2 U}{\delta m^2}(t,\cdot,m,\cdot,\cdot) \in 
\cC^{n+2+\alpha'}(\T^d) \times \left[\cC^{n+\alpha'}(\T^d)\right]^2
\ee
is continuous. 
As already said,  a solution of \eqref{MasterCN} satisfying the above properties has been built in Theorem \ref{theo:2nd-order:master:equation}. When $\beta=0$, one just needs to replace the above assumptions by 
those of Theorem \ref{theo:ex}, which does not require the second order differentiability of $F$ and $G$ with respect to $m$. 

The main idea for proving the convergence of the $(v^{N,i})_{i \in \{1,\dots,N\}}$ towards the solution $U$ is to use the fact that suitable finite dimensional projections of $U$ are nearly solutions to the Nash equilibrium equation. 
Actually, 
as we already alluded to 
at the end
of Section 
\ref{sec:prelim}, 
this strategy works under weaker 
 assumptions than that required in the statement 
of Theorem 
\ref{theo:2nd-order:master:equation}. What is really needed is 
that $H$ and $D_{p} H$
are globally Lipschitz continuous
and
that the master 
equation has a classical solution satisfying the conclusion of Theorem 
\ref{theo:2nd-order:master:equation}
(or Theorem \ref{theo:ex}
if $\beta=0$). 
In particular,  
the monotonicity properties of $F$ and $G$ have no role
in the proof of the convergence of the $N$-Nash system.
We refer to Remarks
\ref{rem:quelles:H:utiles}
and 
\ref{rem:quelles:H:utiles:2}
 below and we let the interesting reader reformulate the statements of Theorems
\ref{thm:mainCV} and \ref{thm:CvMFG}
accordingly.

\subsection{Finite dimensional projections of $U$}

For $N\geq 2$ and $i\in \{1,\dots, N\}$ we set 
$$
u^{N,i}(t,{\bx})= U(t,x_i, m^{N,i}_{\bx})\quad {\rm where }\; {\bx}=(x_1, \dots, x_N)\in (\T^{d})^N, \; m^{N,i}_{\bx}=\frac{1}{N-1}\sum_{j\neq i} \delta_{x_j}.
$$
Note that the $u^{N,i}$ are at least $\cC^2$ with respect to the $x_i$ variable because so is $U$. Moreover, $\partial_t u^{N,i}$ exists and is continuous because of the regularity of $U$. The next statement says that $u^{N,i}$ is actually globally $\cC^2$ in the space variables: 

\begin{Proposition}\label{uNC2} For any $N\geq 2$, $i\in \{1, \dots, N\}$, $u^{N,i}$ is of class $\cC^2$ in the space variables, with 
\begin{align*}
&D_{x_j}  u^{N,i}(t,{\bx})= \frac{1}{N-1} D_mU(t,x_i,m^{N,i}_{\bx},x_j) &&\qquad (j\neq i),
\\
&D^2_{x_i,x_j}  u^{N,i}(t,{\bx})= \frac{1}{N-1} D_xD_mU(t,x_i,m^{N,i}_{\bx},x_j) &&\qquad (j\neq i),
\\
&D^2_{x_j,x_j}  u^{N,i}(t,{\bx})= \frac{1}{N-1} D_y\left[D_mU\right](t,x_i,m^{N,i}_{\bx},x_j)
\\
&\hspace{90pt} +\frac{1}{(N-1)^2} D^2_{mm}U(t,x_i,m^{N,i}_{\bx},x_j,x_j)  &&\qquad (j\neq i)
\\
\textrm{\rm while, \ if} \ j\neq k, \quad 
&D^2_{x_j,x_k}  u^{N,i}(t,{\bx})= \frac{1}{(N-1)^2} D^2_{mm}U(t,x_i,m^{N,i}_{\bx},x_j,x_k) &&\qquad (i,j,k \textrm{\rm \ distinct}).
\end{align*}
\end{Proposition}

\begin{Remark}{\rm  If we only assume that $U$ has a first order derivative with respect to $m$, one can show that, for any $N\geq 2$, $i\in \{1, \dots, N\}$, $u^{N,i}$ is of class $\cC^1$ in all
the variables, with 
$$
D_{x_j} u^{N,i}(t,{\bx})= \frac{1}{N-1} D_mU(t,x_i,m^{N,i}_{\bx},x_j)\qquad \forall j\neq i,
$$
with a globally Lipschitz continuous space derivative.  The proof is the same except that one uses Proposition \ref{prop:diffdeltaDm} instead of Proposition \ref{prop:diffdeltaDm2}. 
}\end{Remark}

\begin{proof} For ${\bx}=(x_j)_{j \in \{1,\dots,N\}}$ such that $x_j\neq x_k$ for any $j\neq k$, let $\ep = \min_{j\neq k}|x_j-x_k|$. For ${\boldsymbol v}=(v_j)\in (\R^d)^N$ with $v_i=0$
(the value of $i \in \{1,\dots,N\}$ being fixed), we consider a smooth vector field $\phi$ such that 
$$
\phi(x)= v_j \qquad {\rm if } \; x\in B(x_j,\ep/4),
$$
where $B(x_{j},\ep/4)$ is the ball of center $x_{j}$ and of radius $\ep/4$. 
Then, in view of our assumptions \eqref{Cond:Sol} and \eqref{Cond:Sol2} on $U$, Propositions \ref{prop:diffdeltaDm2} and \ref{prop:diffdeltaDm2bis} in Appendix imply that 
\begin{equation*}
\begin{split}
&\biggl| U\bigl(t,x_i,(id+\phi)\sharp m^{N,i}_{\bx}\bigr)-
U\bigl(t,x_i,m^{N,i}_{\bx}\bigr)-\inte  D_mU\bigl(t,x_i,m^{N,i}_{\bx},y\bigr)\cdot \phi(y)\ dm^{N,i}_{\bx}(y) 
\\
&\hspace{15pt} - \frac12 \inte  D_y\bigl[D_mU\bigr]\bigl(t,x_i,m^{N,i}_{\bx},y\bigr)
\phi(y)\cdot\phi(y)\ dm^{N,i}_{\bx}(y)
\\
&\hspace{15pt} - \frac12 \inte\inte D^2_{mm}U\bigl(t,x_i,m^{N,i}_{\bx},y,y'\bigr)\phi(y)\cdot\phi(y')\ dm^{N,i}_{\bx}(y) dm^{N,i}_{\bx}(y')\biggr|\; \leq \;  \|\phi\|_{L^3_{m^{N,i}_{\bx}}}^2
\omega(\|\phi\|_{L^3_{m^{N,i}_{\bx}}}),
\end{split}
\end{equation*}
for some modulus $\omega$ such that $\omega(s)\to 0$ as $s\to0^+$. Therefore, 
\begin{equation*}
\begin{split}
&u^{N,i}(t,{\bx}+{\boldsymbol v})- u^{N,i}(t,\bx) 
\\
&= U\bigl((id+\phi) \sharp  m^{N,i}_{\bx}\bigr) - U(m^{N,i}_{\bx}) 
\\
&=   \inte  D_mU \bigl(t,x_{i},m^{N,i}_{\bx},y \bigr)\cdot  \phi(y)\ dm^{N,i}_{\bx}(y)
+ \frac12\inte  D_y\bigl[D_mU\bigr]\bigl(t,x_{i},m^{N,i}_{\bx},y\bigr)
\phi(y)\cdot \phi(y) dm^{N,i}_{\bx}(y)
\\
&\hspace{15pt} +\frac12 \inte\inte   D_{mm}^2U\bigl(t,x_{i},m^{N,i}_{\bx},y,z
\bigr) \phi(y)\cdot  \phi(z) dm^{N,i}_{\bx}(y) dm^{N,i}_{\bx}(z)
+ \|\phi\|_{L^3(m^{N,i}_{\bx})}^2\omega
\bigl(\|\phi\|_{L^3(m^{N,i}_{\bx})}\bigr)
\\
&= \frac{1}{N-1} \sum_{j\neq i}  D_mU\bigl(t,x_{i},m^{N,i}_{\bx},x_j\bigr)\cdot  v_j 
+\frac{1}{2(N-1)} \sum_{j\neq i}  D_y\bigl[D_mU\bigr]
\bigl(t,x_{i},m^{N,i}_{\bx},x_j \bigr) v_j\cdot  v_j 
\\
&\hspace{15pt}+\frac{1}{2(N-1)^2} \sum_{j,k\neq i}  D^2_{mm}U
\bigl(t,x_{i},m^{N,i}_{\bx},x_j,x_k \bigr)v_j\cdot  v_k + |{\boldsymbol v}|^2\omega(|{\boldsymbol v}|) .
\end{split}
\end{equation*}
This shows that $u^{N,i}$ has a  second order expansion at ${\bx}$ with respect to the variables $(x_j)_{j\neq i}$ and that
\begin{align*}
&D_{x_j}  u^{N,i}(t,{\bx})= \frac{1}{N-1} D_mU\bigl(t,x_i,m^{N,i}_{\bx},x_j\bigr) \qquad &&(j\neq i),
\\
&D^2_{x_j,x_j}  u^{N,i}(t,{\bx})= \frac{1}{N-1} D_y\bigl[D_mU\bigr]
\bigl(t,x_{i},m^{N,i}_{\bx},x_j \bigr)
\\
&\hspace{90pt}+\frac{1}{(N-1)^2} D^2_{mm}U
\bigl(t,x_{i},m^{N,i}_{\bx},x_j,x_j\bigr)  \qquad &&(j\neq i)
\\
\textrm{\rm while, \ if}\  j\neq k,
\quad 
&D^2_{x_j,x_k}  u^{N,i}(t,{\bx})= \frac{1}{(N-1)^2} D^2_{mm}U
\bigl(t,x_{i},m^{N,i}_{\bx},x_j,x_k\bigr) \qquad &&(i,j,k \ \textrm{\rm distinct}).
\end{align*}
So far we have proved the existence of  first and second order space derivatives of $U$ in the open subset of $[0,T]\times (\T^d)^N$ consisting in the points $(t,\bx) =(t,x_1,\cdots x_N)$ such that $x_i\neq x_j$ for any $i\neq j$. As $D_mU$, $D_y\left[D_mU\right]$ and $D^2_{mm}U$ are continuous, these first and second order derivatives can be continuously extended to the whole space $[0,T]\times (\T^d)^N$, and therefore $u^{N,i}$ is $\cC^2$ with respect to the space variables in $[0,T]\times\T^{Nd}$.
\end{proof}

We now show that $(u^{N,i})_{i \in \{1,\dots,N\}}$ is ``almost" a solution to the Nash system \eqref{Nash0}:   

\begin{Proposition}\label{Prop:equNi} One has, for any $i\in \{1, \dots, N\}$, 
\be\label{eq:uNi}
\left\{\begin{array}{l}
\ds  - \partial_t u^{N,i}   - \sum_j \Delta_{x_j} u^{N,i}  - \beta  \sum_{j,k} {\rm Tr} D^2_{x_j,x_k} u^{N,i} +H(x_i,D_{x_i}u^{N,i}) \\
\qquad \ds  +\sum_{j\neq i}  D_{x_j}u^{N,i}(t,{\bx})\cdot  D_pH 
\bigl(x_j,D_{x_j}u^{N,j}(t,{\bx}) \bigr) 
=F(x_i,m^{N,i}_{\bx}) +r^{N,i}(t,{\bx}) 
\\
\qquad \qquad \qquad \qquad \qquad  \ds  \qquad 
\hspace{4pt}
{\rm in}\; (0,T)\times \T^{Nd},
\vspace{2pt}
\\
u^{N,i}(T,{\bx})= G(x_i,m^{N,i}_{\bx}) \qquad  {\rm in}\;  \T^{Nd},
\end{array}\right.
 \ee
where $r^{N,i}\in \cC^0([0,T]\times \T^d)$ with
$$
\|r^{N,i}\|_\infty\leq \frac{C}{N}.
$$
\end{Proposition}

\begin{Remark}{\rm When $\beta=0$, we can require $U$ to have only a first order derivative with respect to the measure, but in this case equation \eqref{eq:uNi} only holds  a.e. with $r^{N,i}\in L^\infty$ still satisfying $\ds \|r^{N,i}\|_\infty\leq \frac{C}{N}$. 
}\end{Remark}

\begin{proof}
As $U$ solves \eqref{MasterCN}, one has at a point $(t,x_i,m^{N,i}_{\bx})$: 
\begin{align*} 
&- \partial_t U  - (1+\beta)\Delta_x U +H(x_i,D_xU)  -(1+\beta)\inte \dive_y \bigl[D_m U\bigr]\bigl(t,x_i,m^{N,i}_{\bx},y \bigr) d m^{N,i}_{\bx}(y) 
\\
&\hspace{15pt}+ \inte  D_m U\bigl(t,x_i,m^{N,i}_{\bx},y\bigr)\cdot  
D_pH\bigl(y,D_xU(t,y,m^{N,i}_{\bx}) \bigr)
dm^{N,i}_{\bx}(y) 
\\
&\hspace{15pt} -2 \beta  \inte  \dive_x\bigl[D_mU\bigr]
\bigl(t,x_i,m^{N,i}_{\bx},y\bigr)dm^{N,i}_{\bx}(y) 
\\
&\hspace{15pt}
 -\beta \inte {\rm Tr} D^2_{mm}U\bigl(t,x_i,m^{N,i}_{\bx},y,z\bigr) 
 dm^{N,i}_{\bx}(y)dm^{N,i}_{\bx}(z) =F\bigl(x_i,m^{N,i}_{\bx}\bigr). 
\end{align*}
So $u^{N,i}$ satisfies:  
\begin{align*}
&- \partial_t u^{N,i}  -(1+\beta) \Delta_{x_i} u^{N,i} +H(x_i,D_{x_i}u^{N,i})  -  (1+\beta)\inte \dive_y\bigl[ D_m U\bigr]\bigl(t,x_i,m^{N,i}_{\bx},y\bigr) d m^{N,i}_{\bx}(y)
  \\
&\hspace{30pt} +\frac{1}{N-1}\sum_{j\neq i}  D_m U\bigl(t,x_i,m^{N,i}_{\bx},x_j\bigr) \cdot  D_pH \bigl(x_j,D_xU(t,x_j,m^{N,i}_{\bx})\bigr)
\\
&\hspace{15pt}
-2 \beta  \inte  \dive_x\bigl[D_mU\bigr]\bigl(t,x_i,m^{N,i}_{\bx},y\bigr)dm^{N,i}_{\bx}(y) 
\\
&\hspace{15pt}  -\beta \inte {\rm Tr} D^2_{mm}U\bigl(t,x_i,m^{N,i}_{\bx},y,z\bigr) dm^{N,i}_{\bx}(y)dm^{N,i}_{\bx}(z) =F(x_i,m^{N,i}_{\bx}) .
\end{align*}
 Note that, by Proposition \ref{uNC2}, 
 $$
\frac{1}{N-1} D_m U \bigl(t,x_i,m^{N,i}_{\bx},x_j\bigr) = D_{x_j}u^{N,i}(t,\bx).
$$
In particular, 
\be\label{e.boundDxjui}
\| D_{x_j}u^{N,i}\|_\infty\leq \frac{C}{N}.
\ee
By the Lipschitz continuity of $D_xU$ with respect to $m$, we have 
$$
\left|D_xU(t,x_j,m^{N,i}_{\bx})- D_xU(t,x_j,m^{N,j}_{\bx})\right| \leq C \dk(m^{N,i}_{\bx}, m^{N,j}_{\bx})\leq \frac{C}{N-1}, 
$$
so that,
by Lipschitz continuity of $D_pH$, 
\begin{equation}
\label{eq:quelles:H:utiles}
\bigl| D_pH \bigl(x_j,D_xU(t,x_j,m^{N,i}_{\bx})\bigr)- D_pH 
\bigl(x_j,D_{x_j}u^{N,j}(t,{\bx})\bigr)\bigr|\leq  \frac{C}{N}.
\end{equation}
Collecting the above relations, we obtain 
\begin{align*}
&\frac{1}{N-1}\sum_{j\neq i}  D_m U\bigl(t,x_i,m^{N,i}_{\bx},x_j\bigr) \cdot  D_pH \bigl(x_j,D_xU(t,x_j,m^{N,i}_{\bx})\bigr) 
\\
&\hspace{15pt} = \sum_{j\neq i}  D_{x_j}u^{N,i}(t,{\bx})\cdot  D_pH \bigl(x_j,D_xU(t,x_j,m^{N,i}_{\bx})\bigr) 
\\
&\hspace{15pt} =
\sum_{j\neq i}  D_{x_j}u^{N,i}(t,{\bx})\cdot  D_pH \bigl(x_j,D_{x_j}u^{N,j}(t,{\bx})\bigr) + O(1/N),
\end{align*}
where we used \eqref{e.boundDxjui} in the last inequality. 
On the other hand, 
\begin{align*}
\sum_{j=1}^N \Delta_{x_j} u^{N,i}   + \beta  \sum_{j,k=1}^N {\rm Tr} D^2_{x_j,x_k} u^{N,i}
&= 
(1+\beta)\Delta_{x_i} u^{N,i} + (1+\beta) \sum_{j\neq i} \Delta_{x_j} u^{N,i} 
\\
&\hspace{15pt} + 2\beta \sum_{j\neq i} {\rm Tr} D^2_{x_i,x_j} u^{N,i} +
\beta \sum_{j\neq k\neq i } {\rm Tr} D^2_{x_j,x_k} u^{N,i},
\end{align*}
where, using Proposition \ref{uNC2}, 
\begin{align*}
\sum_{j\neq i} \Delta_{x_j} u^{N,i}(t,\bx)  &= 
 \inte \dive_y\bigl[D_mU\bigr]\bigl(t,x_i,m^{N,i}_{\bx},y\bigr)dm^{N,i}_{\bx}(y) 
 \\
&\hspace{15pt} +\frac{1}{N-1} \inte {\rm Tr} \bigl[D^2_{mm}U\bigr]
\bigl(t,x_i,m^{N,i}_{\bx},y,y\bigr) dm^{N,i}_{\bx}(y)
\\
\sum_{j\neq i} {\rm Tr} D^2_{x_i,x_j} u^{N,i}(t,\bx) &= \inte \dive_x\bigl[D_mU\bigr]\bigl(t,x_i,m^{N,i}_{\bx},y\bigr)dm^{N,i}_{\bx}(y)
\\
\sum_{j\neq k\neq i } {\rm Tr} D^2_{x_j,x_k} u^{N,i}(t,\bx) 
&=\inte\inte {\rm Tr} \bigl[D^2_{mm}U\bigr]\bigl(t,x_i,m^{N,i}_{\bx},y,z
\bigr)dm^{N,i}_{\bx}(y)dm^{N,i}_{\bx}(z).
\end{align*}
Therefore 
\begin{align*}
&- \partial_t u^{N,i}(t,\bx) - \sum_j \Delta_{x_j} u^{N,i}(t,\bx)
  - \beta  \sum_{j,k} {\rm Tr} D^2_{x_j,x_k} u^{N,i}(t,\bx)  
  + H\bigl(x_i,D_{x_i}u^{N,i}(t,\bx) \bigr)    
  \\
&\hspace{15pt} + \sum_{j\neq i}  D_{x_j}u^{N,i}(t,{\bx})\cdot  D_pH 
\bigl(x_j,D_{x_j}u^{N,j}(t,{\bx})\bigr)  
\\
&\hspace{15pt} + \frac{1}{N-1} \inte {\rm Tr} D^2_{mm}U
\bigl(t,x_i,m^{N,i}_{\bx},y,y\bigr) dm^{N,i}_{\bx}(y) =F(x_i,m^{N,i}_{\bx})
+ 
O(1/N),
\end{align*}
which shows the result.
\end{proof}

\begin{Remark}
\label{rem:quelles:H:utiles}
The reader may observe that, in addition 
to the existence of a classical solution
$U$ (to the master equation) 
satisfying the conclusion 
of Theorem \ref{theo:2nd-order:master:equation}, 
only the global Lipschitz property of $D_{p} H$
is used in the proof, see 
\eqref{eq:quelles:H:utiles}. 
\end{Remark}

\subsection{Convergence}
\label{subsubse:partie3:convergence}

We now turn to the proof of Theorem \ref{thm:mainCV}. For this, we consider the solution $(v^{N,i})_{i \in \{1,\dots,N\}}$ of the Nash system \eqref{Nash0}.   
By uniqueness of the solution, the $(v^{N,i})_{i \in \{1,\dots,N\}}$ must be symmetrical. 
By symmetrical, we mean that, for any ${\bx}=(x_l)_{l \in \{1,\dots,N\}}\in \T^{Nd}$ and for any indices $j\neq k$, if $\tilde {\bx}=(\tilde x_l)_{l \in \{1,\dots,N\}}$ is the $N$-tuple obtained from ${\bx}$ by permuting the $j$ and $k$ vectors (i.e., $\tilde x_l=x_l$ for $l
\not \in \{ j,k\}$, $\tilde x_j=x_k$, $\tilde x_k= x_j$), then 
$$
v^{N,i}(t, \tilde {\bx})= v^{N,i}(t, {\bx}) \; {\rm if}\; i\not \in \{ j,k\}, \;{\rm while}\;  v^{N,i}(t, \tilde {\bx})= v^{N,k}(t,{\bx}) \; {\rm if}\; i=j,
$$
which may be reformulated as follows: 
There exists a function $V^N : \T^d \times [\T^d]^{N-1} \rightarrow 
\R$ such that, for any $x \in \T^d$, the function $[\T^d]^{N-1}
\ni (y_{1},\dots,y_{N-1}) \mapsto V^N(x,(y_{1},\dots,y_{N-1}))$
is invariant under permutation, and 
\begin{equation*}
\forall i \in \{1,\dots,N\}, \
 \bx \in [\T^d]^N, 
\quad v^{N,i}(t,\bx) = V^N\bigl(x_{i},(x_{1},\dots,x_{i-1},x_{i+1},\dots,x_{N})
\bigr). 
\end{equation*}
Note that the $(u^{N,i})_{i \in \{1,\dots,N\}}$ are also symmetrical. 

The proof of Theorem \ref{thm:mainCV} consists in comparing ``optimal trajectories" for $v^{N,i}$ and for $u^{N,i}$, for any $i \in \{1,\dots,N\}$. For this, let us fix $t_0\in [0,T)$, $m_0\in \Pw$ and let $(Z_i)_{i \in \{1,\dots,N\}}$ be an i.i.d family of $N$ random variables of law $m_0$. We set $\bZ=(Z_i)_{i \in \{1,\dots,N\}}$.  Let also $((B_{t}^{i})_{t \in [ 0,T]})_{i \in \{1,\dots,N\}}$ be a family of $N$ independent $d$-dimensional Brownian Motions which is also independent of $(Z_i)_{i \in \{1,\dots,N\}}$ and let $W$ be a $d$-dimensional Brownian Motion independent of the $((B^i_{t})_{t \in [ 0,T]})_{i \in \{1,\dots,N\}}$ and $(Z_i)_{i \in \{1,\dots,N\}}$. 
We consider the systems of SDEs with variables $(\bX_{t}=(X_{i,t})_{i\in \{1,\dots,N\}})_{t \in [0,T]}$ and $(\bY_{t}=(Y_{i,t})_{i\in \{1,\dots,N\}})_{t \in [0,T]}$(the SDEs being set on $\R^d$ 
with periodic coefficients): 
\be\label{defxit}
\left\{\begin{array}{l}
dX_{i,t}= -D_pH\bigl(X_{i,t}, D_{x_i}u^{N,i}(t, \bX_t)\bigr)dt +\sqrt{2} dB^{i}_t+\sqrt{2\beta} dW_t\qquad t\in [t_0,T]\\
X_{i,t_0}= Z_i,
\end{array}\right.
\ee
and 
\be\label{defyit}
\left\{\begin{array}{l}
dY_{i,t}= -D_pH\bigl(Y_{i,t}, D_{x_i}v^{N,i}(t, \bY_t)\bigr)dt +\sqrt{2} dB^{i}_t+\sqrt{2\beta} dW_t\qquad t\in [t_0,T]\\
Y_{i,t_0}= Z_i.
\end{array}\right.
\ee
Note that, since the $(u^{N,i})_{i \in \{1,\dots,N\}}$ are symmetrical, 
the processes $((X_{i,t})_{t\in [t_0,T]})_{i \in \{1,\dots,N\}}$ are exchangeable. The same holds for the $((Y_{i,t})_{t\in [t_0,T]})_{i \in \{1,\dots,N\}}$
and, actually, the 
$N$ $\R^{2d}$-valued 
processes
$((X_{i,t},Y_{i,t})_{t\in [t_0,T]})_{i \in \{1,\dots,N\}}$ are also exchangeable.
\\

%

\begin{Theorem}\label{thm:Cvyx} 
Under the standing assumptions, 
we have, for any $i\in \{1, \dots, N\}$,   
\begin{align}
&\E\bigl[\sup_{t\in [t_0,T]} |Y_{i,t}-X_{i,t}|\bigr]\leq \frac{C}{{N}},
\label{estixi-yi}
\\
&\E\biggl[\sup_{t\in [t_0,T]}\left|u^{N,i}(t, \bY_t)-v^{N,i}(t, \bY_t)\right|^2 \nonumber
\\
&\hspace{30pt} +\int_{t_0}^T  |D_{x_i}v^{N,i}(t, \bY_t)- D_{x_i}u^{N,i}(t,\bY_t)|^2dt \biggr]  
 \leq CN^{-2},
 \label{CvyxIneq2}
\end{align}
and,
$\P$ almost surely,
\begin{equation}
\frac{1}{N}
\sum_{i=1}^N
\vert 
v^{N,i}(t_{0},\bZ) 
- 
u^{N,i}(t_{0},\bZ) 
\vert
  \leq CN^{-1}, \label{estiAE}
\end{equation}
where $C$ is a (deterministic) constant that does not depend on $t_0$, $m_0$ and $N$.
\end{Theorem}

\begin{proof}[Proof of Theorem \ref{thm:Cvyx}.] 
\textit{First step.}
We start with the proof of 
\eqref{CvyxIneq2}. 
For simplicity, we work with $t_0=0$. Let us first introduce new notations: 
\begin{equation*} 
\begin{split}
&U^{N,i}_{t} = u^{N,i}(t,\bY_{t}), \quad 
V^{N,i}_{t} = v^{N,i}(t,\bY_{t}), 
\\
&DU^{N,i,j}_{t} = D_{x_{j}} u^{N,i}(t,\bY_{t}), 
\quad
DV^{N,i,j}_{t} = D_{x_{j}} v^{N,i}(t,\bY_{t}), 
\quad t \in [0,T].
\end{split}
\end{equation*}
Using equation \eqref{Nash0} satisfied by the 
$(v^{N,i})_{i \in \{1,\dots,N\}}$, 
we deduce from It\^o's formula that, for any $i \in \{1,\dots,N\}$,
\begin{equation}
\label{rep:vNi}
\begin{split}
d V^{N,i}_{t}
&= 
\Bigl[
\partial_{t} v^{N,i}(t,\bY_{t})
- \sum_{j}
D_{x_{j}} v^{N,i}(t,\bY_{t}) 
\cdot
D_pH \bigl(Y_{j,t}, D_{x_j}v^{N,i}(t, \bY_t) \bigr)
\\
&\hspace{15pt}
+ \sum_{j} \Delta_{x_{j}}
v^{N,i}(t,\bY_{t}) 
+ \beta \sum_{j,k} \text{Tr} D^2_{x_{j},x_{k}}
v^{N,i}(t,\bY_{t}) 
\Bigr] dt
\\
&\hspace{15pt} + 
\sqrt{2}
\sum_{j} D_{x_{j}} v^{N,i}(t,\bY_{t})
dB^j_{t}
+ 
\sqrt{2 \beta}
\sum_{j} D_{x_{j}} v^{N,i}(t,\bY_{t})
dW_{t}
\\
&= 
\Bigl[ H \bigl( Y_{i,t},D_{x_i} v^{N,i}(t, \bY_t) \bigr)
- D_{x_{i}} v^{N,i}(t,\bY_{t}) 
\cdot
D_pH\bigl(Y_{i,t}, 
D_{x_i} v^{N,i}(t, \bY_t) \bigr)
\\
&\hspace{150pt}
- F \bigl( Y_{i,t},m^{N,i}_{\bY_t})
\Bigr] dt
\\
&\hspace{15pt} + 
\sqrt{2}
\sum_{j} D_{x_{j}} v^{N,i}(t,\bY_{t})
dB^j_{t}
+ 
\sqrt{2 \beta}
\sum_{j} D_{x_{j}} v^{N,i}(t,\bY_{t})
dW_{t}.
\end{split}
\end{equation}
Similarly, as $(u^{N,i})_{i \in \{1,\dots,N\}}$ satisfies \eqref{eq:uNi}, we have by standard computation
\begin{equation}
\label{rep:uNi}
\begin{split}
&d U^{N,i}_{t}
\\
&= 
\Bigl[ H \bigl( Y_{i,t},D_{x_i} u^{N,i}(t, \bY_t) \bigr)
- D_{x_{i}} u^{N,i}(t,\bY_{t}) 
\cdot
D_pH\bigl(Y_{i,t}, 
D_{x_i} u^{N,i}(t, \bY_t) \bigr)
\\
&\hspace{150pt}
- F \bigl( Y_{i,t},m^{N,i}_{\bY_t})
- r^{N,i}(t,\bY_{t})
\Bigr] dt
\\
&\hspace{5pt}
-
\sum_{j }
D_{x_{j}} u^{N,i}(t,\bY_{t}) 
\cdot
\Bigl( 
D_pH \bigl(Y_{j,t}, D_{x_j}v^{N,j}(t, \bY_t) \bigr)
-
D_pH \bigl(Y_{j,t}, D_{x_j}u^{N,j}(t, \bY_t) \bigr)
\Bigr) dt
\\
&\hspace{5pt} + 
\sqrt{2}
\sum_{j} D_{x_{j}} u^{N,i}(t,\bY_{t})
\cdot
dB^j_{t}
+ 
\sqrt{2 \beta}
\sum_{j} D_{x_{j}} u^{N,i}(t,\bY_{t})
\cdot
dW_{t}.
\end{split}
\end{equation}
Make the difference between 
\eqref{rep:vNi}
and \eqref{rep:uNi}, take the square and apply It\^o's formula
again:
\begin{equation*}
\begin{split}
&d 
\bigl[ 
U^{N,i}_{t}
- 
V^{N,i}_{t}
\bigr]^2
\\
&= \biggl[ 2 \bigl( U^{N,i}_{t}
- 
V^{N,i}_{t}
\bigr) 
\cdot \Bigl(  
 H \bigl( Y_{i,t},DU^{N,i,i}_{t} \bigr)
 - H \bigl( Y_{i,t},DV^{N,i,i}_{t} \bigr)
 \Bigr)
\\ 
&\hspace{30pt} 
- 2 \bigl( U^{N,i}_{t}
- 
V^{N,i}_{t}
\bigr) 
\cdot \Bigl(  
DU^{N,i,i}_{t}
\cdot
\bigl[
 D_{p} H \bigl( Y_{i,t},DU^{N,i,i}_{t} \bigr)
 -
D_{p} H \bigl( Y_{i,t},DV^{N,i,i}_{t}
\bigr) \bigr] \Bigr)
\\ 
&\hspace{30pt} 
- 2 \bigl( U^{N,i}_{t}
- 
V^{N,i}_{t}
\bigr) 
\cdot \Bigl(  
\bigl[
DU^{N,i,i}_{t}
-
DV^{N,i,i}_{t}
\bigr]
\cdot
 D_{p} H \bigl( Y_{i,t},DV^{N,i,i}_{t} \bigr)
 \Bigr)
 \\
 &\hspace{30pt}
 - 2
 \bigl( U^{N,i}_{t}
- 
V^{N,i}_{t}
\bigr)
r^{N,i}(t,\bY_{t})
 \biggr] dt
\\
&\hspace{15pt} 
- 2 \bigl( U^{N,i}_{t}
- 
V^{N,i}_{t}
\bigr) 
\sum_{j }
D U^{N,i,j}_{t} 
\cdot
\Bigl( 
D_pH \bigl(Y_{j,t}, DV^{N,j,j}_{t} \bigr)
-
D_pH \bigl(Y_{j,t}, DU^{N,j,j}_{t} \bigr)
\Bigr) dt
 \\
 &\hspace{15pt}
 + 
 \biggl[ 2
 \sum_{j} \vert 
DU^{N,i,j}_{t} 
-
 DV^{N,i,j}_{t}
 \vert^2 
 + 
 2 \beta 
\Bigl\vert
 \sum_{j} 
 \bigl( DU^{N,i,j}_{t}
 - DV^{N,i,j}_{t} \bigr)
\Bigr\vert^2
\biggr] dt 
 \\
 &\hspace{15pt}
+
\sqrt{2}
\sum_{j} 
\bigl(
DU^{N,i,j}_{t}
- DV^{N,i,j}_{t}
\bigr) 
\cdot
dB^j_{t}
+ 
\sqrt{2 \beta}
\sum_{j} 
\bigl(
DU^{N,i,j}_{t}
- DV^{N,i,j}_{t}
\bigr) 
\cdot
dW_{t}
\end{split}
\end{equation*}
Recall now that 
$H$ 
and $D_{p} H$
are Lipschitz continuous in the variable $p$. 
Recall also that 
$DU^{N,i,i}_{t}
= D_{x_{i}} U(t,Y_{i,t},m_{{\boldsymbol Y}_{t}}^{N,i})$
is bounded, independently of $i$, $N$ and $t$,
and that $DU^{N,i,j}_{t}
= D_{x_{j}} U(t,Y_{i,t},m_{\boldsymbol Y}^{N,i})$
is bounded by $C/N$ when $i \not =j$, for $C$ independent 
of $i$, $j$, $N$ and $t$.  
Recall
 finally
 from Proposition 
 \ref{Prop:equNi}
that $r^{N,i}$
is bounded by $C/N$. 
Integrating from $t$ to $T$ 
in the above formula and 
taking the conditional expectation 
given $\bZ$ (with the shorten notation 
$\E^{\bZ}[\cdot]
= \E[ \cdot \vert \bZ]$), we deduce: 
\begin{equation}
\label{eq:quelles:H:utiles:2}
\begin{split}
&{\mathbb E}^{\bZ}
\bigl[ 
\vert
U^{N,i}_{t}
- 
V^{N,i}_{t}
\vert^2
\bigr]
+ 
2  \sum_{j} \E^{\bZ}
 \biggl[ \int_{t}^T
 \vert 
DU^{N,i,j}_{s} 
-
 DV^{N,i,j}_{s}
 \vert^2 ds 
 \biggr]
\\
&\hspace{15pt} \leq {\mathbb E}^{\bZ}
\bigl[ 
\vert
U^{N,i}_{T}
- 
V^{N,i}_{T}
\vert^2
\bigr]
+ \frac{C}{N}
\int_{t}^T 
\E^{\bZ} 
\bigl[
\vert
U^{N,i}_{s}
- 
V^{N,i}_{s}
\vert
\bigr]
ds
\\
&\hspace{30pt}
+ C 
\int_{t}^T 
\E^{\bZ}
\Bigl[
\vert
U^{N,i}_{s}
- 
V^{N,i}_{s}
\vert
\cdot
\vert 
DU^{N,i,i}_{s} 
-
DV^{N,i,i}_{s}
\vert 
\Bigr]
ds
\\
&\hspace{30pt} + 
\frac{C}{N}
\sum_{j \not = i} 
\int_{t}^T 
\E^{\bZ}
\Bigl[
\vert
U^{N,i}_{s}
- 
V^{N,i}_{s}
\vert
\cdot
\vert 
DU^{N,j,j}_{s} 
-
DV^{N,j,j}_{s}
\vert 
\Bigr]
ds.
\end{split}
\end{equation}
Note that the boundary condition 
$U^{N,i}_{T}
- 
V^{N,i}_{T}$
is zero. 
By a standard convexity argument, we get
\begin{equation*}
\begin{split}
&{\mathbb E}^{\bZ}
\bigl[ 
\vert
U^{N,i}_{t}
- 
V^{N,i}_{t}
\vert^2
\bigr]
+  \E^{\bZ}
 \biggl[ \int_{t}^T
 \vert 
DU^{N,i,i}_{s} 
-
 DV^{N,i,i}_{s}
 \vert^2 ds 
  \biggr]
  \\
 &\leq \frac{C}{N^2}
+
 C 
\int_{t}^T 
\E^{\bZ}
\bigl[
\vert
U^{N,i}_{s}
- 
V^{N,i}_{s}
\vert^2
\bigr]
ds
+ 
\frac12 \sum_{j}
\E^{\bZ}
 \biggl[ \int_{t}^T
 \vert 
DU^{N,j,j}_{s} 
-
 DV^{N,j,j}_{s}
 \vert^2 ds 
  \biggr].
\end{split}
\end{equation*}
By Gronwall's Lemma, we finally get
(modifying the value of the constant $C$): 
\begin{equation}
\label{CvyxIneq2:001}
\begin{split}
&
\sup_{t \in [0,T]}
{\mathbb E}^{\bZ}
\bigl[ 
\vert
U^{N,i}_{t}
- 
V^{N,i}_{t}
\vert^2
\bigr]
+  \E^{\bZ}
 \biggl[ \int_{0}^T
 \vert 
DU^{N,i,i}_{s} 
-
 DV^{N,i,i}_{s}
 \vert^2 ds 
 \biggr]
 \\
&\hspace{15pt}
\leq \frac{C}{N^2}
+ 
\frac12 \sum_{j}
\E^{\bZ}
 \biggl[ \int_{t}^T
 \vert 
DU^{N,j,j}_{s} 
-
 DV^{N,j,j}_{s}
 \vert^2 ds 
  \biggr].
\end{split}
\end{equation}
Taking the expectation
and using 
the exchangeability 
of the processes 
$((X_{j,t},Y_{j,t})_{t\in [t_0,T]})_{j \in \{1,\dots,N\}}$, we obtain
\eqref{CvyxIneq2}. 
\vspace{4pt}

\textit{Second step.}
We now derive 
\eqref{estixi-yi}
and 
\eqref{estiAE}. 
We start with 
\eqref{estiAE}. 
Noticing that 
$U^{N,i}_{0}
- 
V^{N,i}_{0}=
u^{N,i}({0},\bZ)
- v^{N,i}({0},\bZ)$, we 
deduce, by summing 
\eqref{CvyxIneq2:001}
over $i \in \{1,\dots,N\}$,
 that, with probability $1$ under $\P$, 
\begin{equation*}
\frac1N \sum_{i=1}^N 
\vert u^{N,i}({0},\bZ)
- v^{N,i}({0},\bZ) \vert \leq \frac{C}{N},
\end{equation*}
which is exactly
\eqref{estiAE}.

We are now ready to estimate the difference $X_{i,t}-Y_{i,t}$, for $t \in [0,T]$
and $i \in \{1,\dots,N\}$. 
In view of the equation satisfied by the processes $(X_{i,t})_{t \in [0,T]}$ and by $(Y_{i,t})_{t \in [0,T]}$, we have
\begin{equation}
\label{eq:quelles:H:utiles:3}
\begin{split}
 |X_{i,t}-Y_{i,t}| &\leq   \int_0^t 
\bigl|D_pH \bigl(X_{i,s}, D_{x_{i}}
u^{N,i}(s,\bX_{s})
\bigr) -D_pH
\bigl(Y_{i,s}, D_{x_{i}}v^{N,i}(s,\bY_{s})
\bigr)
\bigr|
 ds 
\\
&\leq  C\int_0^t |X_{i,s}-Y_{i,s}|ds  + C\int_0^T  
\bigl|
D U^{N,i,i}_{s}
-
D V^{N,i,i}_{s}
\bigr|  ds.
\end{split}
\end{equation}
By Gronwall inequality
and by 
\eqref{CvyxIneq2:001}, we obtain 
\eqref{estixi-yi}.
\end{proof}

\begin{Remark}
\label{rem:quelles:H:utiles:2}
The reader may observe that, in addition 
to the existence of a classical solution
$U$ (to the master equation) 
satisfying the conclusion 
of Theorem \ref{theo:2nd-order:master:equation}, 
only the global Lipschitz properties of 
$H$ and $D_{p} H$
are used in the proof, see 
\eqref{eq:quelles:H:utiles:2}
and 
\eqref{eq:quelles:H:utiles:3}. 
\end{Remark}

\begin{proof}[Proof of Theorem \ref{thm:mainCV}.] For part (i), let us choose $m_0\equiv 1$ and apply \eqref{estiAE}: 
$$
\frac{1}{N}\sum_{i=1}^N\left|U(t_0, Z_i, m^{N,i}_\bZ)-v^{N,i}(t_0, \bZ)\right|  \leq C
N^{-1}
\qquad {\rm a.e.},
$$
where $\bZ=(Z_1, \dots, Z_N)$ with $Z_1, \dots, Z_N$ i.i.d. random variables with uniform density on $\T^d$. The support of $\bZ$ being $(\T^d)^N$, we derive from the continuity of $U$ and of the $(v^{N,i})_{i \in \{1,\dots,N\}}$ that the above inequality holds for any $\bx\in (\T^d)^N$: 
$$
\frac{1}{N}\sum_{i=1}^N\left|U(t_0, x_i, m^{N,i}_{\bx})-v^{N,i}(t_0, \bx)\right|  \leq C
N^{-1}
\qquad \forall \bx\in (\T^d)^N. 
$$
Then we  use the Lipschitz continuity of $U$ with respect to $m$ to replace $U(t_0, x_i, m^{N,i}_{\bx})$ by $U(t_0, x_i, m^{N}_{\bx})$ in the above inequality, the additional error term being of order $1/N$. 

For proving (ii), we use the the Lipschitz continuity of $U$ and a result by Dereich, Scheutzow and Schottstedt \cite{DeSeSc} to deduce that, for $d\geq 3$  and for any $x_i\in \T^d$, 
$$
\begin{array}{l}
 \ds  \int_{\T^{d(N-1)}} |u^{N,i}(t,\bx)- U(t,x_i, m_0)|\prod_{j\neq i} m_0(dx_j) \\  
\qquad \qquad \ds = \;  \ds \int_{\T^{d(N-1)}} |U(t,x_i,m^{N,i}_\bx)- U(t,x_i, m_0)|\prod_{j\neq i} m_0(dx_j)\\
\qquad \qquad \ds \leq  C\int_{\T^{d(N-1)}} \dk(m^{N,i}_\bx,m_0)\prod_{j\neq i} m_0(dx_j) \; \leq \; CN^{-1/d}.
\end{array}
$$
If $d=2$, following Ajtai, Komlos and Tusnády \cite{AjKoTu}, the right-hand side has to be replaced by $N^{-1/2}\log(N)$. 
Combining Theorem \ref{thm:Cvyx} with the above inequality, we obtain therefore, for $d\geq 3$, 
$$
\begin{array}{l}
\ds \left\| w^{N,i}(t_0,\cdot, m_0)-U(t_0,\cdot, m_0)\right\|_{L^1(m_0)} \\
\qquad  \ds = \inte \left| \int_{\T^{d(N-1)}}v^{N,i}\bigl(t,(x_j)\bigr)\prod_{j\neq i}m_0(dx_j) -U(t,x_i, m_0)\right|\ dm_0(x_i)\\
\qquad \ds \leq \E\bigl[| v^{N,i}(t,\bZ)-u^{N,i}(t,\bZ)|\bigr]  + \int_{\T^{dN}} |u^{N,i}(t,\bx)- U(t,x_i, m_0)|\prod_{j=1}^N m_0(dx_j) \\
\qquad \ds \leq C
N^{-1}+ CN^{-1/d}\; \leq \;  CN^{-1/d}.
\end{array}
$$
As above, the right-hand side is  $N^{-1/2}\log(N)$ if $d=2$. 
This shows part (ii) of the theorem. 
\end{proof}

\begin{proof}[Proof of Corollary \ref{cor.NashLim}.] We fix $(t,x_1,m)\in [0,T]\times \T^d\times \Pk$ and assume that there exists $v\in \R$ such that 
$$
\limsup_{N\to+\infty, \ x_1'\to x_1, \ m^{N,1}_{\bx'}\to m} \left| v^{N,1}(t,\bx')-v\right| =0.
$$
Our aim is to show that, if $x_1$ belongs to the support of $m$, then $v=U(t,x_1,m)$. For this we first note, from a standard application of the maximum principle, that the 
$(v^{N,i})_{i \in \{1,\dots,N\}}$ are uniformly bounded by a constant $M$ (independent 
of $N$).

Fix $\ep>0$. By our assumption there exists $N_0>0$ and $\delta>0$ such that 
\be\label{tagada}
\left| v^{N,1}(t,\bx')-v\right|\leq \ep \qquad {\rm if }\; N\geq N_0, \;  \dk(m^{N,1}_{\bx'},m)\leq \delta\; {\rm and}\; |x_1-x_1'|\leq \delta. 
\ee
As 
$$
\lim_{N\to+\infty} \int_{(\T^d)^{N-1}}\dk\left(m^{N,1}_{\bx'} ,m\right)
\prod_{
j= 2}^N m(dx'_j)=0, 
$$
we can also  choose $N_0$ large enough so that 
$$
 \int_{(\T^d)^{N-1}} \1_{\bigl\{  \dk\left(m^{N,1}_{\bx'} ,m\right)\geq \delta
 \bigr\}} \prod_{j=2}^N m(dx'_j)\leq \ep \qquad {\rm if }\; N\geq N_0 \; {\rm and}\; |x_1-x_1'|\leq \delta.
$$
Then, integrating \eqref{tagada} over $(\T^d)^{N-1}$, we obtain 
$$
\left|w^{N,1}(t,x'_1)- v\right| \leq \ep+M\ep= \ep(M+1) \qquad {\rm if }\; N\geq N_0\; {\rm and}\; |x_1-x_1'|\leq \delta.
$$
We now integrate this inequality with respect to the measure $m$  on the ball 
$B(x_1,\delta)$: 
$$
\int_{B(x_1,\delta)} \left|w^{N,1}(t,x_1')- v\right| dm(x_1') \leq  \ep(M+1)m\bigl(B(x_1,\delta)\bigr). 
$$
Now Theorem \ref{thm:mainCV}-(ii) states that $w^{N,1}(t,\cdot)$ converges in $L^1_m$ to $U(t,\cdot,m)$. Thus, letting $N\to+\infty$ in the above inequality, we get 
$$
\int_{B(x_1,\delta)} \left|U(t,x_1',m)- v\right| dm(x_1') \leq  \ep(M+1) 
m\bigl(B(x_1,\delta)\bigr). 
$$
Since $U$ is continuous and $x_1$ is in the support of $m$, this last inequality implies that $v=  U(t,x_1,m)$. 
\end{proof}

\subsection{Propagation of chaos}

We now prove Theorem \ref{thm:CvMFG}. Let us recall the notation. 
Throughout this part, $(v^{N,i})_{i \in \{1,\dots,N\}}$ is the solution of the Nash system \eqref{Nash0} and the $((Y_{i,t})_{t \in [t_{0},T]})_{i \in \{1,\dots,N\}}$ are ``optimal trajectories" for this system, i.e., solve \eqref{defyit}
with $Y_{i,t_{0}}=Z_{i}$ as initial condition at time 
$t_{0}$. 
Our aim is to understand the behavior of the $((Y_{i,t})_{t \in [t_{0},T]})_{i \in \{1,\dots,N\}}$ for a large number of players $N$.  

For any $i\in \{1,\dots, N\}$, let 
$(\tilde X_{i,t})_{t \in [t_{0},T]}$ be the solution the SDE of McKean-Vlasov type:
$$
d\tilde X_{i,t} =  -D_pH\left(\tilde X_{i,t}, D_xU(t,\tilde X_{i,t},{\mathcal L}(\tilde X_{i,t}\vert W)\right)dt +\sqrt{2}dB^i_t+\sqrt{2\beta} dW_t, \qquad \tilde X_{i,t_0}= Z_i. 
$$
Recall that, for any $i\in \{1,\dots, N\}$, the conditional law ${\mathcal L}(\tilde X_{i,t}\vert W)$ is equal to $(m_t)$ where $(u_t,m_t)$ is the solution of the MFG system with common noise given by \eqref{eq:se:3:tilde:HJB:FP:t0}-\eqref{eq:se:3:tilde:HJB:FP:prescription} (see section \ref{subsub:uniqueness.sec4}).
Solvability of the McKean-Vlasov equation may be discussed on the model of 
\eqref{eq:EDS:MKV:V}.

Our aim is to show that 
$$
\E\Bigl[ \sup_{t\in [t_0,T]} \bigl|Y_{i,t}-\tilde X_{i,t}\bigr| \Bigr]
\leq C N^{-1/(d+8)},
$$
for some $C>0$. 
Before starting the proof of Theorem \ref{thm:CvMFG}, we need to estimate the distance between the empirical measure associated with the $(\tilde X_{i,t})_{i\in \{1,\dots,N\}}$ and $m_t$. For this, let us  set $\tilde \bX_t= (\tilde X_{i,t})_{i\in \{1,\dots,N\}}$. As the $(\tilde X_{i,t})$ are, conditional on $W$, i.i.d. random variables with law $m_t$, we have by a variant of a result due to Horowitz and Karandikar (see for instance Rashev and Rüschendorf \cite{RaRu98}, Theorem 10.2.1): 

\begin{Lemma} \label{lem:LLN} 
$$
\E\Bigl[\sup_{t\in [t_0,T] } \dk\Bigl(m^{N,i}_{\tilde \bX_t}, m_t\Bigr)\Bigr]
\leq C N^{-1/(d+8)}.
$$
\end{Lemma}

\begin{proof} The proof is exactly the same as for Theorem  10.2.7 in \cite{RaRu98} (for the i.i.d. case). In this proof independence is only used twice and, in both cases, one can simply replace the expectation by the conditional expectation. 
\end{proof}

\begin{proof}[Proof of Theorem \ref{thm:CvMFG}.]  The proof is a direct application of  Theorem \ref{thm:Cvyx} combined with the following estimate on the distance between 
$(\tilde X_{i,t})_{t \in [t_{0},T]}$ and the solution 
$(X_{i,t})_{t \in [t_{0},T]}$ of \eqref{defxit}:
\be
\label{xit-tildexit}
\E \Bigl[ \sup_{t\in [t_0,T] } \bigl| X_{i,t}-\tilde X_{i,t}\bigr| \Bigr] \leq 
C N^{-1/(d+8)}.
\ee
Indeed, by the triangle inequality, we have, provided that 
\eqref{xit-tildexit} holds true:
\begin{equation*}
\begin{split}
\E \Bigl[ \sup_{t\in [t_0,T] } \bigl| Y_{i,t}-\tilde X_{i,t}\bigr| \Bigr] &\leq \E \Bigl[ 
\sup_{t\in [t_0,T] } \bigl| Y_{i,t}-X_{i,t}\bigr| \Bigr] 
+\E \Bigl[ \sup_{t\in [t_0,T] } \bigl| X_{i,t}-\tilde X_{i,t}\bigr| \Bigr] 
\\
&\leq C\bigl( N^{-1}+N^{-1/(d+8)} \bigr),
\end{split}
\end{equation*}
where we used
\eqref{estixi-yi} to pass from the first to the second line. 
\vspace{4pt}

It now remains to check \eqref{xit-tildexit}. For this, we fix $ i\in \{1, \dots, N\}$ and let 
$$
\rho(t)= \E \Bigl[ \sup_{s\in [t_0,t] }\bigl| X_{i,s}-\tilde X_{i,s}\bigr|\Bigr]. 
$$
Then, for any $s\in [t_0,t]$, we have  
\begin{align*}
\bigl| X_{i,s}-\tilde X_{i,s}\bigr|
&\leq \int_{t_0}^s 
\bigl| -D_pH \bigl(X_{i,r}, D_{x_i}u^{N,i}(r, \bX_r)
\bigr)+D_pH \bigl(\tilde  X_{i,r}, D_xU \bigl(r, \tilde X_{i,r},
m_{r}\bigr)
\bigr)  \bigr| dr 
\\
&\leq \int_{t_0}^s \bigl| -D_pH \bigl(X_{i,r}, D_{x}U \bigl(r, X_{i,r}, m^{N,i}_{\bX_r}\bigr)
\bigr) + D_pH \bigl(\tilde  X_{i,r}, D_xU \bigl(r, \tilde X_{i,r},m^{N,i}_{\tilde \bX_r}\bigr)\bigr)  \bigr|  dr
\\
&\hspace{5pt}+\int_{t_0}^s \bigl| -D_pH
\bigl(\tilde  X_{i,r}, D_xU \bigl(r, \tilde X_{i,r},m^{N,i}_{\tilde \bX_r}\bigr)\bigr) 
+D_pH\bigl(\tilde  X_{i,r}, D_xU \bigl(r, \tilde X_{i,r},m_{r}\bigr)\bigr)  \bigr|  dr.
\end{align*}
As $(x,m)\to D_xU(t,x,m)$ is uniformly Lipschitz continuous, we get
\begin{align*}
\bigl| X_{i,s}-\tilde X_{i,s}\bigr|
\leq C\int_{t_0}^s 
\Bigl(
|X_{i,r}-\tilde  X_{i,r}|+ \dk\bigl(m^{N,i}_{\bX_r}, m^{N,i}_{\tilde \bX_r}\bigr) 
+ \dk \bigl( m^{N,i}_{\tilde \bX_r}, m_{r}\bigr)
\Bigr) dr, 
\end{align*}
where 
\be\label{dkmNXtildeX}
\dk\bigl(m^{N,i}_{X_t}, m^{N,i}_{\tilde X_t}\bigr)
\leq \frac{1}{N-1} \sum_{j\neq i} |X_{j,s}-\tilde X_{j,s}|.
\ee
Hence
\begin{equation*}
\bigl| X_{i,s}-\tilde X_{i,s}\bigr| \leq 
C\int_{t_0}^s 
\Bigl(
 |X_{i,r}-\tilde  X_{i,r}|+ \frac{1}{N-1} \sum_{j\neq i} |X_{j,r}-\tilde X_{j,r}|+ \dk ( m^{N,i}_{\tilde \bX_r}, m_{r})\Bigr) dr. 
\end{equation*}
Taking the supremum over $s\in [t_0,t]$ and then the expectation, we have, recalling that the random variables $(X_{j,r}-\tilde X_{j,r})_{j \in \{1,\dots,N\}}$ have the same law: 
\begin{align*}
\rho(t) &=\E\Bigl[ \sup_{s\in [t_0,t]} \bigl| X_{i,s}-\tilde X_{i,s}\bigr|\Bigr] 
\\
&\leq  C\int_{t_0}^t \biggl( \E\Bigl[\sup_{r\in [t_0,s]} |X_{i,r}-\tilde  X_{i,r}| 
\Bigr]
+ \frac{1}{N-1}\sum_{j\neq i} \E\Bigl[\sup_{r\in [t_0,s]} |X_{j,r}-\tilde X_{j,r}|\Bigr]
\biggr) ds
\\
&\hspace{15pt} + C \, \E\Bigl[\sup_{r\in [t_0,T]} \dk \bigl( m^{N,i}_{\tilde \bX_r}, m_{r}\bigr) \Bigr]  
\\
& \leq  C\int_{t_0}^t \rho(s) ds +  CN^{-1/(d+8)},
\end{align*}
where we used  Lemma \ref{lem:LLN} for the last inequality. Then Gronwall inequality gives \eqref{xit-tildexit}. 
\end{proof}





\newpage 
\section{Appendix}
We now provide several basic results on the notion of 
differentiability on the space of probability measures
used in the paper, including a short comparison 
with the derivative
on the set of random variables.

\subsection{Link with
the derivative on the set of random variables}
As a first step, we discuss the connection between 
the derivative $\delta U/\delta m$ in Definition 
\ref{def:Diff} and 
the  derivative
introduced by Lions in 
\cite{LcoursColl}
and used  
(among others) in
\cite{BuLiPeRa,CgCrDe}. 

The notion introduced in 
\cite{LcoursColl} consists in lifting up functionals defined on the space of probability measures into functionals defined on the set of random variables. When the underlying probability measures are defined on a (finite dimensional) vector space
$E$ (so that the random variables that are distributed along these probability measures also take values in $E$), this permits to benefit from the standard differential calculus on the Hilbert space formed 
by the square-integrable random variables with values in $E$. 

Here the setting is slightly different as the 
probability measures that are considered throughout the article 
are defined on the torus. 
Some care is thus needed in the definition of the linear structure 
underpinning the argument. 

\subsubsection{First order expansion 
with respect to 
torus-valued random variables.
}
\label{subsubse:1storder:rv:torus}
On the torus $\T^d$, we may consider the group of translations 
$(\tau_{y})_{y \in \R^d}$, parameterized by elements $x$ of $\R^d$. For any $y \in \R^d$, $\tau_{y}$ maps $\T^d$ into itself. 
The mapping $\R^d \ni y \mapsto \tau_{y}(0)$ being obviously measurable, this permits to define, for any square integrable 
random variable $\tilde X \in L^2(\Omega,{\mathcal A},\P;\R^d)$ 
(where $(\Omega,{\mathcal A},\P)$ is an atomless probability space), the random variable $\tau_{\tilde X}(0)$, which takes values in $\T^d$. Given a mapping $U : {\mathcal P}(\T^d) \rightarrow \R$, we may define its \textit{lifted}
version as
\begin{equation}
\label{eq:def:lift:U}
\tilde{U} : L^2(\Omega,{\mathcal A},\P;\R^d) \ni 
\tilde X \mapsto 
\tilde{U}( \tilde X) = U\bigl( {\mathcal L}(\tau_{\tilde X}(0)) \bigr),
\end{equation}
where the argument in the right-hand side denotes the law 
of $\tau_{\tilde X}(0)$ (seen as a $\T^d$-valued random variable). 
Quite obviously, ${\mathcal L}(\tau_{\tilde X}(0))$ only depends on 
the law of $\tilde X$. 

Assume now that the mapping $\tilde{U}$ is continuously Fr\'echet differentiable on $L^2(\Omega,{\mathcal A},\P;\R^d)$. 
What \cite{LcoursColl}
says is that, 
for any 
$\tilde{X} \in L^2(\Omega,{\mathcal A},\P;\R^d)$, the Fr\'echet 
derivative has the form
\begin{equation}
\label{eq:Lions:derivative:torus}
D \tilde{U}(\tilde{X}) = \widetilde{\partial_{\mu}U}\bigl(
{\mathcal L}(\tilde X)
\bigr)(\tilde X), \quad \P \ \text{almost surely,}
\end{equation}
for a mapping $
\{ \widetilde{\partial_{\mu} U}({\mathcal L}(\tilde X))
: \R^d \ni y \mapsto  
\widetilde{\partial_{\mu} U}({\mathcal L}(\tilde X))(y)
\in \R^d
\}
 \in L^2(\R^d,{\mathcal L}(\tilde X))$. 
 This relationship is fundamental. 
 Another key observation
 is that, for any random variables 
 $\tilde{X}$ and $\tilde{Y}$ with values in $\R^d$
 and $\tilde{\xi}$ with values in $\Z^d$, it holds that 
 \begin{equation*}
 \begin{split}
 \lim_{\varepsilon \rightarrow 0}
 \frac1{\varepsilon}
 \Bigl[ \tilde{U}
 \bigl( \tilde{X} + \tilde{\xi} + \varepsilon \tilde{Y}
 \bigr) - \tilde{U}
 \bigl( \tilde{X} \bigr) \Bigr]
 &= \E
 \Bigl[
 \bigl\langle
  D \tilde{U}
 \bigl( 
 \tilde{X} + \tilde{\xi}
\bigr),
\tilde{Y}
 \bigr\rangle
 \Bigr],
 \end{split}
 \end{equation*}
which is, 
by the simple fact that 
$\tau_{\tilde{X}+\tilde{\xi}}(0)
= \tau_{\tilde{X}}(0)$, also equal to 
\begin{equation*}
 \begin{split}
 \lim_{\varepsilon \rightarrow 0}
 \frac1{\varepsilon}
 \Bigl[ \tilde{U}
 \bigl( \tilde{X}  + \varepsilon \tilde{Y}
 \bigr) - \tilde{U}
 \bigl( \tilde{X} \bigr) \Bigr]
 &= \E
 \Bigl[
 \bigl\langle
  D \tilde{U}
 \bigl( 
 \tilde{X}
\bigr),
\tilde{Y}
 \bigr\rangle
 \Bigr],
 \end{split}
 \end{equation*}
proving that 
\begin{equation}
\label{eq:frechet:derivative:torus}
D \tilde{U} \bigl(\tilde{X} \bigr) = 
D \tilde{U} \bigl(\tilde{X} +  \tilde{\xi} \bigr).
\end{equation}

Consider now a random variable $X$ from 
$\Omega$ with values into $\T^d$. 
With $X$, we may associate the random variable
$\hat{X}$, with values in $[0,1)^d$, given 
(pointwise)
as the only representative of $X$ in $[0,1)^d$. 
We observe that the law of $\hat{X}$ is uniquely determined by 
the law of $X$ and that for any Borel function 
$h : \T^d \rightarrow \R$, 
\begin{equation*}
{\mathbb E}[h(X)] = {\mathbb E}[\hat{h}(\hat{X})],
\end{equation*}
where $\hat{h}$ is the identification of $h$ 
as a function from $[0,1)^d$ to $\R$.  

Then, we 
deduce from 
\eqref{eq:Lions:derivative:torus}
that 
\begin{equation*}
D \tilde{U}(\hat{X}) = 
\widetilde{\partial_{\mu}
U} \bigl( {\mathcal L}(\hat X) \bigr) 
( \hat{X} ), \quad \P \ \text{almost surely}.
\end{equation*}
Moreover, from 
\eqref{eq:frechet:derivative:torus},
we also have,
 for any random variable 
$\hat{\xi}$ with values in $\Z^d$,
\begin{equation*}
D \tilde{U}(\hat{X} + \hat{\xi}) = 
\widetilde{
\partial_{\mu}
U} \bigl( {\mathcal L}(\hat X) \bigr) 
( \hat{X} ), \quad \P \ \text{almost surely}.
\end{equation*}
Since $\partial_{\mu} U({\mathcal L}(\hat{X}))
(\cdot)$ is in $L^2(\R^d,{\mathcal L}(\hat{X}))$ and 
$\hat{X}$ takes values in $[0,1)^d$, we can 
identify $\partial_{\mu} U({\mathcal L}(\hat{X}))(\cdot)$
with a function in $L^2(\T^d,{\mathcal L}(X))$. 
Without any ambiguity, we may denote this function
(up to a choice of a version) by 
\begin{equation*}
\T^d \ni y \mapsto \partial_{\mu}
U \bigl( {\mathcal L}(X) \bigr) 
(y). 
\end{equation*}
As an application we have that, 
for any random variables $X$ and $Y$
with values in $\T^d$,
\begin{equation*}
\begin{split}
U\bigl( {\mathcal L}(Y)
\bigr) 
- U \bigl( {\mathcal L}(X)
\bigr) 
&= \tilde{U}
\bigl( \hat{Y} \bigr) -
\tilde{U} \bigl( \hat{X}
\bigr)
\\
&= 
\E \int_{0}^1
\Bigl\langle
D \tilde{U}
\bigl( 
{\mathcal L}( \lambda \hat{Y}
+ (1-\lambda) \hat{X}
) 
\bigr), 
\hat{Y}
- \hat{X}
\Bigr\rangle d\lambda.
\end{split}
\end{equation*}
Now, we can write 
\begin{equation*}
\lambda \hat{Y}
+ (1-\lambda) \hat{X}
= \hat{X} + \lambda ( \hat{Y} - \hat{X})
= \hat{Z}, \quad \text{with} \
Z = \tau_{\lambda (\hat{Y}-\hat{X})}(X).
\end{equation*}
Noticing that $Z$ is a random variable with values in 
$\T^d$, we deduce that 
\begin{equation*}
\begin{split}
U\bigl( {\mathcal L}(Y)
\bigr) 
- U \bigl( {\mathcal L}(X)
\bigr) 
&= 
\E \int_{0}^1
\Bigl\langle
\partial_{\mu} U
\bigl( 
{\mathcal L}( \tau_{\lambda (\hat{Y} - \hat{X})}(X)
) 
\bigr)
( \tau_{\lambda (\hat{Y} - \hat{X})}(X)
) , 
\hat{Y}
- \hat{X}
\Bigr\rangle d\lambda.
\end{split}
\end{equation*}
Similarly, for any random 
variable $\hat{\xi}$ with values in $\Z^d$,
\begin{equation*}
\begin{split}
U\bigl( {\mathcal L}(Y)
\bigr) 
- U \bigl( {\mathcal L}(X)
\bigr) 
&= \tilde{U}
\bigl( \hat{Y} + \hat{\xi} \bigr) -
\tilde{U} \bigl( \hat{X}
\bigr)
\\
&= 
\E \int_{0}^1
\Bigl\langle
D \tilde{U}
\bigl( 
 \hat{X}
+ \lambda (\hat{Y} + \hat{\xi} - \hat{X}
) 
\bigr), 
\hat{Y} +  \hat{\xi}
- \hat{X}
\Bigr\rangle d\lambda.
\end{split}
\end{equation*}
Now, 
$
 \hat{X}
+ \lambda (\hat{Y} + \hat{\xi} - \hat{X}
)$
writes $\hat{Z} + \hat{\zeta}$, where 
$\hat{\zeta}$ is a random variable with values in 
$\Z^d$ and $\hat{Z}$ is associated with 
the $\T^d$-valued random variable
${Z} = \tau_{ \lambda (\hat{Y} + \hat{\xi} - \hat{X}
)}(X)$,
so that 
\begin{equation}
\label{eq:derivative:torus:b}
\begin{split}
U\bigl( {\mathcal L}(Y)
\bigr) 
- U \bigl( {\mathcal L}(X)
\bigr) 
&=
\E \int_{0}^1
\Bigl\langle
D \tilde{U}
\bigl( 
 \hat{Z} 
\bigr), 
\hat{Y} +  \hat{\xi}
- \hat{X}
\Bigr\rangle d\lambda
\\
&=
\E \int_{0}^1
\Bigl\langle
\partial_\mu U
\bigl( 
{\mathcal L}( 
\tau_{\lambda  (\hat{Y} + \hat{\xi} - \hat{X}
) }(X)
\bigr)
\bigl(
\tau_{\lambda  (\hat{Y} + \hat{\xi} - \hat{X}
) }(X)
\bigr)
, 
\hat{Y} +  \hat{\xi}
- \hat{X}
\Bigr\rangle d\lambda.
\end{split}
\end{equation}
The fact that $\hat{\xi}$ 
can be chosen in a completely arbitrary way 
says that the choice of the representatives of 
$X$ and $Y$ in the above formula does not matter. Of course, 
this
is 
a consequence of the periodicity structure underpinning 
the whole analysis. Precisely, for any 
representatives $\bar{X}$ and $\bar{Y}$
(with values in $\R^d$)
 of $X$ and
$Y$,
we can write
\begin{equation}
\label{eq:derivative:torus}
\begin{split}
U\bigl( {\mathcal L}(Y)
\bigr) 
- U \bigl( {\mathcal L}(X)
\bigr) 
&=
\E \int_{0}^1
\Bigl\langle
\partial_\mu U
\bigl( 
{\mathcal L}( 
\tau_{\lambda  (\bar{Y}  - \bar{X}
) }(X)
\bigr)
\bigl(
\tau_{\lambda  (\bar{Y}  - \bar{X}
) }(X)
\bigr)
, 
\bar{Y} 
- \bar{X}
\Bigr\rangle d\lambda.
\end{split}
\end{equation}
Formula 
\eqref{eq:derivative:torus}
gives a rule for expanding, 
along torus-valued random variables, functionals 
depending on 
torus-supported probability measures. 
It is the analogue of the \textit{differentiation rule}
defined in 
\cite{LcoursColl}
on the space of probability measures on $\R^d$
through the differential calculus in 
$L^2(\Omega,{\mathcal A},\P;\R^d)$. 
\vspace{4pt}

In particular, if 
$\tilde{U}$ is continuously differentiable, 
with (say) $D \tilde{U}$ being Lipschitz continuous
on $L^2(\Omega,{\mathcal A},\P;\R^d)$, then
(with the same notations as in 
\eqref{eq:derivative:torus:b}) 
\begin{equation}
\label{eq:regularity:derivative:measure}
\begin{split}
{\mathbb E}
\bigl[ \vert D \tilde U(\hat{Y}) - D \tilde U(\hat{X})
\vert^2 \bigr]
&= {\mathbb E}
\bigl[ \vert D\tilde U(\hat{Y} + \hat{\xi}) - D \tilde U(\hat{X})
\vert^2 \bigr]
\\
&\leq C {\mathbb E}
\bigl[ \vert \hat{Y} + \hat{\xi} - \hat{X}
\vert^2 \bigr]. 
\end{split}
\end{equation}
Now, for two random variables $X$ and $Y$ with values
in the torus, one may find a random variable
$\hat{\xi}$, with values in $\Z^d$, such that, pointwise, 
\begin{equation*} 
\hat{\xi} = \text{argmin}_{c \in \Z^d}
\vert \tau_{c}(\hat{Y}) - \hat{X} \vert,
\end{equation*}
the right-hand side being the distance
$d_{\T^d}(X,Y)$
 between $X$ and $Y$
on the torus. Put it differently,
we may choose $\hat{\xi}$ such that $\vert
\hat{Y}+\hat{\xi} - \hat{X}
\vert =
d_{\T^d}(X,Y)$. Plugged into 
\eqref{eq:regularity:derivative:measure},
this shows that 
the Lipschitz property of $D \tilde U$ (on $L^2(\Omega,{\mathcal A},\P;\R^d)$) reads as 
a Lipschitz property with respect to 
torus-valued
 random variables. 

Next, we make the connection between 
the mapping ${\mathcal P}(\T^d) \times \T^d
\ni (m,y) \mapsto
\partial_{\mu} U(m)(y) \in \R^d$
and the derivative ${\mathcal P}(\T^d) \times \T^d
\ni (m,y) \mapsto
[\delta U/\delta m](m,y) \in \R^d$
defined in 
Definition 
\ref{def:Diff}. 

\subsubsection{From differentiability along random variables
to differentiability in $m$}

\begin{Proposition}
\label{prop:appendix:lions->def}
Assume that the function $U$ 
is differentiable in the sense explained in 
Subsubsection 
\ref{subsubse:1storder:rv:torus}
and thus
satisfies the expansion
formula 
\eqref{eq:derivative:torus}.
Assume moreover that there
exists a continuous version of 
the mapping
$\partial_{\mu} U : {\mathcal P}(\T^d) \times \T^d \ni (m,y) 
\mapsto \partial_{\mu} U(m,y) \in \R^d$. 

Then, $U$ is differentiable in the sense of Definition \ref{def:Diff}. Moreover,
$\delta U/\delta m$ is continuously differentiable with respect to the second variable
and
\begin{equation*}
D_{m} U(m,y) = \partial_{\mu} U(m)(y), \quad m \in {\mathcal P}(\T^d),
\ y \in \T^d. 
\end{equation*}
\end{Proposition}

\begin{proof}
\textit{First step.}
The first step is to prove that, for any $m \in {\mathcal P}(\T^d)$, there exists 
a continuously differentiable map $V(m,\cdot) : \T^d \ni y \mapsto V(m,y) \in \R$ such that 
\begin{equation*}
 \partial_{\mu} U(m)(y) = D_{y} V(m,y), \quad y \in \T^d. 
\end{equation*}
The strategy is to prove that $\partial_{\mu} U(m) : \T^d 
\mapsto \partial_{\mu} U(m)(y)$ is orthogonal (in $L^2(\T^d,dy)$)
to divergence free vector fields. It suffices to prove that, for any smooth  divergence free vector field
$b : \T^d \rightarrow \R^d$,
\begin{equation*}
\int_{\T^d} \langle \partial_{\mu} U(m)(y), b(y) \rangle dy
= 0.  
\end{equation*}
Since $\partial_{\mu} U$ is jointly continuous in $(m,y)$, it is enough to prove the above identity 
for any $m$ with a positive smooth density. When $m$ is not smooth, we may indeed approximate it
by $m \star \rho$, where $\star$ denotes the convolution 
and $\rho$ a smooth kernel on $\R^d$ with full support. 

With such an $m$ and such a $b$, we consider the ODE
(set on $\R^d$ but driven by periodic coefficients)
\begin{equation*}
 dX_{t} = \frac{b(X_{t})}{m(X_{t})}  dt, \quad t \geq 0,
\end{equation*}
the initial condition $X_{0}$ being $[0,1)^d$-valued and distributed
according to some $m \in {\mathcal P}(\T^d)$ (identifying 
$m$ with a probability measure on $[0,1)^d$).  
By periodicity of $b$ and $m$, 
$(X_{t})_{t \geq 0}$ generates on 
$\T^d$ a flow of probability measures 
$(m_{t})_{t \geq 0}$
satisfying the Fokker Planck equation 
\begin{equation*}
\partial_{t} m_{t} = - \textrm{div}( \frac{b}{m} m_{t} \bigr), \quad t \geq 0, \quad m_{0} = m.  
\end{equation*}
Since $b$ is divergence free, we get that $m_{t} = m$ for all $t \geq 0$.  
Then, for all $t \geq 0$,
\begin{equation*}
U\bigl(m_{t}\bigr) - U\bigl(m_{0}\bigr) = 0,
\end{equation*}
so that, with the same notation as
in 
\eqref{eq:def:lift:U}, 
$\lim_{t \searrow 0} 
[( \tilde{U}(X_{t}) - \tilde{U}(X_{0}))/t] = 0$.
Now, choosing 
$\bar{Y} =X_{t}$
and $\bar{X}=X_{0}$ in 
\eqref{eq:derivative:torus},
we get
\begin{equation*}
\begin{split}
 \int_{\T^d} \langle \partial_{\mu} U( m)( y), b(y) \rangle dy =0.
 \end{split}
\end{equation*}
We easily deduce that 
$\partial_{\mu} U( m)$ reads
as a gradient that is 
\begin{equation*}
\partial_{\mu} U(m)(y)
= \partial_{y} V(m,y). 
\end{equation*}
It is given as a solution of the Poisson equation
\begin{equation*}
 \Delta V(m,y) = {\rm div}_y  \ \partial_{\mu} U(m)(y)
\end{equation*}
Of course, $V(m,\cdot)$ is uniquely defined up to an additive constant. We can choose it in such a way that 
\begin{equation*}
\int_{\T^d} V(m,y) dm(y) = 0. 
\end{equation*} 
Using the representation of the solution of the Poisson equation by means of the Poisson kernel, we easily deduce that the 
function
$V$ is jointly continuous. 
\vspace{4pt}

\textit{Second step.}
The second step of the proof is to check that 
Definition \ref{def:Diff} holds true. Let us consider two measures of the form $m^N_X$ and $m^N_Y$, where $N\in \N^*$, $X=(x_1,\dots, x_N)\in (\T^d)^N$ is  such that $x_i\neq x_j$ and $Y=(y_1,\dots,y_N)\in (\T^d)^N$. Without loss of generality we assume that the indices for $Y$ are such that 
\be\label{mrjeznsnd}
\dk(m^N_X, m^N_Y)=
\frac{1}{N} \sum_{i=1}^N d_{{\mathbb T}^d}(x_i,y_i)
= 
\frac{1}{N} \sum_{i=1}^N |\bar x_i- \bar y_i|,
\ee
where $\bar{x}_{1},\dots,\bar{x}_{N}$ 
and $\bar{y}_{1},\dots,\bar{y}_{N}$
are well-chosen representatives, in $\R^d$, 
of the points $x_{1},\dots,x_{N}$
and $y_{1},\dots,y_N$ in $\T^d$ ($d_{\T^d}$
denoting the distance on the torus). 
Let $\bar X$ be a random variable such that 
$\P(\bar X=\bar x_i)= 1/N$ and $\bar Y$ be the random variable defined by $\bar Y=\bar y_i$ if $\bar X=\bar x_i$. Then, with the 
same notations as in 
\eqref{eq:def:lift:U},
$\P_{{\mathcal L}(\tau_{\bar{X}}(0))}= m^N_X$ and 
$\P_{{\mathcal L}(\tau_{\bar{Y}}(0))}= m^N_Y$. 

Thanks to \eqref{eq:derivative:torus}, we 
get
$$
\begin{array}{l}
\ds U(m^N_Y)-U(m^N_X) 
= 
\ds
\int_{0}^1
\E\Bigl[ \Bigl\langle \partial_\mu U\Bigl( {\mathcal L}
\bigl( {\mathcal \tau}_{\lambda \bar{Y}
+(1-\lambda) \bar{X}}(0)
\bigr)
\Bigr)\bigl( 
\tau_{\lambda \bar{Y}
+(1-\lambda) \bar{X}}(0)
\bigr), \bar Y-\bar X\Bigr\rangle \Bigr] 
d\lambda  
\end{array}
$$
So, if $w$ is a modulus of continuity of the map $\partial_\mu U$ on the compact set ${\mathcal P}(\T^d)\times \T^d$, we obtain by \eqref{mrjeznsnd}:   
\begin{equation}
\label{eq:expansion:1storder:torus:bary:barx}
\begin{split}
&\biggl\vert U(m^N_Y)-U(m^N_X) -  \int_0^1
\E\Bigl[ 
 \Bigl\langle \partial_\mu U\bigl( m^N_{X}
\bigr)\bigl( 
\tau_{\lambda \bar{Y}
+(1-\lambda) \bar{X}}(0)
\bigr), \bar Y-\bar X\Bigr\rangle \Bigr] 
d\lambda
\biggr\vert
\\
&\hspace{15pt} \leq  \E\bigl[ |\bar Y-\bar X|\bigr|] 
w \bigl(\dk(m^N_X, m^N_Y)
\bigr) = \dk(m^N_X, m^N_Y) 
w \bigl(\dk(m^N_X, m^N_Y)
\bigr). 
\end{split}
\end{equation}
Moreover, since  
$D_y V(m,y)= \partial_\mu U(m)(y)$, we have  
\begin{equation*}
\begin{split}
&\int_0^1\E\Bigl[ \Bigl\langle \partial_\mu U(m^N_X)
\bigl( \tau_{\lambda \bar{Y} + (1-\lambda) \bar{X}}
(0)
\bigr)\bigl( 
\tau_{\lambda \bar{Y} + (1-\lambda) \bar{X}}
\bigr), \bar Y-\bar X \Bigr\rangle \Bigr] 
d\lambda
\\ 
&\qquad =  \frac{1}{N}\sum_{i=1}^N  \int_0^1  \bigl\langle
D_y V \bigl(m^N_X, 
\tau_{\lambda \bar y_i+(1-\lambda) \bar x_i}(0)
\bigr), \bar y_i- \bar x_i \bigr\rangle d\lambda 
\\
&\qquad =
\frac{1}{N}\sum_{i=1}^N  \int_0^1  \bigl\langle
D_y V \bigl(m^N_X,
\lambda \bar y_i+(1-\lambda) \bar x_i
\bigr), \bar y_i- \bar x_i \bigr\rangle d\lambda, 
\end{split}
\end{equation*}
where we saw $D_{y}V(m^N_{X},\cdot)$ as a 
periodic function defined on the whole $\R^d$. 
Then,
\begin{equation*}
\begin{split}
\int_0^1\E\Bigl[ \Bigl\langle \partial_\mu U(m^N_X)
\bigl( \tau_{\lambda \bar{Y} + (1-\lambda) \bar{X}}
(0)
\bigr)\bigl( 
\tau_{\lambda \bar{Y} + (1-\lambda) \bar{X}}
\bigr), \bar Y-\bar X \Bigr\rangle \Bigr] 
d\lambda
 = \inte V (m^N_X, x)d(m^N_Y-m^N_X)(x) .
\end{split}
\end{equation*}
By density of the measures of the form $m^N_X$ and $m^N_Y$ and by continuity of $V$, 
we deduce from 
\eqref{eq:expansion:1storder:torus:bary:barx}
 that, for any measure $m,m'\in {\mathcal P}(\T^d)$, 
$$
\begin{array}{l}
\ds \left| U(m')-U(m) -  \inte V (m, x)d(m'-m)(x) \right| \leq  \dk(m, m') w (\dk(m, m')) ,
\end{array}
$$
which shows that $U$ is ${\mathcal C}^1$ in the sense of Definition \ref{def:Diff}  with $\frac{\delta U}{\delta m}=V$. 
\end{proof}

\subsubsection{From differentiability in $m$ to differentiability along random variables}

We now discuss the converse
to Proposition \ref{prop:appendix:lions->def} 

\begin{Proposition}
\label{prop:appendix:def->lions}
Assume that $U$ satisfies the assumption of
Definition 
\ref{def:intrinsic:derivative}. 
Then, $U$ 
satisfies the differentiability property
\eqref{eq:derivative:torus}.
Moreover, 
$D_{m} U(m,y) = \partial_{\mu} U(m)(y)$, 
$m \in {\mathcal P}(\T^d)$ and $y \in \T^d$.
\end{Proposition}

\begin{proof}
We are given two random variables $X$ and $Y$
with values in the torus $\T^d$. By Definition 
\ref{def:Diff},
\begin{equation*}
\begin{split}
&U({\mathcal L}(Y) ) - U({\mathcal L}(X) )
\\
&= 
\int_{0}^1
\biggl[
\int_{\T^d}
\frac{\delta U}{\delta m}
\bigl( \lambda {\mathcal L}(Y)  + (1-\lambda) {\mathcal L}(X) ,y
\bigr) d \bigl( {\mathcal L}(Y)  - {\mathcal L}(X) \bigr)(y)
\biggr] d\lambda
\\
&= \int_{0}^1 {\mathbb E}
\biggl[ 
 \frac{\delta U}{\delta m}
\bigl( \lambda {\mathcal L}(Y)  + (1-\lambda) {\mathcal L}(X) ,Y
\bigr)
- 
\frac{\delta U}{\delta m}
\bigl( \lambda {\mathcal L}(Y)  + (1-\lambda) {\mathcal L}(X) ,X
\bigr)
\biggr] d\lambda 
\\
&= \int_{0}^1 \int_{0}^1
{\mathbb E}
\biggl[ 
D_{y} \frac{\delta U}{\delta m}
\bigl( \lambda {\mathcal L}(Y)  + (1-\lambda) {\mathcal L}(X) 
\bigr)\bigl( \lambda' \bar Y + (1-\lambda') \bar X \bigr)
(\bar Y- \bar X)
\biggr] d\lambda d \lambda',  
\end{split}
\end{equation*}
where $\bar X$ and $\bar Y$ are $\R^d$-valued random variables
that represent the $\T^d$-valued random variables $X$ and $Y$, while $D_{y} [\delta U/\delta m](m,\cdot)$ is seen as a periodic function from $\R^d$ into $\R^{d \times d}$.

By uniform continuity of $D_{m} U = D_{y}[\delta U/\delta m]$ on the compact set ${\mathcal P}(\T^d) 
\times \T^d$, we deduce that, 
\begin{equation}
\label{eq:appendix:def->lions}
\begin{split}
&U\bigl({\mathcal L}(Y) \bigr) - U\big({\mathcal L}(X) 
\bigr) 
\\
&\hspace{15pt} = 
{\mathbb E}
\biggl[ 
D_{y} \frac{\delta U}{\delta m}
\bigl(
{\mathcal L}(X) 
\bigr)( \bar X )
\bigl(
\bar Y-
\bar X
\bigr)
\biggr] +  
{\mathbb E}[\vert\bar X- \bar Y \vert^2]^{1/2}
w 
\bigl( {\mathbb E}[\vert\bar X- \bar Y \vert^2]^{1/2}
\bigr),
\end{split}
\end{equation}
for a function $w : \R_{+} \rightarrow \R_{+}$
that tends to $0$ in $0$ ($w$ being independent of $X$ and $Y$).  
Above, we used the fact that $\dk({\mathcal L}(X),{\mathcal L}(Y)) \leq \E[\vert \bar X - \bar Y\vert^2]^{1/2}$.

Let now $Z_{\lambda} = \tau_{\lambda (\bar{Y}-\bar{X})}(X)$, 
for $\lambda \in [0,1]$,
so that $Z_{\lambda + \varepsilon}
= \tau_{\varepsilon (\bar{Y}-\bar{X})}(Z_{\lambda})$,
for $0 \leq \lambda \leq \lambda + \varepsilon \leq 1$.
Then, $(\lambda+\varepsilon) \bar{Y} + 
[1-(\lambda+\varepsilon)] \bar{X}$ and 
$\lambda \bar{Y} + (1-\lambda) \bar{X}$
are representatives of $Z_{\lambda+\varepsilon}$
and $Z_{\lambda}$ and
the distance between both reads
\begin{equation*}
\bigl\vert (\lambda+\varepsilon) \bar{Y} + 
[1-(\lambda+\varepsilon)] \bar{X}
-
\lambda \bar{Y} + (1-\lambda) \bar{X}
\bigr\vert=
\varepsilon \vert \bar{Y} - \bar{X} \vert.
\end{equation*}
Therefore, by 
\eqref{eq:appendix:def->lions},
\begin{equation*}
\frac{d}{d\lambda}
U(Z_{\lambda})
= {\mathbb E}
\biggl[ 
D_{y} \frac{\delta U}{\delta m}
\bigl(
{\mathcal L}(Z_{\lambda}) 
\bigr)( Z_{\lambda} )
\bigl(
\bar Y-
\bar X
\bigr)
\biggr], \quad \lambda \in [0,1]. 
\end{equation*}
Integrating with respect to $\lambda \in [0,1]$, 
we get \eqref{eq:derivative:torus}.
\end{proof}

\subsection{Technical remarks on derivatives}

Here we collect several results related with the notion of derivative defined in Definition 
\ref{def:Diff}. 

The first one is a quantified version of Proposition \ref{prop:DmUderiv}. 

\begin{Proposition}\label{prop:diffdeltaDm} Assume that $U: \T^d\times \Pk\to \R$ is  $\cC^1$, that, for some $n\in \N$,  
$U(\cdot, m)$ and $\ds \frac{\delta U}{\delta m}(\cdot,m,\cdot)$ are in $\cC^{n+\alpha}$ and in $\cC^{n+\alpha}\times \cC^{2}$ respectively, and that there exists a constant $C_n$ such that, for any $m,m'\in \Pk$,  
\begin{equation}
\label{U-U-infinity}
\left\|\frac{\delta U}{\delta m}(\cdot,m,\cdot)\right\|_{(n+\alpha,2)} \leq C_n,
\end{equation}
and 
\be\label{U-U-inte}
\begin{array}{l}
\ds \left\|U(\cdot,m')-U(\cdot,m)- \inte \frac{\delta U}{\delta m}(\cdot,m,y)d(m'-m)(y) \right\|_{n+\alpha}
\leq C_n\dk^2(m, m').
\end{array}
\ee
Fix $m\in \Pk$ and let $\phi\in L^2(m,\R^d)$ be a vector field. Then 
\begin{equation}
\label{eq:U-Uresult}
\left\|U \bigl(\cdot,(id+\phi)\sharp m \bigr)-U(\cdot,m)- \inte  D_mU(\cdot,m, y)\cdot \phi(y)\ dm(y) \right\|_{n+\alpha}\leq (C_n+1)\|\phi\|^2_{L^2(m)}
\end{equation}
\end{Proposition}

Below, we give conditions that ensure that \eqref{U-U-inte} holds true. 

\begin{proof} 
Using \eqref{U-U-inte} we obtain
\begin{equation}
\label{estijhbvkvfd}
\begin{split}
&\left\|U\bigl(\cdot,(id+\phi)\sharp m\bigr)-U(\cdot,m)- 
\inte \frac{\delta U}{\delta m}(\cdot,m,y)d
\bigl((id+\phi)\sharp m-m\bigr)(y) \right\|_{n+\alpha}
\\
&\hspace{15pt} \leq C_n\dk^2\bigl(m, (id+\phi)\sharp m\bigr)
\; \leq \; C_n\|\phi\|_{L^{2}(m)}^{2} .
\end{split}
\end{equation}
Using the regularity of $\frac{\delta U}{\delta m}$, we obtain, for
an $\{1,\cdots,d\}$-valued tuple $\ell$ of length $|\ell|\leq n$ and
for any $x\in \T^d$, 
(omitting the dependence with respect to $m$ for simplicity): 
\begin{equation*}
\begin{split}
&\inte D^\ell_x \frac{\delta U}{\delta m}(x,y)
d\bigl\{(id+\phi)\sharp m\bigr\}(y)- \inte  D^\ell_x\frac{\delta U}{\delta m}(x,y)dm(y)-\inte   D^\ell_xD_m U(x,y)\cdot 
\phi(y)\  dm(y)
\\
&\qquad  =  \inte  \left(D^\ell_x\frac{\delta U}{\delta m}
\bigl(x,y+\phi(y)\bigr)- D^\ell_x\frac{\delta U}{\delta m}(x,y)-  D^\ell_xD_m U(x,y)\right)\cdot \phi(y)\  dm(y)
\\
&\qquad =  \ds  \int_0^1\inte  \left(D^\ell_xD_y\frac{\delta U}{\delta m}\bigl(x,y+s\phi(y)\bigr)- D^\ell_xD_m U(x, y)\right)\cdot \phi(y)\  dm(y)ds 
\\ 
&\qquad  =  \ds  \int_0^1\int_0^1 \inte s  D^\ell_xD_yD_m U\bigl(x,
y + st \phi(y)\bigr)\phi(y)\cdot \phi(y)\  dm(y)\ dsdt \;  \leq \; \ds C_n\|\phi\|_{L^{2}(m)}^{2}, 
\end{split}
\end{equation*}
where we used \eqref{U-U-infinity} in the last line. 

Coming back to \eqref{estijhbvkvfd}, this shows that 
\begin{equation*}
\begin{split}
&\left\|D^\ell U\Bigl(\cdot,(id+\phi)\sharp m\bigr)-D^\ell U(\cdot,m)- \inte   D^\ell_xD_m U(\cdot,y)\cdot 
\phi(y)\  dm(y) \right\|_{\infty}
 \leq \; C_n\|\phi\|_{L^{2}(m)}^{2},
\end{split}
\end{equation*}
which proves 
\eqref{eq:U-Uresult} but with $\alpha=0$. 

The proof of the H\"older estimate goes along the same line: if $x,x'\in \T^d$, then  
\begin{equation*}
\begin{split}
&\inte D^\ell_x \frac{\delta U}{\delta m}(x,y)
d
\bigl\{
(id+\phi)\sharp m
\bigr\}
(y)- \inte  D^\ell_x\frac{\delta U}{\delta m}(x,y)dm(y)-\inte   D^\ell_xD_m U(x,y)\cdot \phi(y)\  dm(y)
\\
&\hspace{15pt} - \biggl( \inte  D^\ell_x\frac{\delta U}{\delta m}(x',y)d
\bigl\{
(id+\phi)\sharp m
\bigr\}
(y)- \inte  D^\ell_x\frac{\delta U}{\delta m}(x',y)dm(y)
\\
&\hspace{250pt} -\inte   D^\ell_x D_m U(x',y)\cdot \phi(y)\  dm(y)\biggr)
\\
&=  \inte  \left(D^\ell_x\frac{\delta U}{\delta m}\bigl(x,y+\phi(y)
\bigr)- D^\ell_x\frac{\delta U}{\delta m}(x,y)-  D^\ell_xD_m U(x,y)\right)\cdot \phi(y)\  dm(y)
\\
&\hspace{15pt}
-\inte \left( D^\ell_x\frac{\delta U}{\delta m}\bigl(x',y+\phi(y)\bigr)- D^\ell_x\frac{\delta U}{\delta m}(x',y)-  D^{\ell}_x
D_m U(x',y)\right) \cdot \phi(y)\  dm(y)
\\
&=  \int_0^1\inte  \left(D^\ell_xD_y\frac{\delta U}{\delta m}(x,y+s\phi(y))- D^\ell_xD_m U(x, y)\right)\cdot \phi(y)\  dm(y)ds 
\\ 
&\hspace{15pt}  - \int_0^1 \inte \left( D^\ell_xD_y\frac{\delta U}{\delta m}(x',y+s\phi(y))- D^\ell_xD_m U(x', y)\right)\cdot \phi(y) \  dm(y)ds 
\\ 
&=  \ds  \int_0^1\int_0^1 \inte s  
\Bigl(D^\ell_xD_yD_m U
\bigl(x, y+ st \phi(y)\bigr)
\\
&\hspace{150pt}-D^\ell_xD_yD_m U
\bigl(x',  y+st \phi(y)\bigr)
\Bigr)\phi(y)\cdot \phi(y)\  dm(y)\ dsdt 
\\ 
&\leq \ds C_n|x-x'|^{\alpha}\|\phi\|_{L^{2}(m)}^{2}.
\end{split}
\end{equation*}
This shows that 
$$
\left\| \inte \frac{\delta U}{\delta m}(\cdot,m,y)d
\Bigl[ \bigl\{(id+\phi)\sharp m
\bigr\}
-m \Bigr](y) 
-\inte  D_mU(\cdot,m,y)\cdot \phi(y)\ dm(y)  \right\|_{n+\alpha} \leq C_n \|\phi\|_{L^{2}(m)}^{2} .
$$
Plugging this inequality into \eqref{estijhbvkvfd} 
shows the result. 
\end{proof}

We now give conditions under which \eqref{U-U-inte} holds. 

\begin{Proposition} Assume that $U: \T^d\times \Pk\to \R$ is  $\cC^1$ and that, for some $n\in \N^*$, 
$$
\left\|\frac{\delta U}{\delta m}(\cdot, m, \cdot)\right\|_{(n+\alpha,n+\alpha)} + {\rm Lip}_n\left(\frac{\delta U}{\delta m}\right) \;  \leq \; C_n.
$$
Then, for any $m,m'\in \Pk$, we have 
$$
\begin{array}{l}
\ds \left\|U(\cdot,m')-U(\cdot,m)- \inte \frac{\delta U}{\delta m}(\cdot,m,y)d(m'-m)(y) \right\|_{n+\alpha}
\leq C_n\dk^2(m, m').
\end{array}
$$
\end{Proposition}

\begin{proof} We only show the Holder regularity: the $L^\infty$ estimates go along the same line and are simpler. 
For any $\ell \in \N^d$ with $|\ell|\leq n$ and any $x,x'\in \T^d$, we have 
\begin{equation*}
\begin{split}
&\left| D^\ell_x U(x,m')-D^\ell_xU(x,m)- \inte D^\ell_x \frac{\delta U}{\delta m}(x,m,y)d(m'-m)(y)\right.
\\
&\hspace{30pt}
 \left. - \biggl(D^\ell_x U(x',m')-D^\ell_xU(x',m)- \inte D^\ell_x\frac{\delta U}{\delta m}(x',m,y)d(m'-m)(y)
 \biggr)
 \right|
 \\
&\leq \int_0^1\left| \inte \left(D^\ell_x\frac{\delta U}{\delta m}
\bigl(x, (1-s)m+sm',y\bigr)- D^\ell_x \frac{\delta U}{\delta m}(x,m,y) \right.\right.
\\
&\hspace{30pt} -\left.\left. \Bigl[ D^\ell_x\frac{\delta U}{\delta m}\bigl(x', (1-s)m+sm',y\bigr)- D^\ell_x \frac{\delta U}{\delta m}(x',m,y) 
\Bigr]
\right)d(m'-m)(y)\right|ds 
\\
&\leq \ds \sup_{s,y} \left| D_yD^\ell_x\frac{\delta U}{\delta m}
\bigl(x, (1-s)m+sm',y\bigr)- D_yD^\ell_x \frac{\delta U}{\delta m}(x,m,y) \right.
\\
&\hspace{30pt}
-\left. \Bigl[ D_yD^\ell_x\frac{\delta U}{\delta m}
\bigl(x', (1-s)m+sm',y\bigr)- D_yD^\ell_x \frac{\delta U}{\delta m}(x',m,y) 
\Bigr]
\right|\dk(m,m') 
\\
&\leq \ds {\rm Lip}_n\left(\frac{\delta U}{\delta m}\right) |x-x'|^{\alpha} \dk^2(m,m').
\end{split}
\end{equation*}
This proves our claim. 
\end{proof}

\begin{Proposition}\label{prop:diffdeltaDm2} Assume that $U:\Pk\to \R$ is $\cC^2$ with, for any $m,m'\in \Pk$,  
\begin{equation}
\label{e.condprop:diffdeltaDm2}
\begin{split}
&\left| U(m')-U(m)-\inte \frac{\delta U}{\delta m}(m,y)d(m'-m)(y) \right.
\\
&\hspace{15pt} \left.-\frac12 \inte\inte  \frac{\delta^2 U}{\delta m^2}(m,y,y')d(m'-m)(y) d(m'-m)(y')\right|\; \leq \; \dk^2(m,m') w
\bigl(\dk(m,m')\bigr),
\end{split}
\end{equation}
where $w(t)\to 0$ as $t\to0$, 
and that
$$
\left\| \frac{\delta U}{\delta m}(m, \cdot)\right\|_{3}+\left\|\frac{\delta^2 U}{\delta m^2}(m, \cdot,\cdot) \right\|_{(2,2)}\leq C_0.
$$
Then, for any $m\in \Pk$ and any vector field $\phi\in L^3(m,\R^d)$, we have 
\begin{equation*}
\begin{split}
&\biggl| U\bigl((id+\phi)\sharp m\bigr)-U(m)-\inte  D_mU(m,y)\cdot \phi(y)\ dm(y) 
\\
&\hspace{45pt} -\frac12\inte  D_yD_mU(m,y)\phi(y)\cdot\phi(y)\ dm(y)
\\
&\hspace{45pt}
- \frac12 \inte\inte   D^2_{mm}U(m,y,y')\phi(y)\cdot\phi(y')\ dm(y) dm(y')\biggr|
\\
&\leq  \|\phi\|_{L^3_m}^2\tilde w(\|\phi\|_{L^3_m}),
\end{split}
\end{equation*}
where the modulus $\tilde w$ depends on $w$ and on $C_0$. 
\end{Proposition}

\begin{proof} We argue as in Proposition \ref{prop:diffdeltaDm}: by our assumption, we have 
\begin{equation*}
\begin{split}
&\left| U \bigl((id+\phi)\sharp m\bigr)-U(m)
-\inte \frac{\delta U}{\delta m}(m,y)d\Bigl(
\bigl\{
(id+\phi)\sharp m
\bigr\}
-m\Bigr)(y) \right.
\\
&\hspace{30pt} \left.-\frac12 \inte\inte  \frac{\delta^2 U}{\delta m^2}(m,y,z)d\Bigl(\bigl\{(id+\phi)\sharp m\bigr\}-m\Bigr)(y) d\Bigl( \bigl\{ (id+\phi)\sharp m \bigr\}-m\Bigr)(z)\right|
\\
&\leq \dk^2\bigl(m,(id+\phi)\sharp m\bigr)
w\bigl(\dk(m,(id+\phi)\sharp m)\bigr)  \leq \|\phi\|_{L^3_m}^2w(\|\phi\|_{L^3_m}).
\end{split}
\end{equation*}
Now 
\begin{equation*}
\begin{split}
&\inte \frac{\delta U}{\delta m}(m,y)d
\Bigl[ \bigl\{ (id+\phi)\sharp m
\bigr\}
-m \Bigr](y)
\\
&= \inte \left(\frac{\delta U}{\delta m}(m,y+\phi(y))-\frac{\delta U}{\delta m}(m,y) \right)dm(y) 
\\
&= \inte \left( D_y\frac{\delta U}{\delta m}(m,y)\cdot \phi(y) +\frac12  D^2_y\frac{\delta U}{\delta m}(m,y)\phi(y)\cdot\phi(y)+O(|\phi(y)|^3)\right)dm(y) 
\\
&= \inte \left( D_mU(m,y)\cdot \phi(y) +\frac12  D_yD_mU(m,y)\phi(y)\cdot\phi(y) +O(|\phi(y)|^3)\right)dm(y),
\end{split}
\end{equation*}
where 
$$
\inte \left|O(|\phi(y)|^3)\right|\ dm(y) \leq \left\| D^2_yD_mU\right\|_\infty  \inte |\phi(y)|^3dm(y)\leq C_0  \|\phi\|_{L^3_m}^3.
$$
Moreover, 
\begin{equation*}
\begin{split}
&\inte\inte  \frac{\delta^2 U}{\delta m^2}(m,y,z)
d\Bigl[ \bigl\{ 
(id+\phi)\sharp m
\bigr\}
-m
\Bigr]
(y) d\Bigl[ \bigl\{ (id+\phi)\sharp m
\bigr\}
-m
\Bigr]
(z) 
\\
&= \inte\inte  \left(\frac{\delta^2 U}{\delta m^2}
\bigl(m,y+\phi(y),z+\phi(z)\bigr)-\frac{\delta^2 U}{\delta m^2}\bigl(m,y+\phi(y),z\bigr)
- \frac{\delta^2 U}{\delta m^2}(m,y,z+\phi(z))\right.
\\
&\hspace{300pt}
 \left.
+ \frac{\delta^2 U}{\delta m^2}(m,y,z)\right) dm(y) dm(z) 
\\
&= \inte\inte\left(  D^2_{y,z} \frac{\delta^2 U}{\delta m^2} (m, y,z)\phi(y)\cdot \phi(z) + O\bigl(|\phi(y)|^2|\phi(z)|+|\phi(y)||\phi(z)|^2 \bigr) \right)dm(y)dm(z) 
\\
&= \inte\inte\biggl(  D^2_{mm}U(m, y,z)\phi(y)\cdot\phi(z) + O\bigl(|\phi(y)|^2|\phi(z)|+|\phi(y)||\phi(z)|^2\bigr) \biggr)dm(y)dm(z),
\end{split}
\end{equation*}
where 
\begin{equation*}
\begin{split}
&\inte\Bigl|O \Bigl( \vert \phi(y)|^2|\phi(z)|+|\phi(y)||\phi(z)|^2\Bigr)\Bigr| dm(y)dm(z) 
\\
&\hspace{15pt} \leq \sup_m \| D^2_{mm}U(m, \cdot, \cdot)\|_{(\cC^1)^2}\|\phi\|^3_{L^3_m} 
\leq C_0\|\phi\|^3_{L^3_m}. 
\end{split}
\end{equation*}
Putting the above estimates together gives the result.  
\end{proof}

We complete the section by giving conditions under which inequality \eqref{e.condprop:diffdeltaDm2} holds: 

\begin{Proposition}\label{prop:diffdeltaDm2bis} Assume that the mapping ${\mathcal P}(\T^d) \ni m\mapsto \frac{\delta^2 U}{\delta m^2}(m, \cdot,\cdot)$ is continuous from $\Pk$ into $(\cC^2(\T^d))^2$ with a modulus $w$. Then \eqref{e.condprop:diffdeltaDm2} holds. 
\end{Proposition}

\begin{proof} We have 
\begin{equation*}
\begin{split}
U(m')-U(m) &= \int_0^1\inte \frac{\delta U}{\delta m}
\bigl((1-s)m+sm',y\bigr)d(m'-m)(y)
\\
&= \inte \frac{\delta U}{\delta m}(m,y)d(m'-m)(y) 
\\ 
&\hspace{15pt} + \int_0^1\int_0^1 \inte s \frac{\delta^2 U}{\delta m^2}\bigl((1-s\tau)m+s\tau m',y,y'\bigr)d(m'-m)(y)d(m'-m)(y').
\end{split}
\end{equation*}
Hence 
\begin{equation*}
\begin{split}
&\left| 
U(m')-U(m) - \inte \frac{\delta U}{\delta m}(m,y)d(m'-m)(y) \right.
\\
&\hspace{15pt} \left. - \frac12 \inte\inte  \frac{\delta^2 U}{\delta m^2}(m,y,y')d(m'-m)(y)d(m'-m)(y')\right| 
\\
&\leq  \dk(m,m')^2
 \int_0^1\int_0^1 s 
\left\|  D^2_{yy'}\frac{\delta^2 U}{\delta m^2}
\bigl((1-s\tau)m+s\tau m',\cdot,\cdot\bigr) - D^2_{yy'}\frac{\delta^2 U}{\delta m^2}(m,\cdot,\cdot)\right\|_\infty d\tau ds 
\\
&\leq   \dk(m,m')^2 w\bigl( \dk(m,m')^2\bigr).
\end{split}
\end{equation*}

\end{proof}

{\bf Acknowledgement:} The first author was partially supported by the ANR (Agence Nationale de la Recherche) projects  ANR-12-BS01-0008-01 and ANR-14-ACHN-0030-01.



\end{document}